%% file: main.tex
\documentclass[11pt,reqno]{article}

\usepackage[T1]{fontenc}
\usepackage[utf8]{inputenc}

\usepackage{amsfonts}
\usepackage{amscd}
\usepackage{color}
\usepackage{amsmath}
\usepackage{float}
\usepackage{amsthm}
\usepackage{amssymb}

\usepackage{makeidx}
\makeindex

\usepackage{mathrsfs}
\usepackage{amstext}
\usepackage{graphics}
\usepackage{tikz}

\usepackage{color}

\usetikzlibrary{mindmap,trees}
\usetikzlibrary{shapes,arrows}
\usepackage{pgfplots}
\usepgflibrary{plotmarks}

\tikzset{
 basic/.style = {draw, rectangle, rounded corners=6pt, align=center},
 root/.style = {basic, text width=6cm},
 rect1/.style = {basic, font=\small},
 ellip/.style = {ellipse, fill=blue!20, minimum width=6cm}
}

\usepackage{verbatim}
\usepackage{enumerate}

\numberwithin{equation}{section}

\theoremstyle{plain}
\newtheorem{thm}{Theorem}[section]
\newtheorem{defi}[thm]{Definition}
\newtheorem{cor}[thm]{Corollary}
\newtheorem{lem}[thm]{Lemma}
\newtheorem{prop}[thm]{Proposition}
\newtheorem{rem}[thm]{Remark}
\newtheorem{conj}[thm]{Conjecture}

\newcommand{\I}{\ensuremath{\mathbb{I}}}
\newcommand{\R}{\ensuremath{\mathbb{R}}}
\newcommand{\bR}{\ensuremath{\mathbb{R}}}
\newcommand{\N}{\ensuremath{\mathbb{N}}}
\newcommand{\zz}{\ensuremath{\mathbb{Z}}}
\newcommand{\C}{\ensuremath{\mathbb{C}}}
\newcommand{\T}{\ensuremath{\mathbb{T}}}
\newcommand{\Z}{\ensuremath{\mathbb{Z}}}
\newcommand{\n}{\ensuremath{{\N}_0}}

\newcommand\dint{{\rm d}}

\newcommand{\cf}{\ensuremath{\mathcal F}}

\newcommand{\cU}{\ensuremath{\mathcal U}}

\def\DD{\mathbb D}
\def\OO{\mathbb O}

\def\bw{\mathbf w}
\def\D{{\mathcal D}}
\def\Di{{\mathcal D}}
\def\di{\mathbb D}
\def\og{\mathbb O}

\def \Tr{\mathcal T}

\def \V{\mathcal V}
\def\RR{\mathcal R}
\def \<{\langle}
\def\>{\rangle}
\def \L{\Lambda}
\def\La{\Lambda}
\def\Ga{\Gamma}
\def\Th{{\Theta}}
\def \e{\epsilon}
\def \ep{\epsilon}

\def \al{\alpha}

\def \ff{\varphi}
\def \fk{\hat f(\mathbf k)}
\def\ba{\mathbf a}
\def\bk{\mathbf k}
\def\bn{\mathbf n}
\def\bm{\mathbf m}
\def\bs{\mathbf s}
\def\bx{\mathbf x}
\def\by{\mathbf y}
\def\bN{\mathbf N}
\def\bF{\mathbf F}
\def\bE{\mathbf E}
\def\bp{\mathbf p}
\def\bb{\mathbf b}
\def\bH{\mathbf H}

\newcommand{\Span}{{\rm span \, }}
\newcommand{\rank}{{\rm rank \, }}
\newcommand{\supp}{{\rm supp \, }}

\newcommand{\eps}{\varepsilon}

\newcommand{\bproof}{\begin{proof}}
\newcommand{\eproof}{\end{proof}}

\newlength{\fixboxwidth}
\newcommand{\be}{\begin{equation}}
\newcommand{\ee}{\end{equation}}

\newcommand{\bh}{{\bf h}}
\newcommand{\br}{{\bf r}}
\newcommand{\bt}{{\bf t}}

\newcommand{\bj}{{\bf j}}
\newcommand{\bW}{\ensuremath{\mathbf{W}}}
\newcommand{\bB}{\ensuremath{\mathbf{B}}}

\newcommand{\Wrp}{\ensuremath{\mathbf{W}^r_p}}
\newcommand{\Brpt}{\ensuremath{\mathbf{B}^r_{p,\theta}}}
\newcommand{\Hrp}{\ensuremath{\mathbf{H}^r_p}}
\def\Wr2{{\bf W}^r_2}
\def\Wt2{W^t_2}

\def\bBr{{\bf B}^r_{p,\theta}}

\def\F{\mathcal{F}}
\newcommand{\wt}{\widetilde}
\newcommand{\X}{\mathbb{X}}
\newcommand{\Wo}{\ensuremath{\mathring{\mathbf W}_p^r}}

\begin{document}

\title{Hyperbolic Cross Approximation}

\author{Dinh D{\~u}ng$^a$, Vladimir  Temlyakov$^b$\footnote{Corresponding
author, Email: temlyakovv@gmail.com}, Tino Ullrich$^c$\\\\
$^a$Information Technology Institute, Vietnam National University, Hanoi, Vietnam\\
$^b$Department of Mathematics, University of South Carolina, 29208 Columbia, USA\\
$^c$Institute for Numerical Simulation, University of Bonn, 53115 Bonn, Germany}

\date{\today}

\maketitle

\input{abstract}

\newpage

\tableofcontents
  \newpage
    \input{introduction}
    \newpage
    \input{trig_polynomials}
    \newpage
    \input{function_spaces}

   \newpage
   \input{linear_approx}
    \newpage
    \input{sampl_recovery}
       \newpage
        \input{entropy_numbers}
     \newpage
   \input{nonlinear_approx}

   \newpage
    \input{numerical_integr}

    \newpage
   \input{related_problems}

    \newpage
    \input{highdim}
   \newpage
   \input{appendix}
 \newpage

\input{bibliography}

\newpage
\listoffigures

\printindex

\end{document}

%% file: abstract.tex
\begin{abstract} 
  Hyperbolic cross approximation is a special type of multivariate approximation. Recently, driven by applications in
engineering, biology, medicine and other areas of science new challenging problems have appeared. The common feature of
these problems is high dimensions. We present here a survey on classical methods developed in multivariate approximation
theory, which are known to work very well for moderate dimensions and which have potential for applications in really
high dimensions. 
The theory of hyperbolic cross approximation and related theory of functions with mixed smoothness are under detailed
study for more than 50 years. It is now well understood that this theory is important both for theoretical study and for
practical applications. It is also understood that both theoretical analysis and construction of practical algorithms
are very difficult problems. This explains why many fundamental problems in this area are still unsolved. Only a few
survey papers and monographs on the topic are published. This and recently discovered deep connections between the
hyperbolic cross approximation (and related sparse grids) and other areas of mathematics such as probability,
discrepancy, and numerical integration motivated us to write this survey. 
We try to put emphases on the development of ideas and methods rather than list all the known results in the area. We
formulate many problems, which, to our knowledge, 
are open problems. We also include some very recent results on the topic, which sometimes highlight new interesting
directions of research. 
We hope that this survey will stimulate further active research in this fascinating and challenging area of
approximation theory and numerical analysis. 
\end{abstract}

%% file: introduction.tex
\newpage 
\section{Introduction}
This book is a survey on multivariate approximation. The 20th century was a period of transition from univariate
problems to multivariate problems in a number of areas of mathematics. For instance, it is a step from Gaussian sums to
Weil's sums in number theory,
a step from ordinary differential equations to PDEs, a step from univariate
trigonometric series to multivariate trigonometric series in harmonic analysis,
a step from quadrature formulas to cubature formulas in numerical integration, a
step from univariate function classes to multivariate function classes in
approximation theory. 
In many cases this step brought not only new phenomena but also required new
techniques to handle the corresponding multivariate problems. In some cases even
a formulation of a multivariate problem requires 
a nontrivial modification of a univariate problem. For instance, the problem of
convergence of the multivariate trigonometric series immediately encounters a
question of which partial sums we should consider -- there is 
no natural ordering in the multivariate case. In other words: What is a natural
multivariate analog of univariate trigonometric polynomials? Answering this
question mathematicians studied different generalizations of the univariate
trigonometric polynomials: with frequencies from a ball, a cube or, most importantly, a {\it hyperbolic
cross \index{Hyperbolic cross}}
$$
    \Gamma(N) = \Big\{\bk=(k_1,\dots,k_d)\in \Z^d~:~\prod\limits_{j=1}^d \max\{|k_j|,1\} \leq N\Big\}\,.
$$

\begin{figure}[H]
\center
\includegraphics[scale = 0.25]{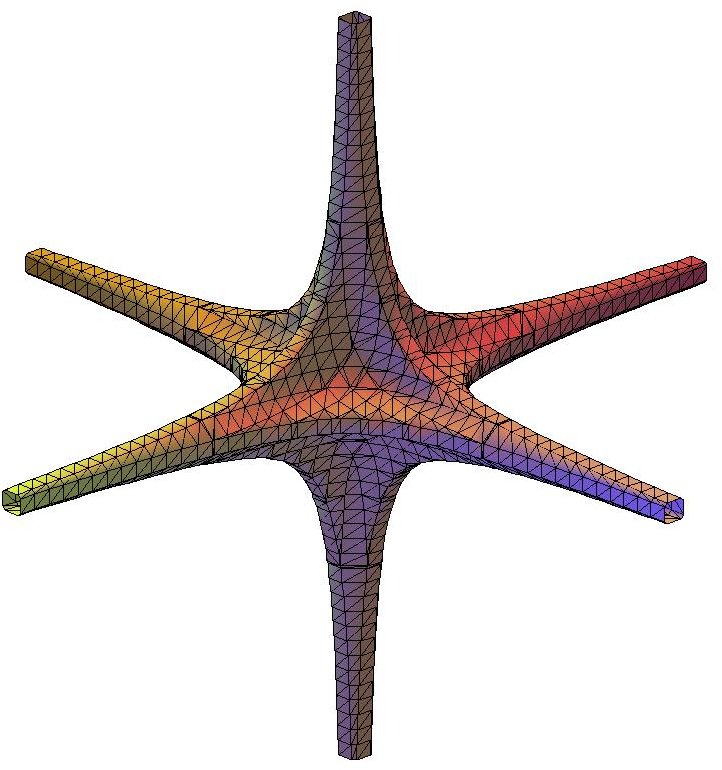}
\caption{A smooth hyperbolic cross in $\R^3$}
\end{figure}

Results discussed in this survey demonstrate that polynomials with frequencies from hyperbolic crosses $\Gamma(N)$ play
the same role in the multivariate approximation as the univariate trigonometric polynomials play in the approximation of
functions on a single variable. 
On a very simple example we show how the hyperbolic cross polynomials\index{Hyperbolic cross!Polynomials} appear naturally in the multivariate
approximation. Let us begin with the univariate case. The natural ordering of the univariate trigonometric system is
closely connected with the ordering of eigenvalues of the differential operator $D:=\frac{d}{dx}$ considered on $2\pi$
periodic functions. The eigenvalues are $\{\pm ik\}$ with $k=0,1,2,\dots$ and the corresponding eigenfunctions \index{Eigenfunction} 
are $\{e^{\pm
ikx}\}$. Nonzero eigenvalues \index{Eigenvalue} of the differential operator of mixed derivative $\prod_{j=1}^d D_j$ are
$\{\prod_{j=1}^d ik_j\,:\,\bk \in \Z^d\}$. Ordering these eigenvalues we immediately obtain the hyperbolic crosses
$\Gamma(N)$. This
simple observation shows that hyperbolic crosses 
are closely connected with the mixed derivative. Results obtained in the multivariate approximation theory for the last
50 years established a deep connection between trigonometric polynomials with frequencies from the hyperbolic crosses
and classes of functions defined with the help of either mixed derivatives or mixed differences. The importance of
these classes was understood in the beginning of 1960s. 

In the 1930s in
connection with applications in mathematical physics, S.L. Sobolev introduced the
classes  of functions by imposing the following restrictions
\be\label{tag1.13i}
\|f^{(n_1,\dots ,n_d)}\|_p  \le 1  \quad
\ee
for all ${\bf n} = (n_1,\dots ,n_d)$ such that $n_1 + \dots + n_d \le R$.
These classes  appeared as natural ways to measure smoothness in many
multivariate problems including numerical integration. It was established
that for Sobolev  classes the optimal error of numerical integration by
formulas with $m$ nodes is of order $m^{-R/d}$. On the other hand, at the end of 1950s,  N.M. Korobov discovered the
following phenomenon: Let us consider the class of functions which satisfy (\ref{tag1.13i}) for all ${\bf n}$ such that
$n_j \le r, \quad j = 1,\dots ,d$ (Korobov considered different, more general classes, but for illustration purposes it
is convenient for us to deal with these classes here). Obviously this new class (class of
functions with \index{Bounded mixed derivative} bounded mixed derivative) is much wider then the Sobolev class with
$R = rd$. For example, all functions of the form
$$
f(x) = \prod_{j=1}^d f_j(x_j) , \quad \|f^{(r)}_j\|_p \le 1,
$$
belong to this class, while not necessarily to the Sobolev class (it would
require, roughly, $\|f_j^{(rd)}\|_p \le 1$). Korobov constructed a cubature
formula with $m$ nodes which guaranteed the accuracy of numerical integration
for this class of order $m^{-r}(\log m)^{rd}$, i.e., almost the same accuracy that 
we had for the Sobolev class.   Korobov's discovery pointed out the importance
of  the classes of functions with bounded mixed derivative  in fields such
as approximation theory and numerical analysis. The simplest versions of Korobov's magic cubature formulas are the
Fibonacci cubature formulas \index{Cubature formula!Fibonacci}, see  Subsection \ref{Sect:Fib} below, given by 
$$     
\Phi_n(f) = b_n^{-1} \sum_{\mu=1}^{b_n}  f(\mu/b_n,\,\{\mu
b_{n-1}/b_n \}), 
$$
where $b_0=b_1=1$,  $\;  b_n=b_{n-1}+b_{n-2}$ are the Fibonacci numbers and
$\{x\}$ is the fractional part of the number $x$. These cubature formulas work optimally 
both for classes of functions with bounded mixed derivative and for classes with bounded mixed difference.
The reason for such an outstanding behavior is the fact that the Fibonacci  cubature
formulas are exact on the hyperbolic cross polynomials associated with $\Gamma(cb_n)$.

The multivariate problems of hyperbolic cross approximation turn out to be much more
involved than their univariate counterparts. For instance, the fundamental Bernstein inequalities
for the trigonometric polynomials are known in the univariate case with explicit
constants and they are not even known in the sense of order for the
trigonometric polynomials with frequencies from a hyperbolic cross (see Open problem 1.1 below). 

 We give a brief historical overview of challenges and open problems of
approximation theory with emphasis put on multivariate approximation. It was
understood in the beginning of the 20th century that smoothness properties of a
univariate function determine the rate of approximation of this function by
polynomials (trigonometric in the periodic case and algebraic in the
non-periodic case).  
A fundamental question is: What is a natural multivariate analog of univariate
smoothness classes? Different function classes were considered in the
multivariate case: isotropic and anisotropic Sobolev and Besov classes, classes
of functions with bounded mixed derivative and others. The simplest case of such a function class is the unit ball of
the mixed Sobolev space \index{Function spaces!Mixed smoothness Sobolev}of bivariate functions given by\index{Mixed smoothness}
$$
    \bW^r_p := \Big\{f\in L_p~:~\|f\|_{\bW^r_p} := \|f\|_p + \Big\|\frac{\partial^rf}{\partial
x_1^r}\Big\|_p+\Big\|\frac{\partial^rf}{\partial x_2^r}\Big\|_p + \Big\|\frac{\partial^{2r}f}{\partial
x_1^r\partial x_2^r}\Big\|_p\le 1\Big\}\,.
$$
These classes are sometimes denoted as classes of functions with dominating mixed derivative \index{Bounded mixed derivative} since the condition on the
mixed derivative is the dominating one. Babenko \cite{Bab2} was the first who introduced such classes and began to study approximation of these classes by the hyperbolic cross polynomials. In Section 3 we will define more general periodic $\bW$ and $\bH$, $\bB$ 
classes
of $d$-variate functions, also with fractional smoothness $r>0$. What concerns $\bH$ and $\bB$ classes we replace
the condition on the mixed derivative, used for the definition of $\bW$ classes, by a condition on a mixed difference. In Section 3 we give a historical comment on the further study of the mixed smoothness classes. 

The next fundamental
question is: How to approximate functions from these classes?   Kolmogorov
introduced the concept of the $n$-width of a function class. This concept is
very useful in answering the above question. The Kolmogorov $n$-width is a
solution to an optimization problem where we minimize the error of best approximation with respect 
to all $n$-dimensional linear subspaces. This concept allows us to understand which $n$-dimensional
linear subspace is the best for approximating a given class of functions.  The
rates of decay of the Kolmogorov $n$-width are known for the univariate
smoothness classes. In some cases even exact values of it are known. The problem
of the rates of decay of the Kolmogorov $n$-width for the classes of
multivariate functions with bounded mixed derivative is still not completely understood. 

We note that the function classes with bounded mixed derivative are not only an
interesting and challenging object for approximation theory. They also represent a suitable model in 
scientific computations. Bungartz and Griebel \cite{BuGr04, BG99, Gr05} and their groups 
use approximation methods designed for these classes in elliptic variational
problems. The recent work of Yserentant \cite{Y2010} on the regularity of eigenfunctions\index{Eigenfunction} of the electronic Schr\"odinger
operator, and Triebel \cite{Tr15} on the regularity of
solutions of Navier-Stokes equations, show that mixed regularity plays a fundamental role in mathematical physics. This
makes approximation
techniques developed for classes of functions with bounded mixed derivative a
proper choice for the numerical
treatment of those problems.      
 
Approximation of classes of functions with bounded mixed derivative (the $\bW$ classes) and 
functions with a restriction on mixed differences (the $\bH$ and, more
generally, $\bB$ classes) have been developed following classical tradition. A systematic
study of different asymptotic characteristics of these classes dates back to the
beginning of 1960s. 
Babenko \cite{Bab1} and Mityagin \cite{Mit} were the first who established the
nowadays well-known classical estimates for the Kolmogorov widths of the classes
$\bW^r_p$ in $L_p$ if $1<p<\infty$, namely
\begin{equation}\label{BabMit}
    d_m(\bW^r_p,L_p) \asymp \left(\frac{(\log m)^{d-1}}{m}\right)^r\quad,\quad m\in \N\,,
\end{equation}
where the constants behind $\asymp$ depend on $d$, $p$ and $r$. Later it turned out that the same order is present in
the
situation $d_m(\bW^r_p,L_q)$ if $1<q\le p<\infty$ and $2\le p\le q<\infty$. With \eqref{BabMit} it has been realized
that
the {\it hyperbolic cross polynomials}
play the same role in the approximation of multivariate functions as classical
univariate polynomials for the approximation of univariate functions. This
discovery resulted in a detailed study of the
hyperbolic cross polynomials. However, it turned out that the study of
properties of hyperbolic cross polynomials is much more difficult than 
the study of their univariate analogs. We discuss the corresponding results in
Section \ref{trigpol}.

In Sections 4 and 5 we consider linear approximation problems -- the problems of approximation by elements of a
given
finite dimensional linear subspace. We discussed a number of the most important asymptotic characteristics related to
the linear approximation. In addition to the Kolmogorov $n$-width we also study the asymptotic behavior of
linear widths and orthowidths. Interesting effects occur when studying the approximation of the class $\bW^r_p$ in $L_q$
if
$p<q$. In contrast to \eqref{BabMit} the influence of the parameters $p$ and $q$ is always visible in the rate of the
order of the linear widths. In fact, if either $1<p\le q\le 2$ or $2\le p\le q<\infty$ then we have
\begin{equation}\nonumber
    \lambda_m(\bW^r_p,L_q) \asymp \left(\frac{(\log m)^{d-1}}{m}\right)^{r-1/p+1/q}\quad,\quad m\in \N\,.
\end{equation}
Then the optimal approximant is realized by a projection on an appropriate linear subspace of the hyperbolic cross
polynomials,
which is not always the case, like for instance in the case $p<2<q$. 
In the definition of linear width we allow all linear operators of rank $m$ to compete in the minimization problem. 
Clearly, we would like to work with nice and simple linear operators. 
This idea motivated researchers to impose additional restrictions 
on the linear operators in the definitions of the corresponding modifications of the linear width. 
One very natural restriction is that the approximating rank $m$ operator is an orthogonal projection operator. 
This leads to the concept of the orthowidth (Fourier width) $\varphi_m(\bF,X)$. It turns out that 
the behavior of the $\varphi_m(\bW^r_p,L_q)$ is different from the 
behavior of the  $\lambda_m(\bW^r_p,L_q)$. For instance, 
it was proved that the operators $S_{Q_n}$ of orthogonal projection onto subspaces $\Tr(Q_n)$ of the 
hyperbolic cross polynomials are optimal (in the sense of order) from the point of view of the 
orthowidth for all $1\le p,q\le\infty$ except $p=q=1$ and $p=q=\infty$. 
That proof required new nontrivial methods for establishing the 
right lower bounds. 

Another natural restriction is that the approximating rank $m$ operator is a recovering operator, which uses function values at $m$ points.
Restricting the set of admissible
rank $m$ operators to such that are based on function evaluations (instead of general linear functionals), we
observe a behavior which is clearly bounded below by $\lambda_m$. We call the corresponding asymptotic
quantities sampling widths $\varrho_m$. However, this is not the end of the story. Already in the situation $\bW^r_p$ in
$L_q$ we are able to determine sets of parameters where $\varrho_m$ is equal to $\lambda_m$ in the sense of order, and
others where $\varrho_m$ behaves strictly worse (already in the main rate). However, the complete picture is still
unknown. In that sense the situation $p=q$ (including $p=2$) is of particular interest. The
result 
$$
      \varrho_m(\bW^r_p,L_p) \lesssim \left(\frac{(\log m)^{(d-1)}}{m}\right)^r (\log m)^{(d-1)/2}\quad,\quad m\in \N\,,
$$
has been a breakthrough since it improved on a standard upper bound by using a non-trivial technique. However, the
exact order is still unknown even in case $p=2$. The so far best-known upper bounds for sampling recovery are all based 
on sparse grid \index{Sparse grid} constructions. Alike the orthowidth results, the optimal in the sense of order subspaces for recovering are 
subspaces $\Tr(Q_n)$ with appropriate $n$ (in all cases, where we know the order of  $\varrho_m(\bW^r_p,L_q)$).

In Section 6 we discuss a very important characteristic of a compact -- its entropy numbers. It quantitatively
determines its ``degree of compactness''. Contrary to the
asymptotic characteristics discussed in Sections 4 and 5 the entropy numbers are not directly connected to the linear
theory of approximation. However, there are very useful inequalities between the entropy numbers and other asymptotic
characteristics, which provide a powerful method of proving good lower bounds, say, for the Kolmogorov widths. In this
case Carl's inequality is used.
Two more points, which motivated us to discuss the entropy numbers of classes of functions with bounded mixed derivative
are the following. (A) The problem of the rate of decay of $\e_n(\bW^r_{p},L_q)$ in a
particular case $r=1$, $p=2$, $q=\infty$ is equivalent  to a
fundamental problem of probability theory (the Small Ball Problem, see Subsection \ref{sect:ESBP} for details). Both of
these problems are still open for $d>2$. (B) The problem of the rate of decay of $\e_n(\bW^r_{p},L_q)$   turns out to be
a very rich and difficult problem, which is not yet completely solved. Those problems that have been resolved required
different nontrivial methods for different pairs $(p,q)$. 

Here is a typical result on the $\e_n(\bW^r_{p},L_q)$: For $1<p,q<\infty$ and $r>(1/p-1/q)_+ $ one has
$$
\e_n(\bW^r_{p},L_q) \asymp \left(\frac{(\log n)^{ d-1}}{n}\right)^r\quad,\quad n\in \N.
 $$
 The above rate of decay does neither depend on $p$ nor on $q$. It is known in approximation theory  that investigation
of asymptotic characteristics of classes $\bW^r_{p}$ in $L_q$ becomes more
difficult when $p$ or $q$ takes value $1$ or $\infty$ than when $1<p,q<\infty$.
This is true for $\e_n(\bW^r_{p},L_q)$, too. There are still fundamental open problems, for instance, Open problem 1.6
below. 

Recently, driven by applications in engineering, biology, medicine and other areas of science
nonlinear approximation began to play an important role. Nonlinear
approximation is important in applications because of its concise
representations  and increased computational efficiency. In Section 7 we discuss a typical 
problem of nonlinear approximation -- the $m$-term approximation. Another name for $m$-term approximation is {\it sparse
approximation}. In this setting we begin with a given system of elements (functions) $\Di$, which is usually called a
{\it dictionary},  in a Banach space $X$. 
The following characteristic 
$$
\sigma_m(f,\Di)_X := \inf_{\substack{g_i\in \Di,c_i \\i=1,\dots,m}} \Big\|f -\sum_{i=1}^m
c_ig_i\Big\|_X
$$
is called best $m$-term approximation of $f$ with regard to $\Di$ and gives us the bottom line of $m$-term approximation
of $f$. For instance, we can use the classical trigonometric system 
$\Tr^d$ as a dictionary $\Di$. Then, clearly, for any $f\in L_p$ we have $\sigma_m(f,\Tr^d)_{L_p} \le E_{Q_n}(f)_p$,
when
$m=|Q_n|$ (see Section 2 for the definition of the step hyperbolic cross $Q_n$). Here, $E_{Q_n}(f)_p$ is the best
approximation by the hyperbolic cross polynomials with frequencies from the
hyperbolic cross $Q_n$. It turns out that for some function classes best $m$-term approximations give the same order of
approximation as the corresponding hyperbolic cross polynomials but for other classes the nonlinear way of approximation
provides a substantial gain over the hyperbolic cross approximation. For instance, when $m=|Q_n|$, we have
$$
\sup_{f\in \bW^r_p}\sigma_m(f,\Tr^d)_{L_p} \asymp \sup_{f\in \bW^r_p}E_{Q_n}(f)_p,\quad 1<p<\infty,
$$
but for $2\le p <q<\infty$ the quantity
$
\sup_{f\in \bW^r_p}\sigma_m(f,\Di)_{L_q} $ is substantially smaller than $\sup_{f\in \bW^r_p}E_{Q_n}(f)_q$.  

In a way similar to optimization over linear subspaces in the case of linear approximation we 
discuss an optimization over dictionaries from a given collection in the $m$-term approximation problem. It turns out
that the wavelet type bases, for instance, the basis $\mathcal U^d$ discussed in Section 7, are very good for sparse
approximation -- in many cases they are optimal (in the sense of order) among all orthogonal bases. A typical result
here is the following: for $1<p,q<\infty$ and large enough $r$ we have
$$
\sup_{f\in \bW^r_p}\sigma_m(f,\mathcal U^d)_{L_q} \asymp \left(\frac{(\log m)^{d-1}}{m}\right)^r\quad,\quad m\in \N.
$$
It is important to point out that the above rate of decay does not depend on $p$ and $q$. 

The characteristic $\sigma_m(f,\Di)_X$ gives us a bench mark, which we can ideally achieve 
in $m$-term approximation of $f$. Clearly, keeping in mind possible numerical applications, we would like to devise good
constructive methods (algorithms) of $m$-term  
approximation. It turns out that greedy approximations (algorithms) work very well for a wide variety of dictionaries
$\Di$ (see Section 7 for more details).

Numerical integration, discussed in Section 8, is one more challenging multivariate problem where approximation theory
methods 
are very useful. In the simplest form (Quasi-Monte Carlo setting), for a given function class $\mathbf F$ we want to
find $m$ points $\bx^1,\dots,\bx^m$ in $D$
such that $\sum_{j=1}^m \frac{1}{m} f(\bx^j)$ approximates well the integral $\int_D fd\mu$, where $\mu$ is the
normalized
Lebesgue measure on $D$. Classical discrepancy theory provides constructions of point sets that are good for numerical
integration of characteristic functions of parallelepipeds of the form $P=\prod_{j=1}^d [a_j,b_j]$. The typical error
bound is of the form $m^{-1}(\log m)^{d-1}$, see Subsection \ref{numintdisc} below. Note that a regular grid for $m=n^d$ provides an error of the order
$m^{-1/d}$. 
The above mentioned results of discrepancy theory are closely related to numerical integration of functions with bounded
mixed derivative (the case of the first mixed derivative) by the Koksma-Hlawka inequality. 

In a somewhat more general setting (optimal cubature formula setting) we are optimizing not only over points
$\bx^1,\dots,\bx^m$ but also over the weights $\lambda_1,\dots,\lambda_m$:
$$
 \kappa_m(\mathbf{F}) \,:=\,
\inf\limits_{X_m=\{\bx^1,\dots,\bx^m\}\subset D}\,
	\inf\limits_{\lambda_1,\dots,\lambda_m\in\R}\sup\limits_{f\in{\mathbf{F}} }
	\Big|\int_{D} f(\bx) \, d\bx-\sum\limits_{i=1}^m \lambda_i f(\bx^i)\Big|\,.
$$
A typical and very nontrivial result here is: For $1<p<\infty$ and $r>\max\{1/p,1/2\}$
we have
$$
\kappa_m(\bW^r_p)\asymp m^{-r} (\log m)^{(d-1)/2}\quad,\quad m\in \N.
$$
In the case of functions of two variables optimal cubature rules are very simple -- the Fibonacci
cubature rules. They represent a special type of cubature rules, so-called Quasi-Monte Carlo rules. In the case $d\geq
3$ the optimal (in the sense of order) cubature rules  are constructive but not as
simple as the Fibonacci cubature formulas -- the Frolov cubature formulae. In fact, for the Frolov cubature formulae
  all weights $\lambda_i$, $i=1,...,m$,  are equal but in general do not sum up to one. This
means in particular that constant functions would not be integrated exactly by Frolov's method. Equal weights which
sum up to one is the main feature of the quasi-Monte Carlo integration. We point out that   there are still fundamental
open
problems: right orders of $\kappa_m(\bW^r_p)$ for $p=1$ and $p=\infty$ are not known (see Open problems 1.8 and 1.9
below). 

In this survey the notion {\it Sparse Grids} \index{Sparse grid} actually refers to the point grid coming out of Smolyak's algorithm applied
to univariate interpolation/cubature rules. The phrase itself is due to Zenger \cite{Ze91}, \cite{GrScZe92} and
co-workers who addressed more general hierarchical methods for avoiding the ``full grid'' decomposition. In any sense
{\em
sparse grids} play an important role in numerical integration of functions with bounded mixed derivative. In the
case of $r$th bounded  mixed derivatives they provide an error of the order $m^{-r}(\log m)^{(d-1)(r+1/2)}$. Also, they
provide the recovery error in the sampling problem of the same order. Note again that the regular grid from above
provides an error of the order $m^{-r/d}$. 
The error bound $m^{-r}(\log m)^{(d-1)(r+1/2)}$ is reasonably good for moderate dimensions $d$, say, $d\le 40$. It turns
out that there are practical computational problems with moderate dimensions where sparse grids work well. 
Sparse grids techniques have applications in quantum mechanics, numerical solutions of stochastic PDEs, data mining, 
finance. 
 
 At the end of this section we give some remarks, which demonstrate a typical
 difficulty of the study of the hyperbolic cross approximations. It is known
that the Dirichlet and the de la Vall\'ee Poussin kernels and operators,
associated
with them, play a significant role in investigation of the trigonometric
approximation. In particular, the boundedness property of the de la Vall\'ee
Poussin kernel: $\|\mathcal V_n(x)\|_1 \le C$ is very helpful. It turns out that there is
no analog of this boundedness property for the hyperbolic cross polynomials: for
the corresponding de la Vall\'ee Poussin kernel we have $\|\mathcal V_{Q_n}(\bx)\|_1 \ge
cn^{d-1}$. 
This phenomenon made the study in the $L_1$ and $L_\infty$ norms difficult.
There are many unsolved problems for approximation in the $L_1$ and $L_\infty$
norms. 

In the case of $L_p$, $1<p<\infty$, the classical tools of harmonic analysis --
the Littlewood-Paley theorem, the Marcinkiewicz multipliers, the
Hardy-Littlewood inequality -- are very useful. However, in some cases other
methods were needed to obtain correct estimates. 
We illustrate it on the following example. Let $\rho(\bs):=\{\bk: [2^{s_j-1}]\le |k_j|<2^{s_j},\, j=1,\dots,d\}$ and
$$
 \Di_{Q_n}(\bx) := \sum_{|\bs|_1\le n} \Di_{\rho(\bs)}(\bx), \quad \Di_{\rho(\bs)}(\bx):=
\sum_{\bk\in\rho(\bs)}e^{i(\bk,\bx)},
$$
be the Dirichlet kernel for the step hyperbolic cross $Q_n$. Then for $2<p<\infty$ by a corollary to the
Littlewood-Paley theorem one gets
\begin{equation}\label{1.1.1}  
  \|\Di_{Q_n}\|_p \lesssim
\left(\sum_{|\bs|_1\le n} \|\Di_{\rho(\bs)}\|_p^2\right)^{1/2}\lesssim2^{(1-1/p)n}n^{(d-1)/2}.
\end{equation}
However,   the
upper bound in 
(\ref{1.1.1}) does not provide the right bound. Other technique (see Theorem \ref{T2.4.6}) gives
\begin{equation}\nonumber
  \|\Di_{Q_n}\|_p \lesssim
\left(\sum_{|\bs|_1\le n} \left(\|\Di_{\rho(\bs)}\|_22^{|\bs|_1(1/2-1/p)}\right)^p\right)^{1/p}\lesssim
2^{(1-1/p)n}n^{(d-1)/p}.
\end{equation}
The above example
demonstrates the problem, which is related to the fact that
in the multivariate Littlewood-Paley formula we have many ($\asymp n^{d-1}$)
dyadic blocks of the same size ($2^n$). 

It is known that in studying asymptotic characteristics of function classes
the discretization technique is useful. Classically,
the Marcinkiewicz theorem served as a powerful tool for discretizing the
$L_p$-norm of a trigonometric polynomial. Unfortunately, there
is no analog of Marcinkiewicz' theorem for hyperbolic cross polynomials (see
Theorem \ref{T2.5.7}). An important new technique, which was developed to
overcome the
above difficulties, is based on volume estimates of the sets of Fourier
coefficients of unit balls of trigonometric polynomials (see Subsection 2.5). 
Later, also within the ``wavelet revolution'', sequence space isomorphisms
were used to discretize function spaces, see Section \ref{Sect:Smolyak} for
the use of the tensorized Faber-Schauder system.

A standard technique of proving lower bounds for asymptotic characteristics of function classes ($d_m$, $\lambda_m$,
$\varphi_m$, $\varrho_m$, $\epsilon_n$ $\sigma_m$, $\kappa_m$) is based on searching for ``bad'' functions in an
appropriate subspaces of the trigonometric polynomials. Afterwards, using the de la
Vall{\'e}e Poussin operator, we reduce the problem of approximation by arbitrary functions to the problem of
approximation by the trigonometric polynomials. The uniform boundedness property of the de la Vall{\'e}e Poussin
operators is fundamentally important in this technique. As we pointed out above the de la Vall{\'e}e Poussin operators
for the hyperbolic cross are not uniformly bounded as operators from $L_1$ to $L_1$ and from $L_\infty$ to $L_\infty$ --
their norms grow with $N$ as $(\log N)^{d-1}$. In some cases we are able to overcome this difficulty by considering the
entropy numbers in the respective situations. In fact, we use some general
inequalities (Carl's inequality, see Section 6), which provide lower bounds for
the Kolmogorov widths in terms of the entropy numbers. However, there are still
outstanding open problems on the behavior of the entropy numbers in 
$L_\infty$ (see Subsection \ref{entropy_gen}). We also point out that classical
techniques, based on Riesz products, turned out to be useful for the 
hyperbolic cross polynomials. 

As we already pointed out we present a survey of results on multivariate approximation -- the hyperbolic cross
approximation. These results provide a natural generalization of the classical univariate approximation theory results
to the case of multivariate functions. We give detailed historical comments only on the multivariate results. Typically,
the corresponding univariate results are well-known and could be found in a book on approximation theory, for instance, 
\cite{Pin85}, \cite{DeLo93}, \cite{lomavo96}, \cite{TBook}. 
We discuss here two types of mixed smoothness classes: 
(I) the $\bW$-type classes, which are defined by a restriction on the mixed derivatives 
(more generally, fractional mixed derivatives); (II) the $\bH$-type and the $\bB$-type classes, 
which are defined by restrictions on the mixed differences. It has become standard in the theory of 
function spaces to call classes defined by restrictions on derivatives Sobolev or Sobolev-type classes. 
We follow this tradition in our survey. However, we point out that S.L. Sobolev did not study the 
classes of functions with bounded mixed derivative. The $\bB$-type classes are usually called 
Besov or Besov-type classes. We follow this tradition in our survey too. 
The $\bH$-type classes are a special case of the $\bB$-type classes, namely $\bH^r_p = \bB^r_{p,\infty}$. 
Historically, the first investigations were conducted on the $\bW$-type and $\bH$-type classes. 
An interesting phenomenon was discovered. It was established that the behavior of the asymptotic characteristics of 
classes $\bW^r_p$ and $\bH^r_p$ measured in $L_q$ are different for $1<p,q<\infty$. 
Typically, for $d\ge 2$
$$
a_m(\bW^r_p,L_q) = o\left(a_m(\bH^r_p,L_q)\right),
$$
where $a_m$ stands for best approximations $E_{Q_n}$, 
the Kolmogorov width $d_m$, the linear width $\lambda_m$, the orthowidth $\varphi_m$, 
the entropy numbers $\varepsilon_m$, and the best $m$-term approximation $\sigma_m$. 
This shows that classes $\bH^r_p$ are substantially larger than their counterparts $\bW^r_p$. 
We point out that it is well-known that in the univariate case, typically, the behavior of the asymptotic 
characteristics of classes $W^r_p$ and $H^r_p$ coincide, even though $H^r_p$ is a wider class than $W^r_p$. 
This phenomenon encouraged researchers to study the $\bB$-type classes $\bB^r_{p,\theta}$ and 
determine the influence of the secondary parameter $\theta$.

In Section 9 we provide some more arguments in favor of a systematic study of the hyperbolic cross approximation and
classes of functions with mixed smoothness. In particular, we discuss the anisotropic mixed smoothness which plays an
important role in high-dimensional approximation and applications; direct and inverse theorems for hyperbolic cross
approximation; widths and hyperbolic cross approximation for the intersection of classes of mixed smoothness; continuous
algorithms in $m$-term approximation and non-linear widths.
 We also comment on the quasi-Banach situation, i.e. the case
$p<1$. Corresponding classes of functions turned out to be suitable not only for best $m$-term approximation problems.
In this section we also complement the results from Section 4 with recent results on other widths ($s$-numbers) like
Weyl and Bernstein numbers relevant for the analysis of Monte Carlo algorithms. Also, we demonstrate there how
classical hyperbolic cross approximation theory can be used in some important contemporary problems of numerical
analysis.

High-dimensional approximation problems appear in several areas of science like for instance in quantum chemistry and
meteorology. As already mentioned above some of our function class models are relevant in this context. In Section 
\ref{Sect:highdim} we comment on some recent results on how the
underlying dimension $d$ affects the multivariate
approximation error. The order of the approximation error is not longer sufficient for determining the {\it information
based complexity} of the problem. We present some recent results and techniques to see the $d$-dependence of the
constants in the approximation error estimates and the convergence rate of widths complemented by sharp {\it
preasymptotical} estimates in the Hilbert space case. In computational mathematics related to high-dimensional problems,
the so-called $\varepsilon$-dimension $n_\varepsilon = n_\varepsilon(\mathbf{F},X)$ is used to quantify the
computational complexity. We discuss $d$-dependence of  the $\varepsilon$-dimension of $d$-variate function classes of
mixed smoothness and of the relevant hyperbolic cross approximations where $d$ may be very large or even infinite, as
well its application for numerical solving of parametric and stochastic elliptic PDEs.

The following list contains monographs and
survey papers directly related to this book: \cite{Korb2}, \cite{Tmon}, \cite{Telyak88}, \cite{TBook},
\cite{D}, \cite{Tsurv}, \cite{BuGr04}, \cite{DeSc89}, \cite{NW08}, \cite{NoWo09}, \cite{NoWo12},
\cite{DiKuSl13}, \cite{No14}, \cite{Rom12}.

{\bf Notation.} As usual $\N$ denotes the natural numbers, $\N_0:=\N\cup\{0\}$,
$\N_{-1}:=\N_0\cup\{-1\}$, $\zz$ denotes the integers, 
$\R$ the real numbers, and $\C$ the complex numbers. The letter $d$ is always
reserved for the underlying dimension in $\R^d, \Z^d$ etc. Elements
$\bx,\by,\bs \in \R^d$ are always typesetted in bold face.  We denote
with $(\bx,\by)$ and $\bx\cdot \by$ the usual Euclidean inner
product in $\R^d$. For
$a\in \R$ we denote $a_+ := \max\{a,0\}$. 
For $0<p\leq \infty$ and $\bx\in \R^d$ we denote $|\bx|_p := (\sum_{i=1}^d
|x_i|^p)^{1/p}$ with the usual modification in the case $p=\infty$. By
$\bx = (x_1,\ldots,x_d)>0$ we mean that each coordinate is positive. By $\T$ we denote
the torus represented by the interval $[0,2\pi]$.
If $X$ and $Y$ are two (quasi-)normed spaces, the (quasi-)norm
of an element $x$ in $X$ will be denoted by $\|x\|_X$. If $T:X\to Y$ is a continuous operator we write $T\in
\mathcal{L}(X,Y)$. The symbol $X \hookrightarrow Y$ indicates that the
identity operator is continuous. For two sequences $a_n$ and $b_n$ we will
write $a_n \lesssim b_n$ if there exists a constant $c>0$ such that $a_n \leq
c\,b_n$ for all $n$. We will write $a_n \asymp b_n$ if $a_n \lesssim b_n$ and
$b_n \lesssim a_n$.

\subsection*{Outstanding open problems}\index{Open problems!Outstanding}
Let us conclude this introductory section with a list of outstanding open problems. 
\begin{itemize}
 \item[\bf 1.1] Find the right form of the Bernstein inequality for $\Tr(N)$ in $L_1$ (Open problem 2.1).
 \item[\bf 1.2] Prove the Small Ball Inequality for $d\ge 3$ (Open problem 2.5, 2.6).
 \item[\bf 1.3] Find the right order of the Kolmogorov widths 
 $d_m(\bW^r_p,L_\infty)$ and $d_m(\bH^r_p, L_\infty)$ for $2\leq p\leq \infty$ and $r>1/p$ in dimension 
 $d\ge 3$ (Open problem 4.2).
\item[\bf 1.4] Find the right order of the optimal sampling recovery $\varrho_m(\bW^r_p,L_p)$, $1\le p\le \infty$, $r>1/p$.
\item[\bf 1.5] Find the right order of the optimal sampling recovery $\varrho_m(\bW^r_p,L_q)$, $1<p<2<q<\infty$.
\item[\bf 1.6] Find the right order of the entropy numbers  $\e_n(\bW^r_p,L_\infty)$ and $\e_n(\bH^r_p, L_\infty)$ for $1\leq p\leq \infty$ and $r>1/p$ in dimension 
 $d\ge 3$ (Open problem 6.3).
\item[\bf 1.7] Find the right order of the best $m$-term trigonometric approximation $\sigma_m(\bW^r_p)_\infty$, $1\le
p\le \infty$ (Open problem 7.2).
\item[\bf 1.8] Find the right order of the optimal error of numerical integration $\kappa_m(\bW^r_\infty)$ (see
Conjecture \ref{C6.5.2} in Subsection \ref{OP_int}). 
\item[\bf 1.9] Find the right order of the optimal error of numerical integration $\kappa_m(\bW^r_1)$ (see
Conjecture \ref{C6.5.1} in Subsection \ref{OP_int}).
\item[\bf 1.10] Find the right order of the optimal error of numerical integration $\kappa_m(\bW^r_p)$ in the range of
small smoothness (see Conjectures \ref{C6.5.1b}, Subsection \ref{OP_int}). 
\end{itemize}

\newpage

{\bf Acknowledgment.} The authors acknowledge the fruitful discussions with D.B. Bazarkhanov, A. Hinrichs, T. K\"uhn, E. Novak and W. Sickel on
this topic, especially at the ICERM Semester Programme ``High-Dimensional Approximation'' in Providence,
2014, where this project has been initiated, and at the conference ``Approximation Methods and Function Spaces'' in
Hasenwinkel, 2015.  The authors would further like to thank M. Hansen, J. Oettershagen, A.S. Romanyuk, S.A. Stasyuk, M. Ullrich and 
N. Temirgaliev for helpful comments on the material. V.N. Temlyakov and T. Ullrich would like to thank the organizers of the 2016 special semester ``Constructive Approximation and Harmonic Analysis''
at the Centre de Recerca Matem\'atica (Barcelona) for the opportunity to present an advanced course based on this material. 
Moreover, T. Ullrich gratefully acknowledges support by the German Research Foundation (DFG), Ul-403/2-1, 
and the Emmy-Noether programme, Ul-403/1-1. D. D\~ung thanks the Vietnam Institute for Advanced Study in Mathematics
(VIASM) and the Vietnam National Foundation for Science and Technology Development (NAFOSTED), Grant No. 102.01-2017.05,
for partial supports.  D. D\~ung and V.N. Temlyakov would like to thank T. Ullrich and M. Griebel for supporting their visits at
the Institute for Numerical Simulation, University of Bonn, where major parts of this work were discussed. Also, D. D\~ung and V.N. Temlyakov express their gratitude to the VIASM for support during 
Temlyakov's visit of VIASM in September-October 2016, when certain parts of the survey were finished.
Finally, the authors would like to thank G. Byrenheid, J. Oettershagen and S. Mayer (Bonn) for preparing most of the figures in the text. 
\newpage

%% file: trig_polynomials.tex
\section{Trigonometric polynomials}\index{Trigonometric polynomials}
\label{trigpol}

\subsection{Univariate polynomials}\index{Trigonometric polynomials!Univariate}
\label{univpol}

Functions of the form
\begin{equation}\label{2.1.1}
t(x) = \sum_{|k|\le n}c_k e^{ikx} =a_0/2+\sum_{k=1}^n
(a_k\cos kx+b_k\sin kx)
\end{equation}
($c_k$,  $a_k$,  $b_k$ are complex numbers) will be called trigonometric
polynomials of order $n$. The set of such polynomials we shall denote
by $\Tr(n)$,  and by $\RR\Tr(n)$ the subset of $\Tr(n)$ of real polynomials.

We first consider a number of concrete polynomials which play
an important role in approximation theory.

{\bf 1. The Dirichlet kernel.}\index{Kernel!Dirichlet}  The Dirichlet kernel of order $n$:
\begin{align*}
\D_n (x)&:= \sum_{|k|\le n}e^{ikx} = e^{-inx} (e^{i(2n+1)x} - 1)
(e^{ix} - 1)^{-1} =\\
&=\bigl(\sin (n + 1/2)x\bigr)\bigm/\sin (x/2).
\end{align*}

The Dirichlet kernel is an even trigonometric polynomial with
the majorant
\begin{equation}\label{2.1.2}
\bigl|\D_n (x)\bigr|\le \min \bigl\{2n+1, \pi/|x|\bigr\}, \qquad
 |x|\le \pi.
\end{equation}
The estimate
\begin{equation}\label{2.1.3}
\|\D_n\|_1\le C \ln n, \qquad n = 2, 3, \dotsc.
\end{equation}
follows from (\ref{2.1.2}).

With $f\ast g$ we denote the convolution
$$
f\ast g := (2\pi)^{-d}\int_{\T^d}f(\by)g(\bx-\by)d\by.
$$
For any trigonometric polynomial $t\in \Tr(n)$ we have
$$
t * \D_n = t,
$$

Denote
$$
x^l := 2\pi l/(2n+1), \qquad l = 0, 1, ..., 2n.
$$
Clearly,  the points $x^l$,  $l = 1, \dots, 2n$,  are zeros of the Dirichlet
kernel $\D_n$ on $[0, 2\pi]$.

 For any $t\in \Tr(n)$
\begin{equation}\label{2.1.4}
t(x) = (2n+1)^{-1}\sum_{l=0}^{2n} t(x^l) \D_n (x - x^l),
\end{equation}
 and for any $t\in \Tr(n)$,
\begin{equation}\label{2.1.5}
\| t \|_2^2 = (2n+1)^{-1}\sum_{l=0}^{2n} \bigl|t(x^l)\bigr|^2.
\end{equation}

Sometimes it is convenient to consider the following slight modification of
$\D_n$:
\begin{equation}\label{nestedDirichlet}
\D_n^1 (x):=  \D_n(x)-e^{-inx} = e^{-i(n-1)x} (e^{i2nx} - 1)
(e^{ix} - 1)^{-1}.
\end{equation}
Denote
$$
y^l := \pi l/n, \qquad l = 0, 1, ..., 2n-1.
$$
Clearly,  the points $y^l$,  $l = 1, \dots, 2n-1$,  are zeros of the
kernel $\D_n^1$ on $[0, 2\pi]$.

In the same way as (\ref{2.1.4}) and (\ref{2.1.5}) are proved for $t\in \Tr(n)$
one can prove the following identities for $t\in \Tr((-n,n]):=
\Span\{e^{ikx}\}_{k=-n+1}^n$.
\begin{equation}\nonumber
t(x) = (2n)^{-1}\sum_{l=0}^{2n-1} t(y^l) \D^1_n (x - y^l),
\end{equation}
\begin{equation}\nonumber
\| t \|_2^2 = (2n)^{-1}\sum_{l=0}^{2n-1} \bigl|t(y^l)\bigr|^2.
\end{equation}

An advantage of $\D_n^1$ over $\D_n$ is that the sets $\{y^l\}_{l=0}^{2^s-1}$,
$s=1,2,\dots$, are nested.

The following relation  for $1 < q < \infty$ is well known and easy to check
\begin{equation}\label{2.1.6}
\| \D_n \|_q\asymp n^{1-1/q}.
\end{equation}

The relation (\ref{2.1.6}) for $q = \infty$ is obvious.

We denote by $S_n$ the operator of taking the partial sum of
order $n$. Then for $f\in L_1$ we have
$$
S_n (f) = f * \D_n.
$$

\begin {thm}\label{T2.1.1} The operator $S_n$ does not change polynomials from
$\Tr(n)$ and for $p = 1$ or $\infty$ we have
$$
\| S_n \|_{p\to p}\le C \ln n, \qquad n = 2, 3, \dots,
$$
and for $1 < p < \infty$ for all n we have
$$
\| S_n \|_{p\to p}\le C(p).
$$
\end{thm}

This theorem follows from (\ref{2.1.3}) and the Marcinkiewicz
multiplier theorem (see Theorem \ref{T7.3.4}).

{\bf 2. The Fej\'er kernel.}\index{Kernel!Fej\'er} The Fej\'er kernel of order $n - 1$:
\begin{align*}
{\mathcal K}_{n-1} (x) &:= n^{-1}\sum_{m=0}^{n-1} \D_m (x) =
\sum_{|m|\le n} \bigl(1 - |m|/n\bigr) e^{imx} =\\
&=\bigl(\sin (nx/2)\bigr)^2\bigm /\bigl(n (\sin (x/2)\bigr)^2\bigr).
\end{align*}

The Fej\'er kernel is an even nonnegative trigonometric
polynomial in $\Tr(n-1)$ with the majorant
\begin{equation}\nonumber
 {\mathcal K}_{n-1} (x) \le \min \bigl\{n, \pi^2 /(nx^2)\bigr\},
\qquad |x|\le \pi.
\end{equation}

From the obvious relations
$$
\| {\mathcal K}_{n-1} \|_1 = 1, \qquad \| {\mathcal K}_{n-1} \|_{\infty} = n
$$
and the inequality  
$$
\| f \|_q \le \| f \|_1^{1/q} \| f \|_{\infty}^{1-1/q}
$$
we get  
\begin{equation}\nonumber
C n^{1-1/q}\le \| {\mathcal K}_{n-1} \|_q \le n^{1-1/q}, \qquad
 1\le q\le \infty.
\end{equation}

{\bf 3. The de la Vall\'ee Poussin kernels.}\index{Kernel!de la Vall\'ee Poussin} The de la Vall\'ee Poussin kernels:
$$
{\mathcal V}_{m, n} (x) := (n - m)^{-1}\sum_{l=m}^{n-1} \D_l (x), \qquad n > m.
$$
It is convenient to represent these kernels in terms of the Fej\'er
kernels:
\begin{align*}
{\mathcal V}_{m, n} (x)&= (n - m)^{-1} \bigl(n {\mathcal K}_{n-1}(x) -
m {\mathcal K}_{m-1} (x)\bigr) =\\
&= (\cos mx - \cos nx)\bigl(2(n - m) \bigl(\sin (x/2)\bigr)^2 \bigr)^{-1}.
\end{align*}

The de la Vall\'ee Poussin kernels ${\mathcal V}_{m, n}$ are even trigonometric
polynomials of order $n - 1$ with the majorant
\begin{equation}\label{2.1.9}
\bigl| {\mathcal V}_{m, n} (x) \bigr| \le C \min \bigl\{n, \ 1/|x|, \
 1/((n-m)x^2)\bigr\}, \ |x|\le \pi.
\end{equation}
The relation (\ref{2.1.9}) implies the estimate
$$
\| {\mathcal V}_{m, n} \|_1 \le C \ln \bigl(1 + n/(n-m)\bigr).
$$

We shall often use the de la Vall\'ee Poussin kernel with $n = 2m$
and denote it by
$$
{\mathcal V}_m (x) := {\mathcal V}_{m, 2m} (x), \qquad m \ge 1, \qquad {\mathcal
V}_0 (x) = 1.
$$

Then for $m \ge 1$ we have
$$
{\mathcal V}_m = 2{\mathcal K}_{2m-1} - {\mathcal K}_{m-1},
$$
which with the properties of ${\mathcal K}_n$ implies
\begin{equation}\label{2.1.10}
\| {\mathcal V}_m \|_1\le 3.
\end{equation}
In addition
$$
\|{\mathcal V}_m \|_{\infty}\le 3m.
$$
Consequently,  in the same way as above we get
\begin{equation}\nonumber
\| {\mathcal V}_m \|_q \asymp m^{1-1/q}, \qquad 1 \le q \le \infty.
\end{equation}

We denote
$$
x(l) := \pi l/2m, \qquad l = 1, \dots, 4m.
$$
Then as in (\ref{2.1.4}) for each $t\in \Tr(m)$ we have
\begin{equation}\label{2.1.12}
t(x) = (4m)^{-1}\sum_{l=1}^{4m} t\bigl(x(l)\bigr)
{\mathcal V}_m \bigl(x - x(l)\bigr).
\end{equation}

The operator $V_m$ defined on $L_1$ by the formula
$$
V_m (f) = f * {\mathcal V}_m
$$
will be called the de la Vall\'ee Poussin operator.

The following theorem is a corollary of the definition of
kernels ${\mathcal V}_m$ and the relation~(\ref{2.1.10}).

\begin{thm}\label{T2.1.2} The operator $V_m$ does not change polynomials from
$\Tr(m)$ and for all $1\le p\le \infty$ we have
$$
\|V_m \|_{p\to p}\le 3, \qquad m = 1, 2, \dotsc .
$$
\end{thm}

 {\bf 4. The Rudin-Shapiro polynomials.}\index{Trigonometric polynomials!Rudin-Shapiro} For any natural number $N$ there
exists a polynomial of the form
$$
\RR_N (x) =\sum_{|k|\le N}\varepsilon_k e^{ikx},\qquad\varepsilon_k=\pm1,
$$
such that the bound
\begin{equation}\nonumber
\| \RR_N \|_{\infty}\le CN^{1/2}
\end{equation}
holds.

\subsection{Multivariate polynomials}\index{Trigonometric polynomials!Multivariate}
\label{multpol}

The multivariate trigonometric system $\Tr^d:=\Tr\times\cdots\times\Tr$, $d\ge
2$, in contrast to the univariate trigonometric system $\Tr$ does not have a
natural ordering. This leads to different natural ways of building sets of
trigonometric polynomials. 
In this section we define the analogs of the Dirichlet, Fej\'er,
de la Vall\'ee Poussin and Rudin-Shapiro  kernels  for $d$-dimensional
parallelepipeds
$$
\Pi(\mathbf N,d) :=\bigl \{\mathbf a\in \R^d  : |a_j|\le N_j,\
j = 1,\dots,d \bigr\} ,
$$
where $N_j$ are nonnegative integers.
We shall formulate properties of these multivariate 
kernels, which easily follow from the corresponding properties of
univariate kernels. Here $\Tr(\mathbf N,d)$ is the set of
complex trigonometric polynomials with harmonics from  $\Pi(\mathbf N,d)$.
The set of real trigonometric polynomials with  harmonics  from
$\Pi(\mathbf N,d)$ will be denoted by $\RR\Tr(\mathbf N,d)$. In the sequel we will frequently use the following notation
\begin{equation}\label{nutheta}
  \nu(\bar{\mathbf N}) :=
\prod_{j=1}^d \bar N_j\quad\mbox{and}\quad \vartheta(\mathbf N) := \prod_{j=1}^d (2N_j  + 1)\,,
\end{equation}
with $\bar{\mathbf N} := (\bar{N}_1,...,\bar{N}_d)$ and $\bar{N}:=\max\{N,1\}$\,.

{\bf 1d. The Dirichlet kernels}\index{Kernel!Dirichlet}
$$
\mathcal D_{\mathbf N} (\mathbf x) := \prod_{j=1}^d \mathcal D_{N_j}  (x_j),\qquad
   \mathbf N = (N_1 ,\dots,N_d)
$$
have the following properties.
For any trigonometric polynomial $t\in \Tr(\mathbf N,d)$,
$$
t * \mathcal D_{\mathbf N}  = t .
$$
For $1 < q\le \infty$,
\begin{equation}\nonumber
\|\mathcal D_{\mathbf N}\|_q\asymp\nu(\bar{\mathbf N})^{1-1/q}
\end{equation}
where $\bar N_j  := \max \{N_j ,1\}$, $\nu(\bar{\mathbf N}) :=
\prod_{j=1}^d \bar N_j$  and
\begin{equation}\nonumber
\|\mathcal D_{\mathbf N}\|_1\asymp \prod_{j=1}^d  \ln (N_j  + 2) .
\end{equation}
We denote
$$
P(\mathbf N) := \{\mathbf n = (n_1 ,\dots,n_d),\quad n_j\in \N_0,\quad
0\le n_j\le 2N_j  ,\quad j = 1,\dots,d \},
$$
and set
$$
\mathbf x^{\mathbf n}:=\left(\frac{2\pi n_1}{2N_1+1},\dots,\frac{2\pi n_d}
{2N_d+1}\right),\qquad \mathbf n\in P(\mathbf N) .
$$
Then for any $t\in \Tr(\mathbf N,d)$,
\begin{equation}\nonumber
t(\mathbf x) =\vartheta(\mathbf N)^{-1}\sum_{\mathbf n\in P(\mathbf N)}
t(\mathbf x^{\mathbf n})\mathcal D_{\mathbf N} (\mathbf x - \mathbf x^{\mathbf n}) ,
\end{equation}
where $\vartheta(\mathbf N) := \prod_{j=1}^d (2N_j  + 1)$ and for any
$t,u\in \Tr(\mathbf N,d)$,
\begin{equation}\nonumber
\langle t,u \rangle = \vartheta(\mathbf N)^{-1}\sum_{\mathbf n\in P(\mathbf N)}
t(\mathbf x^{\mathbf n})\bar u(\mathbf x^{\mathbf n})  , 
\end{equation}
\begin{equation}\nonumber
\|t\|_2^2  =\vartheta(\mathbf N)^{-1}\sum_{\mathbf n\in P(\mathbf N)}
\bigl|t(\mathbf x^{\mathbf n})\bigr|^2   .
\end{equation}

{\bf 2d. The Fej\'er kernels}\index{Kernel!Fej\'er}
$$
\mathcal K_{\mathbf N} (\mathbf x) :=\prod_{j=1}^d\mathcal K_{N_j}  (x_j)  ,\qquad
\mathbf N = (N_1 ,\dots,N_d) ,
$$
are nonnegative trigonometric polynomials from  $\Tr(\mathbf N,d)$, which have
the following properties (recall \eqref{nutheta}):
\begin{equation}\nonumber
\|\mathcal K_{\mathbf N}\|_1  = 1  ,
\end{equation}
\begin{equation}\nonumber
\|\mathcal K_{\mathbf N}\|_q\asymp \vartheta(\mathbf N)^{1-1/q},\qquad
1\le q\le\infty ,
\end{equation}
\begin{equation}\nonumber
\|\mathcal K_{\mathbf N}\|_{\mathbf q}  \asymp \prod_{j=1}^d
\bigl(\max\{1,N_j\}\bigr)^{1-1/q_j},\qquad
\mathbf 1\le \mathbf q\le \infty, \mathbf{q} = (q_1,...,q_d).
\end{equation}

{\bf 3d. The de la Vall\'ee Poussin kernels}\index{Kernel!de la Vall\'ee Poussin}
$$
\mathcal V_{\mathbf N} (\mathbf x) := \prod_{j=1}^d \mathcal V_{N_j}  (x_j)  ,\qquad
\mathbf N = (N_1 ,\dots,N_d) ,
$$
have the following properties (recall \eqref{nutheta})
\begin{equation}\label{2.2.9}
\|\mathcal V_{\mathbf N}\|_1  \le 3^d  ,
\end{equation}
\begin{equation}\nonumber
\|\mathcal V_{\mathbf N}\|_q\asymp \vartheta(\mathbf N)^{1-1/q},\qquad
1\le q\le\infty ,
\end{equation}
\begin{equation}\nonumber
\|\mathcal V_{\mathbf N}\|_{\mathbf q}  \asymp \prod_{j=1}^d
\bigl(\max\{1,N_j\}\bigr)^{1-1/q_j},\qquad
\mathbf 1\le \mathbf q\le \infty .
\end{equation}
For any $t\in \Tr(\mathbf N,d)$,
$$
V_{\mathbf N} (t) := t * \mathcal V_{\mathbf N}  = t .
$$
We denote
$$
P'(\mathbf N) := \{\mathbf n = (n_1,\dots,n_d),\quad n_j\in\N_0,\quad
1\le n_j\le 4N_j  ,\quad j = 1,\dots,d \}
$$
and set
$$
\mathbf x(\mathbf n) :=\left (\frac{\pi n_1}{2N_1} ,\dots,\frac{\pi n_d}{2N_d}
\right)   ,\qquad \mathbf n\in P'(\mathbf N)  .
$$
In the case $N_j  = 0$ we assume $x_j (\mathbf n) = 0$.
Then for any $t\in \Tr(\mathbf N,d)$ we have the representation (recall \eqref{nutheta})
\begin{equation}\label{2.2.12}
t(\mathbf x) =\nu(4\bar{\mathbf N})^{-1}\sum_{\mathbf n\in P'
(\mathbf N)}
t(\mathbf x(\mathbf n))\mathcal V_{\mathbf N} (\mathbf x - \mathbf x
(\mathbf n))  .
\end{equation}
The relation (\ref{2.2.9}) implies that
\begin{equation}\nonumber
\|V_{\mathbf N}\|_{\mathbf  p\to\mathbf p}\le 3^d,\qquad  \mathbf 1\le\mathbf p
\le \infty .
\end{equation}
Let us define the polynomials $\mathcal A_{\mathbf s} (\mathbf x)$
for $\mathbf s =  (s_1 ,\dots,s_d)\in\N^d_0$
$$
\mathcal A_{\mathbf s} (\mathbf x) :=\prod_{j=1}^d\mathcal A_{s_j}(x_j),
$$
with $\mathcal A_{s_j}(x_j)$ defined as follows: 
$$
\mathcal A_0 (x) := 1, \quad \mathcal A_1 (x) := \mathcal V_1 (x) - 1, \quad
\mathcal A_s (x) := \mathcal V_{2^{s-1}} (x) -\mathcal V_{2^{s-2}} (x),
\quad s\ge 2,
$$
where $\mathcal V_m$ are the de la Vall\'ee Poussin kernels. Then by (\ref{2.1.10}),
\begin{equation}\nonumber
\bigl\|\mathcal A_{\mathbf s} (\mathbf x)\bigr\|_1\le 6^d,
\end{equation}
and consequently we  have  for  the  operator  $A_{\mathbf s}$, which  is
the
convolution with the kernel $\mathcal A_{\mathbf s} (\mathbf x)$, the inequality
\begin{equation}\label{2.2.16}
\|A_{\mathbf s}\|_{\mathbf p\to\mathbf p}\le 6^d,\qquad \mathbf 1\le\mathbf p
\le\infty .
\end{equation}
We note that in the case $\mathbf s\ge \mathbf 2$ for any
$t\in \Tr(2^{\mathbf s-\mathbf 2}   ,d)$,
$$
A_{\mathbf s} (t) = 0 .
$$

{\bf 4d. The Rudin--Shapiro polynomials}\index{Trigonometric polynomials!Rudin-Shapiro}
$$
\RR_{\mathbf N} (\mathbf x) :=\prod_{j=1}^d \RR_{N_j}(x_j),
\qquad\mathbf N = (N_1 ,\dots,N_d) ,
$$
have the following properties: $\RR_{\mathbf N}\in \Tr(\mathbf N,d)$,
\begin{equation}\nonumber
\|\RR_{\mathbf N}\|_{\infty}\le C(d)\vartheta(\mathbf N)^{1/2},\quad \hat
\RR_{\mathbf N} (\mathbf k) =\pm 1  ,\quad |\mathbf k| \le \mathbf N. 
\end{equation}
The Rudin-Shapiro polynomials have all the Fourier coefficients with their
absolute values equal to one. This is similar to the Dirichlet kernels. However,
the $L_p$ norms of $\RR_{\mathbf N}$ behave in a very different way:
$$
\|\RR_{\mathbf N}\|_p \asymp \vartheta(\mathbf N)^{1/2},\quad 1\le p\le\infty.
$$
In some applications we need to construct polynomials with similar properties
in a subspace of the $\Tr(\mathbf N,d)$. We present here one known result in
that direction (see \cite{TBook}, Ch.2, Theorem 1.1 and \cite{TE2}).
\begin{thm}\label{T2.2.1} Let $\varepsilon > 0$ and a subspace
$\Psi\subset \Tr(\mathbf N,d)$ be such that
$\dim \Psi\ge\varepsilon\vartheta(\mathbf N)$.
Then there is a $t \in\Psi$ such that
$$
\|t\|_{\infty}  = 1  ,\qquad \|t\|_2\ge C(\varepsilon,d) > 0 .
$$
\end{thm}

\subsection{Hyperbolic cross polynomials}\index{Hyperbolic cross!Polynomials}
\label{ssect:hyppoly}

Let $\mathbf s=(s_1,\dots,s_d )$ be a  vector  whose  coordinates  are
nonnegative integers
\begin{equation}\label{HC}
\begin{split}
\rho(\mathbf s) &:= \bigl\{ \mathbf k\in\mathbb Z^d:[ 2^{s_j-1}] \le
|k_j| < 2^{s_j},\qquad j=1,\dots,d \bigr\},\\
Q_n &:=   \cup_{|\mathbf s|_1\le n}
\rho(\mathbf s) \quad\text{--}\quad\text{a step hyperbolic cross},\\
\Gamma(N) &:= \bigl\{ \mathbf k\in\mathbb Z^d :\prod_{j=1}^d
\max\bigl\{ |k_j|,1\bigr\} \le N\bigr\}\quad\text{--}\quad\text{a hyperbolic
cross}.
\end{split}
\end{equation}

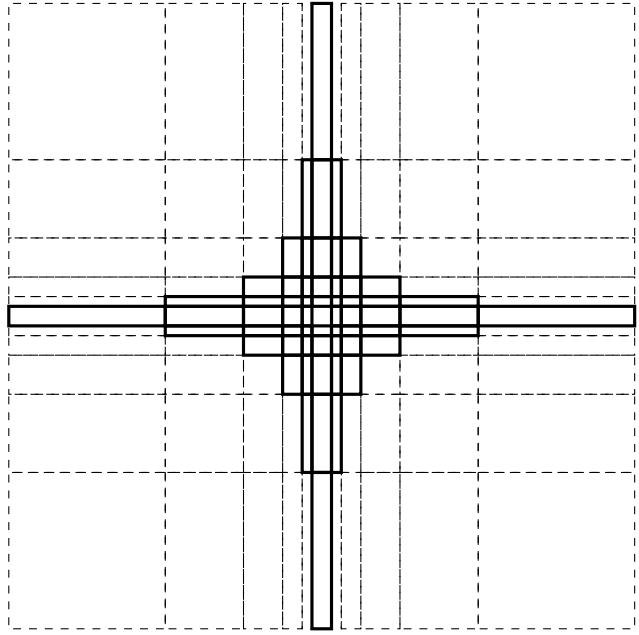
\begin{figure}[H]
 \begin{center}
\begin{tikzpicture}[scale=0.13]
\pgfmathsetmacro{\m}{5};
\pgfmathsetmacro{\rf}{1};
\pgfmathsetmacro{\rs}{1};
\pgfmathtruncatemacro{\xl}{2^(\m / \rf)}
\pgfmathtruncatemacro{\yl}{2^(\m / \rs)}
\pgfmathtruncatemacro{\limf}{round(\m / \rf)}

\pgfmathtruncatemacro{\lima}{round(\m / \rf)}
\pgfmathtruncatemacro{\limb}{round(\m / \rs)}

\foreach \x in {1,...,\lima}
{
\pgfmathtruncatemacro{\lim}{round((\m-\rf*\x)/(\rs))}
\foreach \y in {1,...,\limb}
{
\pgfmathtruncatemacro{\xm}{\x-1}
\pgfmathtruncatemacro{\ym}{\y-1}
    \draw[dashed] (2^\xm,2^\ym)-- (2^\x,2^\ym) --(2^\x,2^\y) -- (2^\xm,2^\y) --
(2^\xm,2^\ym);
        \draw[dashed] (-2^\xm,-2^\ym)-- (-2^\x,-2^\ym) --(-2^\x,-2^\y) --
(-2^\xm,-2^\y) -- (-2^\xm,-2^\ym) ;
        \draw[dashed] (-2^\xm,2^\ym)-- (-2^\x,2^\ym) --(-2^\x,2^\y) --
(-2^\xm,2^\y) -- (-2^\xm,2^\ym) ;
        \draw[dashed] (2^\xm,-2^\ym)-- (2^\x,-2^\ym) --(2^\x,-2^\y) --
(2^\xm,-2^\y) -- (2^\xm,-2^\ym) ;
}
}

\foreach \x in {0,...,\limf}
{
\pgfmathtruncatemacro{\lim}{round((\m-\rf*\x)/(\rs))}
\foreach \y in {0,...,\lim}
{
\pgfmathtruncatemacro{\xm}{\x-1}
\pgfmathtruncatemacro{\ym}{\y-1}
\ifnum \x<1 \relax%
     \ifnum \y<1 \relax%
         draw[color=black,very thick] (-1,-1)-- (1,-1) --(1,1) -- (-1,1) --
(-1,-1);
    \else
        \draw[color=black,very thick] (-1,-2^\ym)-- (1,-2^\ym) --(1,-2^\y) --
(-1,-2^\y) -- (-1,-2^\ym) ;
        \draw[color=black,very thick] (-1,2^\ym)-- (1,2^\ym) --(1,2^\y) --
(-1,2^\y) -- (-1,2^\ym) ;
     \fi
\else%
    \ifnum \y<1 \relax%
            \ifnum \x<1 \relax%

            \else
            \draw[color=black,very thick] (2^\xm,-1)-- (2^\x,-1) --(2^\x,1) --
(2^\xm,1) -- (2^\xm,-1) ;
            \draw[color=black,very thick] (-2^\xm,-1)-- (-2^\x,-1) --(-2^\x,1)
--
(-2^\xm,1) -- (-2^\xm,-1);
            \fi
    \else
    \draw[color=black,very thick] (2^\xm,2^\ym)-- (2^\x,2^\ym) --(2^\x,2^\y) --
(2^\xm,2^\y) -- (2^\xm,2^\ym);
        \draw[color=black,very thick] (-2^\xm,-2^\ym)-- (-2^\x,-2^\ym)
--(-2^\x,-2^\y) -- (-2^\xm,-2^\y) -- (-2^\xm,-2^\ym) ;
        \draw[color=black,very thick] (-2^\xm,2^\ym)-- (-2^\x,2^\ym)
--(-2^\x,2^\y) -- (-2^\xm,2^\y) -- (-2^\xm,2^\ym) ;
        \draw[color=black,very thick] (2^\xm,-2^\ym)-- (2^\x,-2^\ym)
--(2^\x,-2^\y) -- (2^\xm,-2^\y) -- (2^\xm,-2^\ym) ;
    \fi
\fi

}}

\end{tikzpicture} 
  \caption{A step hyperbolic cross $Q_n$ in $d=2$}\index{Hyperbolic cross}
 \end{center}
\end{figure}

For $f\in L_1 (\T^d)$
$$
\delta_{\mathbf s} (f,\mathbf x) :=\sum_{\mathbf k\in\rho(\mathbf s)}
\hat f(\mathbf k)e^{i(\mathbf k,\mathbf x)}.
$$

Let $G$ be a finite set of points in $\mathbb Z^d$, we denote
$$
\Tr(G) :=\left\{ t : t(\mathbf x) =\sum_{\mathbf k\in G}c_{\mathbf k}
e^{i(\mathbf k,\mathbf x)}\right\} .
$$

For the sake of simplicity we shall write  
$\Tr\bigl(\Gamma(N)\bigr) = \Tr(N)$. The unit $L_p$-ball in $\Tr(G)$ we
denote by $\Tr(G)_p$ and in addition
$$
\Tr^{\perp}(G) := \bigl\{ g\in L_1 :\quad\text{ for all }\quad
f\in \Tr(G) ,\qquad \<f,g\> = 0 \bigr\} .
$$

As above for $G = \Gamma(N)$ we write $\Tr^{\perp}(N)$ instead of $\Tr^{\perp}(\Gamma(N))$.
We shall use the following simple relations (recall the notation in \eqref{nutheta})
\begin{equation}\label{hcsums}
\begin{split}
\bigl|\Gamma(N) \bigr|&\asymp N(\log N)^{d-1};\qquad
|Q_n|\asymp 2^n n^{d-1};\\
\sum_{\mathbf k>\mathbf0,\nu(\mathbf k)>N}\nu(\mathbf k)^{-r}&\asymp
N^{-r+1}(\log N)^{d-1},\qquad r>1 ;\\
\sum_{|\bs|_1>n}2^{-r|\bs|_1}&\asymp 2^{-rn}n^{d-1},
\qquad r>0.
\end{split}
\end{equation}
Note, that the sum in the middle can be rewritten (via dyadic blocks) to a sum $\sum_{|\bs|_1>n}2^{-(r-1)|\bs|_1}$ with $|n-\log N| \leq c$. Sums of this type
have been discussed in detail in Lemmas A -- D in the introduction of \cite{Tmon}. Refined estimates for the cardinality of hyperbolic crosses 
of any kind in high dimensions can be found in the recent papers \cite[Lem.\ 3.1, 3.2, Thm.\ 4.9]{KSU15} and \cite{CD13}, see also \eqref{cardhyp} in Section \ref{Sect:highdim}
below. 

It is easy to see that
\begin{equation}\nonumber
Q_n\subset \Gamma(2^n )\subset Q_{n+d}.
\end{equation}
Therefore it is enough to prove a number of properties of polynomials
such as the Bernstein and Nikol'skii inequalities for one set
$\Tr(Q_n)$ or $\Tr(N)$ only.

We shall consider the following trigonometric polynomials.

{\bf 1h. The analogs of the Dirichlet kernels.}\index{Kernel!Dirichlet} Consider
$$
\mathcal D_{Q_n} (\mathbf x) :=\sum_{\mathbf k\in Q_n}e^{i(\mathbf k,\mathbf x)}=
\sum_{|\bs|_1\le n}\mathcal D_{\rho(\mathbf s)}(\mathbf x) ,
$$
where $\mathcal D_{\rho(\mathbf s)}(\mathbf x) :=\sum_{\mathbf k\in\rho(\mathbf s)}
e^{i(\mathbf k,\mathbf x)}$.
It is clear that for $t\in \Tr(Q_n )$,
$$
t * \mathcal D_{Q_n}  = t .
$$
 We have the following behavior of the $L_p$ norms of the Dirichlet kernels
(see \cite{TBook}, Ch.3, Lemma 1.1).
\begin{lem}\label{L2.3.1} Let $1 < p < \infty $. Then
$$
\bigl\|\mathcal D_{Q_n}(\mathbf x)\bigr\|_p\asymp
2^{(1-1/p)n}n^{(d-1)/p}.
$$
\end{lem}

{\bf 2h. The analogs of the de la Vall\'ee Poussin kernels.}\index{Kernel!de la Vall\'ee Poussin} Let
$\mathcal A_{\mathbf s} (x)$ be
the polynomials which have been defined above.
These polynomials are from $\Tr(2^{\mathbf s} ,d)$ and
\begin{equation}\nonumber
\hat{\mathcal A}_{\mathbf s} (\mathbf k)\ne 0\quad\text{ only for }\quad
\mathbf k\quad:\quad 2^{\mathbf s_j-2}< |k_j| <
2^{s_j},\qquad j=1,\dots,d.
\end{equation}

We define the polynomials
$$
\mathcal V_{Q_n}(\mathbf x):=\sum_{|\bs|_1\le n}\mathcal A_{\mathbf s} (\mathbf
x) .
$$
These are polynomials in $\Tr(Q_n)$ with the property
\begin{equation}\nonumber
\hat{\mathcal V}_{Q_{n+d}}(\mathbf k) = 1\quad\text{ for }\quad \mathbf k\in Q_n.
\end{equation}
We shall use the following notation. Let $f\in L_1$
$$
S_{Q_n} (f) := f* \mathcal D_{Q_n},
$$
$$
V_{Q_n} (f) := f* \mathcal V_{Q_n},
$$
$$
A_{\mathbf s} (f) := f* \mathcal A_{\mathbf s}  .
$$

From Corollary \ref{C7.1} to the Littlewood-Paley theorem (see Appendix) it follows
that for $1 < p < \infty $
\begin{equation}\nonumber
\|S_{Q_n}\|_{p\to p}\le C(p,d) .
\end{equation}

In Subsection \ref{multpol} it was established that the $L_1$-norms of the de
la Vall\'ee Poussin kernels for parallelepipeds  are uniformly bounded. This
fact plays an essential role in studying approximation problems in the $L_1$ and
$L_{\infty}$ norms. The following lemma shows that, unfortunately, the kernels
$\mathcal V_{Q_n}$
have no such property (see \cite{TBook}, Ch.3, Lemma 1.2).

\begin{lem}\label{L2.3.2} Let $1 \le p < \infty $. Then   the following relation
$$
\bigl\|\mathcal V_{Q_n}(\mathbf x)\bigr\|_p\asymp
2^{(1-1/p)n}n^{(d-1)/p}.
$$
holds.
\end{lem}

Lemma \ref{L2.3.2} highlights a new phenomenon for hyperbolic cross polynomials -- 
there is no analogs of the de la Vall\'ee Poussin kernels for the hyperbolic
crosses with uniformly bounded $L_1$ norms. In particular, it follows from the
inequality: For any $\ep>0$ there is a number $C_\ep$ such that for all $t\in
\Tr(N)$ (see \cite{Tmon}, Ch.1, Section 2)
\begin{equation}
\sum_{\bk \in \Ga(N)} |\hat t(\bk)| \le C_\ep(\ln N)^\ep N \|t\|_1.
\end{equation}
This new phenomenon substantially complicates the study of approximation by
hyperbolic cross polynomials in the $L_1$ and $L_\infty$ norms. The reader can
find a discussion of related questions in \cite{Tmon}, Chapter 2, Section 5. 

\subsection{The Bernstein-Nikol'skii inequalities  }

{\bf 1. The Bernstein inequalities.}\index{Inequality!Bernstein}
We define the operator $D_{\alpha}^r$,  $r\ge 0$,  $\alpha\in\mathbb R$,
on the set of
trigonometric polynomials as follows:  let $f\in \Tr(n)$;  then
\begin{equation}\label{2.4.1}
D_{\alpha}^r f = f^{(r)} (x, \alpha):=
f(x) * \mathcal V_n^r (x, \alpha),
\end{equation}
\begin{equation}\label{2.4.2}
 \mathcal V_n^r (x, \alpha):=1+2\sum_{k=1}^n k^r\cos(kx
+\al\pi/2)+2\sum_{k=n+1}^{2n-1} k^r(1-(k-n)/n)\cos(kx +\al\pi/2)
\end{equation}
and $f^{(r)} (x, \alpha)$ will be called the $(r, \alpha)$ derivative.
It is clear that
for $f(x)$ such that $\hat f(0) = 0$ we have for natural numbers $r$,
$$
D_r^r f =\frac{d^r}{dx^r} f.
$$

The operator $D_{\alpha}^r$ is defined in such a way that it has an
inverse operator for each $\Tr(n)$. This property distinguishes
$D_{\alpha}^r$ from the differential operator and it will
be convenient for us. On the other hand it is clear that
$$
\frac{d^rf}{dx^r}= D_r^r f - \hat f(0).
$$

\begin{thm}\label{T2.4.1} For any $t\in \Tr(n)$ we have {\rm(}$r > 0$,
$\alpha\in\mathbb R$,  $1\le p\le \infty${\rm)}
$$
\bigl\| t^{(r)} (x, \alpha) \bigr\|_p \le C(r) n^r \| t \|_p,
\qquad n = 1, 2, \dots .
$$
\end{thm}

Theorem \ref{T2.4.1} can be easily generalized   to the
multidimensional case of trigonometric polynomials from $\Tr({\mathbf N},d)$.
 Let $\mathbf r = (r_1 ,\dots,r_d)$, $r_j\ge 0$, $j = 1,\dots,d$,
$\alpha = (\alpha_1,\dots,\alpha_d)$,
$\mathbf N = (N_1,\dots,N_d)$. We consider the polynomials
\begin{equation}\nonumber
\mathcal V_{\mathbf N}^{\mathbf r} (\mathbf x,\alpha) =
\prod_{j=1}^d  \mathcal V_{N_j}^{r_j}   (x_j  ,\alpha_j) ,
\end{equation}
where $\mathcal V_{N_j}^{r_j}   (x_j  ,\alpha_j)$ are defined in (\ref{2.4.2}).

We  define  the  operator  $D_{\alpha}^{\mathbf r}$ on the set of
trigonometric
polynomials as follows: let $f\in \Tr(\mathbf N,d)$, then
$$
D_{\alpha}^{\mathbf r}f:=  f^{(\mathbf r)}(\mathbf x,\alpha)
:=  f(\mathbf x) * \mathcal V_{\mathbf N}^{\mathbf r}
(\mathbf x,\alpha) ,
$$
and we  shall  call  $D_{\alpha}^{\mathbf r}   f$ the
$(\mathbf r,\alpha)$-derivative. In  the  case  of
identical components $r_j=r$, $j=1,\dots,d$, we shall write the
scalar $r$ in place of the vector.

\begin{thm}\label{T2.4.2} Let $\mathbf r\ge \mathbf 0$ and
$\alpha\in \R^d$ be such that for $r_j  =  0$ we
have $\alpha_j=0$. Then for any $t\in \Tr(\mathbf N,d)$,
$\mathbf N >\mathbf 0$, the inequality
$$
\bigl\|t^{(\mathbf r)}   (\cdot ,\alpha)\bigr\|_{\mathbf p}\le
C(\mathbf r)\|t\|_{\mathbf p}\prod_{j=1}^d N_j^{r_j},\qquad
\mathbf 1\le \mathbf p\le \infty ,
$$
holds.
\end{thm}
It is easy to check that the above Bernstein inequalities are sharp. Extension of
Theorem \ref{T2.4.2} to the case of hyperbolic cross polynomials is nontrivial
and brings out a new phenomenon. 
\begin{thm}\label{T2.4.3} For arbitrary $\alpha$
$$
\sup_{t\in \Tr(N)}\bigl\|t^{(r)}(\mathbf x,\alpha)\bigr\|_p\bigm/ \|t\|_p
\asymp\begin{cases} N^r\quad&\text{ for }1<p<\infty,\quad r\ge 0.\\
N^r(\log N)^{d-1}\quad&\text{ for }p=\infty,\qquad r>0.
\end{cases}
$$
\end{thm}
The Bernstein inequalities in Theorem \ref{T2.4.3} have different form for
$p=\infty$ and $1<p<\infty$. The upper bound in the case $p=\infty$ was obtained
by Babenko \cite{Bab2}. The matching lower bound for $p=\infty$ was proved by
Telyakovskii \cite{Tel1}. The case $1<p<\infty$ was settled by Mityagin
\cite{Mit}. 
The right form of the Bernstein inequalities in case $p=1$ is not known. 
It was proved in \cite{TE4} that in the case $d=2$
$$
\sup_{t\in \Tr(N)}\bigl\|t^{(r)}(\mathbf x,\alpha)\bigr\|_1\bigm/ \|t\|_1 \lesssim (\ln N)^{1/2}N^r.
$$

\noindent{\bf 2. The Nikol'skii inequalities.}\index{Inequality!Nikol'skii}
The following inequalities are well known and easy to prove. 
\begin{thm}\label{T2.4.4} For any $t\in \Tr(n)$,  $n > 0$,
we have the inequality
$$
\| t \|_p\le C n^{1/q-1/p} \| t \|_q, \qquad 1\le q < p\le \infty.
$$
\end{thm}
The above univariate inequalities can be extended to the case of polynomials from $\Tr(\bN,d)$. 
\begin{thm}\label{T2.4.5} For any $t\in \Tr(\mathbf N,d)$, $\mathbf N >\mathbf 0$
the  following  inequality
holds $(\mathbf 1\le \mathbf q\le \mathbf p\le \infty)$:
$$
\|t\|_{\mathbf p}\le C(d) \|t\|_{\mathbf q}\prod_{j=1}^d N_j^{1/q_j-1/p_j}.
$$
\end{thm}
We formulate the above inequalities for vector ${\mathbf p}$ and $\mathbf q$
because in this form they are used to prove embedding type inequalities.
We proceed to the problem of estimating $\|f\|_p$ in terms of
the array $\bigl\{ \|\delta_{\mathbf s} (f)\|_q  \bigr\}_{\bs \in \N_0^d}$ where Theorem
\ref{T2.4.5} is heavily used.
Here and below $p$ and $q$  are  scalars  such
that $1\le q,p\le \infty$. Let an array
$\varepsilon = \{\varepsilon_{\mathbf s}\}$ be given, where
$\varepsilon_{\mathbf s}\ge 0$, $\mathbf s = (s_1 ,\dots,s_d)$,
and $s_j$ are nonnegative integers, $j = 1,\dots,d$.
We denote by $G(\varepsilon,q)$ and $F(\varepsilon,q)$
the following sets of functions
$(1\le q\le \infty)$:
$$
G(\varepsilon,q) := \bigl\{ f\in L_q  : \bigl\|\delta_{\mathbf s} (f)
\bigr\|_q
\le\varepsilon_{\mathbf s}\qquad\text{ for all }\mathbf s\bigr\} ,
$$
$$
F(\varepsilon,q) := \bigl\{ f\in L_q  : \bigl\|\delta_{\mathbf s} (f)
\bigr\|_q
\ge\varepsilon_{\mathbf s}\qquad\text{ for all }\mathbf s\bigr\}.
$$

 \begin{thm}\label{T2.4.6} The following relations hold:
\begin{equation}\label{2.4.4}
\sup_{f\in G(\varepsilon,q)}\|f\|_p\asymp\left(\sum_{\mathbf s}
\varepsilon_{\mathbf s}^p2^{|\bs|_1(p/q-1)}\right)^{1/p},
\qquad 1\le q < p < \infty ;
\end{equation}
\begin{equation}\label{2.4.5}
\inf_{f\in F(\varepsilon,q)}\|f\|_p\asymp\left(\sum_{\mathbf s}
\varepsilon_{\mathbf s}^p  2^{|\bs|_1(p/q-1)}\right)^{1/p},
\qquad 1< p < q\le\infty ,
\end{equation}
with constants independent of $\varepsilon$.
\end{thm}
Theorem \ref{T2.4.6} was proved in \cite{Tem10} (see also \cite{Tmon}, Ch.1, Theorem 3.3).
Theorem \ref{T2.4.6} can be formulated in the form of embeddings: relation (\ref{2.4.4}) implies 
Lemma \ref{L3.5c} and relation (\ref{2.4.5}) implies Lemma \ref{L3.5d} (see Section 3 below). 
\begin{rem}\label{RT2.4.6} 
In the proof of the upper bound in (\ref{2.4.4}) from \cite{Tmon} we  used  only
the property $\delta_{\bs} (f)\in \Tr(2^{\bs},d)$. That is, if
$$
f =\sum_{\bs}t_{\bs},\qquad t_{\bs}\in \Tr(2^{\bs},d) ,
$$
then for $1\le q < p < \infty$,
\begin{equation}\label{2.4.4R}
\|f\|_p\le C(q,p,d) \left (\sum_{\bs}\|t_{\bs}\|_q^p
2^{\|\bs\|_1(p/q-1)}\right)^{1/p}    .
\end{equation}
\end{rem}
The above Remark \ref{RT2.4.6} is from \cite{Di91}. This remark is very useful in studying 
sampling recovery by Smolyak's algorithms (see Section 5). 

The Nikol'skii inequalities\index{Inequality!Nikol'skii} for polynomials from $\Tr(N)$ are nontrivial in the case $q=1$.
The following two theorems are from \cite{Tmon}, Ch.1, Section 2.
\begin{thm}\label{T2.4.7} Suppose that $1 \le q < \infty $ and $r\ge 0$. Then
$$
\sup_{t\in \Tr(N)}\bigl\|t^{(r)}(\mathbf x,\alpha)\bigr\|_{\infty}
\bigm/ \|t\|_q\asymp N^{r+1/q}(\log N)^{(d-1)(1-1/q)}.
$$
\end{thm}

\begin{thm}\label{T2.4.8} Suppose that $1 \le q \le p < \infty $, $p > 1$,
$r \ge 0$. Then
$$
\sup_{t\in \Tr(N)}\bigl\|t^{(r)}(\mathbf x,\alpha)\bigr\|_p\bigm /
\|t\|_q\asymp N^{r+1/q-1/p}.
$$
\end{thm}

\noindent{\bf 3. The Marcinkiewicz theorem.}\index{Theorem!Marcinkiewicz}
The set $\Tr(n)$ of trigonometric polynomials is a space of
dimension $2n+1$. Each polynomial $t\in \Tr(n)$ is uniquely defined by its
Fourier coefficients $\bigl\{\hat t(k)\bigr \}_{|k|\le n}$
and by the Parseval identity we have
\begin{equation}\nonumber
\| t \|_2^2 = \sum_{|k|\le n}\bigl|\hat t(k) \bigr|^2,
\end{equation}
which means that the set $\Tr(n)$ as a subspace of $L_2$ is isomorphic to
$\ell_2^{2n+1}$. The relation (\ref{2.1.5}) shows that a similar isomorphism
can be set up in another way:  mapping a polynomial $t\in \Tr(n)$ to the vector
$m(t) := \bigl\{ t(x^l) \bigr\}_{l=0}^{2n}$ of its values at the points
$$
x^l = 2\pi l/(2n+1), \qquad l = 0, \dots, 2n.
$$
The relation (\ref{2.1.5}) gives
$$
\| t \|_2 = (2n+1)^{-1/2} \| m(t) \|_2.
$$

The following statement is the Marcinkiewicz theorem.

\begin{thm}\label{T2.4.9} Let $1 < p < \infty$;  then for $t\in \Tr(n)$,
 $n > 0$,  we have the relation
$$
C_1 (p) \| t \|_p\le n^{-1/p} \bigl\| m(t) \bigr\|_p\le C_2 (p) \| t \|_p.
$$
\end{thm}

The following statement  is analogous to Theorem \ref{T2.4.9} but in
contrast to it includes the cases $p=1$ and $p=\infty$.

\begin{thm}\label{T2.4.10} Let $x(l) = \pi l/(2n)$, $l = 1, \dots, 4n$,\newline
$M(t) :=\bigl(t\bigl(x(1)\bigr), \dots, t\bigl(x(4n)\bigr)\bigr)$.
Then for an arbitrary $t\in \Tr(n)$,  $n > 0$,
$1\le p\le \infty$,
$$
C_1\| t \|_p\le n^{-1/p} \bigl\| M(t) \bigr\|_p\le C_2 \| t \|_p.
$$
\end{thm}

Similar inequalities hold for polynomials from $\Tr(\bN,d)$.  We formulate  the
 equivalence  of  a
mixed norm of a trigonometric polynomial to its mixed lattice norm.
We use the notation
$$
\ell(\mathbf N,d)=\bigl\{\mathbf a=\{a_{\mathbf n}\},\qquad
\mathbf n=(n_1,\dots,n_d),\qquad 1\le n_j\le N_j,\qquad
j = 1,\dots,d \bigr\} ,
$$
and for $\mathbf a\in \ell(\mathbf N,d)$ we define the mixed norm
$$
\|\mathbf a\|_{\mathbf p,\mathbf N} := \left (\sum_{n_d=1}^{N_d}N_d^{-1}
\left (\dots\left (\sum_{n_1=1}^{N_1}N_1^{-1}|a_{\mathbf n}|^{p_1}
\right)^{p_2/p_1}\dots\right)^{p_d/p_{d-1}}\right)^{1/p_d}
$$
\begin{equation}\nonumber
 = \|\mathbf a\|_{\mathbf p}\prod_{j=1}^d  N_j^{-1/p_j}.
\end{equation}

\begin{thm}\label{T2.4.11} Let $t\in \Tr(\mathbf N,d)$, $\mathbf N > \mathbf 0$.
Then for any $\mathbf 1\le\mathbf p\le \infty$,
$$
C_d^{-1}\bigl\|\bigl\{t(\mathbf x(\mathbf n))\bigr\}_{\mathbf n\in P'(\mathbf N)}
\bigr\|_{\mathbf p,4\mathbf N}\le\|t\|_{\mathbf p}\le C_d
\bigl\|\bigl\{t(\mathbf x(\mathbf n))\bigr\}_{\mathbf n\in P'(\mathbf N)}
\bigr\|_{\mathbf p,4\mathbf N},
$$
where $C_d$  is a number depending only on $d$.
\end{thm}

In the case  $\mathbf p  =  (p,\dots,p)$ this  theorem  is  an  immediate
corollary of the  corresponding one-dimensional Theorem   \ref{T2.4.10}.

It turns out that there is no analog of Theorem \ref{T2.4.11} for polynomials
from $\Tr(N)$. We discuss this issue in the next section.

\subsection{Volume estimates}\index{Volume estimates}

\label{vol_est}

  Let $\mathbf s=(s_1,\dots,s_d)$ be a vector with nonnegative integer
coordinates ($\mathbf s \in {\mathbb N}_0^d$).  Denote for a natural number $n$
$$
  \Delta Q_n := Q_n\setminus Q_{n-1} = \cup_{|\bs|_1=n}\rho(\mathbf s)
$$
with $|\bs|_1 = s_1+\dots+s_d$ for $\bs \in {\mathbb N}_0^d$. We call a
set $\Delta Q_n$ {\it hyperbolic layer}. For a set $\La \subset \Z^d$ denote
$$
\Tr(\La) := \{f\in L_1:  \hat f(\mathbf k)=0, \mathbf k\in \mathbb Z^d\setminus \La\} .
$$
For a finite set $\La$ we assign to each $f = \sum_{\mathbf k\in \La} \fk
e^{i(\mathbf k,\mathbf x)}\in \Tr(\La)$ a vector
$$
A(f) := \{(\text{Re}\fk, \text{Im}\fk),\quad \mathbf k\in \La\} \in \R^{2|\L|}
$$
where $|\La|$ denotes the cardinality of $\La$ and define
$$
B_\La(L_p) := \{A(f) : f\in \Tr(\La),\quad \|f\|_p \le 1\}.
$$

In the case $\La =\Pi(\mathbf N):=\Pi(\mathbf N,d):=
[-N_1,N_1]\times\cdots\times[-N_d,N_d]$, $\mathbf N:=(N_1,\dots,N_d)$, the
following volume estimates are known.  
\begin{thm}\label{T2.5.1} For any $1\le p\le \infty$ we have
$$
(vol(B_{\Pi(\mathbf N)}(L_p)))^{(2|\Pi(\mathbf N)|)^{-1}} \asymp |\Pi(\mathbf N)|^{-1/2},
$$
with constants in $\asymp$ that may depend only on $d$.
\end{thm}

We note that the most difficult part of Theorem \ref{T2.5.1} is the lower
estimate for $p=\infty$. The corresponding estimate was proved in the case $d=1$
in \cite{KaE} and in the general case in \cite{TE2} and \cite{T6}. The upper
estimate for $p=1$ in Theorem \ref{T2.5.1} can be easily reduced to the volume
estimate for an octahedron (see, for instance \cite{T4'}). 

The results of \cite{KTE1} imply the following estimate.
\begin{thm}\label{T2.5.2} For any finite set $\La\subset \Z^d$ and any $1\le p\le 2$ we have
$$
vol(B_\La(L_p))^{(2|\La|)^{-1}} \asymp |\La|^{-1/2}.
$$
\end{thm}

The following result of Bourgain-Milman \cite{BM} plays an important role in
the volume estimates of finite dimensional bodies. 
\begin{thm}\label{T2.5.3} For any convex centrally symmetric body $K\subset \R^n$ we have
$$
(vol(K)vol(K^o))^{1/n} \asymp (vol(B^n_2))^{2/n}\asymp 1/n
$$
where $K^o$ is a polar for $K$, that is
$$
K^o:= \{\bx\in \R^n:\sup_{y\in K}(\bx,\by) \le 1\}.
$$
\end{thm}

\begin{rem} For the case of $\ell_p^n$ balls with $1\leq p\leq \infty$ or, more
general, unit balls of symmetric norms on $\R^n$ and their duals, Theorem
\ref{T2.5.3} 
has been proved earlier by Sch\"utt \cite{Schu}.
 
\end{rem}

\noindent The following result is from \cite{KTE3}.
\begin{thm}\label{T2.5.4} Let $\La$ have the form
$\La = \cup_{\mathbf s\in S}\rho(\mathbf s)$, $S\subset {\mathbb N}_0^d$ is
a finite set. Then for any $1\le p<\infty$ we have
$$
vol(B_\La(L_p))^{(2|\La|)^{-1}} \asymp |\La|^{-1/2}.
$$
\end{thm}

We now proceed to results for the case $d=2$. Denote $N:= 2|\Delta Q_n|$. Let
$$
E^\perp_\La(f)_p :=\inf_{g\perp \Tr(\La)}\|f-g\|_p,
$$
$$
B^\perp_\La(L_p) := \{A(f): f\in \Tr(\La),\quad E^\perp_\La(f)_p\le 1\}.
$$
\begin{thm}\label{T2.5.5} In the case $d=2$ we have for $N:=2|\Delta Q_n)|$
\begin{equation}\label{2.5.1}
(vol(B_{\Delta Q_n}(L_\infty)))^{1/N} \asymp (2^nn^2)^{-1/2}; 
\end{equation}
\begin{equation}\nonumber
(vol(B_{\Delta Q_n}^\perp(L_1)))^{1/N} \asymp 2^{-n/2}. 
\end{equation}
\end{thm}

It is interesting to compare the first relation in Theorem \ref{T2.5.5} with
the following estimate for $1\le p<\infty$ that follows from Theorem
\ref{T2.5.4}
\begin{equation}\nonumber
(vol(B_{\Delta Q_n}(L_p)))^{1/N} \asymp (2^nn)^{-1/2}. 
\end{equation}
We see that in the case $\La = \Delta Q_n$ unlike the case
$\La=\Pi(N_1,\dots,N_d)$ the estimate for $p=\infty$ is different from the
estimate for $1\le p<\infty$.

{\bf The discrete $L_\infty$-norm for polynomials from $\Tr(\L)$.}

 We present here some results from \cite{KTE3} (see also \cite{KaTe98} and \cite{KaTe99}). We begin with the following conditional statement.
\begin{thm}\label{T2.5.6} Assume that a finite set $\La\subset \Z^d$ has the following properties:
\begin{equation}\nonumber
(vol(B_\La(L_\infty)))^{1/N} \le K_1N^{-1/2},\quad N:=2|\L|, 
\end{equation}
and a set $\Omega =\{x^1,\dots,x^M\}$ satisfies the condition
\begin{equation}\nonumber
\forall f \in \Tr(\L) \qquad \|f\|_\infty\le K_2\|f\|_{\infty,\Omega},\quad
\|f\|_{\infty,\Omega}:=\max_{x\in \Omega}|f(x)|. 
\end{equation}
Then there exists an absolute constant $C>0$ such that
$$
M\ge Ne^{C(K_1K_2)^{-2}}.
$$
\end{thm}

We now give some corollaries from Theorem \ref{T2.5.6}. 

\begin{thm}\label{T2.5.7} Assume a finite set $\Omega\subset \mathbb T^2$ has
the following property.
\begin{equation}\label{2.5.6}
\forall t\in \Tr(\Delta Q_n) \qquad \|t\|_\infty \le K_2\|t\|_{\infty,\Omega}. 
\end{equation}
Then
$$
|\Omega| \ge 2|\Delta Q_n|e^{Cn/K_2^2}
$$
with an absolute constant $C>0$.
\end{thm}
\begin{proof} By Theorem \ref{T2.5.5} (see (\ref{2.5.1})) we have
$$
(vol(B_{\Delta Q_n}(L_\infty)))^{1/N} \le C(2^nn^2)^{-1/2} \le Cn^{-1/2}N^{-1/2}
$$
with an absolute constant $C>0$. Using Theorem \ref{T2.5.6} we obtain
$$
|\Omega|\ge 2|\Delta Q_n|e^{Cn/K_2^2}.
$$
This proves Theorem \ref{T2.5.7}.
\end{proof}
\begin{rem}\label{R2.5.1} In a particular case $K_2=bn^\al$, $0\le \al\le 1/2$, Theorem \ref{T2.5.7} gives
$$
|\Omega|\ge 2|\Delta Q_n|e^{Cb^{-2}n^{1-2\al}}.
$$
\end{rem}
\begin{cor}\label{C2.5.1} Let a set $\Omega\subset \mathbb T^d$ have a property:
 $$
\forall t \in \Tr(\Delta Q_n) \qquad \|t\|_\infty \le bn^\al\|t\|_{\infty,\Omega}
$$
with some $0\le \al <1/2$. Then 
$$
|\Omega| \ge C_32^nne^{Cb^{-2}n^{1-2\al}} \ge C_1(b,d,\al)|Q_n|e^{C_2(b,d,\al)n^{1-2\al}}.
$$
\end{cor}

\begin{cor}\label{C2.5.2} Let a set $\Omega \subset \mathbb T^2$ be such that $|\Omega|\le C_5|Q_n|$. Then
$$
\sup_{f\in \Tr(Q_n)}\|f\|_\infty/\|f\|_{\infty,\Omega} \ge Cn^{1/2}.
$$
\end{cor}
\begin{proof} Denote
$$
K_2 := \sup_{f\in \Tr(Q_n)}\|f\|_\infty/\|f\|_{\infty,\Omega}.
$$
Then the condition (\ref{2.5.6}) of Theorem \ref{T2.5.7} is satisfied with this $K_2$. Therefore,
 by Theorem \ref{T2.5.7}
$$
2|\Delta Q_n|e^{Cn/K_2^2} \le |\Omega| \le C_5|Q_n|.
$$
This implies that
$$
K_2\gtrsim n^{1/2}.
$$
\end{proof}
\begin{rem}\label{R2.5.2} One can derive from the known results on recovery of
functions from the classes $\bW^r_\infty$ (see \cite{T7}, \cite{T8}) that for
any $n$ 
there is a set $\Omega_n\subset \T^d$ such that $|\Omega_n| \le C|Q_n|$ and
$$
\sup_{f\in \Tr(Q_n)}(\|f\|_\infty/\|f\|_{\infty,\Omega_n})\lesssim n^{d-1}.
$$
\end{rem}
For further results in this direction we refer the reader to a very recent paper \cite{TeTi16}.

\subsection{Riesz products and the Small Ball Inequality}\index{Riesz products}\index{Inequality!Small Ball}
\label{subsect:SBI}

We consider the special
trigonometric polynomial, which falls into a category of Riesz products (see \cite{Tem3})
$$
\Phi_m(\bx) :=\prod_{k=0}^m (1 + \cos 4^kx_1 \cos 4^{m-k}x_2).
$$
The above polynomial was the first example of the hyperbolic cross Riesz products.
Clearly, $\Phi_m(\bx) \ge 0$. It is known that
\begin{equation}\label{2.6.1}
\Phi_m(\bx) = 1 +\sum_{k=0}^{m}\cos 4^kx_1 \cos 4^{m-k}x_2+
t_m (\bx),\qquad t_m\in \mathcal{T}^{\perp}(4^m).
\end{equation}
In particular, relation (\ref{2.6.1}) implies that
$$
\|\Phi_m\|_1 = 1 .
$$
We now consider a more general Riesz product in the case $d=2$ (see \cite{TE3}
and \cite{TE4}). For any two given integers $a\ge 1$ and $0\le b<a$ denote
$AP(a,b)$ the arithmetic progression of the form
$al+b$, $l=0,1,\dots$. Set
$$
H_n(a,b):=\{\bs: \bs\in \N^2_0,\quad |\bs|_1=n,\quad s_1,s_2\ge a,\quad s_1\in
AP(a,b)\}.
$$
It will be convenient for us to consider subspaces $\Tr(\varrho'(\bs))$ of trigonometric polynomials with harmonics in
$$
\varrho'(\bs) :=\{\bk\in \Z^2:[2^{s_j-2}]\le |k_j|<2^{s_j},\quad j=1,2\}.
$$
For a subspace $Y\subset L_2(\T^2)$ we denote $Y^\perp$ its orthogonal complement.

\begin{lem}\label{L2.6.1} Take any trigonometric polynomials $t_\bs\in \Tr(\varrho'(\bs))$ and form the Riesz product
$$
\Phi(n,\bx):=\prod_{\bs\in H_n(a,b)}(1+t_\bs(\bx)).
$$
Then for any $a\ge 6$ and any $0\le b<a$ this function admits the representation
$$
\Phi(n,\bx) = 1+ \sum_{\bs\in H_n(a,b)} t_\bs(\bx) + R(\bx)
$$
with $R\in \Tr(Q_{n+a-6})^\perp$.
\end{lem}
\noindent Usually, Lemma \ref{L2.6.1} is used for $t_\bs(\bx)$ being real
trigonometric polynomials with $\|t_\bs(\bx)\|_\infty\le 1$.

The above Riesz products are useful in proving Small Ball Inequalities. We
describe these inequalities for the Haar and the trigonometric  systems. We
begin with formulating this inequality in the case $d=2$ using the dyadic
enumeration of the Haar system \index{Haar system} 
$$
H_I(\bx):=H_{I_1}(x_1)H_{I_2}(x_2),\quad \bx=(x_1,x_2),\quad I=I_1\times I_2.
$$
Talagrand's inequality \index{Talagrand's inequality} claims that for any coefficients $\{c_I\}$ (see \cite{TaE} and \cite{TE5})
\begin{equation}\label{2.6.2}
\begin{split}
\Big\|\sum_{I:|I|=2^{-n}}c_IH_I(\bx)\Big\|_\infty &\ge
 2^{-{(n+1)}}\sum_{I:|I|=2^{-n}}|c_I| \\
 &=\frac{1}{2}\sum\limits_{m=0}^n\Big\|\sum\limits_{I:|I_1|=2^{-m}, |I_2|=2^{m-n}}c_IH_I(\bx)\Big\|_1\,.
\end{split}
\end{equation}
where $|I|$ means the measure of $I$. 

We now formulate an analogue of (\ref{2.6.2}) for the trigonometric system.  For an even number $n$ define
$$
Y_n:=\{\bs=(2n_1,2n_2),\quad n_1+n_2=n/2\}.
$$
Then for any coefficients $\{c_\bk\}$ (see \cite{TE3})
\begin{equation}\label{2.6.3}
\Big\|\sum_{\bs\in Y_n} \sum_{\bk\in \varrho(\bs)}c_\bk
e^{i(\bk,\bx)}\Big\|_\infty \ge
C\sum_{\bs\in Y_n}\Big\|\sum_{\bk\in\varrho(\bs)} c_\bk e^{i(\bk,\bx)}\Big\|_1,
\end{equation}
where $C$ is a positive number. Inequality (\ref{2.6.3}) plays a key role in
the proof of lower bounds for the entropy numbers.

We proceed to the $d$-dimensional version of (\ref{2.6.3}), which we formulate below.  For even $n$, put
$$
Y^d_n:=\{\bs=(2l_1,\dots,2l_d), l_1+\dots+l_d=n/2, l_j\in \N_0, j=1,\dots,d\}.
$$
It is conjectured (see, for instance, \cite{KTE4}) that the following
inequality, which we call ``small ball inequality'', holds for any coefficients $\{c_\bk\}$
\begin{equation}\label{2.6.4}
n^{(d-2)/2}\Big\|\sum_{\bs\in Y^d_n}\sum_{\bk\in\varrho(\bs)}c_\bk
e^{i(\bk,\bx )}\Big\|_\infty \ge C(d)\sum_{\bs\in Y^d_n}\Big\|\sum_{\bk\in\varrho(\bs)}c_\bk e^{i( \bk,\bx)}\Big\|_1.
\end{equation}
\index{Inequality!Small Ball}
We note that a weaker version of (\ref{2.6.4}) with exponent $(d-2)/2$ replaced
by $(d-1)/2$ is a direct corollary of the Parseval's identity, the Cauchy
inequality and monotonicity of the $L_p$ norms. 

The $d$-dimensional version of the small ball inequality (\ref{2.6.2}), similar to the conjecture in (\ref{2.6.4}), reads as follows:
\begin{equation}\label{2.6.5}
 \begin{split}
    n^{(d-2)/2}\Big\|\sum_{I:|I|=2^{-n}}c_IH_I(\bx)\Big\|_\infty &\geq C(d)\sum_{\bs:s_1+\dots+s_d=n}\Big\|\sum_{I:|I_j|=2^{-s_j},j=1,\dots
,d}c_IH_I(\bx)\Big\|_1\\
&=C(d)2^{-n}\sum\limits_{I:|I|=2^{-n}}|c_I|\,.
 \end{split}
\end{equation}
Recently, the authors of \cite{BL} and \cite{BLV} proved (\ref{2.6.5}) with the
exponent $(d-1)/2-\delta(d)$ with some $0<\delta(d)<1/2$ instead of $(d-2)/2$.  See also its implications for Kolmogorov and entropy 
numbers of the mixed smoothness function classes in $L_\infty$ below and Subsection \ref{sect:ESBP}. Note, that there
is no progress in proving (\ref{2.6.4}). 

\subsection{Comments and open problems}\index{Open problems!Trigonometric polynomials} 

Sections 2.1 and 2.2 mostly contain classical results on univariate
trigonometric polynomials and their straight forward generalizations to the case
of multivariate trigonometric polynomials with frequencies from parallelepipeds.
Theorem \ref{T2.2.1} is from \cite{TE2}. Its proof is based on the volume
estimates from Theorem \ref{T2.5.1} and the classical Brun theorem on sections
of convex bodies. 

Lemma \ref{L2.3.1} is a direct corollary of Lemma 1.4 from \cite{Tem4} and the
Littlewood-Paley theorem (see Appendix, Corollary \ref{C7.1}). Another proof of
Lemma \ref{L2.3.1} was given in \cite{Ga2}. It is easy to derive Lemma
\ref{L2.3.1} from Theorem \ref{T2.4.6} (see \cite{TBook}, Ch.3). For Lemma
\ref{L2.3.2} see \cite{TBook}, Ch.3. 

Theorem \ref{T2.4.1} is the classical Bernstein inequality for the univariate
trigonometric polynomials. Theorem \ref{T2.4.2} is a straight forward
generalization of Theorem \ref{T2.4.1}. 
Theorem \ref{T2.4.3} is discussed above. 

{\bf Open problem 2.1.} Find the order of the quantity 
$$
\sup_{t\in \Tr(N)}\bigl\|t^{(r)}(\mathbf x,\alpha)\bigr\|_1\bigm/ \|t\|_1  
$$
as a function on $N$.

{\bf Open problem 2.2.} Find the order of the quantity 
$$
\sup_{t\in \Tr(N)}\sum_{\bk\in\Gamma(N)} |\hat t (\bk)|\bigm/ \|t\|_1  
$$
as a function on $N$.

Theorems \ref{T2.4.4} and \ref{T2.4.5} are classical Nikol'skii inequalities.
Theorem \ref{T2.4.6}, obtained in \cite{Tmon}, is an important tool in
hyperbolic cross approximation. 
Its proof in \cite{Tmon} is based on a nontrivial application of Theorem
\ref{T2.4.5} and the H{\"o}lder inequalities (\ref{7.3}). Theorems \ref{T2.4.7}
and \ref{T2.4.8} are from \cite{Tem4} (see also \cite{Tmon}). Theorems
\ref{T2.4.9} and \ref{T2.4.10} are classical variants of the Marcinkiewicz
theorem. Theorem \ref{T2.4.11} is from \cite{T32}. 

Historical comments on results from Subsection 2.5 are given in the text above. We
only formulate open problems in this regard. 

{\bf Open problem 2.3.} Find the order of 
$$
(vol(B_{\Delta Q_n}(L_\infty)))^{1/N} ,\qquad N:= 2|\Delta Q_n|,
$$
as $n\to\infty$, for $d\ge 3$. 

{\bf Open problem 2.4.} Find the order of 
$$
(vol(B_{\Delta Q_n}^\perp(L_1)))^{1/N} ,\qquad N:= 2|\Delta Q_n|,
$$
as $n\to\infty$, for $d\ge 3$.

Here are two fundamental open problems in connection with Subsection 2.6. Note, that Lemma \ref{L2.6.1} from \cite{TE4} plays 
the key role in the two-dimensional version of (\ref{2.6.4}).

{\bf Open problem 2.5.} Prove the Small Ball Inequality (\ref{2.6.4}) for $d\ge 3$. 

{\bf Open problem 2.6.} Prove the Small Ball Inequality (\ref{2.6.5}) for $d\ge 3$. 

%% file: function_spaces.tex
\section{Function spaces on $\T^d$}\index{Function spaces}
\label{perfs}

\subsection{Spaces of functions with bounded mixed derivative}\index{Function spaces!Bounded mixed derivative}
\index{Bounded mixed derivative}

We begin with the univariate case in order to illustrate the action of the differential operator on periodic functions. For a trigonometric polynomial $f\in \Tr(n)$ we have
$$
(D_x f)(x) =f'(x) = \sum_{|k|\le n} ik{\hat f}(k)e^{ikx}.
$$
We loose the information of $\hat f(0)$ when we differentiate. We can recover $f$ from $f'$ by the following formula
$$
f(x) = \hat f(0) + f'(x)\ast  \sum_{0<|k|\le n} (ik)^{-1}e^{ikx}.
$$
Note that 
$$
\sum_{0<|k|\le n} (ik)^{-1}e^{ikx} = 2\sum_{k=1}^n k^{-1}\sin kx.
$$
Therefore, the following two definitions of the class $W^1_p$, $1\leq p\leq \infty$, are equivalent

(D1)   $\{f: \|f\|_p\le 1,\, \|f'\|_p\le 1\}$;

(D2)    $\{f: f=\varphi \ast F_1,\, \|\varphi\|_p \le 1\}$,
\newline
where
$$
F_1(x) := 1+2\sum_{k=1}^\infty k^{-1}\sin kx = 1+2\sum_{k=1}^\infty k^{-1}\cos (kx-\pi/2).
$$
The second definition is more convenient than the first one for the following two reasons. It is easy to generalize it to the case of fractional (Weil) derivatives and it is easy to extend it to the multivariate case. We now give the general definition, which we use in this survey. 
This definition is based on the integral representation of a function by the
Bernoulli kernels\index{Kernel!Bernoulli}. Define for $x\in \T$ the univariate Bernoulli kernel
\[
F_{r,\alpha}(x):= 1+2\sum_{k=1}^\infty k^{-r}\cos (kx-\alpha \pi/2)
\]
 and define
the multivariate Bernoulli kernels as the corresponding tensor products \index{Tensor product}
\begin{equation}\label{Bernoulli}
F_{r,\alpha}(\bx):=\prod_{j=1}^dF_{r,\alpha_j}(x_j),\quad \bx=(x_1,\dots,x_d)\in
\T^d,\quad \alpha=(\alpha_1,\dots,\alpha_d).
\end{equation}

\begin{defi} \label{def2Sob}
Let $r>0$, $\alpha \in \R$ and $1\le p \le\infty$. Then $\bW^r_{p,\alpha}$ is defined as the
normed space of all $f\in L_p(\T^d)$ such that 
\[
f=  F_{r,\alpha}\ast
\varphi:=(2\pi)^{-d}\int_{\T^d}F_{r,\alpha}(\bx-\by)\varphi(\by)d\by 
\]
for some $\varphi \in L_p(\T^d)$, equipped with the norm
$\| \, f \, \|_{\bW^r_{p,\alpha}}:= \|\varphi\|_p$.
\end{defi}

It is well known and easy to prove that for all $r>0$ and $\alpha \in \R$ we
have $F_{r,\alpha}\in L_1$ (see \cite{TBook}, Ch.\ 1, Theorem 3.1). The extra
parameter $\alpha$ allows us to treat simultaneously classes of functions with
bounded mixed derivative and classes of functions with bounded trigonometric
conjugate of the mixed derivative. 
In the case $\alpha=r$, Definition \ref{def2Sob} is equivalent to the mentioned below generalization of the definition 
\eqref{def[Sobolev1]}--\eqref{def[Sobolev2]} in terms of the Weil fractional derivatives 
\eqref{def[WeilDerivative]}.
In the case $1<p<\infty$, the parameter
$\alpha$ does not play any role because the corresponding classes with different
$\alpha$ are equivalent. In the case of classical derivative with natural $r$ we
set $\alpha=r$. In the case $\alpha =(r,\dots,r)$ we drop it from the notation:
$\bW^r_p:=\bW^r_{p,(r,\dots,r)}$. For simplicity of notations we formulate the
majority of our results for classes $\bW^r_p$. In those cases, when $\alpha$
affects the result we point it out explicitly. 

We note that in the case $r\in \N$ and $1\leq p\leq \infty$ the above class (space) can be described 
in terms of mixed partial derivatives. The Sobolev space  $\Wrp$
of dominating mixed smoothness of order $r$ can be defined as the collection
of all $f \in L_p (\T^d)$ such that
\begin{equation} \label{def[Sobolev1]}
D^{\br(e)} f \in L_p (\T^d)\, ,\qquad \forall e \subset \{1,...,d\},
\end{equation}
where $\br(e)$ denotes the vector with components $r(e)_i = r$ for $i \in e$ and 
$r(e)_i = 0$ for $i \not\in e$. Derivatives have to be understood in the weak sense. 
We endow this space with the norm
\begin{equation} \label{def[Sobolev2]}
\|  f  \|_{\Wrp(\T^d)} := \sum_{e \subset \{1,...,d\}} |  f  |_{\Wrp(e)}, \qquad
|  f  |_{\Wrp(e)}
:= \
\| \, D^{\br(e)} f\, \|_p\,.
\end{equation}
This definition can be generalized to arbitrary $r \in \R$ based the $\br(e)$-Weil fractional derivatives
in the weak sense 
\begin{equation} \label{def[WeilDerivative]}
D^{\br(e)}f 
:= \
\sum_{\bk \in \Z^d(e)} \Big(\prod_{j \in e}(ik_j)^r \Big) \hat{f}(\bk) e^{i(\bk,\cdot)},
\end{equation}
where $(ia)^b := |a|^be^{(i\pi b \operatorname{sign} a)/2}$ for $a, b \in \R$, and  
$\Z^d(e) := \{\bk \in \Z^d:  k_j \neq 0, \ j \in e\}$ 
(see, e.g., \cite{Di00, DU13}).

For general $r > 0$ and $1 < p < \infty $ one may also use the condition
\[
\sum_{\bk \in \zz^d} \, \hat{f}(\bk) \, \prod_{j=1} \big(\max\{1,|k_j|\}\big)^r \,  e^{i(\bk,\cdot)}
\in L_p (\T^d)\, .
\]
In case $r\in \N$ this leads to an equivalent characterization.

By the Littlewood-Paley theorem \index{Littlewood-Paley decomposition!Classical}
\index{Theorem!Littlewood-Paley}, see Theorem \ref{T7.3.3} in the Appendix, we obtain a useful equivalent
norm for $\Wrp$ in case $1<p<\infty$ and $r>0$, namely
\begin{equation}\label{NeqW}
   \|f\|_{\Wrp} \asymp \Big\|\Big(\sum\limits_{\bs \in \N_0^d} 2^{r|\bs|_12}|\delta_{\bs}(f)(x)|^2\Big)^{1/2}\Big\|_p\,.
\end{equation}
This norm is particularly useful for the analysis of the approximation from the step-hyperbolic cross $Q_n$, see
\eqref{HC} above and Section 4 below. 


\subsection{Spaces of functions with bounded mixed difference}\index{Function spaces!Bounded mixed difference}
\label{Chardiff}

Let us
first recall the basic concepts. For the univariate functions $f:\T \to \C$ the
$m$th difference operator $\Delta_h^{m}$ is defined by
\begin{equation*}
\Delta_h^{m}(f,x) := \sum_{j =0}^{m} (-1)^{m - j} \binom{m}{j} f(x +
jh)\quad,\quad x\in \T, h\in [0,1]\,.
\end{equation*}

Let $e$ be any subset of $\{1,...,d\}$. For multivariate functions $f:\T^d\to
\C$ and $h\in [0,1]^d$ the mixed $(m,e)$th difference operator $\Delta_h^{m,e}$
is defined by 
\begin{equation*}
\Delta_h^{m,e} := \
\prod_{i \in e} \Delta_{h_i,i}^m\quad\mbox{and}\quad \Delta_h^{m,\emptyset} =  \operatorname{Id},
\end{equation*}
where $\operatorname{Id}f = f$ and $\Delta_{h_i,i}^m$ is the univariate
operator applied to the $i$-th variable of $f$ with the other variables kept
fixed. Let us refer to the recent survey \cite{PoSiTi13} for general properties of mixed moduli of smoothness in $L_p$. 

We first introduce spaces/classes $\bH^r_p$ of functions with bounded mixed difference. \index{Function spaces!Mixed smoothness H\"older-Nikol'skii}

\begin{defi}
Let $r > 0$ and $1 \le p \le \infty$. Fixing an integer $m > r$, we define the space $\bH^r_p$ as the set of all
all $f\in L_p(\T^d)$ such that for any $e \subset \{1,...,d\}$
\[
\big\|\Delta_\bh^{m,e}(f,\cdot)\big\|_p
\ \le \  
C\, \prod_{i \in e} |h_i|^r
\]
for some positive constant $C$, and introduce the norm in this space
\[
\| \, f \, \|_{\bH^r_p} :=
\sum_{e \subset \{1,...,d\}}
\, | \, f \, |_{\bH^r_p(e)},
\]
where 
\[
| \, f \, |_{\bH^r_p(e)} :=
\sup_{0 < |h_i| \le 2\pi, \ i \in e} \,  
\left(\prod_{i \in e} |h_i|^{-r} \right) \,\big\| \, \Delta_\bh^{m,e}(f,\cdot) \, \big\|_p\,.
\]
\end{defi} 

\begin{rem}
 Let us define the mixed $(m,e)$th modulus of smoothness by 
\begin{equation}\label{modc}
\omega_{m}^e(f,\bt)_p:= \sup_{|h_i| < t_i, i \in
e}\|\Delta_h^{m,e}(f,\cdot)\|_{p}\quad,\quad \bt \in [0,1]^d,
\end{equation}
for $f \in L_p(\T^d)$ (in particular, $\omega_{m}^{\emptyset}(f,t)_p = \|f\|_{p}$)\,. 
Then there holds the following relation
\begin{equation}\nonumber
| \, f \, |_{\bH^r_p(e)} 
\ \asymp \
\sup_{2\pi>t_i > 0, \ i \in e} \,  
\left(\prod_{i \in e} t_i^{-r}\right)\, \omega_{m}^e(f,\bt)_p\,.
\end{equation}
 \end{rem}  

Based on this remark, we will introduce Besov spaces of mixed smoothness
$\bB^r_{p,\theta}$, a generalization of $\bH^r_p$, see \cite{Am76, ScTr87, Vyb06, Ul06}. 

\index{Function spaces!Mixed smoothness Besov}
\begin{defi}
Let $r > 0$ and $1 \leq p \le \infty$. Fixing an integer $m > r$, we define the space $\bB^r_{p,\theta}$ as the set of all
 $f\in L_p(\T^d)$ such that the norm
\[
\| \, f \, \|_{\bB^r_{p,\theta}} :=
\sum_{e \subset \{1,...,d\}}
\, | \, f \, |_{\bB^r_{p,\theta}(e)}
\]
is finite, where for $e \subset \{1,...,d\}$
\[
| \, f \, |_{\bB^r_{p,\theta}(e)} :=
\begin{cases}
& \sup_{2\pi>t_i > 0, \ i \in e} \,  
\left(\prod_{i \in e} t_i^{-r} \right) \, \omega_{m}^e(f,\bt)_p,\quad \theta=\infty; \\[2ex] 
& \left(\int_{(0,2\pi)^d}\omega^e(f,\bt)_p^\theta
\biggl(\prod\limits_{i\in e} t_i^{-\theta r-1}\biggl) \operatorname{d} \bt \right)^{1/\theta},\quad \theta<\infty.
\end{cases}
\]
\end{defi} 
With this definition we have $\bB^r_{p,\infty}=\bH^r_p$. Notice that the definitions of $\bB^r_{p,\theta}$ and $\bH^r_p$ are independent of $m$ in the sense that different values of $m$  induce equivalent quasi-norms of these spaces. 
With a little abuse of notation, denote the corresponding unit ball
$$
\bB^r_{p,\theta}:= \{f: \|f\|_{\bB^r_{p,\theta}}\le 1\}.
$$

\begin{rem}
In many papers on hyperbolic cross approximation, especially  from the former Soviet Union, instead of the spaces
$\bW^r_p$, $\bH^r_p$ and  $\bB^r_{p,\theta}$, the authors considered their subspaces. 
Namely,  they studied functions $f$ in $\bW^r_p$, $\bH^r_p$ and  $\bB^r_{p,\theta}$, 
which satisfy an extra condition:  $f$ has zero mean value in each 
variable $x_i$, $i = 1,...,d$, that is,
\begin{equation}\nonumber
\int_{\T}f(\bx) \operatorname{d} x_i
\ = \ 0. 
\end{equation}
However, this does not affect generality from the point of view of multivariate approximation 
(but not high-dimensional approximation, when we want to control dependence on dimension $d$) 
due to the following observation. Let $\bF_d$ temporarily denote one of the above spaces in $d$-variables. 
Then we have the following ANOVA-like \index{ANOVA decomposition} decomposition for any $f \in \bF_d$
\[
f 
\ = \
c + \sum_{e \not= \varnothing} f_e,
\]
where $f_e$ are functions of $|e|$ variables $x_i$, $i \in e$, with zero mean values in 
the variables $x_i$, which  can be treated as an element from $\bF_{|e|}$. 
For details and bibliography see, \cite{Di84b, Di86, Tmon, TBook}.
 \end{rem}  
 
\begin{rem} It was understood in the beginning of the 1960s that hyperbolic crosses
are closely related with the approximation and numerical integration of functions 
with {\em dominating mixed smoothness} which initiated a systematic study of these function classes. The
following references have to be mentioned in connection with the development of the theory of
function spaces with dominating mixed smoothness: Nikol'skii \cite{Ni62,Ni63,Ni75}, Babenko \cite{Bab2}, Bakhvalov
\cite{Bak63}, Amanov \cite{Am76}, Temlyakov \cite{Tmon,
TBook}, Tikhomirov \cite{Ti87,Ti90}, Schmei{\ss}er, Triebel \cite{ScTr87}, Vyb{\'i}ral \cite{Vyb06} and Triebel
\cite{Tr10}.
\end{rem}

\subsection{Characterization via Fourier transform}\index{Fourier transform}
\label{sect:FS}

In this subsection, we will give a characterization of spaces $\bH^r_p$ and $\bB^r_{p,\theta}$ via Fourier transform. Let us first comment on the classical Korobov space
$\bE^r_d$ introduced in \cite{Korb1}, see also Subsection \ref{why} below. For
$r>0$ we define the Korobov space 
\begin{equation}\nonumber
  \bE_d^{r}:=\Big\{f\in L_1(\T^d)~:~\sup\limits_{\bk \in \Z^d}
  |\hat{f}(\bk)|\cdot \prod\limits_{i=1}^d \max\{1,|k_i|\}^{r} < \infty\Big\}\,.
\end{equation}
The function $F_{r,\alpha}$ defined in \eqref{Bernoulli} clearly belongs to
$\bE_d^{r}$. Using the Abel transformation twice, see Appendix \ref{App_ineq},
we can prove that for $\bs \in \N_0^d$ it holds
$\|A_{\bs}(F_{r,\alpha})\|_1 \asymp 2^{-|\bs|_1 r}$.
In other words, 
\begin{equation}\label{BernoulliBesov}
      \sup\limits_{\bs \in \N_0^d}
2^{r|\bs|_1}\|A_{\bs}(F_{r,\alpha})\|_1 < \infty\,.
\end{equation}
This immediately ensures that $F_{r,\alpha} \in L_1(\T^d)$. Moreover, \eqref{BernoulliBesov}
is exactly the condition for $F_{r,\alpha}$
belonging to $\bH^{r}_{1}$ as we will see below\,. In this sense, the
classical Korobov space $\bE_d^{r}$ is slightly larger than the space $\bH^r_{1}$. In
Subsection \ref{Chardiff} we have  seen that  Besov
spaces $\bB^r_{p,\theta}$  are defined in a classical way by using exclusively
information on the ``time side'', i.e., without any information on the Fourier
coefficients. Such a useful tool
is so far not available for the classical Korobov space $\bE^{r}_d$.

Let us now characterize spaces $\bH^r_p$ and $\bB^r_{p,\theta}$ via dyadic decompositions of the Fourier transform. We begin with the
simplest version in terms of $\delta_\bs(f)$.
It is known that for $r > 0$ and $1 < p < \infty$,
\begin{equation}\label{DefH}
\|f\|_{\bH^r_p}
\ \asymp \
 \sup_\bs \|\delta_\bs(f)\|_p 2^{r|\bs|_1},
\end{equation}
and
\begin{equation}\label{DefB}
\|f\|_{\bB^r_{p,\theta}}
\ \asymp \
 \left(\sum_{\bs}\left(\|\delta_\bs(f)\|_p
2^{r|\bs|_1}\right)^\theta\right)^{1/\theta},
\end{equation}
and that for $r > 0$ and $1 \le p \le \infty$,
\begin{equation}\nonumber
\|f\|_{\bH^r_p}
\ \asymp \
 \sup_\bs \|A_\bs(f)\|_p 2^{r|\bs|_1},
\end{equation}
and
\begin{equation}\nonumber
\|f\|_{\bB^r_{p,\theta}}
\ \asymp \
 \left(\sum_{\bs}\left(\|A_\bs(f)\|_p
2^{r|\bs|_1}\right)^\theta\right)^{1/\theta}
\end{equation}
(see, e.g., \cite{Di86, Di00, Tmon,TBook, ScTr87, Ul06, NUU15}).

 The characterizations in the right hand side of  \eqref{DefH} \eqref{DefB} 
are simple and work well for $1<p<\infty$. In the cases
$p=1$ and $p=\infty$ the operators $\delta_\bs(\cdot)$ are not uniformly bounded
as operators from $L_p$ to $L_p$. This issue is resolved by replacing operators
$\delta_\bs(\cdot)$ by operators $A_\bs(\cdot)$.
Such a
modification gives equivalent definitions of classes $\bB^r_{p,\theta}$ in the
case $1\le p \le \infty$. We now present a general way for characterizing the Besov classes
for $0<p\leq 1$ in the spirit as done in \cite[Chapt.\ 2]{ScTr87}. In order to
proceed to $0<p\leq 1$ we need the concept of a smooth dyadic decomposition of
unity.
    \begin{defi}\label{cunity} Let $\Phi(\R)$ be the class of all
      systems
       $\varphi = \{\varphi_n(x)\}_{n=0}^{\infty} \subset C^{\infty}_0(\R)$
       satisfying
       \begin{description}
       \item(i) ${\supp}\,\varphi_0 \subset \{x:|x| \leq 2\}$\, ,
       \item(ii) ${\supp}\,\varphi_n \subset \{x:2^{n-1} \leq |x|
       \leq 2^{n+1}\}\quad,\quad n= 1,2,... ,$
       \item(iii) For all $\ell \in \N_0$ it holds
       $\sup\limits_{x,n}
       2^{n\ell}\, |D^{\ell}\varphi_n(x)| \leq c_{\ell} <\infty$\, ,
       \item(iv) $\sum\limits_{n=0}^{\infty} \varphi_n(x) = 1$ for all
       $x\in \R$.
       \end{description}
    \end{defi}

    \begin{rem}\label{speciald}
      The class $\Phi(\R)$ is not empty. We consider the following
      standard example.
      Let $\varphi_0(x)\in C^{\infty}_0(\R)$ be a smooth function with $\varphi_0(x) =
      1$ on $[-1,1]$ and $\varphi_0(x) = 0$
      if $|x|>2$. For $n>0$ we define
      $$
         \varphi_n(x) = \varphi_0(2^{-n}x)-\varphi_0(2^{-n+1}x).
      $$      It is easy to verify that the system $\varphi =
      \{\varphi_n(x)\}_{n=0}^{\infty}$ satisfies (i) - (iv).
    \end{rem}
    \noindent Now we fix a system $\{\varphi_n\}_{n\in \zz} \in \Phi(\R)$, where we 
    put $\varphi_n \equiv 0$ if $n<0$. For $\bs = (s_1,...,s_d) \in \zz^d$ let
the building blocks $f_{\bs}$ be given by
    \begin{equation}\label{f2_2}
	f_{\bs}(\bx) = \sum\limits_{\bk\in \zz^d}
	\varphi_{s_1}(k_1)\cdot...\cdot\varphi_{s_d}(k_d)\hat{f}(\bk)e^{i \bk\cdot
\bx}\quad,\quad \bx\in \T^d\,,\bs\in \N_0^d\,.
    \end{equation}
    \begin{defi}\label{d1} Let $0< p,\theta\leq
    \infty$ and $r>\sigma_p:=(1/p-1)_+$. Then $\Brpt$ is defined as the
collection of all $f\in
    L_1(\T^d)$ such that
    \begin{equation}\nonumber
	\|f\|_{\Brpt}^{\varphi}:=\Big(\sum\limits_{\bs\in \N_0^d}
2^{|\bs|_1r\theta}\|f_{\bs}\|_p^{\theta}\Big)^{1/\theta}
    \end{equation}
    is finite (usual modification in case $q=\infty$).
    \end{defi}
Recall, that this definition is independent of the chosen system $\varphi$ in
the sense of equivalent (quasi-)norms. Moreover, in case $\min\{p,q\}\geq 1$ the
defined spaces are Banach spaces, whereas they are quasi-Banach spaces in case
$\min\{p,q\} < 1$. For details confer \cite[2.2.4]{ScTr87}.

As already mentioned above the two approaches for the definition of the Besov spaces $\Brpt$ of mixed smoothness are equivalent 
if $p,\theta \geq 1$. Concerning difference characterizations for the quasi-Banach range of 
parameters there are still some open questions, see \cite[2.3.4, Rem.\
2]{ScTr87} and Theorem \ref{thm95} below. Let us state the following general equivalence result.

\begin{lem}\label{diff} Let $1\leq p,\theta\leq \infty$ and $m\in \N$ with
$m>r>0$. Then 
$$
    \|f\|_{\Brpt}^{\varphi} \asymp \|f\|_{\Brpt}\quad,\quad
f\in L_1(\T^d)\,.
$$
\end{lem}

\noindent 
As already mentioned above we have the equivalent characterization \eqref{NeqW} for the spaces 
$\bW^r_p$ in case $r > 0$ and $1 < p < \infty$. There is also a characterization in terms of the
so-called rectangular means of differences, i.e.,
\begin{equation}\label{rectm}
\mathcal{R}_{m}^e(f,\bt,\bx):=
\frac{1}{t_1}\int_{[-t_1,t_1]}\cdots\frac{1}{t_d}\int_{[-t_d,t_d]}|\Delta_{
(h_1, ... , h_d) } ^ { m , e } (f , \bx)|dh_d...dh_1\,,\, \bx\in \R^d,
\bt\in (0,1]^d\,.
\end{equation}

The following lemma is a straight-forward modification of \cite[Thm.\
3.4.1]{Ul06}, see also \cite{NUU15}.

\begin{lem}\label{fdiff} Let $1 <p <\infty$ and
$r>0$. Let further $m \in \N$ be a natural number with $m>r$\,. Then 
$$
  \|f\|_{\bW^r_{p}} \asymp \|f\|^{(m)}_{\bW^r_p}\quad,\quad f\in L_1(\R)\,,
$$
where 
$$
  \|f\|^{(m)}_{\bW^r_p}:=\Big\|\Big(\sum\limits_{\bs\in \N_{0}^d}
2^{r|\bs|_12}\mathcal{R}^{e(\bs)}_m(f,2^{-\bs},\cdot)^2\Big)^{1/2}
\Big\|_p
$$ 
and $2^{-\bs}:=(2^{-s_1},...,2^{-s_d})$\,.
\end{lem}

\subsection{Embeddings}\index{Embedding}

Here we review some useful embeddings between the classes $\Brpt$ and
$\bW^r_{p}$. 

\begin{lem}\label{emb} Let $0<p<u\leq \infty$, $\beta = 1/p-1/u$, $r\in \R$, and $0<\theta\leq \infty$.
\begin{description}
 \item (i) It holds
  $$     
    \Brpt \hookrightarrow \bB^{r-\beta}_{u,\theta}\,.\\
  $$
 \item (ii) If in addition $1<p<u<\infty$ then
  $$
    \bW^r_p \hookrightarrow \bW^{r-\beta}_{u}\,.
  $$
  Both embeddings are non-compact. 
  \item (iii) If $r>1/p$ then the embedding
 \begin{equation}\label{BinC}
    \Brpt \hookrightarrow C(\T^d)
 \end{equation}
 is compact. In case $r=1/p$ and $\theta \leq 1$ the embedding \eqref{BinC}
keeps valid but is not compact. 
 \item(iv) If $1<p<\infty$ and $r>0$ then 
 \begin{equation}\label{chainBFB}
    \bB^r_{p,\min\{p,2\}} \hookrightarrow \Wrp \hookrightarrow
\bB^r_{p,\max\{p,2\}}\,.
 \end{equation}
\end{description}
\end{lem}

Let us particularly mention the following two non-trivial embeddings between $\bB$
and $\bW$-spaces for different metrics.

 \begin{lem}\label{L3.5c} Let $0 < p < u<\infty$, $1<u<\infty$ and $\beta:=1/p-1/u$. Then for
$r\ge\beta$ we have
  \begin{equation}\label{Fremb}
    \bB^r_{p,u} \hookrightarrow \bW^{r-\beta}_u\,.
 \end{equation}
 \end{lem}
 
  \begin{lem}\label{L3.5d} Let $1 < p<u\leq \infty$ and $\beta:=1/p-1/u$. Then for
$r\ge\beta$ we have
  \begin{equation}\nonumber
    \bW^r_{p} \hookrightarrow \bB^{r-\beta}_{u,p}\,.
 \end{equation}
 \end{lem}

\begin{rem}\label{JawFr} \index{Embedding!Jawerth-Franke}
Lemma \ref{L3.5c} follows from Theorem \ref{T2.4.6} (see (\ref{2.4.4}) with $p=u$,
$q=p$). The embedding in Lemma \ref{L3.5c} is nontrivial and very useful in
analysis of approximation of classes with mixed smoothness. In the univariate
case an analog of (\ref{Fremb}) was obtained by Ul'yanov \cite{Ul70} and
Timan \cite{Ti74}. They used different methods of proof. Their techniques work
for the multivariate case of isotropic Besov spaces as well. Franke
\cite{Fr86} proved (\ref{Fremb}) for isotropic Besov spaces on $\R^d$ and
obtained its version with the $\bW$ space replaced by the appropriate
Triebel-Lizorkin spaces. The converse embedding in Lemma \ref{L3.5d} for isotropic spaces 
(a Triebel-Lizorkin space embedded in an appropriate Besov space) has been obtained by Jawerth \cite{Ja77}.  Lemma
\ref{L3.5d} is a corollary of Theorem \ref{T2.4.6}. It directly follows from Theorem \ref{T2.4.6} in the special case
$r=0$. The case $r>0$ follows from the case $r=0$ and the well known 
relation $\|t^{(r)}\|_p\asymp 2^{r|\bs|_1}\|t\|_p$ for $t\in \Tr(\rho(\bs))$. 
A new proof of both relations based on atomic decompositions 
has been given recently by Vyb\'iral \cite{Vy08}. The step from the univariate
and isotropic multivariate cases to the case of mixed smoothness spaces required
a modification of technique. In the periodic case  it was done by Temlyakov \cite{Tem10},
\cite{Tmon} (see Theorem \ref{T2.4.6} above) and in the case of $\R^d$ by Hansen
and Vyb\'iral \cite{HaVy09}.
\end{rem}

Let us finally complement the discussion from the beginning of Section \ref{sect:FS} and state useful embedding
relations in the situation $p=1$.

\begin{lem}\label{emb:p=1} Let $r>0$ and $\alpha \in \R$. Then the following continuous embeddings hold true.
\begin{equation}\label{BWB}
      \bB^r_{1,1}\hookrightarrow \bW^r_{1,\alpha} \hookrightarrow
\bB^{r}_{1,\infty} \hookrightarrow \bE^r_d\,.
  \end{equation}
\end{lem}

We note that in case $p=1$ we use operators $A_\bs$ instead of $\delta_\bs$ in the characterization of the $\bB$ classes. 
The first relation follows from Theorem \ref{T2.4.2}. The second relation follows from \eqref{BernoulliBesov} 
and the third embedding is a simple consequence of the characterization of $\bB^r_{1,\infty}$ together with 
$|\hat f(\bk)| \le \|f\|_1$.

Note, that the embedding $\bW^r_{1,\alpha} \hookrightarrow \bB^{r}_{1,2}$,
as a formal counterpart of \eqref{chainBFB}, does not hold true here. In fact,
it does not even hold true with $\bB^r_{1,\theta}$ and $2<\theta<\infty$ on the
right-hand side.  In that sense, the embedding \eqref{BWB} is sharp. Note also, that the embeddings in
Lemma \ref{emb:p=1} are strict. The (tensorized and) periodized hat function, see Figure
\ref{fig_Faber1} below, belongs to $\bB^2_{1,\infty}$ but not to $\bW^2_{1,\alpha}$.

%% file: linear_approx.tex
\section{Linear approximation}\index{Linear approximation}

\subsection{Introduction}

By {\it linear approximation} we understand approximation from a fixed finite
dimensional subspace. In the study of approximation of the univariate periodic
functions the idea of representing a function by its Fourier series is very
natural and traditional. It goes back to 
the work of Fourier from 1807. In this case one can use as a natural tool of
approximation the partial sums of the Fourier expansion. In other words this
means that we use the subspace 
$\Tr(n)$ for a source of approximants and use the orthogonal  projection onto
$\Tr(n)$ as the approximation operator. This natural approach is based on a
standard ordering of the trigonometric system: $1$, $e^{ikx}$, $e^{-ikx}$,
$e^{2ikx}$, $e^{-2ikx}$ $\dots$. We loose this natural approach, when we go from
the univariate case to the multivariate case -- there is no natural ordering of
the $\Tr^d$ for $d>1$. The following idea of choosing appropriate trigonometric
subspaces for approximation of a given class $\bF$ of multivariate functions 
was suggested by Babenko \cite{Bab1}. This idea is based on the concept of
the Kolmogorov width introduced in \cite{Ko36}: for a centrally symmetric compact $\bF\subset X$  define
\index{Width!Kolmogorov}
$$
d_m(\bF,X):= \inf_{\varphi_1,\dots,\varphi_m} \sup_{f\in\bF}
\inf_{c_1,\dots,c_m}\Big\|f-\sum_{k=1}^mc_k\varphi_k\Big\|_X.
$$
Consider a Hilbert space $L_2(\T^d)$ and suppose that the function class
$\bF=A(B(L_2))$ of our interest is an image of the unit ball $B(L_2)$ of
$L_2(\T^d)$ under a mapping $A: L_2(\T^d)\to L_2(\T^d)$ of a compact operator
$A$. For instance, in the case of $\bF=\bW^r_2$ the operator 
$A:=A_r$ is the convolution with the kernel $F_r(\bx)$. It is now well known and
was established by Babenko \cite{Bab1} for a special class of operators $A$ that
\be\nonumber
d_m(A(B(L_2)),L_2) = s_{m+1}(A),
\ee
where $s_j(A)$ are the singular numbers of the operator $A$:
$s_j(A)=(\lambda_j(AA^*))^{1/2}$. 

Suppose now that the eigenfunctions of the operator $AA^*$ are the trigonometric
functions 
$e^{i(\bk^j,\bx)}$. Then the optimal in the sense of the Kolmogorov width
$m$-dimensional subspace will be the $\Span \{ e^{i(\bk^j,\bx)}\}_{j=1}^m$.
Applying this approach to the class 
$\bW^r_2$ we obtain that for $m=|\Gamma(N)|$ the optimal subspace for
approximation in $L_2$ is the subspace of hyperbolic cross polynomials
$\Tr(\Gamma(N))$. This observation led to a thorough study of approximation by
the hyperbolic cross polynomials. We discuss it in Subsection
\ref{Subs:appr_hypcross}. 

B.S. Mityagin \cite{Mit} used the harmonic analysis technique, in particular,
the Marcinkiewicz multipliers (see Theorem \ref{T7.3.4}), to prove that
$$
d_m(\bW^r_p,L_p)\asymp m^{-r}(\log m)^{r(d-1)},\quad 1<p<\infty.
$$
He also proved that optimal, in the sense of order, subspaces are $\Tr(Q_n)$
with $|Q_n|\asymp m$ and $|Q_n|\le m$. In addition, the operator
$S_{Q_n}(\cdot)$ of orthogonal projection onto $\Tr(Q_n)$ can be taken as an
approximation operator. The use of harmonic analysis techniques for the $L_p$
spaces lead to the change from smooth hyperbolic crosses $\Gamma(N)$ to step
hyperbolic crosses  $Q_n$. The idea of application of the theory of widths for
finding good subspaces for approximation of classes of functions with mixed
smoothness is very natural and was used in many papers. A typical problem here
is to study 
approximation of classes $\bW^r_p$ in the $L_q$ for all $1\le p,q \le \infty$.
We give a detailed discussion of these results in further subsections. We only
give a brief qualitative remarks on those results in this subsection. As we
mentioned above, in linear approximation we are interested in approximation from
finite dimensional subspaces. The Kolmogorov width provides a way to determine
optimal (usually, in the sense of order) $m$-dimensional subspaces. The
approximation operator, used in the Kolmogorov width, is the operator of best
approximation. 
Clearly, we would like to use as simple approximation operators as possible. As
a result the following widths were introduced and studied. 

The linear width of a class $\bF$ in a normed space $X$ \index{Width!Linear} has been introduced by V.M. Tikhomirov \cite{Ti60} 
in 1960. It is defined by
\begin{equation}\label{lw}
\lambda_m(\bF,X) := \inf_{\substack{A:X \to X\\\text{linear}\\\rank A\le m}} \sup_{f\in \bF}\|f-A(f)\|_X.
\end{equation}
If $\bF$ is the unit ball of a Banach space $\tilde{\bF} \hookrightarrow X$ then we may compare this quantity
to the $n$th approximation number of the embedding $\text{Id}:\tilde{\bF} \to X$, where $\text{Id}f = f$, 
\begin{equation}\label{an}
    a_n(\text{Id}) := \inf\limits_{\substack{A:\tilde{\bF}\to X \\\text{linear}\\\rank A < n}} \|\text{Id}-A\|_{\tilde{\bF} \to X}\,,
\end{equation}
see Pietsch \cite[6.2.3.1]{Pi07} and Pinkus \cite[Def.\ II.7.3]{Pin85}. Note, that here the 
admissible linear operators $A$ map from $\tilde{\bF}$ to $X$ instead of $A:X\to X$ in the defintion 
of the linear width. However, Heinrich \cite[Cor.\ 3.4]{He89} showed, that if $\bF$ is a compact absolutely convex set in $X$ then 
the quantities $a_{m+1}(\text{Id}:\tilde{\bF} \to X)$ and $\lambda_m(\bF,X)$ are equal. Here, we mainly 
consider compact Banach space embeddings in $L_p$, where these quantities coincide.

V.N. Temlyakov \cite{Te82a} introduced the concept of orthowidth (Fourier
width):\index{Width!Fourier, Ortho-}
$$
\varphi_m(\bF,L_p):= \inf_{u_1,\dots,u_m} \sup_{f\in\bF}\Big\|f-\sum_{j=1}^m
\<f,u_j\>u_j\Big\|_p,
$$
where $u_1,\dots,u_m$ is an orthonormal system. 

It is clear that for any class $\bF$ and $1\le p\le \infty$
\begin{equation}\label{ineq[lambda_n>d_m]}
d_m(\bF,L_p) \le \lambda_m(\bF,L_p) \le \varphi_m(\bF,L_p) 
\end{equation}
and for $p=2$
$$
d_m(\bF,L_2) = \lambda_m(\bF,L_2) = \varphi_m(\bF,L_2).
$$

We now give a brief comparison of the above three widths in the case of classes
$W^r_p$ of univariate functions. For convenience we denote
$$
D_1 := \bigl\{ (p, q) :  1\le p\le q\le 2 \text{ or }1\le q\le
p\le \infty \bigr\}, 
$$
$$
D_2 := \bigl\{ (p, q) :  2\le p < q\le \infty \bigr\}, 
$$
$$
D_3 := \bigl\{ (p, q) :  1\le p < 2 \text{ and }2 < q\le \infty \bigr\},
$$
and let $D_4$ be a part of $D_3$ such that $1/q + 1/p\ge 1$ and
$D_5 = D_3\setminus D_4$. It is convenient to represent the corresponding domains in term 
of points $(1/p,1/q)$ on the square $[0,1]^2$ instead of points $(p,q)$ on $[1,\infty]^2$. Denote
$$
D_i^*:= \{(1/p,1/q): (p,q)\in D_i\},\qquad i=1,\dots,5.
$$

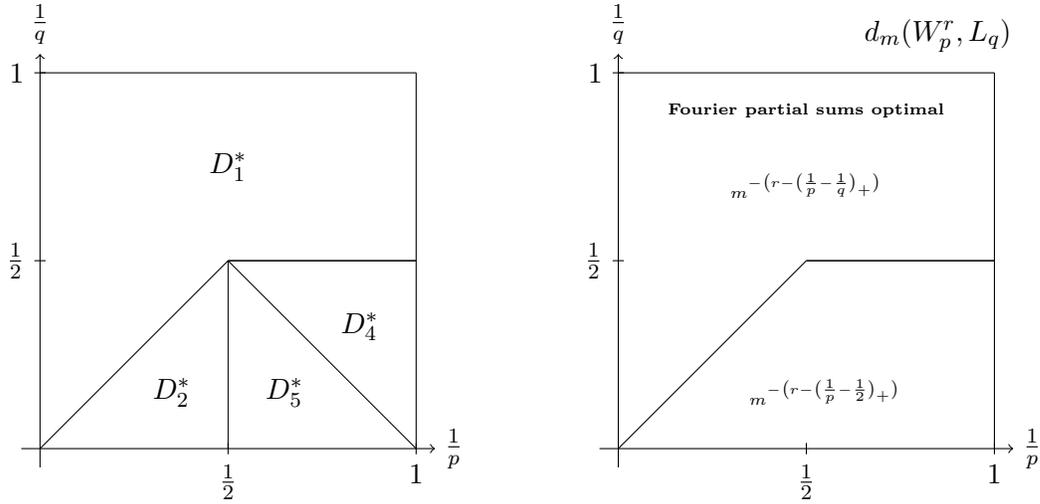
\begin{figure}[H]
\begin{minipage}{0.48\textwidth}
\begin{center}
\begin{tikzpicture}[scale=2.5]

\draw[->] (-0.1,0.0) -- (2.1,0.0) node[right] {$\frac{1}{p}$};
\draw[->] (0.0,-.1) -- (0.0,2.1) node[above] {$\frac{1}{q}$};

\draw (1.0,0.03) -- (1.0,-0.03) node [below] {$\frac{1}{2}$};
\draw (0.03,1) -- (-0.03,1) node [left] {$\frac{1}{2}$};

\draw (0,2) -- (2,2);
\draw (1,1) -- (2,1);
\draw (0,0) -- (1,1);

\draw (1,1) -- (2,0);
\draw (1,0) -- (1,1);
\draw (1,1) -- (2,1);
\draw (2,2) -- (2,0);

\node at (1.3,0.3) {$D^*_5$};
\node at (1.7,0.65){$D^*_4$};          	
\node at (1,1.5) {$D^*_1$};
\node at (0.7,0.3) {$D^*_2$};
\draw (2,0.03) -- (2,-0.03) node [below] {$1$};
\draw (0.03,2) -- (-0.03,2) node [left] {$1$};

\end{tikzpicture}

\end{center}

\end{minipage}
\begin{minipage}{0.48\textwidth}
 \begin{center}
\begin{tikzpicture}[scale=2.5]

\draw[->] (-0.1,0.0) -- (2.1,0.0) node[right] {$\frac{1}{p}$};
\draw[->] (0.0,-.1) -- (0.0,2.1) node[above] {$\frac{1}{q}$};

\draw (1.0,0.03) -- (1.0,-0.03) node [below] {$\frac{1}{2}$};
\draw (0.03,1) -- (-0.03,1) node [left] {$\frac{1}{2}$};

\draw (0,2) -- (2,2);
\draw (1,1) -- (2,1);
\draw (0,0) -- (1,1);

\draw (1,1) -- (2,1);
\draw (2,2) -- (2,1);
\draw (2,1) -- (2,0);

\node at (1,1.4) {\tiny $m^{-\big(r-\big(\frac{1}{p}-\frac{1}{q}\big)_+\big)}$};
\node at (1.1,0.3) {\tiny $m^{-\big(r-\big(\frac{1}{p}-\frac{1}{2}\big)_+\big)}$};
\node at (1,1.8) {\tiny {\bf Fourier partial sums optimal}};
\node at (1.7,2.2) {$d_m(W^r_p,L_q)$};
\draw (2,0.03) -- (2,-0.03) node [below] {$1$};
\draw (0.03,2) -- (-0.03,2) node [left] {$1$};

\end{tikzpicture}

\end{center}

\end{minipage}
\caption{The univariate behavior of $d_m(W^r_p,L_q)$ in different regions}
\end{figure}

It is known that for
$(p, q)\in D_1$ approximation by trigonometric polynomials in $\Tr(n)$ with
$n = \bigl[(m-1)/2\bigr]$ gives the order of decrease of the Kolmogorov
widths $d_m(W^r_p,L_q)$. But for $(p, q)\notin D_1$ this is not the case:
$$
d_{2n+1} (W_p^r, L_q) = o\bigl(E_n (W_p^r)_q \bigr), \qquad
 (p, q)\notin D_1.
$$

For
$(p, q)\in D_1$ the
orders of the Kolmogorov widths can be obtained by   linear
operators $V_n$,  $m = 4n - 1$. It is known that the operators $V_n$ give the
orders of linear widths not
only in the domain $D_1$ but also in the domain $D_2$. In the domain $D_3$  the
relation
$$
\lambda_{2n+1} (W_p^r, L_q) = o\bigl(E_n (W_p^r)_q \bigr)
$$
holds,  which shows,  in particular,  that for $(p, q)\in D_3$ the order of
linear width can not be realized by the operators $V_n$.

For $(p, q)\in D_1\cup D_4$ the orders of
the Kolmogorov widths can be realized by linear methods:  in the case
$(p, q)\in D_1$ by means of the operators $V_n$,  and in the case
$(p, q)\in D_4$
by means of some other linear operators.

For the classes $W_{p}^r$   for all
$(p, q)$, excepting  the case
$(p, q) = (1, 1)$,  $(\infty, \infty)$,  the operators $S_n$  $(m = 2n + 1)$
are optimal
Fourier operators in the sense of order.

The linear operators $A$
providing the orders of the widths $\lambda_m (W_p^r, L_q)$
for $(p, q)\in D_3$,  that
is,  in the case when $A$ differs from $V_n$  ($m = 4n - 1$),  are not
orthogonal projections. Moreover,    operators $A$ can not be bounded uniformly
(over $m$)
as operators from $L_2$ to $L_2$.

Further,  for example for $p = 2$ and $q = \infty$, the Kolmogorov widths
decrease faster than the corresponding linear widths:
$$
d_m (W_2^r, L_{\infty})\asymp m^{-1/2}\lambda_m (W_2^r,
L_{\infty}).
$$
However,  up to now no concrete example of a system $\{\varphi_i \}_{i=1}^m$
is known,
the best approximations by which would give the order of
$d_m (W_2^r, L_{\infty})$ (the same is true for the domain $D_2\cup D_5$).

This discussion shows that the sets $\Tr(n)$ and the operators $V_n$
and $S_n$ are optimal in many cases from the point of view of the
Kolmogorov widths,  linear widths and orthowidths. In the cases when
we can approximate better than by means of the operators $V_n$ and $S_n$,
we must sacrifice some useful properties which these operators have.

We have a similar qualitative picture in the case of approximation of classes of
functions with mixed smoothness. The role of $\Tr(n)$ is played now by
$\Tr(Q_n)$ with $m\asymp |Q_n|$. 
The analog of the univariate de la Vall{\'e}e Poussin kernel $\V_n(x)$ the
kernel $\V_{Q_n}$ (see Section \ref{trigpol}) is not as good as its univariate
version. Lemma \ref{L2.3.2} from Section \ref{trigpol} gives
$$
\|\V_{Q_n}\|_1 \asymp n^{d-1}.
$$
This substantially complicates the study of approximation in $L_1$ and
$L_\infty$ norms. 
Many problems of approximation in these spaces are still open. For $1<p<\infty$
the role of $S_n$ is played by $S_{Q_n}$. 

\subsection{Approximation by the hyperbolic cross polynomials}\index{Hyperbolic cross!Polynomials}
\label{Subs:appr_hypcross}

 The operators $S_{Q_n}$ and $V_{Q_n}$ play an important role in the hyperbolic
cross approximation. These operators can be written in terms of the
corresponding univariate operators in the following form. Denote by $S^i_l$ the
univariate operator 
 $S_l$ acting on functions on the variable $x_i$. Then, it follows from the definition of $Q_n$ that
 \begin{equation}\label{4.2a}
 S_{Q_n} = \sum_{\bs:|\bs|_1\le n} \prod_{i=1}^d (S^i_{2^{s_i}-1} -
S^i_{[2^{s_i-1}]-1}),
 \end{equation}
here $S_{-1}=0$. A similar formula holds for the $V_{Q_n}$. 

\subsubsection*{The Smolyak algorithm}\index{Smolyak's algorithm}

Operators of the
form (\ref{4.2a}) with $S^i$ replaced by other univariate operators are used in
sampling recovery (see Section 5) and other problems. For a generic discussion we refer to 
Novak \cite{Nov00} and Wasilkowski, Wo\'zniakowski \cite{WaWo95}. Sometimes Smolyak's algorithm can also be identified
in the framework of boolean methods, see \cite{DeSc89}.

The approximate
recovery operators of the form (\ref{4.2a}) were first considered by Smolyak
\cite{Sm63}. A standard name for operators of the form (\ref{4.2a}) is {\it
Smolyak-type algorithms}.  Very often analysis of operators $S_{Q_n}$,
$V_{Q_n}$, and other operators of the form (\ref{4.2a}) goes along the same
lines. The following general framework was suggested in \cite{AT}, see also
\cite{SiUl07}. Let three numbers $a>b\ge 0$, and $1\le p\le \infty$ be given.
Consider a family of univariate linear operators $\{Y_s\}_{s=0}^\infty$, which
are defined on the space $W^a_p$ and have the following two properties:

(1) For any $f$ from the class $W^a_p$ we have
$$
\|f-Y_s(f)\|_p \le C_12^{-as}\|f\|_{W^a_p},\qquad s=0,1,2,\dots;
$$

(2) For any trigonometric polynomial $t$ of order $2^u$, we have
$$
\|Y_s(t)\|_p \le C_2 2^{b(u-s)}\|t\|_p,\qquad u\ge s.
$$

Let as above $Y^i_s$ denote the univariate operator 
 $Y_s$ acting on functions on the variable $x_i$. Consider the following $d$-dimensional operator
 $$
 T_n:= \sum_{\bs:|\bs|_1\le n}\Delta_\bs,\qquad \Delta_\bs:= \prod_{i=1}^d
(Y^i_{s_i} - Y^i_{s_i-1}),
 $$
 with $Y_{-1}=0$. We illustrate the above general setting by one result from
\cite{AT}. 
\begin{prop}\label{AT1} Let operators $\{Y_s\}_{s=0}^\infty$ satisfy conditions
(1) and (2). Then for any $r\in (b,a)$ we have for $f\in\bH^r_p$
$$
\|\Delta_\bs(f)\|_p \lesssim 2^{-r|\bs|_1}\|f\|_{\bH^r_p}\quad \text{and}\quad
\|f-T_n(f)\|_p \lesssim 2^{-rn}n^{d-1}\|f\|_{\bH^r_p}.
$$
\end{prop}
We note that the technique developed in \cite{ByUl15}, see Subsection
\ref{SampRep}, allows for extending Proposition \ref{AT1} to classes
$\bB^r_{p,\theta}$. 
\begin{prop}\label{ATB} Let operators $\{Y_s\}_{s=0}^\infty$ satisfy conditions
(1) and (2). Then for any $r\in (b,a)$ we have for $f\in\bB^r_{p,\theta}$
$$
\left(\sum_\bs\left(2^{r|\bs|_1}\|\Delta_\bs(f)\|_p\right)^\theta\right)^{
1/\theta} \lesssim \|f\|_{\bB^r_{p,\theta}}   
$$
and
$$
 \|f-T_n(f)\|_p \lesssim 2^{-rn}n^{(d-1)(1-1/\theta)}\|f\|_{\bB^r_{p,\theta}}.
 $$
\end{prop}

\subsubsection*{Approximation from the hyperbolic cross}\index{Hyperbolic cross!Approximation}

We now proceed to best approximation by the hyperbolic cross polynomials from $\Tr(Q_n)$. 
Denote for a function $f\in L_q$
$$
E_{Q_n}(f)_q := \inf_{t\in \Tr(Q_n)} \|f-t\|_q
$$
and for a function class $\bF$
$$
E_{Q_n}(\bF)_q := \sup_{f\in\bF}E_{Q_n}(f)_q.
$$
We begin with approximation of classes $\bW^r_p$ in $L_q$. It is clear that
approximation of the Bernoulli
kernels $F_r$, which are used for integral representation of a
function from $\bW^r_p$, plays an important role in approximation of classes
$\bW^r_p$.
The following theorem is from \cite{Tem1}
\begin{thm}\label{TBT3.1} For $1\le q \le \infty$ and $r - 1 + 1/q > 0$ we have
$$
E_{Q_n}(F_r)_q\asymp 2^{-(r-1+1/q)n}n^{(d-1)/q}.
$$
\end{thm}
In particular, Theorem \ref{TBT3.1} with $q=1$ implies
$$
E_{Q_n}(\bW^r_\infty)_\infty \lesssim 2^{-rn}n^{d-1}.
$$
By the corollary
to the Littlewood-Paley theorem we get from Theorem \ref{TBT3.1},
\be\nonumber
\bigl\| F_r-S_{Q_n}(F_r)\bigr\|_q\lesssim 2^{-(r-1+1/q)n}n^{(d-1)/q},\qquad
1 < q < \infty .
\ee
In the case of $q=\infty$ we only get
\be\nonumber
\bigl\| F_r-S_{Q_n}(F_r)\bigr\|_\infty\lesssim 2^{-n(r-1)}n^{d-1},\qquad
\ee
which has an extra factor $n^{d-1}$ compared to $E_{Q_n}(F_r)_\infty$. 

 The upper bounds  for  the  best
approximations of the functions $F_r$ in the uniform
metric
were obtained (see \cite{Tem4}) using the Nikol'skii duality theorem (see
Appendix, Theorem \ref{T7.2.2}.  
The use of the duality theorem has the result that we can
determine the order of the best approximation of $F_r $
in the
uniform metric, but we cannot construct a polynomial giving this
approximation. The situation is unusual from the point of view of
approximation of functions of one variable.

\begin{thm}\label{TBT3.2} Suppose that $1 < p,q < \infty$,
$r > (1/p - 1/q)_+$. Then
$$
E_{Q_n}(\bW_p^r)_q\asymp 2^{-n\bigl(r-(1/p-1/q)_+\bigr)}.
$$
\end{thm}

This theorem was proved by harmonic analysis technique for the case $p=q$ in
\cite{Mit}, \cite{NiN74} and for $q\neq p$ in \cite{Ga78}.

We now consider the cases when one or two
parameters $p$, $q$ take the extreme value $1$ or $\infty$. These
results are not as complete as in the case $1 < p,q < \infty$.

\begin{thm}\label{TBT3.4} We have
$$
E_{Q_n}(\bW_{p}^r)_q\asymp
\begin{cases}
2^{-nr},\quad r>0;\quad p=\infty,\quad 1\le q<\infty; \quad
1<p\le\infty, \quad q=1;\\
2^{-n(r-1+1/q)}n^{\frac{d-1}{q}},\quad r > 1-1/q,\quad p=1,\quad
1 \le q \le \infty ;\\
2^{-n(r-1/p)},\quad r>1/p,\quad 1\le p\le 2,\quad q=\infty.
\end{cases}
$$
\end{thm}

The upper bounds in the case $p=1$ follow from Theorem \ref{TBT3.1}. The case
$1\le p\le 2$, $q=\infty$, was established in \cite{Tem3}. For the first case
see \cite{TBook}, Ch. 3, Theorem 3.4.

We now proceed to classes $\bH^r_p$. The problem of finding the right orders of
decay of the $E_{Q_n}(\bH^r_p)_q$ turns out to be more difficult than the
corresponding problem for the $\bW^r_p$ classes. Even in the case $1<p,q<\infty$
a new technique was required. 
\begin{thm}\label{TBT3.3} We have
$$
E_{Q_n}(\bH_p^r )_q \asymp
\begin{cases}
2^{-n(r-1/p+1/q)}n^{\frac{d-1}{q}},\quad 1\le p<q< \infty,\qquad
r>1/p-1/q,\\
2^{-rn}n^{\frac{d-1}{2}},\quad 1<q\le p<\infty,\quad p\ge2;
\quad p=\infty, \quad 1<q<\infty,\\
2^{-rn}n^{\frac{d-1}{p}},\quad1\le q\le p\le2,\qquad r>0.
\end{cases}
$$
\end{thm}
In the case $p=q=2$ Theorem \ref{TBT3.3} was proved in \cite{Bu64} and in the
case $1<p=q<\infty$ in \cite{NiN74}. In the case $p=q=1$ it was proved in
\cite{Tem3} and in the case $1\le p<q\le 2$ in \cite{Tem6}. In the case $1\le
p<q<\infty$, Theorem \ref{TBT3.3} was proved in \cite{Tmon} with the use of
Theorem \ref{T2.4.6} from Section \ref{trigpol}. In the case $2\le q<p\le
\infty$ the
required upper bounds follow from the upper bounds in the case $1<q=p<\infty$
and the lower bounds follow from \cite{Tem2}. In the case $1<q<2\le p<\infty$
Theorem \ref{TBT3.3} was proved in \cite{Ga84} and \cite{Di84c}.
In the case $1\le q<p\le 2$ the proof of lower bounds required a new technique
(see \cite{Tem15}).  

\begin{thm}\label{TBT3.3H} Let $2<p\le\infty$ and $r>0$. Then we have
$$
E_{Q_n}(\bH_p^r )_1 \asymp 2^{-rn}n^{\frac{d-1}{2}}.
$$
\end{thm}
The upper bounds in Theorem \ref{TBT3.3H} follow from the upper bounds for
$E_{Q_n}(\bH^r_2)_2$ from Theorem \ref{TBT3.3}. The lower bounds are nontrivial.
They follow from the corresponding lower bounds for the Kolmogorov widths
$d_m(\bH^r_\infty,L_1)$, which, as it was observed in \cite{Be5}, follow from the
lower bounds for the entropy numbers $\epsilon_k(\bH^r_\infty,L_1)$ from
\cite{TE1} and \cite{TE2}. 

The following result is known in the case of functions of two variables.

\begin{thm}\label{TBT3.5} Let $d = 2$, and $r > 0$. Then
$$
E_{Q_n}(\bH_{\infty}^r )_q\asymp
\begin{cases}
2^{-rn}n&\qquad\text{ for }\qquad q=\infty,\\
2^{-rn}n^{1/2}&\qquad\text{ for }\qquad q=1.
\end{cases}
$$
\end{thm}
This theorem was proved in \cite{Tem3} in the case $q=\infty$ and in
\cite{TBook}, Chapter 3, Theorem 3.5, in the case $q=1$. The proof of lower
bounds is based on the Riesz products (see Subsection 2.6).

In this subsection we studied approximation in the $L_q$-metric of
functions in the classes $\bW_{p}^r$ and
$\bH_{p}^r$, $1 \le p, q \le \infty$ by trigonometric
polynomials whose harmonics lie in the hyperbolic crosses.
Certain specific features of the multidimensional case were observed
in this study.

As is known, in the univariate case the order of the least
upper bounds of the best approximation by trigonometric polynomials
for both classes are the same for all $1 \le p,q \le \infty$, even though
$H_p^r$ is a wider class than $W_{p}^r$. 
It was determined that the least upper
bounds of the best approximation by polynomials in $\Tr(Q_n)$ are
different for the classes $\bW_{p}^r$ and
$\bH_{p}^r$ for all $1 < p,q < \infty$.
Namely,
$$
E_{Q_n}(\bW_{p}^r)_q= o\bigl(E_{Q_n} (
\bH_{p}^r )\bigr)_q,\qquad d \ge 2 .
$$
This phenomenon is related to the following fact. Let $1<p<\infty$. The property
$f\in \bW^r_p$ implies and is very close to the property
$$
\|f_l\|_p \lesssim 2^{-rl},\quad f_l:= \sum_{|\bs|_1=l} \delta_\bs(f).
$$
Contrary to that the property $f\in\bH^r_p$ is equivalent to 
$$
\|\delta_\bs(f)\|_p \lesssim 2^{-r|\bs|_1}.
$$
Therefore, the difference between classes $\bW^r_p$ and $\bH^r_p$ is determined
by the interplay between conditions on the dyadic blocks and the hyperbolic
layers. The number of the dyadic blocks in the $n$th hyperbolic layer is of $\asymp n^{d-1}$. The
quantities $E_{Q_n}(\bW^r_p)_q$ and $E_{Q_n}(\bH^r_p)_q$ differ by a factor of
the order $n^{(d-1)c(p,q)}$. 

In the case $p=q=1$ the classes $\bW$ and $\bH$ are alike in the sense of best
approximation by the hyperbolic cross polynomials (see Theorems \ref{TBT3.4} and
\ref{TBT3.3} above):
$$
E_{Q_n}(\bW^r_1)_1 \asymp E_{Q_n}(\bH^r_1)_1\asymp 2^{-rn}n^{d-1}.
$$

It turns out that approximation in the uniform metric differs
essentially from approximation in the $L_q$-metric, $1 < q < \infty$, not
only in the methods of proof, but also in that the results are
fundamentally different. For example, in approximation  in the
$L_q$-metric, $1 < q < \infty$, the partial Fourier sums $S_{Q_n} (f)$
give the
order of the best approximation $E_{Q_n}(f)_q$ and thus, if we are not
interested in the dependence of $E_{Q_n}(f)_q$ on $q$,
then we can confine
ourselves to the study of $S_{Q_n}(f)$.

In the univariate case and the uniform metric the partial
sums of the Fourier series give good approximation for the functions
in the classes $W_{p}^r$ and $H_p^r$,
$1 < p < \infty$:
$$
E_n(F)_{\infty}\asymp\sup_{f\in F}\bigl\|f - S_n (f)\bigr\|_{\infty},
$$
where $F$ denotes either $W_{p}^r$ or $H_p^r$.

In the case of the classes $\bW_{p}^r$ and
$\bH_{p}^r$, $1 < p < \infty$, not only
the Fourier sums do not give the orders of the least upper bounds of
the best approximations in the $L_{\infty}$-norm, but also no linear method
gives the orders of the least upper bounds of the best approximations
with respect to the classes $\bW_{p}^r$ or
$\bH_{p}^r$, $d \ge 2$, $1 < p < \infty$ (see \cite{Tmon}, Chapter 2, Section
5).
In other words, the operator of the best approximation in the
uniform metric by polynomials in $\Tr(Q_n)$ cannot be replaced by any linear
operator without detriment to the order of approximation on the classes
$\bW_{p}^r$ and $\bH_{p}^r$, $1 < p < \infty$.

Let us continue with results on the Besov class $\Brpt$. We will see how the
third parameter $\theta$ in this class is reflected on the asymptotic order of
$E_{Q_n}(\Brpt)_q$. 


\begin{thm} \label{Theorem[HC-Approx-Brpt]}
Let $1 < p,q < \infty$,  $1\le \theta < \infty$, $r > (1/p - 1/q)_+$.
 Then we have
\begin{equation}\nonumber
E_{Q_n}(\Brpt)_q
\ \asymp \
\begin{cases}
2^{-(r - 1/p + 1/q) n} \, n^{(1/q - 1/\theta)_+ (d-1)}, \ & p < q;
\\[1ex]
2^{- rn}, & q \le p, \ \theta \le \min\{2,p\};
\\[1ex]
2^{- rn} \, n^{(1/2 - 1/\theta) (d-1)}, & q \le p, \ p \ge 2, \ \theta > 2;
\\[1ex]
2^{- rn} \, n^{(1/p - 1/\theta)(d-1)}, & p = q < 2, \ \theta > p.
\end{cases}
\end{equation}
\end{thm}
Theorem~\ref{Theorem[HC-Approx-Brpt]} was proved in \cite{Di85}. 
The upper bounds are realized by the approximation by the operator $S_{Q_n}$. 
Although in the case $q \le p$, $ \theta < \min\{2,p\}$,
$\Brpt \varsubsetneq \Wrp$, we still have $E_{Q_n}(\Brpt)_q  \asymp E_{Q_n}(\Wrp)_q$. 
While in the case $q \le p$, $ \theta > \min\{2,p\}$ where $\Brpt \not\subset
\Wrp$ and the approximation properties of 
$\Brpt$ are closer to those of $\Hrp$, the asymptotic order of 
$E_{Q_n}(\Brpt)_q$ has the additional logarithm term $n^{(1/\min\{2,p\} -
1/\theta)(d-1)}$.

\begin{thm} \label{Theorem[HC-Approx-Brpt,q=1]}
Let $1 \le p < \infty$,  $1 \le \theta < \infty$, $r > 0$. 
Then we have
\begin{equation}\nonumber
E_{Q_n}(\Brpt)_1
\ \asymp \
\begin{cases}
2^{-rn} \, n^{(1/p - 1/\theta)_+ (d-1)}, \ & p \le 2,  \\[1ex]
2^{-rn}, \ & p > 2 , \ \theta \le 2 \,.
\end{cases}
\end{equation}
\end{thm}

Theorem~\ref{Theorem[HC-Approx-Brpt,q=1]} was proved in the case $p=1$ in
\cite{Ro04} and in the case 
$1 < p \le 2$ in \cite{Ro08}. 
In the case $p>2$ and $\theta\le 2$ the upper bound follows from the embedding of the 
$\Brpt$ into $\bW^r_2$ and Theorem \ref{TBT3.2}. The lower bound in this case is trivial. 
Similarly to
Theorem~\ref{Theorem[HC-Approx-Brpt]}, the upper bounds in this theorem are
realized by the approximation by the operator $S_{Q_n}$. The lower bounds are
proved by the construction of a ``fooling'' test function.

\begin{thm} \label{Theorem[HC-Approx-Brpt,q=infty]}
Let $1 \le p \le 2$,  $1 \le \theta \le 2$, $r > 1/p$. 
Then we have
\begin{equation} \label{[HC-Approx-Brpt,q=infty]}
E_{Q_n}(\Brpt)_\infty
\ \asymp \
2^{-(r-1/p)n}.
\end{equation}
\end{thm}

Theorem~\ref{Theorem[HC-Approx-Brpt,q=infty]} was proved in \cite{Ro08}. The
upper bound follows from the embedding of the 
$\Brpt$ into $\bW^{r-1/p+1/2}_2$ and Theorem \ref{TBT3.4}. The lower bound
follows from relation \eqref{[HC-Approx-Brpt,q=infty]} for the univariate
case ($d=1$).

We note that in some cases the upper bounds are trivial. For instance, in the case $\theta =1$, using the Nikol'skii inequalities, we obtain for $1\le p\le q\le \infty$, $r>1/p-1/q$, that
\begin{equation}\nonumber
 \begin{split}
E_{Q_n}(f)_q &\le \sum_{|\bs|_1\ge n} \|A_\bs(f)\|_q \lesssim \sum_{|\bs|_1\ge n} 2^{|\bs|_1(1/p-1/q)}\|A_\bs(f)\|_p\\
&\le 2^{(-r+1/p-1/q)n}\sum_{|\bs|_1\ge n}2^{r|\bs|_1} \|A_\bs(f)\|_p \ll 2^{(-r+1/p-1/q)n}\|f\|_{\bB^r_{p,1}}.
\end{split}
\end{equation}
The corresponding lower bounds follow from the univariate case. Thus, we obtain for $1\le p\le q\le \infty$, $r>1/p-1/q$ (see \cite{Di85} 
for $q<\infty$ and \cite{Rom10} for $q=\infty$)
$$
    E_{Q_n}(\bB^r_{p,1})_q \asymp 2^{(-r+1/p-1/q)n}.
$$

In other cases the lower bounds follow from the corresponding examples, used for the $\bH$ classes, and a simple inequality for $f\in \Tr(\Delta Q_n)$
\begin{equation}\label{RoStB1}
\|f\|_{\bB^r_{p,\theta}} \lesssim n^{(d-1)/\theta}\|f\|_{\bH^r_p}. 
\end{equation}
For instance, in this way we obtain (see \cite{Rom10})

$$
E_{Q_n}({\bB}_{\infty,\theta}^r)_q\asymp 2^{-rn} n^{\left(\frac{1}{2}-\frac{1}{\theta}\right)_+(d-1)} 
$$
for $1<q<\infty$, $r>0$, $1\leq \theta <\infty$;
$$
E_{Q_n}({\bB}_{p,\theta}^r)_1\asymp 2^{-rn} n^{\left(\frac{1}{2}-\frac{1}{\theta}\right)} 
$$
if $d=2$ and  $2<\theta <\infty$, $2<p\leq \infty$, $r>0$; and (see \cite{Rom11})
$$
E_{Q_n}({\bB}_{\infty,\theta}^r)_\infty\asymp 2^{-rn} n^{1-\frac{1}{\theta}} 
$$
if $d=2$ and $1\leq \theta <\infty$, $r>0$.

Here is a result from \cite{Rom06}.
\begin{thm}\label{Rom5} Let $1<q<\infty$, $r>1-\frac{1}{q}$, $1\leq \theta <\infty$. Then we have
$$
E_{Q_n}({\bB}_{1,\theta}^r)_q\asymp 2^{-\left(r-1+\frac{1}{q} \right)n} n^{\left(\frac{1}{q}-\frac{1}{\theta}\right)_+(d-1)} .
$$
\end{thm}
We note that the upper bounds in Theorem \ref{Rom5} are derived from Theorem \ref{T2.4.6}.

\subsection{The Kolmogorov widths} \index{Width!Kolmogorov}

We begin with results on $\bW^r_p$ classes. Denote
$$
r(p, q) := \begin{cases} (1/p-1/q)_+\qquad&\text{ for }1\le p\le q \le2; \qquad
1\le q\le p\le\infty, \\
\max\{1/2, 1/p\}\qquad&\text{ otherwise }.\end{cases}
$$
\begin{thm}\label{TBT4.4} Let $r(p,q)$ be   as above.
Then for $1<p,q<\infty$, $r>r(p,q)$
$$
d_m(\bW_{p}^r,L_q) \asymp
\left( \frac{(\log m)^{d-1}}{m}\right)^{r-\bigl(1/p-\max\{1/2,1/q\}\bigr)_+}.
$$
\end{thm} 
As we already mentioned in the Introduction in the case $p=q=2$
Theorem \ref{TBT4.4} follows from a general result by Babenko \cite{Bab1} and in
the case $1<p=q<\infty$ it was proved by Mityagin \cite{Mit}. Note that in
\cite{Mit} only the case of natural $r$ was considered. The result was extended
to real $r$ in \cite{Ga78}. In the case $1<p<q\le 2$
the theorem was obtained in \cite{Tem2} and \cite{Tem6}. In the case $1<
p<q<\infty$, $2\le q<\infty$ it was obtained in \cite{Te82a} and \cite{Tem10}.
In the case $1<q<p<\infty$ Theorem \ref{TBT4.4} was proved in \cite{Ga85}. In all cases $1 < p \le q \le 2$ and
$1 < q \le p < \infty$, included in
Theorem \ref{TBT4.4}, the upper estimates follow from approximation by the
hyperbolic cross polynomials from $\Tr(Q_n)$ with $m\asymp |Q_n|$, $|Q_n|\le m$.

 {
 In the case $2 \leq p \leq q<\infty$ we encounter an interesting and important phenomenon in Theorem \ref{TBT4.4}. The main rate of convergence is $r$ and the rate does
neither depend on $p$ nor on $q$. In the univariate case this effect has been first observed  by Kashin in his seminal paper \cite{Ka77}. It makes use of the Maiorov 
discretization technique \index{Maiorov's discretization technique}\cite{Mai75} where the problem of $n$-widths for function classes is reduced to the study of $n$-widths in finite dimensional normed spaces, see 
Theorem \ref{TBT1.4.6} below. In the multivariate case (suppose $2 \leq p \leq q<\infty$) the proof of Theorem \ref{TBT4.4} (and Theorem \ref{TBT4.4'} below) is based on the following result.
\begin{thm}\label{TBT2.4.6} One has the estimate
$$
d_m \bigl(\Tr(\bN,d)_2, L_{\infty}\bigr)\lesssim\bigl(\vartheta(\bN)/
m\bigr)^{1/2}\ln \bigl(e\vartheta(\bN)/m) \bigr)^{1/2}.
$$
\end{thm}
We illustrate on the example of estimating from above the $d_m(\bW^r_2,L_q)$,
$2\le q<\infty$, how Theorem \ref{TBT2.4.6}  is applied. First of all, we derive the
following lemma from Theorem \ref{TBT2.4.6}. 
\begin{lem}\label{L4.3.1} Let $2\le q<\infty$.   We have for $\Delta Q_n:=Q_n\setminus Q_{n-1}$
$$
d_m(\Tr(\Delta Q_n)_2,L_q) \lesssim (|\Delta Q_n|/m)^{1/2} (\ln (e|\Delta Q_n|/m))^{1/2}.
$$
\end{lem}
\begin{proof} By Corollaries 1 and 2 of the Littlewood-Paley theorem it is easy to see that
$$
d_m(\Tr(\Delta Q_n)_2,L_q) \lesssim \max_{\bs\in \theta_n} d_{[m/|\theta_n|]}(\Tr(\rho(\bs))_2,L_q),
$$
where $\theta_n:=\{\bs:|\bs|_1=n\}$. 
Applying Theorem \ref{TBT2.4.6} we obtain Lemma \ref{L4.3.1}.
\end{proof}
Second, let $r>1/2$ and $\kappa>0$ be such that $r>1/2+\kappa$. For $n\in\N$ set
$$
m_l:=[|\Delta Q_n| 2^{\kappa(n-l)}],\quad l=n+1,\dots;
$$
$$
m:= |Q_n| +\sum_{l>n} m_l \le C(\kappa,d)2^nn^{d-1}.
$$
Then
$$
d_m(\bW^r_2,L_q) \le \sum_{l>n}d_{m_l}(\Tr(\Delta Q_l)_2,L_q) 2^{-rl} \lesssim 2^{-rn}\asymp m^{-r}(\log m)^{(d-1)r}.
$$

Theorem \ref{TBT2.4.6} is a corollary of the following fundamental result of
Kashin \cite{Ka77}, Gluskin \cite{Gl83} and Garnaev, Gluskin \cite{GaGl84}. See
also \cite{CaPa84}, \cite{PaJa86} and the recent papers  \cite{Vy08b}, \cite{FoPaRaUl10} for
dual versions of the result.
\begin{thm}\label{TBT1.4.6} For any natural numbers $d,  m$,  $m < d$ we have
$$
d_m(B_2^d,  \ell_{\infty}^d)\le C m^{-1/2} \bigl(\ln (ed/m) \bigr)^{1/2}.
$$
\end{thm}
Theorem \ref{TBT2.4.6} is derived from Theorem \ref{TBT1.4.6} by discretization
using the multivariate version of the Marcinkiewicz Theorem \ref{T2.4.11} from
Section \ref{trigpol}. We note here that
it would be natural to try to apply the discretization technique to the
hyperbolic cross polynomials in $\Tr(Q_n)$. However, it follows from the
discussion in Subsection 2.5 of Section \ref{trigpol} that this way does not
work. This brings technical difficulties in the analysis of the case
$2<q<\infty$. 

We note that the proof of Theorem \ref{TBT1.4.6} both in \cite{Ka77} and in
\cite{GaGl84} is probabilistic. The authors prove existence of matrices with
special properties. This makes Theorem \ref{TBT1.4.6} and all results obtained
with its help non-constructive. It turns out that matrices similar to those
constructed in \cite{Ka77} and \cite{GaGl84} are very important in sparse
approximation, namely, in compressed sensing \index{Compressed sensing}\cite{FoRa13}, \cite{Tbook}. Such matrices are called the
matrices with the Restricted Isometry Property (RIP) \index{Restricted Isometry Property} or simply the RIP matrices.
The reader can find a discussion of such matrices and an introduction to
compressed sensing in \cite[Chapt.\ 6 \& 9]{FoRa13}, \cite[Chapt.\ 5]{Tbook}, \cite{FoPaRaUl10}. We point out that the
problem of deterministic construction of good (near optimal) RIP matrices is a fundamental
open problem of the compressed sensing theory, see \cite[p. 170]{FoRa13}. 

The lower bounds in the case $1<p<q\le 2$ in Theorem \ref{TBT4.4} required a new technique. A new
element of the technique is that in the multivariate case a subspace, from which
we choose a ``bad'' function,
depends on the system $\ff_1,\dots,\ff_m$  of functions used for approximation
in the definition of the Kolmogorov width. In the univariate case we can always
take a ``bad'' function from $\Tr(n)$ with appropriate $n\asymp m$. 

 Let us now proceed with some limiting cases.

\begin{thm}\label{TBT4.4'} Let $r>1$ and $2\le q<\infty$.
Then
$$
d_m(\bW_{1}^r,L_q) \asymp
\left( \frac{(\log m)^{d-1}}{m}\right)^{r-1/2}(\log m)^{(d-1)/2}.
$$
\end{thm} 
\noindent This theorem was obtained in \cite{Te82a} and \cite{Tem10}.

\begin{thm}\label{KashT95}  For all $1\le q <\infty$ and $r>0$ we have
$$
d_m(\bW^r_\infty,L_q) \asymp  m^{-r}(\log m)^{r(d-1)}.
$$
\end{thm} 
The upper bounds in Theorem \ref{KashT95} follow from the case $1<p=q<\infty$.
The corresponding lower bounds in
Theorem \ref{KashT95} were proved in \cite{KTE1} in the case $1<q<\infty$ and in
\cite{KTE2}
in the case $q=1$. These proofs are based on finite dimensional geometry results
on volume estimates (see Subsection 2.5 of Section \ref{trigpol}). As a
corollary of the
lower bound in Theorem \ref{KashT95} and upper bounds from Theorem \ref{TBT4.4}
we obtain the following result.

\begin{thm}\label{S3}  For all $1< p \le \infty$ and $r>0$ we have
$$
d_m(\bW^r_p,L_1) \asymp  m^{-r}(\log m)^{r(d-1)}.
$$
\end{thm} 

The first result on the right order of $d_m(\bW^r_p,L_\infty)$ was obtained in
\cite{VT59} in the case $d=2$.
\begin{thm}\label{Te96} Let $d=2$ and $2\le p\le \infty$, $r>1/2$. Then
$$
d_m(\bW^r_p,L_\infty) \asymp m^{-r} (\log m)^{r+1/2}.
$$
\end{thm}
The most difficult part of Theorem \ref{Te96} is the lower bounds. The proof of
the lower bounds is based on the Small Ball Inequality (see \eqref{2.6.3}) from Section
2). In the case $2\le p<\infty$ Theorem \ref{Te96} is proved in \cite{VT59} and
in the case $p=\infty$ in \cite{TE4}. The region
$$
R_1 := \{(1/p,1/q) : \quad 0<1/p\le 1/q<1\quad \text{or} \quad 1/2\le 1/q \le 1/p<1\}
$$
is covered by Theorem \ref{TBT4.4}. For this region the upper bounds follow
from Theorem \ref{TBT3.2}, which means that in this case the subspaces
$\Tr(Q_n)$ of the hyperbolic cross polynomials are optimal in the sense of
order. Theorem \ref{TBT4.4} shows that for the region $R_2:=(0,1)^2 \setminus
R_1$ the subspaces $\Tr(Q_n)$ are not optimal in the sense of order.  Theorem
\ref{TBT4.4} gives the orders of the $d_m(\bW^r_p,L_q)$ for all $(1/p,1/q)$ from
the open square $(0,1)^2$ under some restrictions on $r$. The situation on the
boundary of this square is more difficult. 
Theorem \ref{TBT4.4'} covers the segment
$
S_1:=\{(1,1/q):0<1/q\le 1/2\}.
$
The segment 
$
S_2 :=\{(0,1/q):0<1/q\le 1\}
$
is covered by Theorem \ref{KashT95}. The segment $S_3:=\{(1/p,1):0\le 1/p<1\}$
is covered by Theorem \ref{S3}. Finally, the segment $S_4:=\{(1/p,0):0\le 1/p\le
1/2\}$ in the case $d=2$ is covered by Theorem \ref{Te96}. In all other cases
the right order of the $d_m(\bW^r_p,L_q)$ is not known. 

 Let us discuss an extension of Theorem \ref{Te96} to the case $d\ge 3$. The following upper bounds are known (see \cite{Be5}, the book \cite{TrBe04}, and \cite{DuLiKuLi99} 
for the special case $r=1$)
\be\nonumber
d_m(\bW^r_p,L_\infty) \lesssim m^{-r}(\log m)^{(d-1)r+1/2},\quad 2\le p\le \infty,\quad r>1/2.
\ee
Recent results on the Small Ball Inequality for the Haar system\index{Haar system} (see \cite{BL}, \cite{BLV}) allow us to improve a trivial 
lower bound to the following one for $r=1$ and all $p<\infty$, $d\geq 3$:
$$
d_m(\bW^1_p,L_\infty) \gtrsim m^{-1}(\log m)^{d-1+\delta(d)},\quad \delta(d)>0.
$$
Theorem \ref{Te96} and the above upper and lower bounds support the following
\begin{conj} Let $d\geq 3$, $2\le p\le \infty$, $r>1/2$. Then
$$
    d_m(\bW^r_p,L_\infty) \asymp  m^{-r}(\log m)^{(d-1)r+1/2}\,.
$$
\end{conj}

We summarize the above results on the $d_m(\bW^r_p,L_q)$ in the following figure.

\begin{figure}[H]
\begin{center}
\begin{tikzpicture}[scale=2.5]

\draw (-0.1,0) -- (0,0);
\draw[->] (2,0.0) -- (2.1,0.0) node[right] {$\frac{1}{p}$};
\draw[very thick][dashed]  (0,0.0) -- (1,0.0);
\draw[very thick][dashed]  (1,.0) -- (2,0.0);
\draw[->] (0.0,-.1) -- (0.0,2.1) node[above] {$\frac{1}{q}$};

\draw (1.0,0.03) -- (1.0,-0.03) node [below] {$\frac{1}{2}$};
\draw (0.03,1) -- (-0.03,1) node [left] {$\frac{1}{2}$};

\draw[ultra thick](0,2) -- (2,2);
\draw (1,1) -- (2,1);
\draw (0,0) -- (1,1);

\draw (1,1) -- (2,1);
\draw[very thick][dashed] (2,2) -- (2,1);
\draw[ultra thick] (2,1) -- (2,0);
\draw[ultra thick] (0,2) -- (0,0);
\node at (1,1.4) {\tiny $\Big(\frac{(\log
m)^{d-1}}{m}\Big)^{r-(\frac{1}{p}-\frac{1}{q})_+}$};
\node at (1.1,0.3) {\tiny $\Big(\frac{(\log
m)^{d-1}}{m}\Big)^{r-(\frac{1}{p}-\frac{1}{2})_+}$};
\node at (1,1.8) {\tiny {\bf  {Hyperbolic cross optimal}}};
\draw (2,0.03) -- (2,-0.03) node [below] {$1$};
\draw (0.03,2) -- (-0.03,2) node [left] {$1$};

\end{tikzpicture}

\end{center}
   \caption{The asymptotical order of $d_m(\bW^r_p,L_q)$} \label{dmW}
\end{figure}
\noindent A straight line on the boundary indicates that in the respective parameter region the correct order is known. In the
dashed line region we do not know the correct order.

We now proceed to classes $\bH^r_p$. The first result on the right order of the
Kolmogorov width for $\bH^r_p$ classes was obtained in \cite{Tem2}. The proper
lower bound for $d_m(\bH^r_\infty,L_2)$ was proved in \cite{Tem2}. 

\begin{thm}\label{TBT4.5} Let $r(p,q)$ be the same as in Theorem \ref{TBT4.4}.
Then for
$1< p\le q<\infty$, $r > r(p,q)$
$$
d_m(\bH_p^r,L_q)\asymp
\left(\frac{(\log m)^{d-1}}{m}\right)^{r-\bigl(1/p-\max\{1/2,1/q\}\bigr)_+}
(\log m)^{(d-1)\max\{1/2,1/q\}}.
$$
\end{thm}
    In the case $1\le p<2$, $q=2$ Theorem \ref{TBT4.5} was obtained in \cite{Tem6}. 
 In the case $1< p<q <\infty$, $q\ge 2$ -- in \cite{Te82a} and \cite{Tem10}. 
 In the case $1<p\le q<2$ -- in \cite{Ga90}.  

\begin{thm}\label{TBT4.5'} Let $1<q<p\le\infty$, $p\ge 2$ and $r>0$.
Then
$$
d_m(\bH_p^r,L_q)\asymp
\left(\frac{(\log m)^{d-1}}{m}\right)^{r}
(\log m)^{(d-1)/2}.
$$
\end{thm}

In the case $2\le q\le p \le \infty$, $q<\infty$, Theorem \ref{TBT4.5'} was
obtained in \cite{Tem2}.  In the case $1<q<2\le p<\infty$ Theorem \ref{TBT4.5'} was
obtained in \cite{Ga85}, \cite{Ga84} and in
\cite{Di84c}. In the case $p=\infty$ and
$1<q<\infty$ -- in \cite{TE1} and \cite{TE2}. 

\begin{thm}\label{Tp=1} Let $2\le q<\infty$, $r>1$. Then
$$
d_m(\bH_1^r,L_q)\asymp
\left(\frac{(\log m)^{d-1}}{m}\right)^{r-1/2}
(\log m)^{(d-1)/2}.
$$
\end{thm}
This theorem is from \cite{Te82a} and \cite{Tem10}.

\begin{thm}\label{Tq=1} Let $2\le p\le \infty$, $r>0$. Then
$$
d_m(\bH_p^r,L_1)\asymp
\left(\frac{(\log m)^{d-1}}{m}\right)^{r}
(\log m)^{(d-1)/2}.
$$
\end{thm}
The upper bounds in Theorem \ref{Tq=1} follow from the upper bounds for
$E_{Q_n}(\bH^r_2)_2$ from Theorem \ref{TBT3.3}.   It was observed in \cite{Be5}
that the lower bounds follow from the corresponding lower bounds for the entropy
numbers $\epsilon_k(\bH^r_\infty,L_1)$ from \cite{TE1} and \cite{TE2}. 

 \begin{thm}\label{Te96H} Let $d=2$, $2\le p\le \infty$ and $r>1/2$ when $p<\infty$, and $r>0$ when $p=\infty$. Then
$$
d_m(\bH^r_p,L_\infty) \asymp m^{-r} (\log m)^{r+1}.
$$
\end{thm}
The most difficult part of Theorem \ref{Te96H} is the lower bounds for $p=\infty$. The proof of
the lower bounds is based on the Small Ball Inequality (see \eqref{2.6.3} from Section
2). Theorem \ref{Te96H} is proved in \cite{VT59} with the assumption $r>0$ for the lower bounds in case $p=\infty$ and with the assumption $r>1/2$ for the upper bounds. 
We note that in case $p=\infty$ the matching upper bounds follow from Theorem \ref{TBT3.5} under assumption $r>0$.   

In case $d\ge 3$ the right order of the $d_m(\bH^r_p,L_\infty)$, $2\le p\le \infty$, is not known. The following upper bounds are known (see \cite{Be5, Be91})
\be\nonumber
d_m(\bH^r_p,L_\infty) \lesssim m^{-r}(\log m)^{(d-1)(r+1/2)+1/2},\quad 2\le p\le \infty,\quad r>1/2.
\ee
Theorem \ref{Te96H} and the above upper bounds support the following conjecture.

\begin{conj}\label{CH} Let $d\ge 3$, $2\le p\le \infty$, $r>1/2$. Then
$$
 d_m(\bH^r_p,L_\infty) \asymp m^{-r}(\log m)^{(d-1)(r+1/2)+1/2}.
 $$
 \end{conj}

We summarize the above results
on the $d_m(\bH^r_p,L_q)$ in the following picture.

\begin{figure}[H]
 \begin{minipage}{0.48\textwidth}
\begin{tikzpicture}[scale=2.5]
\draw (0,0) -- (0,-0.1);
\draw (-0.1,0) -- (0.0,0.0);
\draw[->] (2,0.0) -- (2.1,0.0) node[right] {$\frac{1}{p}$};
\draw[->] (0.0,2) -- (0.0,2.1) node[above] {$\frac{1}{q}$};
\draw[very thick][dashed] (0,0) -- (2,0);
\draw[very thick][dashed] (0,0) -- (0,2);
\draw (0.03,1) -- (-0.03,1) node [left] {$\frac{1}{2}$};
\draw (1,0.03) -- (1,-0.03) node [below] {$\frac{1}{2}$};

\draw[ultra thick] (0,2) -- (1,2);
\draw[very thick][dashed] (1,2) -- (2,2);
\draw (1,1) -- (1,2);
\draw (0,0) -- (2,2);
\draw[very thick][dashed] (2,2) -- (2,1);
\draw[ultra thick] (2,1) -- (2,0);

\node at (0.5,1.4) {\small $\alpha = r, \beta = \frac{1}{2}$};
\node at (1.3,1.6) {\huge ?};
\node at (1.1,0.15) {\small $\alpha = r-\Big(\frac{1}{p}-\max\big\{\frac{1}{2},\frac{1}{q}\big\}\Big)_+$};
\node at (1.1,0.5) {\small $\beta = \max\big\{\frac{1}{2},\frac{1}{q}\big\}$};
\draw (2,0.03) -- (2,-0.03) node [below] {$1$};
\draw (0.03,2) -- (-0.03,2) node [left] {$1$};
\end{tikzpicture}
\end{minipage}
\begin{minipage}{0.48\textwidth}
 \begin{center}
    $$
        d_m(\bH^r_p,L_q) \asymp \Big(\frac{\log^{(d-1)} m}{m}\Big)^{\alpha} (\log m)^{(d-1)\beta} 
    $$
 \end{center}

\end{minipage}
\caption{The asymptotical order of $d_m(\bH^r_p,L_q)$}\label{dmH}
\end{figure}

The region $P_1:=\{(1/p,1/q):0\le 1/p<1/q<1,1/p\le 1/2\}$ is covered by  Theorem
\ref{TBT4.5'}. For this region the upper bounds follow from Theorem
\ref{TBT3.3}. Therefore, in this case the subspaces $\Tr(Q_n)$ are optimal in
the sense of order. In the case of region $P_2:=\{(1/p,1/q):1/2<1/p<1/q\le1\}$
the right order of $d_m(\bH^r_p,L_q)$ is not known. The
region $P_3:=\{(1/p,1/q): 0<1/q\le 1/p<1\}$ is covered by Theorem \ref{TBT4.5}.
The boundary segment $A_1:=\{(1,1/q):0<1/q\le 1/2\}$ is covered by Theorem
\ref{Tp=1}, the segment $A_2:=\{(1/p,1): 0\le 1/p\le 1/2\}$ is covered by
Theorem \ref{Tq=1}, and, finally, the segment $A_3:=\{(1/p,0): 0\le 1/p\le
1/2\}$ in the case $d=2$ is covered by Theorem \ref{Te96H}. In all other cases
the right orders of $d_m(\bH^r_p,L_q)$ are not known.
 
Let us continue with results on $d_m(\Brpt,L_q)$. As for $E_{Q_n}(\Brpt)_q$,
the third parameter $\theta$ is reflected on the asymptotic order of
$d_m(\Brpt,L_q)$.

\begin{thm} \label{thm[d_mB,p<q]}
Let $r > 0$, $1 < p \le q < \infty$, $1\le \theta < \infty$, $\beta := \frac{1/p - 1/q}{1 - 2/q}$.
Then we have
\begin{equation}\nonumber
d_m(\Brpt,L_q)
\ \asymp \
\left(\frac{(\log m)^{d-1}}{m}\right)^{r - [1/p - \max\{1/2,1/q\}]_+}
(\log m)^{(d-1)\delta},
\end{equation}
where
\begin{equation}\nonumber
\delta
 := \
\begin{cases}
{(1/2 - 1/\theta)_+}, & \ \mbox{if} 
  \  2 \le p < q, \ r > \beta,  \ 
\\[1ex]
{[\max\{1/2,1/q\} - 1/\theta]_+}, & \ \mbox{if} \   
p \le q \le 2, \ r > 1/p - 1/q, \
\mbox{or} \\
 &  \quad \  p \le 2 < q, \  r > 1/p.
\end{cases}
\end{equation}
\end{thm}

In Theorem \ref{thm[d_mB,p<q]}, the cases $2 \le p < q$ and  $ p \le 2 < q$ were proved in
\cite{Ga01}, the upper bound of the case $p \le q \le 2$ was proved in
\cite{Di85, Di86}, the lower bound of the case $p \le q \le 2$  was proved in
\cite{Ro91} for $\theta \ge 1$.

\begin{thm} \label{thm[d_mB,p>q]}
Let $r > 0$, $1\le \theta < \infty$. Let $2 \le q \le p < \infty$ or $1 < q < 2
\le p < \infty$ and $\theta \ge 2$.
Then we have
\begin{equation}\nonumber
d_m(\Brpt,L_q)
\ \asymp \
\left(\frac{(\log m)^{d-1}}{m}\right)^r
(\log m)^{(d-1)(1/2 - 1/\theta)_+}\,.
\end{equation}
\end{thm}

In Theorem \ref{thm[d_mB,p>q]}, the upper bounds were proved in \cite{Di85,
Di86}, the lower bounds were proved in \cite{Ga01}.

Let us give a brief comment on the proofs of Theorems \ref{thm[d_mB,p<q]} and
\ref{thm[d_mB,p>q]} for which $1 < p,q < \infty$. For details, the reader can
see \cite{Di85, Di86, Ga01, Ro91}. We first treat the upper bounds.
The cases $p \le q \le 2$, $2 \le q \le p < \infty$ and $1 < q < 2 \le p < \infty$ are derived from the linear
approximation by the operators $S_{Q_n}$, the case $p \le 2 < q$ can be reduced by the embedding 
$\Brpt \hookrightarrow {\bf B}^{r-1/p+1/2}_{2,\theta}$ (see Lemma \ref{emb}) to the case $2 \le p < q$. 
By using Corollary~\ref{C7.1} the case $2 \le p < q$ can be reduced to the upper bounds of 
Kolmogorov widths of finite-dimensional sets which can be estimated by the following 
Kashin-Gluskin's Lemma \ref{Lemma[Kashin-Gluskin]}, see \cite{Ka77}, \cite{Gl83} below.  Below, in Lemmas \ref{Lemma[Kashin-Gluskin]}-\ref{Lemma[MaRu16]}, Theorem \ref{TBT1.4.6} 
and Lemma \ref{Lemma[Gluskin]} we will state results on $n$-widths and linear widths of finite dimensional sets. For a rather complete survey on these results see e.g.\ Vyb\'iral \cite{Vy08b}.

\begin{lem} \label{Lemma[Kashin-Gluskin]}
Let $2 \le p < q < \infty$ and $1/p + 1/q \ge 1$ and $\beta := \frac{1/p - 1/q}{1 - 2/q}$. Then we have for $n > m$, 
\begin{equation} \nonumber
d_m(B^n_p, \ell^n_q)
\ \asymp \
\min\{1, n^{2\beta/q}m^{-\beta}\}.  
\end{equation}
\end{lem}

Notice also that the restriction on the smoothness $r > \beta$ in the case $2
\le p < q$ is important in the proof of the upper bound, for details see
\cite{Gl83}.

We next consider the lower bounds in Theorem \ref{thm[d_mB,p>q]}. The cases $p \le q \le 2$,  $2 \le q \le p <
\infty$ and $1 < q < 2 \le p < \infty$ can be reduced to the  Kolmogorov widths
of finite-dimensional sets which can be upper
estimated for the cases $2 \le q \le p < \infty$ and $1 < q < 2 \le p < \infty$
 by Pietsch-Stesin's lemma below. The cases $ 2 \le p < q$, $ p \le 2 < q$ can be reduced by the norm inequality  
$\|\cdot\|_q \ge \|\cdot\|_2$ to the cases $p \le q \le 2$, $p \ge q \ge 2$ with  $q = 2$. 

\begin{lem} \label{Lemma[Pietsch-Stesin]}
Let $1 \le q \le p \le \infty$. Then we have for $n > m$, 
\begin{equation} \nonumber
d_m(B^n_p, \ell^n_q)
\ \asymp \
(n - m)^{1/q - 1/p}.
\end{equation}
\end{lem}
\noindent
For the case $p \le q \le 2$, the
upper bounds of Kolmogorov widths of finite-dimensional sets are estimated by
Kolmogorov-Petrov-Smirnov-Mal'tsev-Stechkin's equation 
$d_m(B^n_1, \ell^n_2) = (1 - m/n)^{-1/2}$ (see, e.g., \cite{Ga01} for
references). 

Another important tool is given by Galeev's lemma \cite{Ga01} below. Let for $1 \le p,q \le \infty$ the finite dimensional normed space 
$\ell^{n,s}_{p,q}$ of vectors $\bx \in \R^{sn}$ be equipped with the norm 
\[
\|x\|_{\ell^{n,s}_{p,q}}
:= \
\biggl(\sum_{j=1}^s \biggl(\ \sum_{(j-1)n < k \le jn} \ |x_k|^p \biggl)^{q/p}
\biggl)^{1/q}. 
\]
$B^{n,s}_{p,q}$ denotes its unit ball.
\noindent

 \begin{lem} \label{Lemma[Galeev]}
Let $1 < q < \infty$. Then we have for $m\leq ns/2$ that  
\begin{equation} \nonumber
d_m(B^{n,s}_{1,\infty}, \ell^{n,s}_{2,q})
\ \asymp \
s^{1/q}.
\end{equation}
\end{lem} 
Izaak \cite{Iz94, Iz96} extended the above result to $q=1$ and obtained the following. 

\begin{lem} \label{Lemma[Izaak]}
For $m\leq ns/2$ it holds
\begin{equation} \nonumber
s\frac{\sqrt{\log\log s}}{\log s} \lesssim d_m(B^{n,s}_{1,\infty}, \ell^{n,s}_{2,1}) \lesssim s. 
\end{equation}
\end{lem} 

The case $q=1$ is particularly important for proving some of the sharp lower bounds for the linear widths of the H\"older-Nikolskii 
\index{Function spaces!Mixed smoothness H\"older-Nikol'skii}classes  
in Theorem \ref{thm[lambda_nH]} below. See the proof of the main theorem in \cite{Ga96}. Note that there is a very tiny gap in the result by Izaak which has been
closed only recently by Malykhin and Ryutin \cite{MaRu16} with a new method based on a probabilistic argument. In fact, we have 

\begin{lem} \label{Lemma[MaRu16]}
For $m\leq ns/2$ it holds
\begin{equation} \nonumber
d_m(B^{n,s}_{1,\infty}, \ell^{n,s}_{2,1}) \asymp s. 
\end{equation}
\end{lem}

\begin{thm} \label{thm[d_mB,p=1]}
Let $1 < q \le 2$, $r > 1 - 1/q$, $1 \le \theta \le q$ or $2 \le q < \infty$,
$r > 1$, $1 \le \theta < \infty$.
Then we have
\begin{equation}\nonumber
d_m({\bf B}^r_{1,\theta},L_q)
\ \asymp \
\left(\frac{(\log m)^{d-1}}{m}\right)^{r - [1 - \max\{1/2,1/q\}]_+}
(\log m)^{(d-1)\delta},
\end{equation}
where
\begin{equation}\nonumber
\delta
 := \
\begin{cases}
{(1/2 - 1/\theta)_+}, & \ \mbox{if} 
  \  2 \le q < \infty,  
\\[1ex]
0, & \ \mbox{if} \ 1 < q \le 2.
\end{cases}
\end{equation}
\end{thm}

Theorem \ref{thm[d_mB,p=1]} is proved in \cite{Ro05}.

\begin{thm} \label{thm[d_mB,p=q=infty]}
Let $r > 0$.
Then we have
\begin{equation}\nonumber
d_m({\bf B}^r_{\infty,1},L_\infty)
\ \asymp \
\left(\frac{(\log m)^{d-1}}{m}\right)^r.
\end{equation}
\end{thm}

Theorem \ref{thm[d_mB,p=q=infty]} was proved in \cite{Ro06}.
Notice that in both Theorems \ref{thm[d_mB,p=1]} and \ref{thm[d_mB,p=q=infty]} 
the upper bounds are derived from the linear approximation by the operators
$V_{Q_n}$. We note that the following relation 
\begin{equation}\nonumber
d_m({\bf B}^r_{\infty,1},L_q)
\ \asymp \
\left(\frac{(\log m)^{d-1}}{m}\right)^r,\quad 2\le q\le \infty
\end{equation}
holds. The upper bounds are derived from the linear approximation by the operators
$V_{Q_n}$. The lower bounds follow from Example 1 of Chapter 3 in \cite{TBook}. 


\subsection{Orthowidths}\index{Width!Fourier, Ortho-}
We begin with the classes $\bW^r_p$. For numbers $m$ and $r$ and parameters $1
\le p,q \le \infty$ we define the
functions
$$
w(m,r,p,q) :=\bigl(m^{-1}(\log m)^{d-1}\bigr)^{r-(1/p-1/q)_+}
(\log m)^{(d-1)\xi(p,q)},
$$
where $(a)_+:= \max\{a,0\}$ and
$$
\xi(p,q) =
\begin{cases}
0&\qquad\text{ for }\qquad 1< p\le q<\infty,\qquad 1\le q<p\le\infty,\\
1-1/p&\qquad\text{ for }\qquad 1 < p \le \infty,\qquad q = \infty,\\
1/q&\qquad\text{ for }\qquad p=1,\qquad 1\le q\le\infty.
\end{cases}
$$

\begin{thm}\label{TBT4.1} Suppose that $r > (1/p-1/q)_+$ and
$(p,q)\ne(1,1)$, $(\infty,\infty )$, $(1,\infty)$. Then
$$
\varphi_m(\bW_{p}^r,L_q) \asymp w(m,r,p,q),
$$
and orthogonal projections $S_{Q_n}$ onto the subspaces $\Tr(Q_n)$ of trigonometric
polynomials with harmonics in the step hyperbolic crosses with $n$ such that
$|Q_n|\le m$, $|Q_n|\asymp m$
give the order of the quantities
$\varphi_m(\bW_{p}^r,L_q)$.
\end{thm}
Theorem \ref{TBT4.1} in the case $1<p\le q<\infty$ was obtained in \cite{Te82a}
and \cite{Tem10} and in the general case in \cite{TE1} and \cite{TE2}. 

The following theorem from \cite{AT} covers the case $(p,q) = (1,1)$ or
$(\infty,\infty)$.
\begin{thm}\label{AT} The following order estimates hold for $p=1$ and
$p=\infty$
$$
\ff_m(\bW^r_p,L_p)\asymp m^{-r}(\log m)^{(d-1)(r+1)},\qquad r>0.
$$
\end{thm}

The only case, which is not completely studied, is the case $p=1$ and
$q=\infty$. We have here the following partial result, which shows an
interesting behavior of $\ff_m(\bW^r_1,L_\infty)$ in this case. 

\begin{thm}\label{TBT4.2}  For  odd $r > 1$ and arbitrary $d$ we have
$$
\varphi_m(\bW_{1}^r,L_{\infty})\asymp  \bigl(m^{-1}(\log m)^{d-1}
\bigr)^{r-1},
$$
and for even $r$ and $d = 2$ we have
$$
\varphi_m(\bW_{1}^r,L_{\infty} )\asymp m^{1-r}(\log m)^r .
$$
\end{thm}

We now proceed to classes $\bH^r_p$. For numbers $m$ and $r$ and parameters $1
\le p,q \le \infty$ we define the
functions
$$
h(m,r,p,q)=\bigl(m^{-1}(\log m)^{d-1}\bigr)^{r-(1/p-1/q)_+}
(\log m)^{(d-1)\eta(p,q)},
$$
where
$$
\eta(p,q) =
\begin{cases}
1/q&\qquad\text{ for }\qquad 1<p<q<\infty;\qquad p=1,\qquad 1\le q<\infty;\\
1&\qquad\text{ for }\qquad 1\le p\le \infty,\quad\qquad q=\infty;\\
1/2&\qquad\text{ for }\qquad 1\le q\le p\le\infty,\qquad p\ge 2,\qquad
q<\infty;\\
1/q&\qquad\text{ for }\qquad 1\le q\le p \le2.
\end{cases}
$$

\begin{thm}\label{TBT4.3} Suppose that $r > (1/p-1/q)_+$ and
$(p,q)\neq(1,1),(\infty ,\infty )$.
Then
$$
\varphi_m(\bH_p^r,L_q) \asymp h(m,r,p,q)
$$
and subspaces, optimal in the sense of order, are given by $\Tr(Q_n)$
with appropriate $n$ (as in Theorem \ref{TBT4.1}).
\end{thm}
Theorem \ref{TBT4.3} in the case $1\le p<q<\infty$ and $1<p=q\le 2$ was obtained
in \cite{Te82a} and \cite{Tem10}. In the case $1<q<p<2$ it was proved in
\cite{Ga88}. In the general form it was obtained in \cite{TE1} and \cite{TE2}. 

The following theorem from \cite{AT} covers the case $(p,q) = (1,1)$ or
$(\infty,\infty)$.
\begin{thm}\label{ATH} The following order estimates hold for $p=1$ and
$p=\infty$
$$
\ff_m(\bH^r_p,L_p)\asymp m^{-r}(\log m)^{(d-1)(r+1)},\qquad r>0.
$$
\end{thm}

In the cases $p=q=1$ and $p=q=\infty$ the operators $S_{Q_n}$ do not provide
optimal in the sense of order approximation for classes $\bW$ and $\bH$. In
these cases other orthonormal system -- the wavelet type system of orthogonal
trigonometric polynomials -- was used (see \cite{AT}). The proofs of upper
bounds in Theorems \ref{AT} and \ref{ATH} are based on the approach discussed in
the beginning of Subsection 4.2 (see Proposition \ref{AT1}).

The main point of the proof of Theorems \ref{TBT4.1} and \ref{TBT4.3} is in
lower bounds. The lower bound proofs are based on special examples. Some of
these examples are simple, like, $e^{i(\bk^0,\bx)}$ with $\bk^0$ determined by
the system $u_1, u_2,\dots,u_m$ from the definition of the orthowidth $\ff_m$.
Other examples are more involved. For instance, a function $g(\bx+\by^*)$ of the
form
$$
g(\bx):=\sum_{\bs\in \theta_n^1} e^{i(\bk^\bs,\bx) },
$$
with $\by^*$, $\{\bk^\bs\}_{\bs\in\theta_n^1}$, $\theta_n^1\subset
\theta_n:=\{\bs:|\bs|_1=n\}$ are determined by the system $u_1,
u_2,\dots,u_m$. The reader can find a detailed discussion of these examples in
\cite{TE2} and \cite{TBook}.

We note that the case $1\le q<p<2$ turns out to be difficult for classes
$\bH^r_p$. Even 
the corresponding result for $E_{Q_n}(\bH^r_p)_q$ was difficult and required a
new technique. The right order of the Kolmogorov widths in this case is still
unknown.

We now proceed to classes $\Brpt$. For numbers $m$ and $r$ and parameters 
$1 < p,q < \infty$, $1 \le \theta < \infty$, we define the
functions
\[
g(m,r,p,q, \theta) =
\left(\frac{(\log m)^{d-1}}{m}\right)^{r - (1/p - 1/q)_+}
(\log m)^{(d-1)\delta(p,q,\theta)},
\]
where
\begin{equation}\nonumber
\delta(p,q,\theta) =
\begin{cases}
(1/q - 1/\theta)_+ &\qquad\text{ if }\qquad  p < q;\\
1 &\qquad\text{ if }\qquad  q\le p,\ \theta \le \min\{p,2\}; \\
1/2 - 1/\theta &\qquad\text{ if }\qquad  q\le p,\ p \ge 2, \ \theta\ge 2;\\
(1/p - 1/\theta)_+ &\qquad\text{ if }\qquad  q\le p \le 2.
\end{cases}
\end{equation}

\begin{thm} \label{thm[phi_mB]}
 Let  $1 < p,q < \infty$, $1 \le \theta < \infty$, $r > (1/p-1/q)_+$.
Then
\[
\varphi_m(\Brpt,L_q) 
\asymp 
g(m,r,p,q,\theta)
\]
and subspaces, optimal in the sense of order, are given by $\Tr(Q_n)$
with appropriate $n$ (as in Theorem \ref{TBT4.1}).
\end{thm}

Theorem \ref{thm[phi_mB]} was obtained
in \cite{Di86} excepting the cases $q < p \le 2$ and $q = p \le 2$, $q > \theta$. 
The upper bounds of the last two cases  were proved in \cite{Di85}, the lower
bounds follow from a result on nuclear widths in 
\cite{Ga90b} (the  nuclear width which is smaller than orthowidth is defined in  
a way similar to the definition of orthowidth by replacing orthogonal
projectors by more general nuclear operators, for details see \cite{Ga90b}).

\begin{thm} \label{thm[phi_mB,q=1]}
 Let  $1 < p \le \infty$, $1 \le \theta < \infty$, $r > 0$.
Then
\[
\varphi_m(\Brpt,L_1) 
\asymp 
\left(\frac{(\log m)^{d-1}}{m}\right)^r
(\log m)^{(d-1)(1/p - 1/\theta)_+}
\]
and subspaces, optimal in the sense of order, are given by $\Tr(Q_n)$
with appropriate $n$ (as in Theorem \ref{TBT4.1}).
\end{thm}
\noindent Theorem \ref{thm[phi_mB,q=1]} was proved in \cite{Ro08}. 

 The lower bounds (the main part) in the following result (see \cite{Rom11}) follow from 
\eqref{RoStB1} and the corresponding example for the class $\bH^r_p$. 

\begin{thm}\label{RomFour} Let $1\leq p<\infty$, $r>\frac{1}{p}$,
$1\leq \theta <\infty$. Then
$$
\varphi_m({\bB}_{p,\theta}^r, L_\infty) \asymp \left( \frac{\left(\log m \right)^{d-1}}{m}\right)^{r-\frac{1}{p}} \left(\log m \right)^{(d-1)\left(1-\frac{1}{\theta}\right)},
$$
and subspaces, optimal in the sense of order, are given by $\mathcal{T}(Q_n)$ with appropriate $n$
(as in Theorem 4.32).
\end{thm}
The right orders of the $\varphi_m({\bB}_{p,\theta}^r, L_q)$ in the cases $p=q=1$ and $p=q=\infty$, 
where the wavelet-type systems are used to prove the upper bounds, were obtained by D.B. Bazarkhanov (see \cite{Ba10a,Ba10b}). 
D.B. Bazarkhanov also obtained right orders of the Fourier widths for the Triebel-Lizorkin classes.

Results of Subsection 4.4 on the orthowidths show that if we want to approximate classes $\bW^r_p$ or $\bH^r_p$ in $L_q$
with $(p,q)$ distinct from $(1,1)$ and $(\infty,\infty)$ by operators of orthogonal projection of rank $m = |Q_n|$, then
the best (in the sense of order) operator is  $S_{Q_n}$. The operator $S_{Q_n}$ is a very natural operator for
approximating classes $\bW^r_p$ and $\bH^r_p$. Therefore, we can ask the following question: How much can we weaken the
assumption that the rank $m=|Q_n|$ linear operator is an orthogonal projection and still get that $S_{Q_n}$ is the best
(in the sense of order)? Here is a result in this direction. In \cite{TBook} along with the quantities
$\varphi_m(\bF,L_p)$ we consider the quantities
$$
\varphi_m^B (\bF,L_q ) =\inf_{G\in \mathcal L_m(B)_q}
\sup_{f\in \bF\cap D(G)}\bigl\|f - G(f)\bigr\|_q,
$$
where $B \ge 1$ is a number and $\mathcal L_m (B)_q$ is the set of
linear operators $G$
with domains $D(G)$ containing all trigonometric polynomials, and
with ranges contained in an $m$-dimensional subspace of $L_q$, such that
$\|Ge^{i(\bk,\bx)}\|_2 \le B$ for all $\bk$.
It is clear that $\mathcal L_m(1)_2$ contains the operators of orthogonal
projection onto $m$-dimensional subspaces, as well as operators given
by multipliers $\{\lambda_l\}$ with $|\lambda_l|\le 1$ for all $l$
with respect to an orthonormal system of functions. 
It is known (see \cite{TBook}) that in the case $(p,q)$ distinct from $(1,1)$ and $(\infty,\infty)$
$S_{Q_n}$ gives the order of $\varphi_m^B (\bF,L_q )$ for both $\bF=\bW^r_p$ and $\bF=\bH^r_p$. 
Another result of that same flavor is about nuclear widths mentioned above (see \cite{Ga90b}).


\subsection{The linear widths}\index{Width!Linear}
\label{linearw}

As in the previous subsections, we begin with results on  classes $\Wrp$.

\begin{thm} \label{thm[lambda_nW]}
Let $r > (1/p - 1/q)_+$, $1 < p < \infty$ and $1 \le q < \infty$.
Then we have 
\begin{equation} \label{lambda_nW}
\lambda_m(\Wrp,L_q)
\ \asymp \
\begin{cases}
\left(\frac{(\log m)^{d-1}}{m}\right)^{r - (1/p - 1/q)_+}, & \ \mbox{for}  
\quad q \le 2, \ \mbox{or} \quad p \ge 2;
\\[1ex]
\left(\frac{(\log m)^{d-1}}{m}\right)^{r - 1/p + 1/2}, & \ \mbox{for}  
\quad 1/p + 1/q \ge 1, \ q > 2, \ r > 1/p; 
\\[1ex]
\left(\frac{(\log m)^{d-1}}{m}\right)^{r - 1/2 + 1/q}, & \ \mbox{for}  
\quad 1/p + 1/q \le 1, \ p < 2, \ r > 1 - 1/q.
\end{cases}
\end{equation}
\end{thm}

Theorem \ref{thm[lambda_nW]} for $1 < q < \infty$ was proved by
Galeev~\cite{Ga87,Ga96}, and for $q=1$ by Romanyuk~\cite{Ro08}. It is interesting to
notice that by putting 
$M = \frac{m}{(\log m)^{d-1}}$, the relations \eqref{lambda_nW} look like the
asymptotic order of 
$\lambda_m(W^r_p,L_q)$ in the univariate case ($d=1$), see e.g. \cite{ET92}.

We summarize the above results on the $\lambda_m(\bW^r_p,L_q)$ in the following picture.

\begin{figure}[H]
\begin{minipage}{0.48\textwidth}
\begin{center}
\begin{tikzpicture}[scale=2.5]
\draw[->] (2,0.0) -- (2.1,0.0) node[right] {$\frac{1}{p}$};
\draw (0,0) -- (-0.1,0.0);
\draw (0,0) -- (0,-0.1);
\draw (0.0,-.1) -- (0.0,2.0);
\draw[very thick][dashed] (-0,0.0) -- (2,0.0);
\draw[->] (0.0,2) -- (0.0,2.1) node[above] {$\frac{1}{q}$};
\draw[ultra thick] (0,0) -- (0,2);
\draw (1.0,0.03) -- (1.0,-0.03) node [below] {$\frac{1}{2}$};
\draw (0.03,1) -- (-0.03,1) node [left] {$\frac{1}{2}$};

\node at (1.7,2.2) {$\lambda_m(\bW^r_p,L_q)$};

\draw[ultra thick] (0,2) -- (2,2);
\draw (1,1) -- (2,1);

\draw (1,1) -- (2,0);
\draw (1,0) -- (1,1);
\draw (1,1) -- (2,1);
\draw[very thick][dashed] (2,2) -- (2,0);

\node at (1.4,0.2) {\tiny ${r-\frac{1}{2}+\frac{1}{q}}$};
\node at (1.6,0.8) {\tiny ${r-\frac{1}{p}+\frac{1}{2}}$};
\node at (1,1.4) {\tiny $\Big(\frac{(\log
m)^{d-1}}{m}\Big)^{r-(\frac{1}{p}-\frac{1}{q})_+}$};

\node at (1,1.8) {\tiny {\bf  {Hyperbolic cross optimal}}};
\draw (2,0.03) -- (2,-0.03) node [below] {$1$};
\draw (0.03,2) -- (-0.03,2) node [left] {$1$};

\end{tikzpicture}

\end{center}

\end{minipage}
\begin{minipage}{0.48\textwidth}
 \begin{center}
\begin{tikzpicture}[scale=2.5]

\draw (-0.1,0) -- (0,0);
\draw[->] (2,0.0) -- (2.1,0.0) node[right] {$\frac{1}{p}$};
\draw[very thick][dashed]  (0,0.0) -- (1,0.0);
\draw[very thick][dashed]  (1,.0) -- (2,0.0);
\draw[->] (0.0,-.1) -- (0.0,2.1) node[above] {$\frac{1}{q}$};

\draw (1.0,0.03) -- (1.0,-0.03) node [below] {$\frac{1}{2}$};
\draw (0.03,1) -- (-0.03,1) node [left] {$\frac{1}{2}$};

\draw[ultra thick](0,2) -- (2,2);
\draw (1,1) -- (2,1);
\draw (0,0) -- (1,1);

\draw (1,1) -- (2,1);
\draw[very thick][dashed] (2,2) -- (2,1);
\draw[ultra thick] (2,1) -- (2,0);
\draw[ultra thick] (0,2) -- (0,0);
\node at (1,1.4) {\tiny $\Big(\frac{(\log
		m)^{d-1}}{m}\Big)^{r-(\frac{1}{p}-\frac{1}{q})_+}$};
\node at (1.1,0.3) {\tiny $\Big(\frac{(\log
		m)^{d-1}}{m}\Big)^{r-(\frac{1}{p}-\frac{1}{2})_+}$};
\node at (1,1.8) {\tiny {\bf  {Hyperbolic cross optimal}}};
\draw (2,0.03) -- (2,-0.03) node [below] {$1$};
\draw (0.03,2) -- (-0.03,2) node [left] {$1$};
\node at (1.7,2.2) {$d_m(\bW^r_p,L_q)$};
\end{tikzpicture}

\end{center}

\end{minipage}
\caption{Comparison of $\lambda_m(\bW^r_p,L_q)$ and $d_m(\bW^r_p,L_q)$}
\end{figure}

\noindent The region
$$
R_1 := \{(1/p,1/q) : \quad 0<1/p \le 1/2, \  \quad \text{or} \quad  1/2 \le 1/q \le 1\}
$$
is covered by Theorem \ref{thm[lambda_nW]}. For this region the upper bounds
follow from Theorem \ref{TBT3.2} and the fact that the operators $S_{Q_n}$ 
are uniformly bounded from $L_q$ to $L_q$, $1<q<\infty$. It means that in this case the subspaces
$\Tr(Q_n)$ of the hyperbolic cross polynomials are optimal in the sense of
order. Theorem \ref{thm[lambda_nW]} shows that for the small square $(1/2,1)
\times (0,1/2)$ the subspaces $\Tr(Q_n)$ are not optimal in the sense of order. 
Theorem \ref{thm[lambda_nW]} gives the orders of the $\lambda_m(\Wrp,L_q)$ for
all $(1/p,1/q)$ from the  square $(0,1) \times (0,1]$ under some restrictions on
$r$. In all other cases the right orders of $\lambda_m(\Wrp,L_q)$ are not known
(see Open problem 4.8).

\begin{thm} \label{thm[lambda_nH]}  Let $1 \leq p,q \leq \infty$ and $r > (1/p - 1/q)_+$.
Then we have
 \begin{equation}\nonumber
\begin{split}
&\lambda_m(\Hrp,L_q)\\
&\hspace{0.1cm}\ \asymp \
\begin{cases}
\left(\frac{(\log m)^{d-1}}{m}\right)^r 
(\log m)^{(d-1)/2}, & \ \mbox{for}  
 \quad 1\leq q\leq p\leq \infty,\\
 &\hspace{1cm} p\geq 2, q<\infty; 
\\[1ex]
\left(\frac{(\log m)^{d-1}}{m}\right)^{r - 1/p + 1/q} 
(\log m)^{(d-1)/q}, & \ \mbox{for}  
\quad 1<p \le q \le 2;  \\
&\hspace{1cm} 2\leq p<q<\infty;
\\[1ex]
\left(\frac{(\log m)^{d-1}}{m}\right)^{r - 1/p + 1/2} (\log m)^{(d-1)/2}, & \
\mbox{for}  
\quad 1/p + 1/q \ge 1,\\
&\hspace{1cm} 2<q<\infty, \ r > 1/p;
\\[1ex]
\left(\frac{(\log m)^{d-1}}{m}\right)^{r - 1/2 + 1/q}
(\log m)^{(d-1)/q}, & \ \mbox{for}  \quad
1/p+1/q<1, q<\infty,\\
&\hspace{1cm} 1<p\leq 2, r>1-1/q.
\end{cases}
\end{split}
\end{equation}
\end{thm}
\noindent  The first line is due to Temlyakov \cite{TE2}. We make a remark on this result, which is Theorem 3.2 from
\cite{TE2}. The most important part of this result is the lower bound for $\lambda_m(\bH^r_\infty,L_1)$, which is proved in 
\cite{TE2}. For the upper bounds in Theorem 3.2 from \cite{TE2} the reference is given. 
However, it is not pointed out there that the upper bounds are proved for $q<\infty$. In other words, 
the condition $q<\infty$ is missing there. In the case $p=q=\infty$ the following relation holds for $d=2$, $r>0$
$$
\lambda_m(\bH^r_\infty,L_\infty) \asymp m^{-r}(\log m)^{r+1}.
$$
Comments on the lower bound are given after Theorem \ref{Te96H} above. Approximation by $V_{Q_n}$ with an appropriate $n$ implies the upper bound. 
The right order of $\lambda_m(\bH^r_\infty,L_\infty)$ in case $d\ge 3$ is not known. The cases $1<p\leq q\leq 2$ as well as $1/p+1/q\geq 1$, $2<q<\infty$, $r>1/p$ in 
Theorem \ref{thm[lambda_nH]} are due to Galeev \cite{Ga96}. The remaining cases have been settled very recently by Malykhin and Ryutin \cite{MaRu16}. These authors provided 
sharp lower bounds (based on Lemma \ref{Lemma[MaRu16]} above) to well-known upper bounds given by Galeev \cite{Ga96}. 
Note, that the case $1\leq q\leq p<2$ is still open. For an upper bound we refer the reader to the third case in Theorem \ref{TBT3.3} above.

The fact that the class $\Hrp$ are properly larger than the class $\Wrp$ is reflected to
$\lambda_m(\Hrp,L_q)$ by the additional logarithm term $(\log
m)^{\max\{1/2,1/q\}\,(d-1)}$. We summarize the above results on the
$\lambda_m(\bH^r_p,L_q)$ in the following picture,

\begin{figure}[H]
\begin{minipage}{0.48\textwidth}
 \begin{center}
\begin{tikzpicture}[scale=2.5]
\draw[->] (2.0,0.0) -- (2.1,0.0) node[right] {$\frac{1}{p}$};
\draw[->] (0,2.0) -- (0,2.1) node[above] {$\frac{1}{q}$};
\draw (0,0) -- (-0.1,0.0);
\draw (0,0) -- (0,-0.1);
\draw[very thick][dashed] (1,0.0) -- (2.0,0.0);
\draw[very thick][dashed] (0,0.0) -- (1,0.0);
\draw[very thick][dashed] (0.0,0) -- (0.0,2.0);

\node at (1.7,2.2) {$\lambda_m(\bH^r_p,L_q)$};
\draw (0.03,1) -- (-0.03,1) node [left] {$\frac{1}{2}$};
\draw (1,0.03) -- (1,-0.03) node [below] {$\frac{1}{2}$};

\draw[ultra thick] (0,2) -- (1,2);
\draw[very thick][dashed] (1,2) -- (2,2);
\draw (1,1) -- (1,2);
\draw (0,0) -- (2,2);
\draw[very thick][dashed] (2,2) -- (2,1);
\draw[very thick][dashed] (2,1) -- (2,0);
\draw (1,1) -- (2,0);
\draw (1,1) -- (2,1);
\draw (1,0) -- (1,1); 
\node at (0.22,1.4) {\small $ \alpha = r  $};
\node at (0.22,1.2) {\small $\beta = \frac{1}{2}$};
\node at (0.45,0.82) {\small \bf  {HC opt.}};
\node at (1.6,0.85) { \tiny$\alpha = r-\frac{1}{p}+\frac{1}{2}$};
\node at (1.7,0.65) {\tiny $\beta = \frac{1}{2}$};
\node at (1.8,1.55) {\tiny \bf  {HC opt.}};
\node at (1.6,1.15) { \tiny $\alpha = r-\frac{1}{p}+\frac{1}{q}$};
\node at (1.7	,1.35) { \tiny $\beta = \frac{1}{q}$};

\node at (1.2,1.5) {\huge ?};
\node at (0.55,0.1) {\tiny $\alpha = r-1/p+1/q$};
\node at (0.5,0.27) {\tiny $\beta = 1/q$};
\node at (0.73,0.48) {\tiny \bf  {HC opt.}};

\node at (1.45,0.1) {\tiny $\alpha = r-1/2+1/q$};
\node at (1.3,0.27) {\tiny $\beta = 1/q$};


\draw (2,0.03) -- (2,-0.03) node [below] {$1$};
\draw (0.03,2) -- (-0.03,2) node [left] {$1$};

\end{tikzpicture}
\end{center}
\end{minipage}
\begin{minipage}{0.48\textwidth}
\begin{center}
\begin{tikzpicture}[scale=2.5]
\draw (0,0) -- (0,-0.1);
\draw (-0.1,0) -- (0.0,0.0);
\draw[->] (2,0.0) -- (2.1,0.0) node[right] {$\frac{1}{p}$};
\draw[->] (0.0,2) -- (0.0,2.1) node[above] {$\frac{1}{q}$};
\draw[very thick][dashed] (0,0) -- (2,0);
\draw[very thick][dashed] (0,0) -- (0,2);
\draw (0.03,1) -- (-0.03,1) node [left] {$\frac{1}{2}$};
\draw (1,0.03) -- (1,-0.03) node [below] {$\frac{1}{2}$};

\draw[ultra thick] (0,2) -- (1,2);
\draw[very thick][dashed] (1,2) -- (2,2);
\draw (1,1) -- (1,2);
\draw (0,0) -- (2,2);
\draw[very thick][dashed] (2,2) -- (2,1);
\draw[ultra thick] (2,1) -- (2,0);

\node at (0.5,1.4) {\small $\alpha = r, \beta = \frac{1}{2}$};
\node at (1.3,1.6) {\huge ?};
\node at (1.1,0.15) {\small $\alpha = r-\Big(\frac{1}{p}-\max\big\{\frac{1}{2},\frac{1}{q}\big\}\Big)_+$};
\node at (1.1,0.5) {\small $\beta = \max\big\{\frac{1}{2},\frac{1}{q}\big\}$};
\draw (2,0.03) -- (2,-0.03) node [below] {$1$};
\draw (0.03,2) -- (-0.03,2) node [left] {$1$};
\draw[dashed, color = red] (1,1) -- (2,1);
\node at (1.7,2.2) {$d_m(\bH^r_p,L_q)$};
\node at (1.7,1.3) {\small \bf  {HC opt.}};
\node at (0.45,0.82) {\small \bf  {HC opt.}};
\end{tikzpicture}
\end{center}
\end{minipage}
\caption{Comparison of $\lambda_m(\bH^r_p,L_q)$ and $d_m(\bH^r_p,L_q)$} \label{compH}
\end{figure}

where $\alpha$ and $\beta$ refers to the asymptotic order
$$\Big(\frac{(\log m)^{d-1}}{m}\Big)^{\alpha}(\log m)^{(d-1)\beta}.$$
The region
$$
R_1 := \{(1/p,1/q) : \quad 0<1/p \le \min\{1/2,1/q\} \quad \text{or} \quad  1/2
\le 1/q \le 1/p \le 1\}
$$
is covered by Theorem \ref{thm[lambda_nH]}. For this region the upper bounds
follow from Theorem \ref{TBT3.2}, which means that in this case the subspaces
$\Tr(Q_n)$ of the hyperbolic cross polynomials are optimal in the sense of
order. 
The region
$$
R_2 := \{(1/p,1/q) : \quad 0<1/q <1/p <1, \   1/p + 1/q  \ge 1\}
$$
is covered by Theorem \ref{thm[lambda_nH]}.
Theorem \ref{thm[lambda_nH]}
shows that for this  region the subspaces $\Tr(Q_n)$ are not optimal in the
sense of order. Theorem \ref{thm[lambda_nH]} gives the orders of 
the $\lambda_m(\Hrp,L_q)$ for all $(1/p,1/q)$ from the  regions  $R_1$ and
$R_2$ under some restrictions on $r$. In all other cases the right orders of
$\lambda_m(\bH^r_p,L_q)$ are not known (see Open problem 4.8).


\begin{thm} \label{thm[lambda_nB]}
Let $r > (1/p - 1/q)_+$, $1 < p,q < \infty$ and $1 \le \theta < \infty$.
Then we have
\begin{equation}\nonumber
\begin{split}
&\lambda_m(\Brpt,L_q)\\
&\hspace{0.1cm}\ \asymp \
\begin{cases}
\left(\frac{(\log m)^{d-1}}{m}\right)^r  
(\log m)^{(1/2 - 1/\theta)_+(d-1)}, & \ \mbox{for}  
\ \ 2 \le q < p; 
\\
& \quad \quad q \le 2 \le p, \ \theta \ge 2;
\\[1ex]
 \left(\frac{(\log m)^{d-1}}{m}\right)^{r - 1/p + 1/q}  
, & \ \mbox{for}  
\quad  2 \le p < q,
\\
& \quad \quad  
 \ 2 \le \theta \le q, 
\ r > 1 - 1/q; 
\\[1ex]
\left(\frac{(\log m)^{d-1}}{m}\right)^{r - 1/p + 1/q}  
(\log m)^{(1/q - 1/\theta)_+(d-1)}, & \ \mbox{for}  
\quad  1 < p < q \le 2;
\\[1ex]
\left(\frac{(\log m)^{d-1}}{m}\right)^{r - 1/p + 1/2} 
(\log m)^{(1/2 - 1/\theta)_+(d-1)}, & \ \mbox{for}  
\quad 1/p + 1/q \ge 1, \ q \ge 2, 
\\
& \qquad \quad  r > 1/p;
\\[1ex]
\left(\frac{(\log m)^{d-1}}{m}\right)^{r - 1/2 + 1/q}, & \ \mbox{for}  
\quad 1/p + 1/q \le 1, \ 1< p \le 2, 
\\
& \qquad \quad 2 \le \theta \le q, \ r > 1 - 1/q.
\end{cases}
\end{split}
\end{equation}
\end{thm}

It is too complicated to summarize the above results 
on the $\lambda_m(\bB^r_{p,\theta},L_q)$ in a picture, since there are three parameters
$p,\theta, q$ (with fixed $r$) which require a three-dimensional picture.

\begin{thm} \label{thm[lambda_nB,q=1]}
Let $r > 0$, $2 \le p \le \infty$ and $2 \le \theta < \infty$.
Then we have
\begin{equation}\nonumber
\lambda_m(\Brpt,L_1)
\ \asymp \
\left(\frac{(\log m)^{d-1}}{m}\right)^r  
(\log m)^{(1/2 - 1/\theta)(d-1)}.
\end{equation}
\end{thm}

\begin{thm} \label{thm[lambda_nB,p=1]}
Let $r > 1 - 1/q$, $1 < q \le 2$ and $1 \le \theta \le q$.
Then we have
\begin{equation}\nonumber
\lambda_m({\bf B}^r_{1,\theta},L_q)
\ \asymp \
\left(\frac{(\log m)^{d-1}}{m}\right)^{r - 1 + 1/q}  
(\log m)^{(1/q - 1/\theta)(d-1)}.
\end{equation}
\end{thm}

\begin{thm} \label{thm[lambda_nB,p=q=infty]}
Let $r > 1/2$.
Then we have
\begin{equation}\nonumber
\lambda_m({\bf B}^r_{\infty,1},L_\infty)
\ \asymp \
\left(\frac{(\log m)^{d-1}}{m}\right)^r.
\end{equation}
\end{thm}
In Theorems \ref{thm[lambda_nB]} -- \ref{thm[lambda_nB,p=q=infty]}, 
the case  $1/p + 1/q \ge 1, \  2 \le q < \infty$ was proved  in\cite{Ro01a}, 
the case $q=1, \ p \ge 2$ in \cite{Ro08}, the other cases 
were proved in \cite{Ro01b}.

Let us give a brief comment on the proofs of Theorems \ref{thm[lambda_nW]} --
\ref{thm[lambda_nB,p=q=infty]}. For details, the reader can see \cite{Ga96}, \cite{Ro01a},
\cite{Ro01b}, \cite{Ro08}.
For the cases $p \ge 2$ or $q \le 2$, the upper bounds are derived from the linear
approximation by the operators $S_{Q_n}$, and the lower bounds
from the inequality \eqref{ineq[lambda_n>d_m]} and corresponding lower bounds
for Kolmogorov widths in Theorems \ref{TBT4.4} -- \ref{thm[d_mB,p=q=infty]}. For
the case $p \le 2 < q$ the upper bounds can be reduced to
the upper bounds of linear $n$-widths of finite-dimensional sets which can be
estimated by the following Gluskin's lemma \cite{Gl83}.
\begin{lem} \label{Lemma[Gluskin]}
Let $1 \le p < 2 \le q < \infty$ and $1/p + 1/q \ge 1$. Then we have for $n >
m$, 
\begin{equation}\nonumber 
\lambda_m(B^n_p, \ell^n_q)
\ \asymp \
\max \big\{m^{1/q - 1/p}, (1 - n/m)^{1/2}\, \min\{1, m^{1/q}n^{-1/2}\}\big\}.  
\end{equation}
\end{lem}
The lower bounds for  the case $1/p + 1/q \ge 1$, $q \ge 2$ can be reduced by
the norm inequality  
$\|\cdot\|_q \ge \|\cdot\|_2$ to the case $p \le 2$, $q = 2$, for the case $1/p
+ 1/q \le 1$, $p \le 2$ to the upper bounds of linear $n$-widths of
finite-dimensional sets which can be estimated by Lemma \ref{Lemma[Gluskin]}.

There are cases in which the right order of $\lambda_m(\Brpt,L_q)$ are 
not known.

\subsection{Open problems}\index{Open problems!Linear approximation}

 We presented historical comments and a discussion of results, including open problems, in the above text of Section 4.
We summarize here the most important comments on open problems. A number of asymptotic characteristics is discussed in
this section: the Kolmogorov widths, the linear widths, and the orthowidths (the Fourier widths). It seems like the most
complete results are obtained for the orthowidths (see Subsection 4.4). However, even in the case of orthowidths there
are still unresolved problems. We mention the one for the $\bW$ classes. 
 \newline\newline
{\bf Open problem 4.1.} Find the order of $\varphi_m(\bW^r_1,L_\infty)$ for  all
$r>1$ and $d$.\newline
Results of Subsection 4.3 show that the right order of the Kolmogorov widths $d_m(\bW^r_p,L_q)$ are known for all
$1<p,q<\infty$ and $r>r(p,q)$. However, in the case of extreme values of $p$ or $q$ ($p$ or $q$ takes a value $1$ or
$\infty$) not much is known, see also Figure \ref{dmW}. Here are some open problems.
\newline
 {\bf Open problem 4.2.} Find the order of  $d_m(\bW^r_p,L_\infty)$ and $d_m(\bH^r_p,L_\infty)$ for $2\le
p\le\infty$ and $r>1/p$ in dimension $d\ge 3$. 
 \newline
 {\bf Open problem 4.3.} Find the order of $d_m(\bW^r_p,L_\infty)$ for $1\le p<2$.
 \newline
  {\bf Open problem 4.4} Find the order of $d_m(\bW^r_1,L_q)$ for $1\le q<2$.  
  
  It turns out that the problem of the right orders of the Kolmogorov widths for the $\bH$ classes is more difficult
than this problem for the $\bW$ classes. In addition to some open problems in the case of extreme values of $p$ and $q$
the following case is not settled.
 \newline
{\bf Open problem 4.5.} Find the order of $d_m(\bH^r_p,L_q)$ for $1\le q<p<2$, see also Figure \ref{dmH}.  

For the linear widths we have a picture similar to that for the Kolmogorov widths. We point out that the study of
approximation of functions with mixed smoothness in the uniform norm ($L_\infty$ norm) is very hard. Even the right
orders of approximation by polynomials of special form -- the hyperbolic cross polynomials -- are not known in this
case. 
\newline
{\bf Open problem 4.6.}
 Find the order of $E_{Q_n}(\bW^r_q)_\infty$ for $2<q\le\infty$. 
\newline
{\bf Open problem 4.7.} Find the order of $E_{Q_n}(\bH^r_\infty)_\infty$ for $d\ge
3$.
\newline
{\bf Open Problem 4.8.} Find the right order of $\lambda_m(\bH^r_p,L_q)$ in the missing cases in Figure \ref{compH}.

%% file: sampl_recovery.tex
\section{Sampling recovery}\index{Sampling}\index{Sampling!Recovery}
\label{Sect:Smolyak}

In Section 4 we discussed approximation of functions with mixed smoothness by elements of 
finite dimensional subspaces. On the base of three asymptotic characteristics we concluded that subspaces $\Tr(Q_n)$ of
the hyperbolic cross polynomials are optimal in the sense of order in many situations. The Kolmogorov width $d_m(\bF,Y)$
gives the lower bound for approximation from any $m$-dimensional linear subspace. The linear width $\lambda_m(\bF,Y)$ 
gives the lower bound for approximation from any $m$-dimensional linear subspace by linear operators. Finally, the
orthowidth  $\varphi_m(\bF,Y)$ gives the lower bound for approximation from any $m$-dimensional linear subspace by
operators of orthogonal projections (assuming that the setting makes sense). In addition, from an applied
point of view, restrictions on approximation methods in a form of linear operator and orthogonal projection, there is
one more natural setting. In this setting we approximate (as above) by elements from finite dimensional subspaces but
our methods of approximation are restricted to linear methods, which may only use the 
function values at a certain set of points. We discuss this setting in detail in this section. 
We begin with precise definitions. The general goal of this section is to recover a multivariate continuous
periodic function $f:\T^d \to \C$ belonging to a function class $\bF$ from a
finite set of $m$ function values. To be more precise, we consider linear
reconstruction algorithms of type
$$
  \Psi_m(f,X_m):=\sum\limits_{i=1}^m f(\bx^i)\psi_i(\cdot)
$$
for multivariate functions. The set of sampling nodes $X_m:=\{\bx^i\}_{i=1}^m
\subset \T^d$ and associated (continuous) functions $\Psi_m:=\{\psi_i\}_{i=1}^m$
is fixed in advance. To guarantee reasonable access to function values we need
that the function class $\bF$ consists of functions or equivalence classes which
have a continuous representative. When considering classes $\bW^r_p$ or
$\bB^r_{p,\theta}$ the embedding into $C(\T^d)$ holds whenever
$r>1/p$, see Lemma \ref{emb} above. 

As usual we are interested in the minimal error (sampling 
numbers/widths) in a Banach space $Y$

\begin{equation}\label{sampling-number}
    \varrho_m(\bF,Y) := \inf\limits_{X_m}\inf\limits_{\Psi_m}
\Psi_m(\bF,X_m)_Y\,,
\end{equation}
where 
$$
\Psi_m(\bF,X_m)_Y:=\sup\limits_{f \in \bF} \|f-\Psi_m(f,X_m)\|_Y\,.
$$
The quantities $\varrho_m(\bF,Y)$ are called {\it sampling widths} or {\it
sampling numbers}. The following inequalities hold
\be\label{dlr}
d_m(\bF,Y)\le \lambda_m(\bF,Y) \le \varrho_m(\bF,Y)\,,
\ee
see \eqref{lw} and the relation to the approximation numbers \eqref{an}. Note, that the right
inequality in \eqref{dlr} can not directly be deduced from the formal defintion of the linear width, since the above 
defined sampling operators do not make sense as operators from $L_p$ to $L_p$.

There are no general inequalities relating the characteristics $\varphi_m(\bF,Y)$ and $\varrho_m(\bF,Y)$. In both cases
in addition to the linearity assumption on the approximation operator we impose additional restrictions but those
restrictions are of a very different nature -- orthogonal projections and sampling operators. However, it turns out that
similarly to 
the orthowidth setting, where $\Tr(Q_n)$ are optimal (in the sense of order) in all cases with a few exceptions, the
subspaces $\Tr(Q_n)$ are optimal (in the sense of order) from the point of view of $\varrho_m(\bF,Y)$ in all cases where
we know their right orders. 

We make a more detailed comment on this issue on the example of the univariate problem. Classically, the interpolation problem by polynomials (algebraic and trigonometric) was only considered in the space of continuous functions and the error of approximation was measured in the uniform norm (here, for notational convenience, we denote it $L_\infty$ norm). Restriction to continuous functions is very natural because we need point evaluations of the function in the definition of the recovery operator $\Psi_m$. We are interested in the recovery error estimates not only in the uniform norm $L_\infty$, but in the whole range of $L_q$, $1\le q\le \infty$. It was understood in the first papers on this topic in the mid of 1980s (see, for instance, \cite{T7}) that recovery in $L_q$, $q<\infty$, instead of $L_\infty$ brings difficulties and new phenomena. However, the problem of optimal recovery on classes $W^r_p$ and $H^r_p$ in $L_q$ was solved for all $1\le p,q\le \infty$, $r>1/p$ (see \cite{TBook}):
$$
\varrho_m(W^r_p,L_q) \asymp \varrho_m(H^r_p,L_q) \asymp m^{-r+(1/p-1/q)_+}.
$$
Thus, in the univariate case the asymptotic characteristics $\varrho_m$ and $\varphi_m$ behave similarly and the optimal subspaces for recovery ($1\le p,q\le\infty$) and orthowidth ($1\le p,q\le\infty$, $(p,q)\neq (1,1)$, $(p,q)\neq (\infty,\infty)$) are the trigonometric polynomials. 

It was established in Section 4 that in some cases the right order of the linear width can be realized by an orthogonal
projection operator and in other cases it cannot be realized that way. It means that in the first case $\lambda_m(\bF,Y)
\asymp \varphi_m(\bF,Y)$ and in the second case $\lambda_m(\bF,Y) = o(\varphi_m(\bF,Y))$. It is an interesting problem:
When the linear width can be realized (in the sense of order) by a sampling operator? In other words: When
$\lambda_m(\bF,Y) \asymp \varrho_m(\bF,Y)$?
 Note, that in the univariate setting an
optimal algorithm in the sense of order of $\lambda_m$ for the classes $W^r_p$ in
$L_p$, $1<p<\infty$, $r>1/p$, consists in the standard equidistant interpolation
method. In that sense, $\lambda_m$ and
$\varrho_m$ are equal in order. However, when considering multivariate classes the
notion of ``equidistant'' is not clear anymore. In other words, what are
optimal point sets $X_m$ in the $d$-dimensional cube to sample the function and
build optimal sampling algorithms in the sense of order of $\varrho_m$? Is such an
operator then also optimal  in the sense of order of $\lambda_m$? For the classes
$\bW^r_p$ this represents a well-known open problem. 

Let us begin our discussion of known results with H\"older-Nikol'skii classes $\bH^r_p$.  Temlyakov
\cite{Te85b} proved for $1\leq p\leq \infty$ and $r>1/p$ the relation
\begin{equation}\label{eq:SampH}
    \varrho_m(\bH^r_p,L_p) \lesssim m^{-r}(\log m)^{(r+1)(d-1)}\,.
\end{equation}
The correct order is not known here, except in the situation $d=2$, $p=\infty$,
and $r>1/2$, see Remark \ref{d=2} below. Let the Smolyak-type sampling operators $T_n$ (see Subsection 4.2) be induced
by the univariate linear operators
$$
Y_s(f) := R_{2^s}(f) := 2^{-s-2}\sum_{l=1}^{2^{s+2}}f(x(l))\mathcal V_{2^s}(x-x(l)),\quad x(l):=\pi l 2^{-s-1}.
$$
Then (\ref{eq:SampH}) follows from the error bound \cite{Te85b} 
\begin{equation}\label{ineq:SampHT}
    \sup_{f\in\bH^r_p}\|f-T_n(f)\|_p \lesssim 2^{-rn}n^{d-1}\,.
\end{equation}
Note that with the above specification of $Y_s$ we have $\Delta_\bs(f) \in \Tr(2^{\bs+\bf 1})$ and $T_n(f)\in
\Tr(Q_{n+d})$. Therefore, \eqref{ineq:SampHT} implies
$$
E_{Q_n}(\bH^r_\infty)_\infty \lesssim 2^{-rn}n^{d-1},
$$
which is known to be the right order in the case $d=2$ (see Theorem \ref{TBT3.5}). Comparing (\ref{ineq:SampHT}) with Theorem \ref{TBT3.3} we see that the above sampling operator $T_n$ with
$m\asymp |Q_n|$ does not provide the order of best approximation $E_{Q_n}(\bH^r_p)_p$ for $1<p<\infty$. We note
that the first result on recovering by sampling operators of the type of $T_n$ (Somolyak-type operators) was obtained by
Smolyak \cite{Sm63}. He proved the error bound
\begin{equation}\label{eq:SampWS}
    \sup_{f\in\bW^r_\infty}\|f-T_n(f)\|_\infty \lesssim 2^{-rn}n^{d-1}\,.
\end{equation}
The error bound (\ref{eq:SampWS}) and similar bound for $1\le p< \infty$ follow  from  (\ref{ineq:SampHT}) by embedding, see
\eqref{chainBFB}.

The first correct order
\begin{equation}\label{eq:SampH[p<q<2]}
    \varrho_m(\bH^r_p,L_q)
\ \asymp \ 
m^{-(r- 1/p + 1/q)}(\log m)^{(r-1/p + 2/q)(d-1)}\quad,\quad m\in \N\,,
\end{equation}
for $1< p < q \le 2$, $r > 1/p,$ was proven by Dinh D\~ung \cite{Di91}.
The upper bound is given by the above sampling operator $T_n$ with $m\asymp |Q_n|$: for $1\le
p<q<\infty$,
\be\label{sampub}
 \sup\limits_{f\in \bH^r_p}\|f-T_n(f)\|_q \lesssim 2^{-(r-1/p+1/q)n}n^{(d-1)/q}.
\ee
 The proof is based on the case $p=q$, discussed above, and Remark \ref{RT2.4.6}  to Theorem \ref{T2.4.6} (see inequality (\ref{2.4.4R})). Indeed, $f\in \bH^r_p$ implies $\|v_\bj(f)\|_p \lesssim 2^{-r|\bj|_1}$ (see (\ref{v_s(f)<H}) below). Application of inequality (\ref{2.4.4R}) completes the proof.   The lower
bound for $1< p < q
\le 2$ follows from the inequality $\varrho_m \ge d_m$ and the lower bound of  $d_m(\bH^r_p,L_q)$ for $ 1< p < q \le 2$
obtained by Galeev \cite{Ga90}.
Comparing (\ref{sampub}) with Theorem \ref{TBT3.3} we see that the sampling operator $T_n$ provides the best order of
approximation $E_{Q_n}(\bH^r_p)_q$ in the case $1\le p<q<\infty$. 

Surprisingly, even in the Hilbert
space setting, i.e., for Sobolev classes $\bW^r_2$ there are only
partial results for $\varrho_m(\bW^r_2,L_2)$. Let us mention the
following result due to Temlyakov \cite{T8}. The following situation deals
with the error norm $L_{\infty}$ and provides a sharp result. For $r>1/2$ we
have
\begin{equation}\label{Linfty}
    \varrho_m(\bW^r_2,L_{\infty}) \asymp m^{-(r-1/2)}(\log m)^{r(d-1)}\,.
\end{equation}
Interestingly, we have here $\varrho_m \asymp \lambda_m$. The difficult part of (\ref{Linfty}) is the upper bound. Its proof uses Theorem \ref{T2.4.6}.   
The lower bounds for
$\lambda_m$ were reduced in \cite{T8} to known lower bounds for $d_m(\bW^r_1,L_2)$ using the
Ismagilov duality result for the $\lambda_m$ \cite{Is74}. Notice that using
properties of $2$-summing operators the corresponding lower bound can be derived from a result proven recently 
by Cobos, K\"uhn, Sickel \cite{CoKuSi15}. They showed the beautiful identity 
\be\label{CKS}
      \lambda_m(\bW^r_2,L_{\infty}) = \Big(\sum\limits_{j=m}^{\infty}
\lambda_j(\bW^r_2,L_2)^2\Big)^{1/2}\,,
\ee
which immediately gives the lower bound in \eqref{Linfty}.

The known results on the hyperbolic cross approximation (see Theorem 3.7 from \cite{TBook}, Chapter 3) imply: for $1<
p<\infty$, $r>1/p$
$$
E_{Q_n}(\bW^r_p)_\infty = o\left(\sup_{f\in \bW^r_p}\|f-T_n(f)\|_\infty\right).
$$

Similarly to the class $\Hrp$, there holds the correct order relation
\begin{equation}
    \varrho_m(\Wrp,L_q)
\ \asymp \ 
m^{-(r- 1/p + 1/q)}(\log m)^{(r-1/p + 1/q)(d-1)}
\end{equation}
if either $1< p < q \le 2$ or $2\leq p<q< \infty$ and $r > 1/p$, see Theorem \ref{SampWidthW}(i)-(ii) below 
which has been proved recently in \cite{ByUl15} for $r > 1/p$, and in 
\cite{Di16_1} for $r > \max\{1/p,1/2\}$.
The special case $p=2<q$ is proved in \cite[Thm.\ 6.10]{ByDuUl14}. 
If we replace the uniform error norm by the $L_p$-error in \eqref{Linfty} we can
only say the
following for $r>\max\{1/p,1/2\}$ (including $p=2$), namely 
\begin{equation}\label{op}
    \varrho_m(\bW^r_p,L_p) \lesssim m^{-r} (\log m)^{(r+1/2)(d-1)}\,.
\end{equation}
This result has been first observed by Sickel \cite{Si3} for the case $d=2$. It was extended later 
to arbitrary $d$ by Sickel, Ullrich \cite{SiUl07}, \cite{Ul08}. A corresponding lower bound is not known. Note, that the
linear widths in this
situation are smaller than the right-hand side of \eqref{op}. In fact, there
we have $(\log m)^{(d-1)r}$ instead of $(\log
m)^{(r+1/2)(d-1)}$. Let us mention once more that even in the case $p=2$ it is
not known 
\begin{itemize}
 \item whether the upper bound in \eqref{op} is sharp and 
 \item whether the sampling widths $\varrho_m$ coincide with the 
linear widths $\lambda_m$ in the sense of order. 
\end{itemize}

It is remarkable to notice that we so far only have sharp bounds for the order of
$\varrho_m(\bW^r_p,L_q)$ and $\varrho_m(\bH^r_p,L_q)$ in case $p<q$ and for these cases the Smolyak algorithm is
optimal in the sense of order.
 
In general we do not assume that
the approximant $\Psi_m(f,X_m)$ on $f$
satisfies the interpolation property 
\begin{equation}\label{intprop}
     \Psi_m(f,X_m)(\bx^i) = f(\bx^i)\quad,\quad i=1,...,m\,.
\end{equation}
The already described upper bounds for the sampling widths in the various
situations are based on sampling algorithms on sparse grids obtained by
applying the Smolyak algorithm to univariate sampling operators, see Subsection
\ref{Subs:appr_hypcross} above. In the sequel
we will describe how to create and analyze such operators. As an ingredient we may
use the univariate classical de la Vall\'ee Poussin interpolation which is
described in the
next subsection. The second step is a tensorization procedure. Afterwards we
are going to prove a characterization of the spaces of interest in terms of
those sampling operators which finally yield the stated error bounds. 

\subsection{The univariate de la Vall\'ee Poussin interpolation}
\index{Kernel!de la Vall\'ee Poussin} For $m\in \N$ let $\mathcal{V}_m:=\mathcal{V}_{m,2m}$ be the univariate de la
Vall\'ee Poussin kernel as introduced in Paragraph \ref{univpol}.3. An
elementary calculation, see Paragraph 2.1.3, shows that 
$$
  \mathcal{V}_{m}(t):= \frac1{m} \sum_{k= m}^{2m-1}\mathcal{D}_k(t)
  = \frac{\sin (m t/2)\sin (3m t/2)}{m \sin^2 (t/2)}\quad,\quad m\in \N\,.
$$
In the manner of \eqref{2.1.12} we define for $f\in C(\T)$ the interpolation
operator
\begin{equation}\label{dvp}
  V_m(f,J_{3m}) := \frac{1}{3m}\sum\limits_{\ell=0}^{3m-1}
f\Big(\frac{2\pi\ell}{3m}\Big)\mathcal{V}_{m}
\Big(\cdot-\frac{2\pi\ell}{3m}\Big)\quad , \quad m\in \N\,,
\end{equation}
with respect to the equidistant grid 
$$
    J_N:=\Big\{\frac{2\pi \ell}{N}~:~\ell = 0,...,N-1\Big\}\subset \T\quad,\quad N=3m.
$$
Instead of $J_{3m}$ one can also use more redundant grids in the definition of
\eqref{dvp} like, e.g., $J_{4m}$ in \cite{TBook} or $J_{8m}$ in \cite{Te85b}.
Note, that the particular choice $J_{3m}$ in \eqref{dvp} leads to
$\mathcal{V}_{m}(0) = 3m$ and $\mathcal{V}_{m}(2\pi\ell/3m) = 0$ if $\ell \neq 0$.
Using that, \eqref{dvp} implies the interpolation property \eqref{intprop}. It
is straight-forward to compute the Fourier coefficients of the approximant
$V_m(f,J_{3m})$. This directly leads to the following crucial reproduction
result similar to Theorem \ref{T2.1.2}.

\begin{lem}\label{reprpoly} Let $m\in \N$. The operator $V_m(\cdot,J_{3m})$
does not change trigonometric polynomials of degree at most $m$, i.e., if $f \in
\mathcal{T}(m)$ is a trigonometric polynomial then
$$
      V_m(f,J_{3m}) = f\,.
$$
\end{lem}
Let us also recall the classical trigonometric interpolation from
\eqref{2.1.4} by 
$$
D_m(f,J_{2m+1}):=\frac{1}{2m+1}\sum\limits_{\ell=0}^{2m}
f\Big(\frac{2\pi\ell}{2m+1}\Big)\mathcal{D}
_m\Big(\cdot-\frac{2\pi\ell}{2m+1}\Big)\quad , \quad m\in \N\,.
$$
This approximant is also interpolating its argument in the nodes $J_{2m+1}$ and
Lemma \ref{reprpoly} keeps valid also for this interpolation operator. However,
it has two disadvantages. The first one is the fact that subsequent
dyadic grids are not nested, i.e., when comparing $D_{2^{j+1}}$ and $D_{2^j}$.
However, this can be fixed by using the modified Dirichlet kernel
$\mathcal{D}^1_n$, see \eqref{nestedDirichlet}. The second issue is related to
the effects described in Theorems \ref{2.1.1} and
\ref{2.1.2}. 
In the case of extreme value of $p$ or $q$ it might be a problem to work with the Dirichlet kernels. 
In fact, it is
known that the classical trigonometric interpolation
does not provide the optimal rate of recovery in case $p=1$ and $p=\infty$. That
might be no problem when
considering Sobolev classes $\bW^r_p$ for $1<p<\infty$. However, the cases $p=1$
and $p=\infty$ represent important special cases for Besov classes $\Brpt$. All
the mentioned de la Vall\'ee Poussin sampling operators show the following
behavior in the univariate setting, see for instance \cite{TBook}, \cite{SiSp99}
and \cite{Ul04} for further details.

\begin{thm}\label{univ_bound} {\em (i)} Let $1\leq p,\theta \leq \infty$ and
$r>1/p$ then 
$$
    \|f-V_m(f,J_{3m})\|_p \lesssim m^{-r}\|f\|_{B^r_{p,\theta}}\,.
$$
{(ii)} Let $1<p<\infty$ and $r>1/p$ then 
$$
  \|f-V_m(f,J_{3m})\|_p \lesssim m^{-r}\|f\|_{W^r_p}\,.
$$
\end{thm}

Let us emphasize that the use of a sampling operator in $L_p$, $p<\infty$, immediately brings 
problems. For instance, it is easy to prove that 
$$
\|V_m(f,J_{3m})\|_\infty \le C\|f\|_\infty.
$$
It is also easy to understand that we do not have an analog of the above inequality in the $L_p$ spaces with $p<\infty$.
It was understood in early papers on this topic (see \cite{Te85b}) that the technically convenient way out of the above
problem is to consider a 
superposition of a sampling operator and the de la Vall{\'e}e Poussin operator. In particular, the following inequality
was established (see Lemma 6.2 from \cite{TBook}, Chapter 1, and Corollary 3 from \cite{Di91})
\be\label{sampV}
\|V_m(V_n(f),J_{3m})\|_p \le C\|f\|_p(n/m)^{1/p},\quad n\ge m.
\ee

\subsection{Frequency-limited sampling representations - discrete Littlewood-Paley}\index{Sampling!Representation!Frequency limited}
\index{Littlewood-Paley decomposition!Discrete}
\label{SampRep}

The definition of the $\bH$, $\bB$ and $\bW$ classes in terms of $\{\delta_\bs(f)\}_{\bs}$ is convenient for 
analyzing performance of the operators $S_{Q_n}$ and not very convenient for analyzing operators $T_n$. 
The idea is to replace the convolutions $\delta_{\bs}(f)$ in \eqref{DefB}, \eqref{NeqW} in the definition of the
Besov and Sobolev space by discrete convolutions of type \eqref{dvp}. Having
such a sampling representation at hand one can easily obtain bounds for the
sampling approximation error as we will show below. 

In order to obtain proper characterizations of $\Brpt$ we need the following
dyadic differences of $\eqref{dvp}$ given by 
$$
  v_0(f):=V_1(f,J_{3})\quad,\quad v_j(f) := V_{2^j}(f,J_{3\cdot
2^j})-V_{2^{j-1}}(f,J_{3\cdot 2^{j-1}})\quad,\quad j\geq 1\,.
$$
In other words we set $Y_j:= V_{2^j}(\cdot,J_{3\cdot
2^j})$   from the Smolyak-type algorithm construction in Subsection 4.2.
For the $d$-variate situation we need their tensor product
$\Delta_\bj :=v_{\bj}:=\bigotimes_{n=1}^d v_{j_n}$\,. This operator has to be understood
componentwise. For $f\in C(\T^d)$ the function $(\bigotimes_{i=1}^d v_{j_n})f$
is the trigonometric polynomial from $\mathcal{T}(2^{\bj +1},d)$ which we obtain
by applying each $v_{j_n}$ to the respective component of $f$. This gives 
the representation
\begin{equation}\label{tensor1}
   f = \sum\limits_{\bj \in \N_0^d} v_{\bj}(f)\,.
\end{equation}    In case of $\bH^r_p$ the inequality
\begin{equation}  \label{v_s(f)<H}
\sup_{\bj \in \N_0^d} \ 2^{r|\bj|_1}\|v_\bj(f)\|_p \lesssim \|f\|_{\bH^r_p}
\end{equation}
is due to Temlyakov~\cite{Te85b}. The inverse inequality (see \cite{Di91})
\begin{equation}  \label{v_s(f)>H}
\|f\|_{\bH^r_p} \lesssim \sup_{\bj \in \N_0^d} \ 2^{r|\bj|_1}\|v_\bj(f)\|_p
\end{equation}
follows from a simple observation: suppose that 
$$
f= \sum_\bj t_\bj,\quad \|t_\bj\|_p \le 2^{-r|\bj|_1},\quad t_\bj\in \Tr(2^{\bj},d),
$$
then $\|f\|_{\bH^r_p}\le C(d)$.  

The following proposition is a nontrivial generalization of inequalities (\ref{v_s(f)<H}) and (\ref{v_s(f)>H}) to the
case of $\bB$ and $\bW$ classes. 

\begin{prop} \label{samp_repr1} {(i)} Let $1\leq p,\theta \leq \infty$ and
$r>1/p$.
Then we have for any $f\in \Brpt$ 
\begin{equation}\label{lhs1}
\Big(\sum_{\bj \in \N_0^d}2^{r|\bj|_1 \theta}\|v_\bj(f)\|^{\theta}_p
\Big)^{1/\theta} \asymp \|f\|_{\Brpt}
\end{equation}
with the sum being replaced by a supremum for $\theta = \infty$.\\
{\em (ii)} Let $1<p<\infty$ and $r>\max\{1/p,1/2\}$ then we have for
any $f\in C(\T^d)$
\begin{equation}\label{lhs2}
\Big\|\Big(\sum_{\bj \in \N_0^d}2^{r|\bj|_1 2}
|v_{\bj}(f)|^2\Big)^{1/2}\Big\|_p \asymp \|f\|_{\Wrp}\,.
\end{equation}
\end{prop}

 The idea to replace the Littlewood-Paley convolutions $A_{\bj}(f)$, see
\eqref{2.2.16}, in the description of Besov classes by dyadic blocks $v_{\bj}(f)$ coming out of a discrete
convolution goes back to Dinh D\~ung \cite{Di91,Di01}.
The relation \eqref{lhs1} in Proposition \ref{samp_repr1} was proved in \cite{Di01}, and 
the relation \eqref{lhs2} has been proved recently in \cite{ByUl15}, see also \cite{ByDuUl14} for the case
$p=2$. As one would expect we do not necessarily have to choose de la Vall{\'e}e Poussin type building blocks $v_{\bj}(f)$ in order 
to obtain relations of type \eqref{lhs1} and \eqref{lhs2} if $1<p<\infty$. One may use building blocks
based on the classical trigonometric interpolation, i.e.,  Dirichlet kernels \eqref{2.1.4}, \eqref{nestedDirichlet}, 
which has been recently proved in \cite[Thms.\ 5.13, 5.14]{ByUl15}.\index{Kernel!Dirichlet} 
The relation in Proposition \ref{samp_repr1}(ii) requires further tools from
Fourier analysis, i.e., maximal functions of Peetre and Hardy-Littlewood
type, see \cite[Thm.\ 5.7]{ByUl15}.

\bproof For the convenience of the reader let us give
a short proof of the upper bound in (\ref{lhs1}) to show the main techniques. 
This proof alike the proof of (\ref{eq:SampH}) is based on relation (\ref{sampV}).
We start with the decomposition $f = \sum_{\bm \in
\N_0^d} f_{\bm}$, where $f_{\bm}$ is given by \eqref{f2_2}. Due to the
reproduction property in Lemma \ref{reprpoly} we obtain 
\begin{equation}\label{tria}
    \|v_{\bj}(f)\|_p \leq \sum\limits_{\substack{m_n\geq j_n-3\\n=1,...,d}}
\|v_{\bj}(f_{\bm})\|_p\,.
\end{equation}
Choose now $1/p<a<r$. Relation (\ref{sampV}) yields 
\begin{equation}\label{tria2}
    \|v_{\bj}(f_{\bm})\|_p \lesssim  
2^{(|\bm|_1-|\bj|_1)/p}\|f_{\bm}\|_p
\end{equation}
which gives 
$$
    2^{|\bj|_1 /p}\|v_{\bj}(f)\|_p \lesssim \sum\limits_{\substack{m_n\geq
j_n-3\\n=1,...,d}}2^{-|\bm|_1(a-1/p)}2^{|\bm|_1a}\|f_{\bm}\|_p\,.
$$
What remains is a standard argument based on a discrete Hardy type inequality,
see \cite[Lem.\ 2.3.4]{DeLo93}. H\"older's inequality with $1/\theta + 1/\theta'
= 1$ yields
$$
    2^{|\bj|_1/p}\|v_{\bj}(f)\|_p \lesssim
2^{-|\bj|_1(a-1/p)}\Big(\sum\limits_{\substack{m_n\geq
j_n-3\\n=1,...,d}}2^{|\bm|_1a\theta}\|f_{\bm}\|^{\theta}_p\Big)^{1/\theta}\,.
$$
Taking the $\ell_\theta$-norm (with respect to $\bj$) on both sides and
interchanging the summation on the right-hand side yields
\begin{equation}\nonumber
\begin{split}
\sum\limits_{\bj \in \N_0^d} 2^{|\bj|_1r\theta}\|v_{\bj}(f)\|^{\theta}_p
&\lesssim  \sum\limits_{\bm}
2^{|\bm|_1a\theta}\|f_{\bm}\|_p^{\theta}\sum\limits_{\substack{j_n\leq
m_n+3\\n=1,...,d}}2^{|\bj|_1(r-a)\theta}\\
&\lesssim \sum\limits_{\bm} 2^{|\bm|_1r\theta}\|f_{\bm}\|_p^{\theta}\,.
\end{split}
\end{equation} 
\eproof 

The method in the proof of Proposition \ref{samp_repr1} can be adapted to more
general Smolyak algorithms, see also Propositions \ref{AT1}, \ref{ATB} above and
\cite{SiUl07}.  

In the following one-sided relation the condition on $r$ can be relaxed to $r>0$. 

\begin{prop}\label{samp_repr2} Let $\{t_{\bj}\}_{\bj}$ be a sequence of
trigonometric polynomials with $t_{\bj} \in \mathcal{T}(2^{\bj},d)$ such that
the respective right-hand side \eqref{rhs1} or \eqref{rhs2} below is
finite.\\
{\em (i)} Let $1\leq p,\theta \leq \infty$ and $r>0$. Then
$f=\sum_{\bj\in \N_0^d} t_{\bj}$ belongs to $\Brpt$ and 
\begin{equation}\label{rhs1}
    \|f\|_{\Brpt} \lesssim \Big(\sum_{\bj \in \N_0^d}2^{r|\bj|_1
\theta}\|t_\bj\|^{\theta}_p
\Big)^{1/\theta}\,.
\end{equation}
{\em (ii)} Let $1<p<\infty$ and $r>0$. Then
$f=\sum_{\bj\in \N_0^d} t_{\bj}$ belongs to $\Wrp$ and 
\begin{equation}\label{rhs2}
 \|f\|_{\Wrp} \lesssim \Big\|\Big(\sum_{\bj \in \N_0^d}2^{r|\bj|_1 2}
|t_{\bj}(\cdot)|^2\Big)^{1/2}\Big\|_p\,.
\end{equation}
\end{prop}

The proof is similar to the one of Proposition \ref{samp_repr1}, for details see \cite{ByDuUl14, ByUl15}. 
 
\subsection{Sampling on the Smolyak grids}\index{Sampling!Smolyak grids}
\label{sampsmol}

It is now a standard way to obtain good errors of recovering functions with mixed smoothness by using 
the Smolyak-type algorithms described in Subsection \ref{Subs:appr_hypcross}. Different realizations differ
by the family of operators $\{Y_s\}_{s=0}^\infty$ used in the construction. Smolyak \cite{Smo} and other authors 
(see, for instance, \cite{Te85b,Di91,Di01}) used interpolation-type operators based on the Dirichlet or de la Vall\'ee
Poussin kernels.

For $n \in \N_0^d$ we define the sampling operator $T_n$ by 
\begin{equation}\label{sampop_per}
T_n(f) 
:= \ 
\sum\limits_{\substack{\bj \in \N_0^d\\|\bj|_1 \leq n}} v_{\bj}(f)\quad,\quad
f\in C(\T^d)\,,
\end{equation}
where $v_{\bj}(f)$ are defined in Subsection \ref{SampRep}.

From Lemma \ref{reprpoly} we see that $T_n$
does not change trigonometric polynomials from hyperbolic crosses $Q_{n-3d}$,
see \eqref{HC}.
In addition, $T_n$ interpolates $f$ at every grid
point $\by \in \wt{SG}^d(n)$, i.e.,
\begin{equation} \label{InterpolationT_n(f)}
T_n(f,\by) = f(\by), \quad \by \in \wt{SG}^d(n)\,,
\end{equation}
where the Smolyak Grid $\wt{SG}^d(n)$ of level $n$ is given by 
\begin{equation}\label{SG}
    \wt{SG}^d(n):= \bigcup\limits_{j_1+...+j_d \leq n} J_{3\cdot 2^{j_1}}
\times...\times J_{3\cdot 2^{j_d}}\,.
\end{equation}

\begin{figure}[H]
\begin{minipage}{0.48\textwidth}
\vspace{0.33cm}
\includegraphics[scale = 1.3]{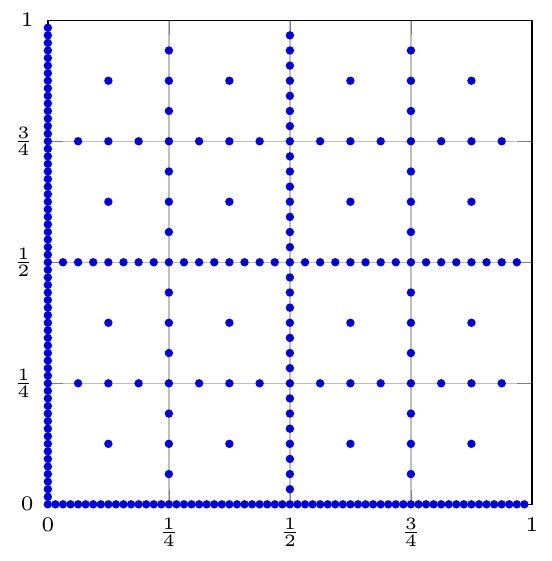}
\end{minipage}
\begin{minipage}{0.48\textwidth}
\begin{tikzpicture}[scale=0.051]
\pgfmathsetmacro{\m}{6};
\pgfmathsetmacro{\rf}{1};
\pgfmathsetmacro{\rs}{1};
\pgfmathtruncatemacro{\xl}{2^(\m / \rf)}
\pgfmathtruncatemacro{\yl}{2^(\m / \rs)}
\pgfmathtruncatemacro{\limf}{round(\m / \rf)}

\pgfmathtruncatemacro{\lima}{round(\m / \rf)}
\pgfmathtruncatemacro{\limb}{round(\m / \rs)}

\foreach \x in {1,...,\lima}
{
\pgfmathtruncatemacro{\lim}{round((\m-\rf*\x)/(\rs))}
\foreach \y in {1,...,\limb}
{
\pgfmathtruncatemacro{\xm}{\x-1}
\pgfmathtruncatemacro{\ym}{\y-1}
    \draw[dashed] (2^\xm,2^\ym)-- (2^\x,2^\ym) --(2^\x,2^\y) -- (2^\xm,2^\y) --
(2^\xm,2^\ym);
        \draw[dashed] (-2^\xm,-2^\ym)-- (-2^\x,-2^\ym) --(-2^\x,-2^\y) --
(-2^\xm,-2^\y) -- (-2^\xm,-2^\ym) ;
        \draw[dashed] (-2^\xm,2^\ym)-- (-2^\x,2^\ym) --(-2^\x,2^\y) --
(-2^\xm,2^\y) -- (-2^\xm,2^\ym) ;
        \draw[dashed] (2^\xm,-2^\ym)-- (2^\x,-2^\ym) --(2^\x,-2^\y) --
(2^\xm,-2^\y) -- (2^\xm,-2^\ym) ;
}
}

\foreach \x in {0,...,\limf}
{
\pgfmathtruncatemacro{\lim}{round((\m-\rf*\x)/(\rs))}
\foreach \y in {0,...,\lim}
{
\pgfmathtruncatemacro{\xm}{\x-1}
\pgfmathtruncatemacro{\ym}{\y-1}
\ifnum \x<1 \relax%
     \ifnum \y<1 \relax%
         draw[color=red,very thick] (-1,-1)-- (1,-1) --(1,1) -- (-1,1) --
(-1,-1);
    \else
        \draw[color=red,very thick] (-1,-2^\ym)-- (1,-2^\ym) --(1,-2^\y) --
(-1,-2^\y) -- (-1,-2^\ym) ;
        \draw[color=red,very thick] (-1,2^\ym)-- (1,2^\ym) --(1,2^\y) --
(-1,2^\y) -- (-1,2^\ym) ;
     \fi
\else%
    \ifnum \y<1 \relax%
            \ifnum \x<1 \relax%

            \else
            \draw[color=red,very thick] (2^\xm,-1)-- (2^\x,-1) --(2^\x,1) --
(2^\xm,1) -- (2^\xm,-1) ;
            \draw[color=red,very thick] (-2^\xm,-1)-- (-2^\x,-1) --(-2^\x,1) --
(-2^\xm,1) -- (-2^\xm,-1);
            \fi
    \else
    \draw[color=red,very thick] (2^\xm,2^\ym)-- (2^\x,2^\ym) --(2^\x,2^\y) --
(2^\xm,2^\y) -- (2^\xm,2^\ym);
        \draw[color=red,very thick] (-2^\xm,-2^\ym)-- (-2^\x,-2^\ym)
--(-2^\x,-2^\y) -- (-2^\xm,-2^\y) -- (-2^\xm,-2^\ym) ;
        \draw[color=red,very thick] (-2^\xm,2^\ym)-- (-2^\x,2^\ym)
--(-2^\x,2^\y) -- (-2^\xm,2^\y) -- (-2^\xm,2^\ym) ;
        \draw[color=red,very thick] (2^\xm,-2^\ym)-- (2^\x,-2^\ym)
--(2^\x,-2^\y) -- (2^\xm,-2^\y) -- (2^\xm,-2^\ym) ;
    \fi
\fi

}}

\end{tikzpicture} 

\end{minipage}
\caption{A sparse grid and associated hyperbolic cross in $d=2$}\label{sg_hc}\index{Sparse grid}\index{Hyperbolic cross}
\end{figure}
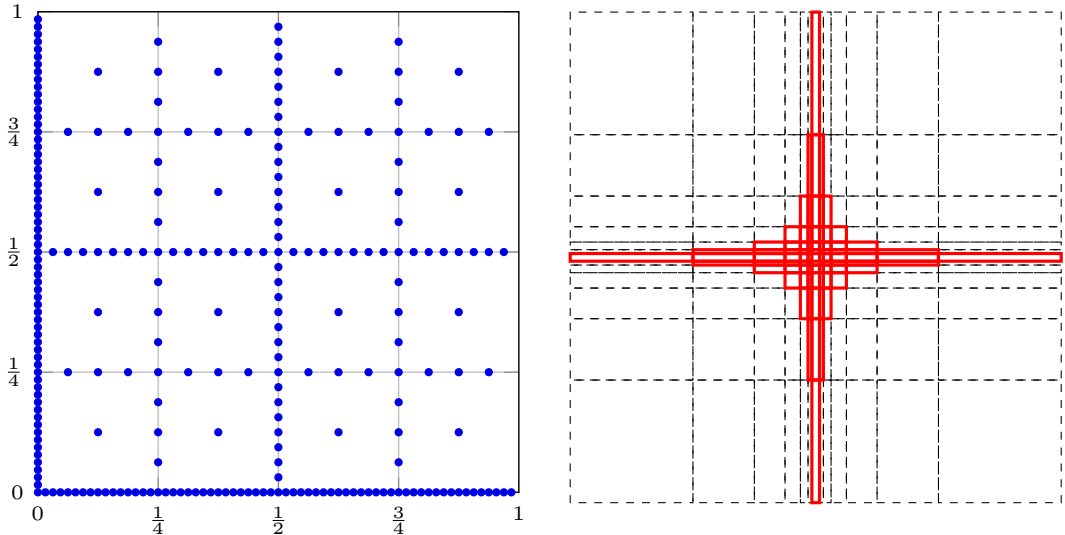

The following estimate concerning the grid size is known, see for instance
\cite{BuGr04} or the recent paper \cite[Lem.\ 3.10]{DU13}. Clearly,
$\mbox{card }
\wt{SG}^d(n) = 3^d\mbox{card }SG^d(n)$, where $SG^d(n)$ is given in this section below (see also \eqref{eq53}). It holds 
\begin{equation}\label{Sg}
  2^{n}\binom{n+d-1}{d-1} \leq \mbox{card }SG^d(n) \leq
2^{n+1}\binom{n+d-1}{d-1}\,.
\end{equation}
This gives $\mbox{card }SG^d(n) \asymp 2^n n^{d-1}$\,.
Here the nestedness of the univariate grids plays an important role. In fact,
we can replace $j_1+...+j_d \leq n$ by $j_1+...+j_d = n$ in \eqref{SG}\,.

Let us begin with the recovery on the Smolyak grids of Sobolev classes $\Wrp$ in $L_q$. 

\begin{thm}\label{Sobsamp_q} Let $1<p,q<\infty$. \\
{\em (i)} In case $p\geq q$ and $r>\max\{1/p,1/2\}$ it holds
$$
    \sup\limits_{f\in \Wrp}\|f-T_n(f)\|_q \asymp 2^{-nr}n^{(d-1)/2}\,.
$$
{\em (ii)} In case $1<p<q<\infty$ and $r>1/p$ we have 
$$
    \sup\limits_{f\in \Wrp}\|f-T_n(f)\|_q \asymp 2^{-n(r-1/p+1/q)}\,.
$$
\end{thm}

\bproof Let us demonstrate the general principle on the special case
$1<p=q<\infty$. Starting with \eqref{error} (see below) we estimate using H\"older's
inequality 
\begin{equation}
  \begin{split}
    \|f-T_n(f)\|_p &\leq  \Big\|\sum\limits_{|\bj|_1>n}
|v_{\bj}(f)(\cdot)|\Big\|_p\\
&\lesssim
\Big\|\Big(\sum\limits_{|\bj|_1>n}2^{-2r|\bj|_1}\Big)^{1/2}
\cdot\Big(\sum\limits_ { |\bj|_1>n}
2^{2r|\bj|_1}|v_{\bj}(f)(\cdot)|^2\Big)^{1/2}\Big\|_p
  \end{split}
\end{equation}
Proposition \ref{samp_repr1} together with \eqref{|j|>n} finishes the
proof.\eproof

In case $q=\infty$ we observe an extra $\log$-term. The result in Theorem
\ref{Sobsamp_infty} below has been obtained by Temlyakov, see
\cite{T8}. For the convenience of the reader we give a short proof of this result based on the non-trivial embedding in
Lemma \ref{L3.5d}, which in turn is a corollary of Theorem \ref{T2.4.6}.

\begin{thm}\label{Sobsamp_infty} For $1<p<\infty$ and $r>1/p$ we have 
$$
    \sup\limits_{f\in \bW^r_p}\|f-T_n(f)\|_{\infty} \asymp
2^{-n(r-1/p)}n^{(d-1)(1-1/p)}\,. 
$$
\end{thm}

\bproof The proof from \cite{T8} goes as follows. First, the upper bounds are proved. The key role in that proof is played by Theorem \ref{T2.4.6}. Second,
the lower bound in the case $p=2$ is proved. It is derived from the lower bound for 
$\lambda_m(\bW^r_2,L_\infty)$. Finally, using the lower bound for $p=2$ and the upper bounds
for $1<p<\infty$ we derive the lower bounds for $1<p'<\infty$. We only illustrate how to prove 
the upper bounds. In the proof below instead of direct use of Theorem \ref{T2.4.6} we use its corollary in the form of Lemma \ref{L3.5d}.

We use the triangle inequality to obtain
\begin{equation}\nonumber
 \begin{split}
\|f-T_n(f)\|_{\infty} &\leq \sum\limits_{|\bj|_1>n} \|v_{\bj}(f)\|_{\infty} \\
&= \sum\limits_{|\bj|_1>n}2^{-(r-1/p)|\bj|_1}2^{(r-1/p)|\bj|_1}\|v_{\bj}(f)\|_{\infty}\,.
 \end{split}
\end{equation}
Using H\"older's inequality with respect to $1/p+1/p'=1$ we obtain
$$
    \|f-T_n(f)\|_{\infty} \lesssim \Big(\sum\limits_{|\bj|_1>n}
2^{-(r-1/p)|\bj|_1p'}\Big)^{1/p'}\Big(\sum\limits_{|\bj|_1>n}2^{p(r-1/p)|\bj|_1}
\|v_{\bj}(f)\|_\infty^p\Big)^{1/p}\,.
$$
Applying (\ref{lhs1}) we have
$$
    \|f-T_n(f)\|_{\infty} \lesssim 2^{-n(r-1/p)}n^{(d-1)(1-1/p)}\|f\|_{\bB^{r-1/p}_{\infty,p}}\,.
$$
Finally, the embedding $\bW^r_p \hookrightarrow \bB^{r-1/p}_{\infty,p}$ (see Lemma \ref{L3.5d}) concludes the
proof. \eproof

The results in \cite{Tmon} (see Theorems 5.4 and 3.1 of Chapter 2) imply the inequality
 \be\label{2.17*}
 \sup_{f\in\bW^r_p} \|f-U_n(f)\|_\infty \gtrsim 2^{-(r-1/p)n}n^{(d-1)(1-1/p)},\quad 1<p<\infty,
\ee
which is valid for any sequence of linear operators $U_n$ : $\bW^r_p \to \Tr(Q_n)$.
Theorem \ref{Sobsamp_infty} and relation (\ref{2.17*}) show that the sequence of operators $T_n$ is optimal (in the
sense of order) among all linear operators of approximating by means of polynomials in $\Tr(Q_{n+d})$.

Let us now proceed to the $\bH$-classes. 

\begin{thm} \label{Theorem[T_n(H)<]}
Let $1 \le p, q \le \infty$ and $r > 1/p$. Then we have the following. 
\begin{itemize}
\item[{\rm (i)}] For $p \ge q$,
\begin{equation*} 
\sup_{f \in \bH^r_p}\|f - T_n(f)\|_q
 \ \asymp \ 
2^{- r n} n^{d-1}.
\end{equation*}
\item[{\rm (ii)}] For $p < q$, 
\begin{equation*} 
\sup_{f \in \bH^r_p}\|f - T_n(f)\|_q
 \ \asymp \ 
\begin{cases}
2^{- (r - 1/p + 1/q) n} n^{(d-1)/q}, \ & q < \infty, \\
2^{- (r - 1/p) n} n^{d-1}, \ &  q = \infty.
\end{cases}
\end{equation*}
\end{itemize} 
\end{thm}

The upper bounds in Theorem \ref{Theorem[T_n(H)<]} for the case $p \ge q$ are already in \eqref{ineq:SampHT}.
The corresponding lower bounds follow from  more general results on lower bounds for numerical integration (see, for
instance, Theorem \ref{T4.4} below). In particular, Theorem \ref{T4.4} provides these lower bounds not only for the
operator $T_n$ but for any recovering operator, which uses the same nodes as $T_n$.  In the case $p <
q$ the lower bounds follow from Theorem \ref{TBT3.3}. For the upper bounds see \cite{Di90, Di91} and \cite{TBook},
Chapter 4, \S 5, Remark 2. Finally, in the case $q = \infty$ the upper bounds can
be easily derived from \eqref{v_s(f)<H}. The lower bounds follow from a general statement (see \cite{Tmon}, Chapter 2,
Theorem 5.7):
 \begin{prop}\label{TmonT5.7} Let $U_n$ be a bounded linear operator from $\bH^r_p$ to $\Tr(Q_n)$. Then for $1\le p\le
\infty$
 $$
 \sup_{f\in\bH^r_p}\|f-U_n(f)\|_\infty \ge C(d,p) 2^{-(r-1/p)n}n^{d-1}.
 $$
\end{prop}
Theorem \ref{Theorem[T_n(H)<]} and Proposition \ref{TmonT5.7} show that the sequence of operators $T_n$ is optimal
(in the sense of order) among all linear operators of approximating by means of the hyperbolic cross polynomials in
$\Tr(Q_{n+d})$.

Finally, we study the situation $\theta<\infty$. We will see that the
third index $\theta$ influences the estimates significantly. We refer to
\cite{Di90, Di91} and the recent papers \cite{SiUl07, Ul08, SU11,
ByDuUl14, Di11, DU14, VT152}. 

\begin{thm} \label{Theorem[T_n(B)<]}
Let $1 \leq  p, q, \theta \le \infty$ and $r > 1/p$. Then we have the
following relations
\begin{itemize}
\item[{\rm (i)}] For $p \ge q$,
\begin{equation*} 
\sup_{f \in \Brpt}\|f - T_n(f)\|_q
 \ \asymp \ 
 2^{-rn}n^{(d-1)(1-1/\theta)}\,. 
 \end{equation*}
\item[{\rm (ii)}] For $p < q$, 
\begin{equation*} 
\sup_{f \in \Brpt}\|f - T_n(f)\|_q
 \ \asymp \ 
\begin{cases}
2^{- (r - 1/p + 1/q) n} n^{(d-1)(1/q - 1/\theta)_+}, \ & q < \infty, \\
2^{- (r - 1/p) n} n^{(d-1)(1 - 1/\theta)}, \ &  q = \infty.
\end{cases}
\end{equation*}
\end{itemize} 
\end{thm}

For the respective lower bounds we refer to \cite[Sect.\ 5]{DU14} and
\cite{VT152}, see also Section \ref{cubsmol} below. Note, that lower bounds for
numerical integration also serve as lower bounds for approximation. Regarding upper bounds we have several options how to proceed. Starting with
\begin{equation}\label{error}
    \|f-T_n(f)\|_q = \Big\|\sum\limits_{|\bj|_1>n}v_{\bj}(f)\Big\|_q
\end{equation}
one option is to use embeddings of Besov-Sobolev
classes into $L_q$, see Lemma \ref{emb}, in order to replace the right-hand
side of \eqref{error} by an expression of type \eqref{rhs1} or \eqref{rhs2}.
Here we need Proposition \ref{samp_repr2} above combined with the embedding
relations from Lemma \ref{emb}. Another option is to simply use the triangle
inequality for the error \eqref{error} and afterwards the Bernstein-Nikol'skii
inequality (Theorem \ref{T2.4.1}) to change the $L_q$-norm to an $L_p$-norm for
instance. Then we get immediately an expression of type \eqref{rhs1} or
\eqref{rhs2}. In the latter expressions $r$ will be $1/p-1/q$ in the case
$p<q$ (not the $r$ from the class above). Using H\"older's inequality and
standard procedures we end up with expressions like on the left-hand side of
\eqref{lhs1}, \eqref{lhs2} multiplied with a certain rate which is always
generated from infinite sums of type 
\begin{equation}\label{|j|>n}
      \Big(\sum\limits_{|\bj|_1>n} 2^{-|\bj|_1s\eta}\Big)^{1/\eta} \asymp
2^{-sn}n^{(d-1)/\eta}\,.
\end{equation}
Finally we apply our Proposition \ref{samp_repr1}.\\

A general linear sampling operator on Smolyak grids $SG^d(n)$ is given by 
\begin{equation} \label{def[Psi(f,SG^d(n))]}
\Psi(f,SG^d(n))
\ = \
\sum_{\by \in SG^d(n)} f(\by) \psi_\by\,,
\end{equation}
where $\Psi_n = \{\psi_\by\}_{\by \in SG^d(n)}$ denotes a family of functions indexed by the grid points in $SG^d(n)$, see \eqref{eq53} below. 
On the basis of the time-limited  B-spline sampling representation \eqref{eq:B-splineRepresentation} which is given 
in Subsection \ref{timelimsamprep} below, we construct the sampling algorithms $R_n$ on 
the Smolyak grids $SG^d(n)$ by 
\begin{equation}\label{def[R_n]}
R_n(f) 
:= \ 
\sum\limits_{\substack{\bj \in \N_0^d\\|\bj|_1 \leq n}} q_{\bj}(f)\quad,\quad
f\in C(\T^d)\,,
\end{equation}
which induces a linear sampling algorithm  of the form  \eqref{def[Psi(f,SG^d(n))]}
where $\psi_\by$ are explicitly constructed as linear combinations of
at most $n_0$ B-splines $N_{\bs,\bk}$ for some $n_0 \in \N$ which is independent of $\bs,\bk,n$ and $f$. This fact can be proven in the same way as the proof of its counterpart for non-periodic functions \cite{Di11}.
Differing from $T_n$, the operator $R_n$ does not possess an interpolation property similar 
 to \eqref{InterpolationT_n(f)}, which is why it is called {\em quasi-interpolation}.

\begin{thm}\label{Theorem[analogs]}
All Theorems \ref{Sobsamp_q} --  \ref{Theorem[T_n(B)<]} hold true with frequency-limited sampling operators  $T_n$ replaced by the time-limited sampling operators $R_n$ 
(with proper restrictions on $r$ according to the smoothness of the B-spline).  
\end{thm}

Theorem \ref{Theorem[analogs]} is proved in \cite{ByUl16}, \cite{Di16_1} for Sobolev classes $\Wrp$.
The upper bounds in  Theorem \ref{Theorem[analogs]} have been proven
in \cite[Theorem 3.1]{Di11} for Besov classes $\Brpt$ and  the lower bounds follow from \cite[Theorem 5.1]{DU14}.

Both sampling operators $T_n$ and $R_n$ provide approximate recovery with similar error bounds for classes $\bW^r_p$ and $\bB^r_{p,\theta}$. 
The common feature of the operators $T_n$ and $R_n$ is that they use the Smolyak grid points as the sampling points.
This motivates us to ask the following question. Are there (non-)linear operators, 
that use the Smolyak grids points for sampling, which give better bounds than $T_n$ and $R_n$? 
It is proved in \cite{Di16_1,DU14} that the answer to this question is ``No''. We refer to a more 
general result in \cite[Sect.\ 5]{DU14} and \cite{VT152}, to establish the lower bounds in the theorems above. Note also that in case
$p\geq q\geq 1$ the lower bounds for numerical integration also serve as lower bounds for approximation, see also Subsection \ref{cubsmol} below. 
In general, the proof of the corresponding lower bounds is based on the construction of test functions from a class of our interest which is zero at the grid points. 
A nontrivial part of it is the proof of the fact that the constructed function belongs to the class. The inverse time-limited B-spline representation theorems 
(see \cite{Di16_1,DU14}) and Proposition \ref{atomic}, see \cite{UU14}, is used at this step. 

Moreover, it is proved in \cite{VT152} that even if we allow more general than the Smolyak grids sampling sets we do 
not gain better error bounds. We give a precise formulation of this result.}  
Let $\bs = (s_1,\dots,s_d)$, $s_j\in \N_0$, $j=1,\dots,d$. We associate with $\bs$ a web
$W(\bs)$ as follows: denote 
$$
w(\bs,\bx) := \prod_{j=1}^d \sin (2^{s_j}x_j)
$$
and define
$$
W(\bs) := \{\bx: w(\bs,\bx)=0\}.
$$
We say that a set of nodes $X_m:=\{\bx^i\}_{i=1}^m$ is
an $(n,\ell)$-net if $|X_m\setminus W(\bs)| \le 2^\ell$ for all $\bs$ such that
$|\bs|_1=n$. It is easy
to
check that $SG^d(n)\subset W(\bs)$ with any $\bs$ such
that $|\bs|_1=n$. This means that $SG^d(n)$ is an $(n,\ell)$-net for any $\ell$.
\begin{thm}\label{T5.2} For any recovering operator $\Psi(\cdot,X_m)$ with respect to a \newline $(n,n-1)$-net $X_m$ we
have for $1\le p< q<\infty$, $r>\beta$,
$$
\Psi(\bB^r_{p,\theta},X_m)_q \gtrsim 2^{-n(r-\beta)}n^{(d-1)(1/q-1/\theta)},\quad \beta:=1/p-1/q.
$$
\end{thm}

The reader may find some further results on Smolyak type algorithms in papers \cite{TeKuSh10} and \cite{TeNaSh15}.

The exact recovery of trigonometric polynomials with frequencies in hyperbolic crosses from a discrete set of samples also plays a role in many applications. 
Several authors, see for instance \cite{Ha92}, \cite{FeKuPo06}, \cite{DoKuPo10} and the references therein, considered the problem of 
adapting the well-known fast Fourier transform to the sparse grid spacial discretization (HCFFT) for the recovery of multivariate trigonometric polynomials with frequencies on a hyperbolic cross, see Figure \ref{sg_hc}. 
However, there are some stability issues as \cite{KaKu11} shows. This is related to the fact that the grid size still grows exponentially in $d$, however 
considerably slower than the full grid. In \cite{KaKu12} the authors proposed to use discretization points generated by a oversampled lattice rule 
coming from numerical integration (see also Subsection \ref{samp_lattices}). Due to the lattice structure one may use the classical one dimensional FFT here. 
 
\subsection{Sampling widths}\index{Width!Sampling}
\label{sampwidths}
 
Let us transfer the preceding results to the language of sampling widths
$\varrho_m(\bF,X)$, see \eqref{sampling-number}. Based on
the results in the previous subsection we can give
reasonable upper bounds in those situations. Note that, although
we
provided lower bounds for the Smolyak worst case sampling error, we do not have
sharp lower bounds for the sampling numbers $\varrho_m(\bF,X)$ in most of the
situations. Due to the specific framework (linear sampling
numbers) we can of course use linear widths  for the
estimates from below. This leads to sharp results in some of the cases, see
also Section \ref{quasiB} for results where $p,q,\theta <1$. Note
that the upper bounds in
Theorems \ref{Theorem[T_n(H)<]}, \ref{Theorem[T_n(B)<]}, \ref{Sobsamp_q} and
\ref{Sobsamp_infty} can be directly transferred to upper estimates for sampling
numbers $\varrho_m$ by taking the number $$m:=\# \wt{SG}^{d}(n)
\asymp2^{n}{n}^{d-1}$$
of grid points in a Smolyak grid of level $n$ into account. Note, that this
follows from \eqref{Sg}. 

\begin{figure}[H]

\begin{minipage}{0.48\textwidth}
\begin{tikzpicture}[scale=2.5]

\draw[->] (-0.1,0.0) -- (2.1,0.0) node[right] {$\frac{1}{p}$};
\draw[->] (0.0,-.1) -- (0.0,2.1) node[above] {$\frac{1}{q}$};

\draw (1.0,0.03) -- (1.0,-0.03) node [below] {$\frac{1}{2}$};
\draw (0.03,1) -- (-0.03,1) node [left] {$\frac{1}{2}$};

\node at (1.7,2.2) {$\varrho_m(\bW^r_p,L_q)$};

\draw (0,2) -- (2,2);
\draw (1,1) -- (2,1);
	\node at (.6,0.15) {\tiny  $\alpha=r-\frac{1}{p}+\frac{1}{q}$};
	\node at (1.62,1.2) {\tiny  $\alpha=r-\frac{1}{p}+\frac{1}{q}$};

\draw (1,0) -- (1,1);
\draw (1,1) -- (2,1);
\draw (2,2) -- (2,0);
\draw (0,0) -- (2,2);

\node at (1.5,0.5) {{\huge ?}};


\node at (0.65,1.45) {{\huge ?}};


\draw (2,0.03) -- (2,-0.03) node [below] {$1$};
\draw (0.03,2) -- (-0.03,2) node [left] {$1$};

\end{tikzpicture}

\end{minipage}
\begin{minipage}{0.48\textwidth}
\begin{tikzpicture}[scale=2.5]

\draw[->] (-0.1,0.0) -- (2.1,0.0) node[right] {$\frac{1}{p}$};
\draw[->] (0.0,-.1) -- (0.0,2.1) node[above] {$\frac{1}{q}$};

\draw (1.0,0.03) -- (1.0,-0.03) node [below] {$\frac{1}{2}$};
\draw (0.03,1) -- (-0.03,1) node [left] {$\frac{1}{2}$};

\node at (1.7,2.2) {$\lambda_m(\bW^r_p,L_q)$};

\draw (0,2) -- (2,2);
\draw (1,1) -- (2,1);

\draw (1,1) -- (2,0);
\draw (1,0) -- (1,1);
\draw (1,1) -- (2,1);
\draw (2,2) -- (2,0);

\node at (1.4,0.2) {\tiny ${\alpha=r-\frac{1}{2}+\frac{1}{q}}$};
\node at (1.6,0.8) {\tiny ${\alpha=r-\frac{1}{p}+\frac{1}{2}}$};
  \node at (1,1.6) {\tiny$\Big(\frac{(\log
m)^{d-1}}{m}\Big)^{r-(\frac{1}{p}-\frac{1}{q})_+}$};

\draw (2,0.03) -- (2,-0.03) node [below] {$1$};
\draw (0.03,2) -- (-0.03,2) node [left] {$1$};

\end{tikzpicture}
\end{minipage}
\caption{Comparison of $\varrho_m(\bW^r_p,L_q)$ and $\lambda_m(\bW^r_p,L_q)$, rate $\big(m^{-1}\log^{d-1}m\big)^{\alpha}$}\label{fig11}
\end{figure}
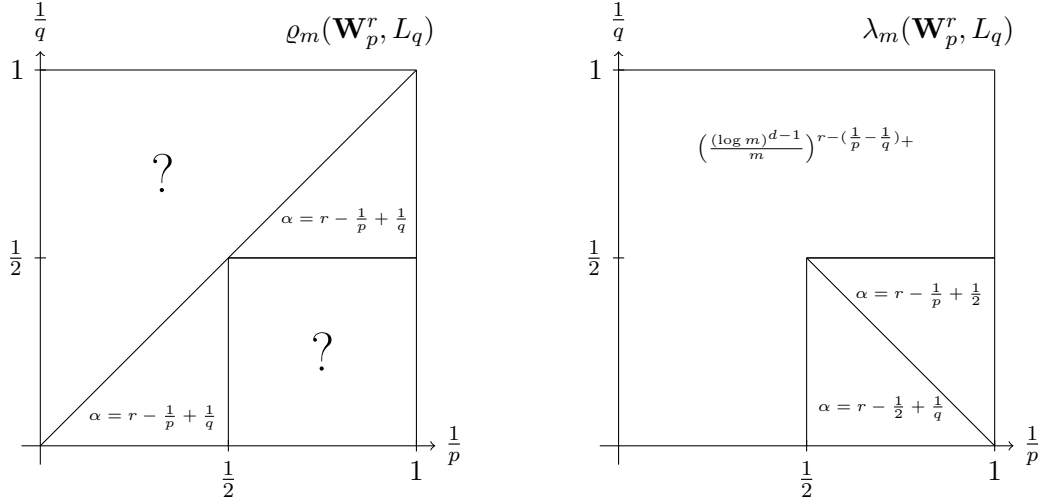

\begin{thm} \label{SampWidthW}
(i) Let $1<p<q\leq 2$ and $r>1/p$. Then 
 $$
    \varrho_m(\Wrp, L_q)\asymp m^{-(r-1/p+1/q)}(\log m)^{(d-1)(r-1/p+1/q)}\,.
 $$
(ii) Let $2\leq p<q<\infty$ and $r>1/p$. Then 
$$
   \varrho_m(\Wrp, L_q)\asymp m^{-(r-1/p+1/q)}(\log m)^{(d-1)(r-1/p+1/q)}\,. 
$$
(iii) Let $r > 1/2$.
Then we have  
\begin{equation*} 
\varrho_m(\bW^r_2,L_\infty)
 \ \asymp \ 
m^{-(r-1/2)}(\log m)^{(d-1)r}\,.
 \end{equation*}
\end{thm}

The results in (i) and (ii) have been proved recently in \cite{ByUl15} for $r > 1/p$, and  in 
\cite{Di16_1} for $r > \max\{1/p,1/2\}$. The special case $p=2<q$ is proved in \cite[Thm.\ 6.10]{ByDuUl14}. The result (iii) is proved in \cite{T8}. Clearly, (i) and (ii) follow from Subsection \ref{linearw} together with Theorem
\ref{Sobsamp_q}. The upper bound in (iii) follows from
Theorem \ref{Sobsamp_infty}. The lower bound is based on Ismagilov's \cite{Is74}
duality theorem 
$$
 \lambda_m(\Wr2,L_\infty) \asymp \lambda_m({\bf W}^r_1,L_2)
$$
(see also \cite{CoKuSi15}) together with Theorem~\ref{TBT4.4'}\,. This together
with  Theorem \ref{TBT4.4} also shows a sharp lower bound. 

For completeness we reformulate \eqref{eq:SampH[p<q<2]} as follows. 
\begin{thm} \label{Theorem[rho_n(H)]}
Let $ 1 < p < q \le 2$ and $r > 1/p$. Then we have  
\begin{equation*} 
\varrho_m(\Hrp,L_q)
 \ \asymp m^{-(r-1/p+1/q)} (\log m)^{(d-1)(r-1/p+2/q)}\,.
\end{equation*}
\end{thm}

\begin{figure}[H]

\begin{minipage}{0.48\textwidth}
 \begin{center}
\begin{tikzpicture}[scale=2.5]
\draw[->] (1,0.0) -- (2.1,0.0) node[right] {$\frac{1}{p}$};
\draw (0,0.0) -- (1,0.0);
\draw[->] (0.0,-.1) -- (0.0,2.1) node[above] {$\frac{1}{q}$};

\node at (1.7,2.2) {$\varrho_m(\bH^r_p,L_q)$};
\draw (0.03,1) -- (-0.03,1) node [left] {$\frac{1}{2}$};
\draw (1,0.03) -- (1,-0.03) node [below] {$\frac{1}{2}$};

\draw (0,2) -- (2,2);
\draw (1,1) -- (1,2);
\draw (0,0) -- (2,2);
\draw (2,2) -- (2,1);	
\draw (2,1) -- (2,0);
\draw (1,1) -- (2,0);
\draw (1,1) -- (2,1);
\node at (0.5,1.2) {\huge ?};
\node at (1.7,0.65) {\huge ?};

\node at (1.6,1.15) { \tiny $\alpha = r-\frac{1}{p}+\frac{1}{q}$};
\node at (1.7	,1.35) { \tiny $\beta = \frac{1}{q}$};

\node at (1.2,1.5) {\huge ?};
\node at (1.0,0.4) {\huge ?};

\draw (2,0.03) -- (2,-0.03) node [below] {$1$};
\draw (0.03,2) -- (-0.03,2) node [left] {$1$};

\end{tikzpicture}
\end{center}
\end{minipage}
\begin{minipage}{0.48\textwidth}
 \begin{center}
\begin{tikzpicture}[scale=2.5]
\draw[->] (2.0,0.0) -- (2.1,0.0) node[right] {$\frac{1}{p}$};
\draw[->] (0,2.0) -- (0,2.1) node[above] {$\frac{1}{q}$};
\draw (0,0) -- (-0.1,0.0);
\draw (0,0) -- (0,-0.1);
\draw[very thick][dashed] (1,0.0) -- (2.0,0.0);
\draw[very thick][dashed] (0,0.0) -- (1,0.0);
\draw[very thick][dashed] (0.0,0) -- (0.0,2.0);

\node at (1.7,2.2) {$\lambda_m(\bH^r_p,L_q)$};
\draw (0.03,1) -- (-0.03,1) node [left] {$\frac{1}{2}$};
\draw (1,0.03) -- (1,-0.03) node [below] {$\frac{1}{2}$};

\draw[ultra thick] (0,2) -- (1,2);
\draw[very thick][dashed] (1,2) -- (2,2);
\draw (1,1) -- (1,2);
\draw (0,0) -- (2,2);
\draw[very thick][dashed] (2,2) -- (2,1);
\draw[very thick][dashed] (2,1) -- (2,0);
\draw (1,1) -- (2,0);
\draw (1,1) -- (2,1);
\draw (1,0) -- (1,1); 
\node at (0.22,1.4) {\small $ \alpha = r  $};
\node at (0.22,1.2) {\small $\beta = \frac{1}{2}$};
\node at (1.6,0.85) { \tiny$\alpha = r-\frac{1}{p}+\frac{1}{2}$};
\node at (1.7,0.65) {\tiny $\beta = \frac{1}{2}$};
\node at (1.6,1.15) { \tiny $\alpha = r-\frac{1}{p}+\frac{1}{q}$};
\node at (1.7	,1.35) { \tiny $\beta = \frac{1}{q}$};

\node at (1.2,1.5) {\huge ?};
\node at (0.55,0.1) {\tiny $\alpha = r-1/p+1/q$};
\node at (0.5,0.27) {\tiny $\beta = 1/q$};

\node at (1.45,0.1) {\tiny $\alpha = r-1/2+1/q$};
\node at (1.3,0.27) {\tiny $\beta = 1/q$};


\draw (2,0.03) -- (2,-0.03) node [below] {$1$};
\draw (0.03,2) -- (-0.03,2) node [left] {$1$};

\end{tikzpicture}
\end{center}
\end{minipage}

\caption{Comparison of $\varrho_m(\bH^r_p,L_q)$ and $\lambda_m(\bH^r_p,L_q)$}\label{fig12}
\end{figure}
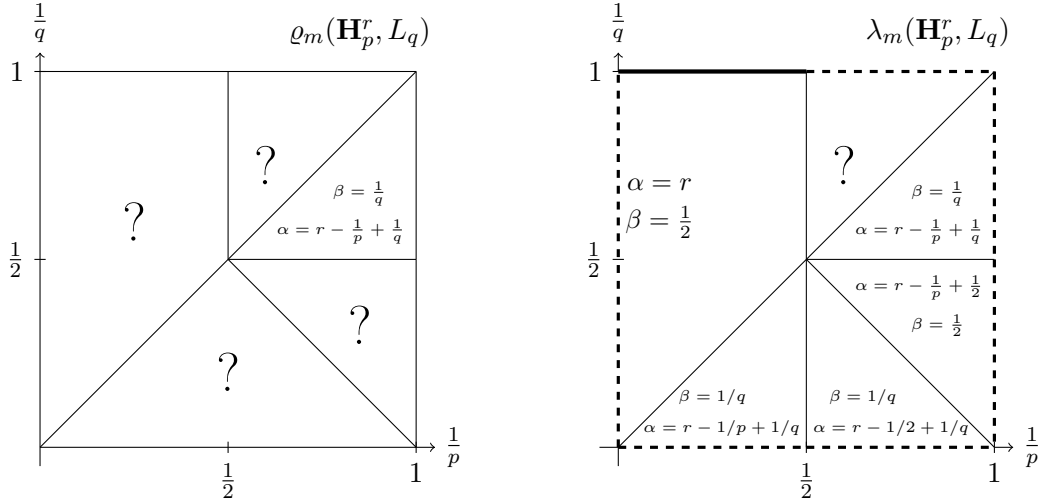

\noindent In the figure $\alpha$ and $\beta$ refer to the asymptotic order
$$\Big(\frac{(\log m)^{d-1}}{m}\Big)^{\alpha}(\log m)^{(d-1)\beta}.$$
 
\begin{rem}\label{d=2} It is pointed out in \cite{VT152} that the upper bound
in \eqref{eq:SampH} and the lower bound for the Kolmogorov widths in
Theorem \ref{Te96H}, proved in \cite{VT59}, imply for $d=2$, $r>1/2$,
$$
    \varrho_m(\bH^r_{\infty},L_{\infty}) \asymp m^{-r}(\log m)^{r+1}\,.
$$
For $d>2$ the correct order is not known. 
\end{rem}

\begin{thm} \label{Th[rho_n]}
Let $1 < p, q, \theta \le \infty$ and $r > 1/p$.
Then we have the following. 
\begin{itemize}
\item[{\rm (i)}] For $p \ge q$,
\begin{equation*} 
\varrho_m({\bf B}^r_{p,1},L_q)
 \ \asymp \ 
 (m^{-1} \log^{d-1}m)^r\quad,\quad
\begin{cases}
 2 \le q < p < \infty, \\
1 < p = q \le \infty.
\end{cases}
\end{equation*}
\item[{\rm (ii)}] For $1 < p < q < \infty$, 
\begin{equation*} 
\varrho_m(\bBr,L_q)
 \ \asymp \ 
(m^{-1} \log^{d-1}m)^{r - 1/p + 1/q}(\log^{d-1}m)^{(1/q -
1/\theta)_+}\quad,\quad 
 \begin{cases}
2 \le p, \ 2 \le \theta \le q,  \\
q \le 2.
\end{cases}
 \end{equation*}
\end{itemize} 
\end{thm}

\bproof 
This theorem directly follows from Theorem \ref{Theorem[T_n(B)<]} together with the lower bounds for linear
widths in Theorem \ref{thm[lambda_nB]}. Notice that it is also easily obtained from a non-periodic version in 
\cite{Di11}.
\eproof

\subsection{Time-limited sampling representations--B-splines}\index{Sampling!Representation!Time limited}\index{B-spline}
\label{timelimsamprep}
This section is devoted to a second method of constructing sampling operators based
on Smolyak's algorithm. This time the approximant is not longer a trigonometric
polynomial (like in \eqref{sampop_per}), it is rather a superposition of tensor products of compactly suppoerted
functions such as hat functions and more general B-splines. 
The potential of this technique for the approximation and integration of functions with dominating mixed smoothness has
been recently observed by Triebel \cite{Tr10,Tr12} and, independently, Dinh D\~ung \cite{Di11}. The latter reference
deals with B-spline representation, where the Faber-Schauder system is a special case. One striking advantage of
this approach is its potential for non-periodic spaces of dominating mixed smoothness, see \cite{Di11}. Apart
from that the relations \eqref{f52},\eqref{rws2}, and \eqref{flelambda} below are very well suited for construction for
construction ``fooling functions'' for sampling recovery and numerical integration, see Section \ref{sect:lbounds_int}
and Remark \ref{uppersmol} below. In this subsection, for convenience we will use $\T$ for the interval $[0,1]$ (instead of $[0,2\pi]$) with the usual identification
of the end points. 

\subsubsection*{The Faber-Schauder basis}\index{Faber-Schauder system}

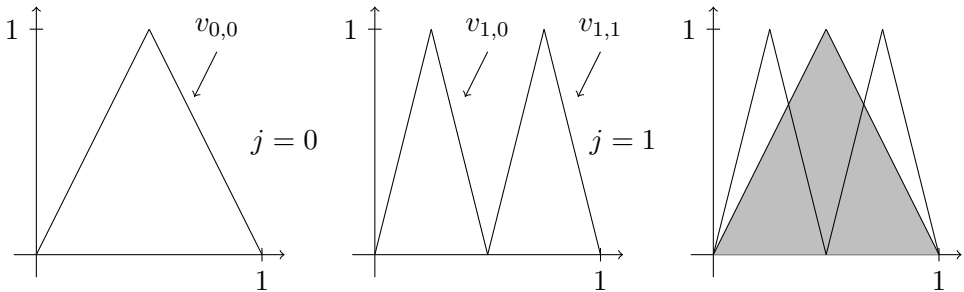
\begin{figure}[H]\label{fig:3}
\centering
\begin{tikzpicture}[scale=3]

\draw[->] (-0.1,0.0) -- (1.1,0.0);
\draw[->] (0.0,-0.1) -- (0.0,1.1); 

\draw (1.0,0.03) -- (1.0,-0.03) node [below] {$1$};
\draw (0.03,1.0) -- (-0.03,1.00) node [left] {$1$};
\draw[->] (0.8,0.9) -- (0.7,0.7);

\node at (0.8,1) {$v_{0,0}$};

\node at (1.1,0.5) {$j=0$};

\draw (0,0) -- (0.5,1);
\draw (0.5,1) -- (1,0);

\draw[->] (1.4,0.0) -- (2.6,0.0); 
\draw[->] (1.5,-0.1) -- (1.5,1.1); 

\draw (2.5,0.03) -- (2.5,-0.03) node [below] {$1$};
\draw (1.53,1.0) -- (1.47,1.00) node [left] {$1$};
\draw[->] (2.5,0.9) -- (2.4,0.7) ;
\draw[->] (2,0.9) -- (1.9,0.7) ;

\node at (2.5,1) {$v_{1,1}$};
\node at (2,1) {$v_{1,0}$};

\node at (2.6,0.5) {$j=1$};

\draw (1.5,0,0) -- (1.75,1.0);
\draw (1.75,1) -- (2,0.0);
\draw (2,0) -- (2.25,1.0);
\draw (2.25,1) -- (2.5,0.0);

\draw[->] (2.9,0.0) -- (4.1,0.0); 
\draw[->] (3.0,-0.1) -- (3.0,1.1);

\draw (4.0,0.03) -- (4.0,-0.03) node [below] {$1$};
\draw (3.03,1.0) -- (2.97,1.00) node [left] {$1$};

\fill[lightgray] plot[domain=3:3.5] (\x,-6+2*\x)%
                      -- plot[domain=4:3.5] (\x,0);%


\draw (3,0) -- (3.5,1);
\draw (3.5,1) -- (4,0);

\draw (3.0,0,0) -- (3.25,1.0);
\draw (3.25,1) -- (3.5,0.0);
\draw (3.5,0) -- (3.75,1.0);
\draw (3.75,1) -- (4,0.0);

\end{tikzpicture}
  \caption{The univariate Faber-Schauder basis, levels $j = 0,1$}
\label{fig_Faber1}
\end{figure}

Let us briefly recall the basic facts about the Faber-Schauder basis taken from
\cite[3.2.1, 3.2.2]{Tr10}. Faber \cite{Fa09} observed that every continuous
(non-periodic) function $f$ on $[0,1]$ can be represented (point-wise) as
\begin{equation}\label{f51}
    f(x) = f(0)\cdot (1-x)+f(1)\cdot x -
\frac{1}{2}\sum\limits_{j=0}^{\infty}\sum\limits_{k=0}^{2^j-1}
\Delta^2_{2^{-j-1}}(f,2^{-j}k)v_{j,k}(x)
\end{equation}
with convergence at least point-wise. Consequently, every periodic function on
$C(\T)$ can be represented by
\begin{equation}\label{f5}
    f(x) = f(0) - \frac{1}{2}\sum\limits_{j=0}^{\infty}\sum\limits_{k=0}^{2^j-1}
    \Delta^2_{2^{-j-1}}(f,2^{-j}k)v_{j,k}(x)\,.
\end{equation}

\begin{defi} The univariate periodic Faber-Schauder system is given by the system
of functions on $\T$
$$
      \{1,v_{j,k}:j\in \n, k\in \D_j\}\,,
$$
where $\D_j := \{0,...,2^j-1\}$ if $j\in \N_0$, $\D_{-1}:=\{0\}$ and 
\begin{equation}\label{f19}
    v_{j,m}(x)=\left\{\begin{array}{lcl}
	   2^{j+1}(x-2^{-j}m)&:&2^{-j}m \leq x \leq 2^{-j}m+2^{-j-1},\\
	   2^{j+1}(2^{-j}(m+1)-x)&:&2^{-j}m+2^{-j-1}\leq x \leq 2^{-j}(m+1),\\
	   0&:& \mbox{otherwise}\,.
        \end{array}\right.
\end{equation}
For notational reasons we let $v_{-1,0}:=1$ and obtain the Faber-Schauder system
$$
    \mathcal{F}:=\{v_{j,k}:j\in \N_{-1}, k\in \D_j\}\,.
$$
We denote by 
$$
    v:=v_{0,0}
$$
the Faber basis function on level zero. 
\end{defi}

\subsubsection*{The tensor Faber-Schauder system}
\index{Faber-Schauder system!Tensorized}
Let now $f(x_1,...,x_d)$ be a $d$-variate function $f\in C(\T^d)$. By fixing
all variables except $x_i$ we obtain by $g(\cdot) =
f(x_1,...x_{i-1},\cdot,x_{i+1},...,x_d)$ a univariate periodic
continuous function. By applying \eqref{f5} in every such component we obtain the point-wise
representation 
\begin{equation}\label{repr}
  f(\bx) = \sum\limits_{\bs \in \N_{-1}^d} \sum\limits_{\bk\in \D_{\bs}} d^2_{\bs,\bk}(f)
  v_{\bs,\bk}(\bx)\quad,\quad \bx\in \T^d\,,
\end{equation}
where $\D_{\bs} = \D_{s_1} \times...\times \D_{s_d}$, 
$$
  v_{\bs,\bk}(x_1,...,x_d):=v_{s_1,k_1}(x_1)\cdot...\cdot
v_{s_d,k_d}(x_d)\quad,\quad \bs\in \N_{-1}^d, \bk\in \D_{\bs}\,,
$$
and 
\begin{equation}\label{f_100}
  d^2_{\bs,\bk}(f):=
  (-2)^{-|e(\bs)|}\Delta^{2,e(\bs)}_{2^{-(\bs+1)}}(f,\bx_{\bs,\bk})\quad,\quad
  \bs\in \N_{-1}^d, \bk\in \D_{\bs}\,.
\end{equation}
Here we put $e(\bs) = \{i:s_i \neq -1\}$ and $\bx_{\bs,\bk} = (2^{-(s_1)_+}k_1,...,2^{-(s_d)_+}k_d)$\,.

\subsubsection*{The Faber-Schauder basis for Besov spaces}\index{Faber-Schauder system!Besov spaces}

Our next goal is to discretize the spaces $\bB^r_{p,\theta}$ using the
Faber-Schauder system $\mathcal{F}^d:=\{v_{\bs,\bk}\,:\,\bs\in \N_{-1}^d, \bk\in \D_{\bs}\}$. We obtain a sequence space isomorphism
performed by the
coefficient
mapping $d^2_{\bs,\bk}(f)$ above. In \cite[3.2.3, 3.2.4]{Tr10} and \cite[Thm.\
4.1]{Di11} this was done for the non-periodic setting $\bB^r_{p,\theta}(Q_d)$.
For the results stated below we refer to the recent paper \cite{HiMaOeUl14}.

\begin{defi}\label{defsequ} Let $0<p,\theta\leq \infty$ and $r\in \R$. Then
$\bb^r_{p,\theta}$ is
the collection of all sequences $\{\lambda_{\bs,\bk}\}_{\bs\in \N_{-1}^d, \bk\in \D_{\bs}}$
  such that
$$
  \|\lambda_{\bs,\bk}\|_{\bb^r_{p,\theta}}:=\Big[\sum\limits_{\bs\in
\N_{-1}^d}2^{|\bs|_1(r-1/p)q}\Big(\sum\limits_{\bk\in
\D_{\bs}}|\lambda_{\bs,\bk}|^p\Big)^{q/p} \Big]^{1/q}
$$
is finite.
\end{defi}

\begin{figure}[ht]
\begin{minipage}{0.48\textwidth}
  \includegraphics[scale = 0.45]{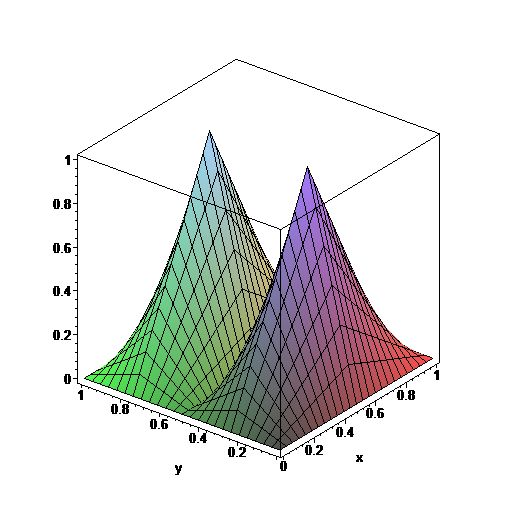}
\end{minipage}
\begin{minipage}{0.48\textwidth}
  \includegraphics[scale = 0.3]{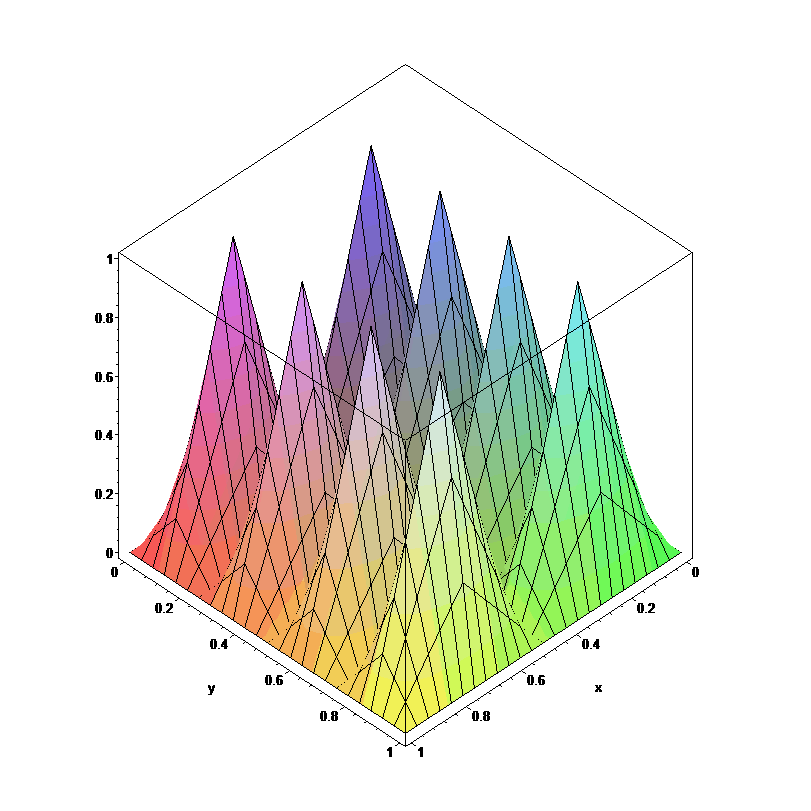}
\end{minipage}
\caption{Faber-Schauder levels in $d=2$}
\end{figure}

\begin{prop}\label{discvscont} Let $1 \leq p,\theta \leq \infty$ and
$1/p<r<2$.
Then there
exists a constant $c>0$ such that
\begin{equation}\label{f7}
    \big\|d^2_{\bs,\bk}(f)\|_{\bb^r_{p,\theta}} \leq c
    \|f\|_{\bB^r_{p,\theta}}
\end{equation}
for all $f\in C(\T^d)$.

\end{prop}

Let us also give a converse statement. Note, that we do not need the condition
$r>1/p$ here. 

\begin{prop}\label{contvsdisc} Let $1\leq p,\theta \leq \infty$ and $0<r<1+1/p$.
Let further $\{\lambda_{\bs,\bk}\}$ be a sequence belonging to
$\bb^r_{p,\theta}$. Then the function 
$$
   f := \sum\limits_{\bs\in \N_{-1}^d}\sum\limits_{\bk \in
\mathcal{D}_{\bs}} \lambda_{\bs,\bk}v_{\bs,\bk}
$$
belongs to $\bB^r_{p,\theta}$ and 
\begin{equation}\label{f52}
    \|f\|_{\bB^r_{p,\theta}} \leq
c\|\lambda_{\bs,\bk}(f)\|_{\bb^r_{p,\theta}}\,.
\end{equation}
\end{prop}

\subsubsection*{The Faber-Schauder basis for Sobolev spaces}\index{Faber-Schauder system!Sobolev spaces}

Let us now come to the Sobolev spaces of mixed smoothness. We will state
counterparts of the relations in Part (ii) of the Propositions
\ref{samp_repr1}, \ref{samp_repr2} above. Partial results have been already
obtained in \cite{Tr10} if $r=1$\,.

\begin{defi}\label{defsequw} Let $1<p < \infty$ and $r\in \R$. Then
$\bw^r_{p}$ is the collection of all sequences $\{\lambda_{\bs,\bk}\}_{\bs\in
\N_{-1}^d, \bk\in \D_{\bs}}$
  such that
$$
  \|\lambda_{\bs,\bk}\|_{\bw^r_{p}}:=\Big\|\Big[\sum\limits_{\bs\in
\N_{-1}^d}2^{|\bs|_1r2}\Big(\sum\limits_{\bk\in
\D_{\bs}}|\lambda_{\bs,\bk}(f)v_{\bs,\bk}|^2\Big) \Big]^{1/2}\Big\|_p
$$
is finite.
\end{defi}

\begin{prop}\label{sob_faber_repr1} Let $1<p<\infty$ and $\max\{1/p,1/2\}<r<2$
then we have for
any $f\in C(\T^d)$
\begin{equation}\label{lws2}
\|d^2_{\bs,\bk}(f)\|_{\bw^r_p} \lesssim \|f\|_{\Wrp}\,.
\end{equation}
\end{prop}

\begin{prop}\label{sob_faber_repr2}  Let $1<p<\infty$ and $0<r<\min\{1+1/p,
3/2\}$. Let furthermore $\lambda_{\bs,\bk}\in \bw_{p}^r$. Then
$f=\sum\limits_{\bs\in
\N_{-1}^d}\sum\limits_{\bk\in
\D_{\bs}} \lambda_{\bs,\bk}v_{\bs,\bk}$ belongs to $\Wrp$ and 
\begin{equation}\label{rws2}
 \|f\|_{\Wrp} \lesssim \|\lambda_{\bs,\bk}\|_{\bw_p^r}.
\end{equation}
\end{prop}

In case $r=1$ both assertions can be found in \cite[Chapt.\ 3]{Tr10}. In the stated form the relation will be
rigorously proved in \cite{ByUl16}. In fact, the result can be immediately deduced from its one dimensional counterpart
via the tensor-product structure of the spaces and the corresponding operators, see \cite[Thm.\ 2.1, 2.5]{SiUl09}.
Note, that Propositions \ref{sob_faber_repr1}, \ref{sob_faber_repr2} imply that the tensorized Faber-Schauder system
represents an unconditional basis in $\bW^r_p$ in the respective parameter domain. Let 
us illustrate this parameter domain in the following diagram.

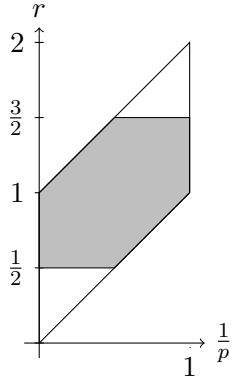
\begin{figure}[H]
 \begin{center}
\begin{tikzpicture}[scale=2]
\draw[->] (-0.1,-1.0) -- (1.1,-1.0) node[right] {$\frac{1}{p}$};
\draw[->] (0.0,-1.1) -- (0.0,1.1) node[above] {$r$};

\draw (1.0,-1.03) -- (1.0,-1.03) node [below] {$1$};
\draw (0.03,1.0) -- (-0.03,1.00) node [left] {$2$};
\draw (0.03,.5) -- (-0.03,.5) node [left] {$\frac{3}{2}$};
\draw (0.03,0) -- (-0.03,0) node [left] {$1$};
\draw (0.03,-.5) -- (-0.03,-.5) node [left] {$\frac{1}{2}$};
\draw (0.03,-1.0) -- (-0.03,-1.00) ;

\draw[fill=lightgray] (0.0,-.5) -- (0.0,0.0) -- (.5,.5) -- (1.0,0.5) -- (1.0,0.0) -- (.5,-.5) --
(0.0,-.5);
\draw (0.0,-1.0) -- (0.0,0.0) -- (1.0,1.0) -- (1.0,0.0) -- (0.0,-1.0);
\draw (0.85,0.65) node {};
\draw (0.15,-0.65) node {};
\end{tikzpicture}
\end{center}
\caption{The parameter domain for the unconditional Faber-Schauder basis in $\bW^r_p$}
\end{figure}

\begin{rem}\index{Haar system} There seems to be a fundamental difference between Besov spaces and
Sobolev spaces in this context. A corresponding problem for Haar bases has been
studied by Seeger, Ullrich in the recent papers \cite{SeUl15, SeUl15_2} which leads to the conclusion that
the above described region is sharp. It also indicates that the parameter domain for the representation in Proposition
\ref{samp_repr1}, (ii) is sharp, i.e., it can not be extended to $1/p<r\leq 1/2$ if $p>2$. This fundamental difference
seems to be reflected in the unknown behavior of optimal cubature/sampling
recovery in the region of small smoothness (lower triangles), see Section
\ref{numint} below. 
\end{rem}

\subsubsection*{B-spline representations and general atoms}\index{Sampling!Representation!Time limited}
\index{B-spline}
\label{Bspline}

When using the Faber-Schauder system for the discretization of function
spaces we always get the restriction $r<2$, see Propositions \ref{discvscont}
and \ref{contvsdisc} above, due to the limited smoothness of the tensorized hat
functions. The question arises whether one can use smoother basis functions such
as B-splines. A smooth hierarchical basis approach simililar to the one we will describe below has been developed by Bungartz \cite{Bu_habil}. Here 
we will focus on D. D\~ung's \cite{Di15,Di16_1} approach towards a B-spline quasi-interpolation representation for periodic continuous functions on
$\T^d$. For a non-periodic counterpart, see \cite{Di11}. 
For a given $\ell \in \N,$ denote by  $M=M_{2\ell}$ the cardinal B-spline of order $2\ell$ defined as the $2\ell$-fold
convolution $M:= (\chi_{[0,1]} \ast \cdots \ast \chi_{[0,1]})$, 
where $\chi_{[0,1]}$ denotes the characteristic function of the interval $[0,1]$\,.
%
%

Let $\Lambda = \{\lambda(j)\}_{|j| \le \mu}$ be a given finite even sequence, i.e., 
$\lambda(-j) = \lambda(j)$ for some $\mu \ge \ell - 1$. We define the linear operator $Q$ for functions $f$ on $\R$ by  
\begin{equation} \label{def:Q}
Q(f,x):= \ \sum_{k \in \Z} \Lambda (f,k)M(x-k), \quad
\Lambda (f,k):= \ \sum_{|j| \le \mu} \lambda (j) f(k-j + \ell).
\end{equation}
The operator $Q$ is called a 
{\em quasi-interpolation operator} if  $Q(g) = g$ for every polynomial $g$ of degree at most $2\ell - 1$. 
Since $M(2\ell\,2^s x)=0$ for every $s \in \N_0$ and every $x \notin (0,1)$, we can extend  the univariate B-spline  
$M(2\ell\,2^s\cdot)$ to an $1$-periodic function on the whole 
$\R$ which can be considered as a function on $\T$. Denote this function on $\T$ by $N_s$ and define 
$N_{s,k}(x):= \ N_s(x - (2\ell)^{-1}2^{-s}k), \ s\in \N_0, \ k \in I(s)$, 
where $I(s) := \{0,1,..., 2\ell2^s - 1\}$. 
The quasi-interpolation operator $Q$ induces the periodic quasi-interpolation operator $Q_s$ on $\T$ which is defined for $s\in
\N_0$ and a function $f$ on $\T$
through
\begin{equation} \nonumber
Q_s(f)  := \ 
\sum_{k \in I(s)} a_{s,k}(f)N_{s,k} \quad , \quad  
a_{s,k}(f):= \ 
\sum_{|j| \le \mu} \lambda (j) f((2\ell)^{-1}2^{-s}(k-j+ \ell)).
\end{equation}

\noindent A procedure similar to \eqref{tensor1} yields an operator $q_{\bs}$
such that every continuous function $f$ on $\T^d$ is represented as
$B$-spline series 
\begin{equation} \label{eq:B-splineRepresentation}
f \ = \sum_{\bs \in \N_0^d} \ q_\bs(f) = 
\sum_{\bs \in \N_0^d} \sum_{\bk \in I^d(\bs)} c_{\bs,\bk}(f)N_{\bs,\bk}, 
\end{equation}  
converging in the norm of $C(\T^d)$, where $I^d(\bs):=\prod_{i=1}^d I(s_i)$ and the coefficient functionals
$c_{\bs,\bk}(f)$ are explicitly constructed as
linear combinations of at most $n_0$ function
values of $f$ for some $n_0 \in \N$ which is independent of $\bs,\bk$ and $f$. 
The following proposition represents a counterpart of Proposition
\ref{samp_repr1} and a generalization of Propositions \ref{sob_faber_repr1} and \ref{sob_faber_repr2}.

\begin{prop} \label{samp_repr1Q} {\em (i)} Let $1\leq p,\theta \leq \infty$ and
 $1/p < r < \min\{2\ell, 2\ell - 1 + 1/p\}$. Then we have for any $f\in \Brpt$, 
\begin{equation}\label{lhs1Q}
\Big(\sum_{\bj \in \N_0^d}2^{r|\bj|_1 \theta}\|q_\bj(f)\|^{\theta}_p
\Big)^{1/\theta} \asymp \|f\|_{\Brpt}
\end{equation}
with the sum being replaced by a supremum for $\theta = \infty$.\\
{\em (ii)} Let $1<p<\infty$ and $\max\{1/p,1/2\} < r < 2\ell - 1$. Then we have for
any $f\in \Wrp$,
\begin{equation}\label{lhs2Q}
\Big\|\Big(\sum_{\bj \in \N_0^d}2^{r|\bj|_1 2}
|q_{\bj}(f)|^2\Big)^{1/2}\Big\|_p \asymp \|f\|_{\Wrp}\,.
\end{equation}
\end{prop}

Proposition \ref{samp_repr1Q} as well as a counterpart for the 
B-spline representation \eqref{eq:B-splineRepresentation} of 
Proposition \ref{samp_repr2} have been proven in \cite{Di16_1}.
As in the proofs of  Proposition \ref{samp_repr1}(ii) and 
Propositions \ref{sob_faber_repr1} and \ref{sob_faber_repr2}, the proof of 
Proposition \ref{samp_repr1Q}(ii) requires tools from
Fourier analysis, i.e., maximal functions of Peetre and Hardy-Littlewood
type. Moreover, it is essentially based on a special explicit formula for the coefficients 
$c_{\bk,\bs}(f)$ in the representation \eqref{eq:B-splineRepresentation}, and on the specific property of the representation \eqref{eq:B-splineRepresentation} that the component functions
$q_\bs(f)$ can be split into a finite sum of 
the B-splines $N_{\bs,\bk}$ having non-overlap interiors of their supports. One can probably extend the smoothness range in Proposition 
\ref{samp_repr1Q}, (ii) to $\max\{1/p,1/2\} < r < 2\ell-1 + \min\{1/p,1/2\}$ using similar techniques as in \cite{ByUl16}.

There are indeed many ways to construct quasi-interpolation
operators built on $B$-splines, see, e.g., \cite{CD87,C92,BF73}. 
We give two examples of quasi-interpolation operators. For more examples, see
\cite{C92}.
A piecewise linear quasi-interpolation operator is defined as
\begin{equation} \nonumber
Q(f,x):= \ \sum_{k \in \Z} f(k) M(x-k), 
\end{equation} 
where $M$ is the   
symmetric piecewise linear B-spline $\ell = 1$).
It is related to the classical Faber-Schauder basis of the hat functions 
(see, e.g., \cite{Di11}, \cite{Tr10}, for details). 
Another example is the cubic quasi-interpolation operator generated by the symmetric cubic B-spline $M$ ($\ell=2$): 
\begin{equation*} 
Q(f,x):= \ \sum_{k \in \Z} \frac {1}{6} \{- f(k-1) + 8f(k) - f(k+1)\} M(x-k). 
\end{equation*} 
We are interested in the most general form for $C^{\infty}$-bumps
for all parameters $r>0$. This has been shown
by Vyb\'iral \cite{Vyb06} based on the approach of Frazier, Jawerth
\cite{FrJa90}\,. In fact, consider a smooth bump
function $\varphi$ (atom) supported in $[0,1]^d$. We
define
$$
   a_{\bs,\bk}(x_1,...,x_d) := \varphi(2^{s_1}x_1-k_1)\cdot...\cdot
\varphi(2^{s_d}x_d-k_d)\quad,\quad \bx\in \R^d,\,\bs\in \N_0^d,\, \bk\in
\Z^d\,.
$$

\begin{prop}\label{atomic} Let $1\leq p, \theta \leq \infty$ and
$r>0$. Let further $\{\lambda_{\bs,\bk}\}_{\bs,\bk}$ be a sequence belonging to
$\bb^r_{p,\theta}$. Then the function 
\begin{equation}\label{tf}
	f:=\sum\limits_{\bs\in \N_0^d} \sum\limits_{\bk\in \mathcal{D}_{\bs}}
\lambda_{\bs,\bk}a_{\bs,\bk}(x)
\end{equation}
exists and is supported in $[0,1]^d$\,. Moreover, 
\begin{equation}\label{flelambda}
 \|f\|_{\Brpt} \lesssim \Big(\sum\limits_{\bs\in \N_0^d}
 2^{|\bs|_1(r-1/p)\theta}\Big[\sum\limits_{\bk\in
 \Z^d}|\lambda_{\bs,\bk}|^{p}\Big]^{\theta/p}\Big)^{1/\theta}\,.
\end{equation}
\end{prop}

\subsection{Open problems}\index{Open problems!Sampling recovery}

Below the reader may find a list of important open problems in this field. Afterwards we will comment on some of those.
\\

\noindent{\bf Open problem 5.1} Find the right order of the optimal sampling recovery $\varrho_m(\bW^r_p,L_p)$ in case 
$1\le p\le \infty$ and $r>1/p$.\\

\noindent{\bf Open problem 5.2} Find the right order of the optimal sampling recovery $\varrho_m(\bW^r_p,L_q)$,
$1<p<2<q<\infty$, see also Figure \ref{fig11}.\\

\noindent{\bf Open problem 5.3} Find the right order of the optimal sampling recovery $\varrho_m(\bH^r_p,L_q)$ in the
question-marked regions in Figure \ref{fig12} (left picture). \\

\noindent{\bf Open problem 5.4} Find the right order of the optimal sampling recovery $\varrho_m(\bW^r_p,L_p)$ in the
case of small smoothness, $2<p<\infty$, $1/p<r\leq 1/2$.\\

\noindent{\bf Open problem 5.5} Find the right order of the optimal sampling recovery $\varrho_m(\bB^r_{p,\theta},L_p)$
for $\theta>1$ and $\varrho_m(\bB^r_{1,1},L_1)$. The latter problem seems to reduce to the problem of finding a lower
bound for $\lambda_m(\bB^r_{1,1},L_1)$, see Section 4.\\

The following simple example refers to Open problem 5.2 (see also the lower right triangle in Figure \ref{fig11}). It
shows that even in case $d=1$ sampling numbers $\varrho_m$ and linear widths $\lambda_m$ do not coincide. 

\begin{thm}\label{lambda<rho} Let $1<p<2<q<\infty$ and $r>1/p$. Then we have

\begin{equation}
 \begin{split}
  \lambda_m(\bW^r_p,L_q) &\asymp m^{-(r-1/2+1/q)}(\log m)^{(d-1)(r-1/2+1/q)}\\
  &=o(m^{-(r-1/p+1/q)})\\
  &=o(\varrho_m(\bW^r_p,L_q))\,.
 \end{split}
\end{equation}
\end{thm}
\bproof Already in the univariate case it holds
$$m^{-(r-1/p+1/q)} \lesssim \varrho_m(\bW^r_p,L_q)\,.$$
\eproof

The case $p=q$ looks rather simple. However, Open Problem 5.l seems to be a hard problem. We state the following
conjecture. 

\begin{conj}\label{W1} Let $1<p<\infty$ and $r>\max\{1/p,1/2\}$. Then 
$$
    \varrho_m(\Wrp, L_p) \asymp m^{-r} (\log m)^{(d-1)(r+1/2)}\,.
$$
\end{conj}

The upper bound is known, see  Theorem \ref{Sobsamp_q}(i). Problematic is the
lower bound. The linear widths are smaller than the bound in
Conjecture \ref{W1}, see Theorem \ref{thm[lambda_nW]}. In other words, we
conjecture two things. First, sampling is worse than approximation also in
this situation and second, Smolyak's algorithm is optimal for sampling numbers.
Note, that the Hilbert space situation $\bW^r_2$ in $L_2$ is also open in this
respect. 

Let us comment on Open problem 5.4. In the case of small smoothness we consider the situation $2<p<\infty$ and
$1/p<r\leq 1/2$. This problem is also relevant for numerical integration, see Section
\ref{numint} below, where we were able to give a partial answer.
In fact, it is possible to prove an interesting upper bound for
$\kappa_m(\bW^r_p)$ in the situation $2<p<\infty$ and $1/p<r\leq 1/2$. 

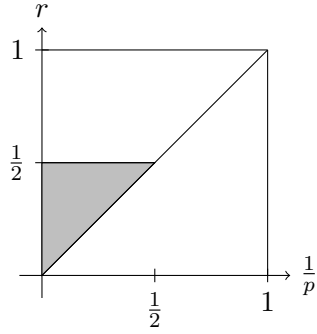
\begin{figure}[H]
  \begin{center}
  \begin{tikzpicture}[scale=3]

\draw[->] (-0.1,0.0) -- (1.1,0.0) node[right] {$\frac{1}{p}$};
\draw[->] (0.0,-.1) -- (0.0,1.1) node[above] {$r$};
\draw (1.0,0.03) -- (1.0,-0.03) node [below] {$1$};
\draw (0.03,1.0) -- (-0.03,1.00) node [left] {$1$};
\draw (0.03,0.5) -- (-0.03,0.5) node [left] {$\frac{1}{2}$};
\draw (1.0,0) -- (1.0,1.0);
\draw (0,1) -- (1.0,1.0);
\filldraw[fill=lightgray,draw=black] (0,0.5)--(0.5,0.5)--(0,0)--(0,0.5); 
\draw (0,0.5) -- (0.5,0.5);
\draw (0,0) -- (1,1);
\draw (1/2,0.03) -- (1/2,-0.03) node [below] {$\frac{1}{2}$};
\end{tikzpicture}
\caption{The region of ``small smoothness''}
\end{center}
\end{figure}

Let us start with an interesting
observation for small $r$ in the Besov setting. Here the ``difference'' between
$\lambda_m$ and $\varrho_m$ will be much smaller than in Theorem
\ref{lambda<rho}
and only  apparent if $d\geq 2$\,. 

\begin{thm} Let $2< p \leq \infty$ and $1/p<r<1/2$. Then 
\begin{equation}
 \begin{split}
  \lambda_m(\bB^r_{p,p},L_p) &\asymp m^{-r}(\log
m)^{(d-1)(r+1/2-1/p)}\\
  &=o(m^{-r}(\log m)^{(d-1)(1-1/p)})\\
  &=o(\varrho_m(\bB^r_{p,p},L_p)).
 \end{split}
\end{equation}
\end{thm}
\bproof Note that $$m^{-r}(\log m)^{(d-1)(1-1/p)} \asymp \kappa_m(\bB^r_{p,p}) \lesssim \varrho_m(\bB^r_{p,p},L_p) \lesssim
m^{-r}(\log m)^{(d-1)(r+1-1/p)}\,.$$
\eproof

%

%% file: entropy_numbers.tex
\section{Entropy numbers}\index{Entropy number}

\subsection{General notions and inequalities}
\label{entropy_gen}

The concept of entropy is also known as {\it Kolmogorov entropy} and {\it
metric entropy}. This concept allows us to measure how big is a compact set. In
the case of finite dimensional compacts it is convenient  to compare compact
sets by their volumes. In the case of infinite dimensional Banach spaces this
way does not work. The concept of entropy is a good replacement of the concept
of volume in infinite dimensional Banach spaces.

Let $X$ be a Banach space and let $B_X$ denote the unit ball of $X$ with the
center at $0$. Denote by $B_X(y,r)$ a ball with center $y$ and radius $r$:
$\{x\in X:\|x-y\|\le r\}$. For a compact set $A$ and a positive number $\eps$ we
define the covering number $N_\eps(A)$
 as follows
$$
N_\eps(A) := N_\eps(A,X)  
:=\min \{n : \exists y^1,\dots,y^n :A\subseteq \cup_{j=1}^n B_X(y^j,\eps)\}.
$$

It is convenient to consider along with the entropy $H_\eps(A,X):= \log_2
N_\eps(A,X)$ the entropy numbers $\e_n(A,X)$:
$$
\e_n(A,X) :=  \inf \{\eps>0 : \exists y^1,\dots ,y^{2^{n}} \in X : A \subseteq
\cup_{j=1}
^{2^{n}} B_X(y^j,\eps)\}.
$$
The definition of the entropy numbers can be written in a form similar to the definition of the Kolmogorov widths
(numbers):
$$
\e_n(A,X) :=\inf_{y^1,\dots,y^{2^n}}\sup_{f\in A}\inf_j\|f-y^j\|_X.
$$

In the case of finite dimensional spaces ${\mathbb R}^d$ equipped with different
norms, say, norms $\|\cdot\|_X$ and $\|\cdot\|_Y$ the volume argument gives some
bounds on the $N_\eps(B_Y,X)$. For a Lebesgue measurable set $E\subset {\mathbb
R}^d$ we denote its Lebesgue measure by $vol(E):=vol_d(E)$.  Let further $C=A\oplus B:=\{c:c=a+b, a\in A,b\in B\}$.
\begin{thm}\label{Theorem 32.1} For any two norms $X$ and $Y$ and any $\eps>0$
we have
\begin{equation}\nonumber
\frac{1}{\eps^d}\frac{vol(B_Y)}{vol(B_X)} \le N_\eps(B_Y,X) \le
\frac{vol(B_Y(0,2/\eps)\oplus B_X)}{vol(B_X)}. 
\end{equation}
\end{thm}

Let us formulate one immediate corollary of Theorem \ref{Theorem 32.1}.
\begin{cor}\label{Corollary 32.1} For any $d$-dimensional Banach space $X$ we
have
$$
\eps^{-d} \le N_\eps(B_X,X) \le (1+2/\eps)^d,
$$
and, therefore,
$$
\e_n(B_X,X) \le 3(2^{-n/d}).
$$
\end{cor}

 Let us consider some typical $d$-dimensional Banach spaces. These are the
spaces $\ell^d_p$: the linear space ${\mathbb R}^d$ equipped with the norms
$$
|x|_p:=\|x\|_{\ell^d_p}:= (\sum_{j=1}^d|x_j|^p)^{1/p},\quad 1\le p<\infty,
$$
$$
|x|_\infty:=\|x\|_{\ell^d_\infty}:= \max_{j}|x_j|.
$$
Denote $B^d_p := B_{\ell^d_p}$.
\begin{thm}\label{Theorem 32.2} Let $0<p<q\leq \infty$. Then we have the
following bounds for the entropy numbers $\e_n(B_p^d,\ell_q^d)$
$$\e_n(B_p^d,\ell_q^d) \asymp 
\begin{cases}
1,&1\leq n \leq \log_2 d;\\
\left[\frac{\log(\frac{d}{n}+1)}{n}\right]^{\frac{1}{p}-\frac{1}{q}},&\log_2
d\leq n \leq d;\\
d^{\frac{1}{q}-\frac{1}{p}}2^{-n/d},&d\leq n.
\end{cases} $$
\end{thm}

Theorem \ref{Theorem 32.2} has been obtained by Sch\"utt \cite{Schu} for the
Banach-space range of parameters. He considered the more general situation of
symmetric Banach spaces, see also \cite{Ho} and \cite{Mai}. For the quasi-Banach
situation we refer to \cite{ET} and finally to  K\"uhn \cite{Ku01}, Gu\'edon, Litvak \cite{GuLi00} and Edmunds, Netrusov \cite[Thm.\ 2]{EdNe98}, where the lower bound
in the middle case is provided.

We define by $\sigma(x)$ the normalized $(d-1)$-dimensional measure on the
sphere $S^{d-1}$ -- the boundary of $B^d_2$. 
 
\begin{thm}\label{Theorem 32.3} Let $X$ be $\bR^d$ equipped with $\|\cdot\|$ and
$$
M_X:= \int_{S^{d-1}}\|x\|d\sigma(x).
$$
Then we have
$$
 \e_n(B^d_2,X) \lesssim M_X \left\{\begin{array}{ll}   (d/n)^{1/2}, & n\le d\\
 2^{-n/d} , &  n\ge d.\end{array} \right.
$$
\end{thm}

 Theorem \ref{Theorem 32.3} is a dual version of
the corresponding result from \cite{Su}. Theorem \ref{Theorem 32.3} was proved
in \cite{PaJa86}. 

There are several general results  which give 
lower estimates of the Kolmogorov widths $d_n(F,X)$ in terms of the entropy
numbers $\e_k(F,X)$. 

\begin{thm}\label{thm:carl} For any $r>0$ we have 
\begin{equation}\label{1.3}
\max_{1\le k \le n} k^{r} \e_k(F,X) \le C(r) \max _{1\le m \le n} m^{r}
d_{m-1}(F,X).
\end{equation}
\end{thm}

\noindent Theorem \ref{thm:carl} is due to Carl \cite{C}, which is\index{Inequality!Carl}
actually stated for general $s-$numbers  and unit balls $F$ stemming from a norm, see also Carl, Stephani \cite{CS}. Recently, 
Hinrichs et al. \cite{HiKoVy15} proved a version for quasi-Banach spaces (non-convex $F$), which is particularly interesting for Gelfand numbers
$c_m$ on the right-hand side, see Subsection \ref{sect:snumbers} below.

Let us introduce a nonlinear Kolmogorov's $(N,m)$-width:
$$
d_m(F,X,N) :=  \inf_{\L_N, \#\L_N \le N} \sup_{f\in F} \inf_{L\in \L_N}
\inf_{g\in L} \|f-g\|_X,
$$
where $\L_N$ is a set of at most $N$ $m$-dimensional subspaces $L$. It is clear
that
$$
d_m(F,X,1) = d_m(F,X).
$$
The new feature of $d_m(F,X,N)$ is  that we allow to choose a subspace 
$L\in \L_N$ depending on $f\in F$. It is clear that the bigger $N$ the more
flexibility we have 
to approximate $f$. The following inequality from \cite{T1} (see also
\cite{Tbook}, Section 3.5) 
is a generalization of inequality (\ref{1.3}).   

\begin{thm}\label{thm:gen_carl} Let $r>0$. Then 
 \begin{equation}\label{1.4}
   \max_{1\le k \le n} k^r \e_k(F,X) \le C(r,K) \max _{1\le m \le n} 
   m^r d_{m-1}(F,X,K^m),
  \end{equation}
where we denote
$$
d_0(F,X,N) := \sup_{f\in F}\|f\|_X .
$$
\end{thm}

The possibility of replacing $K^m$ by $(Kn/m)^m$ in (\ref{1.4}) was discussed in
\cite{T1}, in Section 3.5 of \cite{Tbook}, and in \cite{Tappr}.  The
corresponding remark is from \cite{Tappr}.
\begin{rem}\label{R1.1c} Examining the proof of (\ref{1.4}) one can check that
the following inequality holds
$$
n^r \e_n(F,X) \le C(r,K) \max_{1\le m \le n} 
m^r d_{m-1}(F,X,(Kn/m)^m).
$$
\end{rem} 
Finally, Theorem
\ref{T2.1} is from \cite{Tappr}.
\begin{thm}\label{T2.1} Let a compact $F\subset X$ and  a number $r>0$ be such
that for some
$n \in \N$
$$
  d_{m-1}(F,X,(Kn/m)^m) \le m^{-r},\quad m\le n.
$$
Then for $k\le n$
$$
\e_k(F,X) \le C(r,K) \left(\frac{\log(2n/k)}{k}\right)^r.
$$
\end{thm}

We discuss an application which motivated a study of $d_m(F,X,N)$ with
$N=(Kn/m)^m$. Let $\D=\{g_j\}_{j=1}^n$ be a system of normalized elements of
cardinality $|\D|=n$ in a Banach space $X$. Consider best $m$-term
approximations of $f$ with respect to $\D$
$$
\sigma_m(f,\D)_X:= \inf_{\{c_j\};\Lambda:|\Lambda|=m}\|f-\sum_{j\in
\Lambda}c_jg_j\|.
$$
For a function class $F$ set
$$
\sigma_m(F,\D)_X:=\sup_{f\in F}\sigma_m(f,\D)_X.
$$
Then it is clear that for any system $\D$, $|\D|=n$,
$$
d_m(F,X,\binom{n}{m})\le \sigma_m(F,\D)_X.
$$
Next,
$$
\binom{n}{m} \le (en/m)^m.
$$
Thus Theorem \ref{T2.1} implies the following theorem (see \cite{Tappr}).
\begin{thm}\label{T3.1} Let a compact $F\subset X$ be such that there exists a
normalized system $\D$, $|\D|=n$, and  a number $r>0$ such that 
$$
  \sigma_m(F,\D)_X \le m^{-r},\quad m\le n.
$$
Then for $k\le n$
\begin{equation}\nonumber
\e_k(F,X) \le C(r) \left(\frac{\log(2n/k)}{k}\right)^r.
\end{equation}
\end{thm}

Theorem \ref{T3.1} is useful in proving lower bounds for best $m$-term
approximations. Recently, it was shown in \cite{VT156} how Theorem \ref{T3.1} can be used for proving sharp upper bounds for the entropy numbers of classes of mixed smoothness. 

We now proceed to two multiplicative inequalities for the $L_p$ spaces. Let $D$
be a domain in $\bR^d$ and let $L_p:= L_p(D)$ denote the corresponding $L_p$
space, $1\le p\le \infty$, with respect to the Lebesgue measure. We note that
the inequalities below hold for any measure $\mu$ on $D$. 

\begin{thm}\label{T35.2} Let $A\subset L_1\cap L_\infty$. Then for any $1\le
p\le\infty$ we have $A\subset L_p$ and
$$
\e_{n+m}(A,L_p) \le 2\e_n(A,L_1)^{1/p}\e_m(A,L_\infty)^{1-1/p}.
$$
\end{thm}

It will be convenient for us to formulate one more inequality in terms of
entropy numbers of operators. Let $S$ be a linear operator from $X$ to $Y$. We
define the $n$th entropy number of $S$ as
$$
\e_n(S:X\to Y) := \e_n(S(B_X),Y)
$$
where $S(B_X)$ is the image of $B_X$ under mapping $S$. 
\begin{thm}\label{T35.3} For any $1\le p \le \infty$ and any Banach space $Y$ we
have
$$
\e_{n+m}(S:L_p\to Y)\le 2\e_n(S:L_1\to Y)^{1/p}\e_m(S:L_\infty\to Y)^{1-1/p}.
$$
\end{thm}

\subsection{Entropy numbers for $\bW$ classes in $L_q$} \index{Entropy number}

It is well known that in the univariate case   
\begin{equation}\label{36.1}
\e_n(W^r_{p},L_q)\asymp n^{-r}
\end{equation}
 holds for all $1\le p,q \le \infty$ and $r>(1/p-1/q)_+$. We note that condition
$r>(1/p-1/q)_+$ is a necessary and sufficient condition for a compact embedding
of $W^r_{p}$ into $L_q$. Thus (\ref{36.1}) provides a complete description of
the rate of $\e_n(W^r_{p},L_q)$ in the univariate case. We point out that
(\ref{36.1}) shows that the rate of decay of $\e_n(W^r_{p},L_q)$ depends only on
$r$ and does not depend on $p$ and $q$. In this sense the strongest upper bound
(for $r>1$) is $\e_n(W^r_{1},L_\infty) \lesssim n^{-r}$ and the strongest lower
bound is $\e_n(W^r_{\infty},L_1)\gtrsim n^{-r}$. 

There are different generalizations of classes $W^r_{p}$ to the case of
multivariate functions. In this section we only discuss classes $\bW^r_{p}$ of
functions with bounded mixed derivative and $\bB^r_{p,\theta}$ of functions with
bounded mixed difference. These classes are of special interest for several
reasons: 
\begin{itemize}
 \item[(A)] The problem of the rate of decay of $\e_n(\bW^r_{p},L_q)$ in a
particular case $r=1$, $p=2$, $q=\infty$ is equivalent (see \cite{KL}) to a
fundamental problem of probability theory (the small ball problem). Both of
these problems are still open for $d>2$.
 \item[(B)] The problem of the rate of decay of $\e_n(\bW^r_{p},L_q)$ and
$\e_n(\Brpt,L_q)$ turns out to be a very rich and difficult problem. There are
still many open problems. Those problems that have been resolved required
different nontrivial methods for different pairs $(p,q)$.
\end{itemize}
 
 We begin with classes $\Wrp$. 

\begin{thm}\label{thm:ent_pq} For $1<p,q<\infty$ and $r>(1/p-1/q)_+ $ one has
\begin{equation}\label{36.2}
\e_n(\bW^r_{p},L_q) \asymp n^{-r}(\log n)^{r (d-1)}.
\end{equation}
\end{thm}

\begin{thm}\label{Wlo} For $r>0$ and $1\le p <\infty$ one has
$$	
\e_n(\bW^r_{p},L_1) \gtrsim n^{-r}(\log n)^{r (d-1)}.
$$
\end{thm} 

The problem of estimating $\e_n(\bW^r_{p},L_q)$ has a long history. The first
result on the right order of $\e_n(\bW^r_{2},L_2)$ was obtained by Smolyak
\cite{Smo}.  The case $1<q=p<\infty$, $r>0$ was
established by Dinh D\~ung \cite{Di85}. Theorem \ref{thm:ent_pq} was established
by Temlyakov, see \cite{TE1},
\cite{TE2} for all $1<p,q<\infty$ and  $r>1$. Belinskii \cite{Bel} extended
(\ref{36.2}) to
the
case $r>(1/p-1/q)_+$. Later, this result was extended by Vyb\'iral to a
non-periodic setting, see
\cite{Vyb06}.

It is known in approximation theory (see \cite{TBook}) that investigation of
asymptotic characteristics of classes $\bW^r_{p}$ in $L_q$ becomes more
difficult when $p$ or $q$ takes value $1$ or $\infty$ than when $1<p,q<\infty$.
This is true for $\e_n(\bW^r_{p},L_q)$, too. Theorem \ref{Wlo} is taken from
\cite{TE1,TE2}. It was discovered that in some of these extreme cases ($p$ or
$q$ equals $1$ or $\infty$) relation (\ref{36.2}) holds and in other cases it
does not hold. We describe the picture in detail. It was proved in \cite{TE2}
that (\ref{36.2}) holds for $q=1$, $1<p<\infty$, $r>0$. It was also proved that
(\ref{36.2}) holds for $q=1$, $p=\infty$ (see \cite{Bel} for $r>1/2$ and 
\cite{KTE2} for $r>0$). Summarizing, we state that (\ref{36.2}) holds for
$1<p,q<\infty$ and $q=1$, $1<p\le\infty$ for all $d$ (with appropriate 
restrictions on $r$). This easily implies that (\ref{36.2}) also holds for
$p=\infty$, $1\le q<\infty$. We formulate this as a theorem.

\begin{thm}\label{Twext} Let $q=1$, $1<p\le\infty$ or $p=\infty$, $1\le
q<\infty$ and $r>0$. Then
$$
\e_n(\bW^r_{p},L_q) \asymp n^{-r}(\log n)^{r (d-1)}.
$$
\end{thm}

For all other pairs $(p,q)$, namely, for
$q=\infty$, $1\le p\le\infty$ and 
$p=1$, $1\le q\le \infty$ the rate of $\e_n(\bW^r_{p},L_q)$ is not known in the
case $d>2$. It is an outstanding open problem. 

\subsubsection*{Entropy numbers in $L_\infty$}

Let us start with the following result for the case $d=2$. 

\begin{thm}\label{Td=2} Let $d=2$, $1<p\le\infty$, $r>\max\{1/p,1/2\}$. Then
\begin{equation}\label{36.3}
\e_n(\bW^r_{p},L_\infty)\asymp n^{-r}(\log n)^{r+1/2}\,.
\end{equation}
\end{thm}

The first result on the right order of
$\e_n(\bW^r_{p},L_q)$ in the case $q=\infty$ was obtained by Kuelbs and Li
\cite{KL} for $p=2$, $r=1$. It was proved in \cite{TE3} that (\ref{36.3})
holds for $1<p<\infty$, $r>1$. We note that the upper bound in (\ref{36.3}) was
proved under condition $r>1$ and the lower bound in (\ref{36.3}) was proved
under condition $r>1/p$. Belinskii \cite{Bel} proved the upper bound in
(\ref{36.3}) for $1<p<\infty$ under condition $r>\max\{1/p,1/2\}$. Relation
(\ref{36.3}) for $p=\infty$ under assumption $r>1/2$ was proved in \cite{TE4}. 

 An analogue of the upper bound in (\ref{36.3}) for any $d$ was obtained by
Belinskii \cite{Be5, Bel} and, independently, by Dunker, Linde, Lifshits, K\"uhn \cite{DuLiKuLi99} (in case $r=1$)
\begin{equation}\label{36.4'} 
\e_n(\bW^r_{p},L_\infty)\lesssim n^{-r}(\log n)^{(d-1)r+1/2},\quad
r>\max\{1/p,1/2\}.
\end{equation}
That proof is based on Theorem \ref{Theorem 32.3} (see also the book \cite{TrBe04} for a detailed description of the technique). 
Recent results on the Small Ball Inequality for the Haar system\index{Haar system} (see \cite{BL}, \cite{BLV} and the remark after \eqref{2.6.5} in Subsection \ref{subsect:SBI} above) 
allow us to improve a trivial lower bound to the 
following one for $r=1$ and all $p<\infty$, $d\geq 3$:
\be\nonumber
\e_n(\bW^1_p,L_\infty) \gtrsim n^{-1}(\log n)^{d-1+\delta(d)}\quad,\quad 0<\delta(d)<1/2.
\ee
Theorem \ref{Td=2} and the above upper and lower bounds support the following conjecture.

\begin{conj}\label{small_ball_conj} Let $d\ge 3$, $2\le p\le \infty$, $r>1/2$. Then
$$
 \e_n(\bW^r_p,L_\infty) \asymp n^{-r}(\log n)^{(d-1)r+1/2}.
 $$
 \end{conj}
It is known that the corresponding lower bound in Conjecture \ref{small_ball_conj}
would follow from the $d$-dimensional version of the Small Ball Inequality for
the trigonometric system (\ref{2.6.4}).

The case $p=1$, $1\le q\le \infty$ was settled by Kashin and Temlyakov
\cite{KTE3}. The authors proved the following results.
\begin{thm}\label{KaTentr1} Let $d=2$, $1\le q<\infty$, $r>\max\{1/2,1-1/q\}$.
Then
\begin{equation}\label{36.3'} 
\e_n(\bW^r_{1},L_q)\asymp n^{-r}(\log n)^{r+1/2}
\end{equation}
\end{thm}
\begin{thm}\label{KaTentr2} Let $d=2$, $r>1$. Then 
\begin{equation}\label{36.4} 
\e_n(\bW^r_{1,0},L_\infty)\asymp n^{-r}(\log n)^{r+1} .
\end{equation}
\end{thm}
The most difficult part of Theorems \ref{KaTentr1} and \ref{KaTentr2} -- the lower bounds -- is proved with the help of the volume 
estimates of the appropriate sets of the Fourier coefficients of bounded trigonometric polynomials.  
 With the notation from Subsection
\ref{vol_est} the volume estimates of the sets $B_\L(L_p)$ and related
questions have been studied in a number of papers: the case $\L=[-n,n]$,
$p=\infty$ in \cite{KaE}; the case $\L=[-N_1,N_1]\times[-N_2,N_2]$, $p=\infty$
in \cite{TE2}, \cite{T6}; the case of arbitrary $\L$ and $p=1$ in \cite{KTE1}.
In particular, the results of \cite{KTE1} imply for $d=2$ and $1\le p<\infty$
that
$$
(vol(B_{\Delta Q_n}(L_p)))^{(2|\Delta Q_n|)^{-1}} \asymp |\Delta Q_n|^{-1/2}
\asymp (2^nn)^{-1/2}.
$$
It was proved in \cite{KTE3} that in the case $p=\infty$ the volume estimate is
different (see Theorem \ref{T2.5.5}):
\begin{equation}\label{36.9}
(vol(B_{\Delta Q_n}(L_\infty)))^{(2|\Delta Q_n|)^{-1}} \asymp   (2^nn^2)^{-1/2}.
\end{equation}
We note that in the case $\L=[-N_1,N_1]\times[-N_2,N_2]$ the volume estimate is
the same for all $1\le p\le \infty$. 
The volume estimate (\ref{36.9}) plays the key role in the proof of
(\ref{36.3'}) and (\ref{36.4}).

Let us make an observation on the base of the above discussion. In the
univariate case the entropy numbers $\e_n(W^r_{p},L_q)$ have the same order of
decay with respect to $n$ for all pairs $(p,q)$, $1\le p,q\le\infty$. In the
case $d=2$ we have three different orders of decay of $\e_n(\bW^r_{p},L_q)$
which depend on the pair $(p,q)$. For instance, in the case $1<p,q<\infty$ it is
$n^{-r}(\log n)^r$, in the case $p=1$, $1<q<\infty$, it is $n^{-r}(\log
n)^{r+1/2}$ and in the case $p=1$, $q=\infty$ it is $n^{-r}(\log n)^{r+1}$. 

We discussed above known results on the rate of decay of  $\e_n(\bW^r_{p},L_q)$.
In the case $d=2$ the picture is almost complete. In the case $d>2$ the
situation is fundamentally different. The problem of the right order of decay of
$\e_n(\bW^r_{p},L_q)$ is still open for $p=1$, $1\le q\le \infty$ and
$q=\infty$, $1\le p\le\infty$. In particular, it is open in the case $p=2$,
$q=\infty$, $r=1$ that is related to the small ball problem. We discussed in more
detail the case $q=\infty$, $1\le p\le \infty$. We pointed out above that in the
case $d=2$ the proof of lower bounds (the most difficult part) was based on the
Small Ball Inequalities for the Haar system for $r=1$ and for the trigonometric
system for all $r$. 

\subsection{Entropy numbers for $\bH$ and $\bB$ classes in $L_q$}\index{Entropy number}
Let us proceed with the entropy numbers of the $\bH$ classes in $L_q$.

\begin{thm}\label{Hup} For $r>1$ one has
$$
\e_n(\bH^r_{1},\bB^0_{\infty,2}) \lesssim n^{-r}(\log
n)^{(r+\frac{1}{2})(d-1)}.
$$
\end{thm}
Note that the space $\bB^0_{\infty,2}$ is ``close'' to $L_\infty(\T^d)$, however not comparable. In fact, we have
$$
   \bB^{0}_{\infty,1} \hookrightarrow L_{\infty}(\T^d) \hookrightarrow \bB^{0}_{\infty,\infty}\,. 
$$
\begin{thm}\label{Hlo} For $r>0$ one has
$$
\e_n(\bH^{r}_{\infty},L_1) \gtrsim n^{-r}(\log n)^{(r+\frac{1}{2})(d-1)}.
$$
\end{thm}

Theorems \ref{Hup} and $\ref{Hlo}$ were obtained in \cite{TE1} and \cite{TE2}.
Theorem \ref{Hlo} was proved earlier by N.S. Bakhvalov \cite{Bakh4} in the
following special cases: (I) $d=2$; (II) any $d$ but $L_1$ is replaced by $L_2$.
These theorems give the right order of $\e_n(\bH^r_{p},L_q)$ for all
$1\le p,q\le \infty$, except the case $q=\infty$. In this case we have the following result by 
Belinskii \cite{Bel}.

\begin{thm}\label{Bel_H} Let $1<p<\infty$ and $r>\max\{1/p,1/2\}$. Then 
$$
   \Big(\frac{(\log n)^{d-1}}{n}\Big)^r (\log n)^{(d-1)/2} \lesssim \epsilon_n(\bH^r_p,L_\infty) \lesssim \Big(\frac{(\log n)^{d-1}}{n}\Big)^r(\log n)^{d/2}\,.
$$
\end{thm}
 The following theorem shows that in the case $d=2$ the upper bounds in Theorem \ref{Bel_H} are sharp (in the sense of order). 
 
 \begin{thm}\label{THlo2} In the case $d=2$ for any $1\le p\le \infty$, $r>1/p$ we have
 $$
 \e_n(\bH^r_p,L_\infty) \asymp n^{-r}(\log n)^{r+1}.
 $$
 \end{thm}
One can prove the upper bounds in Theorem \ref{THlo2} using the discretization method based directly on discretization and on Theorem 6.3. The difficult part of Theorem \ref{THlo2} is the lower bound for $p=\infty$. This lower bound was proved in \cite{TE3}.

In the same way as we argued before formulating Conjecture \ref{small_ball_conj} we can argue in support of the following conjecture. 
 \begin{conj}\label{CHe} Let $d\ge 3$, $1\le p\le \infty$, $r>1/p$. Then
$$
 \e_n(\bH^r_p,L_\infty) \asymp n^{-r}(\log n)^{(d-1)(r+1/2)+1/2}.
 $$
 \end{conj}

We discussed above results on the right order of decay of the entropy numbers.
Clearly, each order relation $\asymp$ is a combination of the upper bound
$\lesssim$ and the matching lower bound $\gtrsim$. We now briefly discuss
methods that were used for proving upper and lower bounds. The upper bounds in
Theorems \ref{Hup} and \ref{thm:ent_pq} were proved by the standard method of
reduction
by discretization to estimates of the entropy numbers of finite-dimensional
sets.   Theorem \ref{Theorem 32.2} plays a key role in this method. It is clear
from the above discussion that it was sufficient to prove the lower bound in
(\ref{36.2}) in the case $q=1$. The proof of this lower bound (see Theorems
\ref{Hlo} and \ref{Wlo}) is more difficult and is based on nontrivial estimates
of the volumes of the sets of Fourier coefficients of bounded trigonometric
polynomials. Theorem \ref{T2.5.1} plays a key role in this method. 

We continue with the $\Brpt$ classes. A non-periodic version of the following result has been proved by 
Vyb\'iral \cite[Thms.\ 3.19, 4.11]{Vyb06}.

\begin{thm}\label{thm:ent_Brpt} Let $0<p,\theta \leq \infty$ and $1<q<\infty$.\\
{\em (i)} If $r>(1/p-1/q)_+$. Then 
$$    \e_n(\Brpt,L_q) \gtrsim n^{-r}(\log n)^{(d-1)(r+1/2-1/\theta)_+}\,.
$$
{\em (ii)} If $r>(1/p-1/q)_+ + 1/\min\{p,q,\theta\} - 1/\min\{p,q\} + 1/q-
1/\max\{q,2\}$ then
$$
    \e_n(\Brpt,L_q) \lesssim n^{-r}(\log n)^{(d-1)(r+1/2-1/\theta)}\,.
$$
{\em (iii)} If $r>\max\{1/p-1/2, 1/\theta-1/2\}$ then
$$
    \e_n(\Brpt,L_2) \asymp n^{-r}(\log n)^{(d-1)(r+1/2-1/\theta)}\,.
$$
\end{thm}
Note, that the third index $\theta$ influences the power of the logarithm. There is an interesting effect of small smoothness observed by Vyb\'iral in \cite[Thm.\ 4.11]{Vyb06}. The condition in Theorem 
\ref{thm:ent_Brpt}, (iii) seems to be crucial in the following sense. If $\theta < p$ and $1/p-1/2 < r \leq 1/\theta-1/2$ then for any $\varepsilon>0$ there is 
$c_{\varepsilon}>0$ such that
\begin{equation}\label{entr_ss}
   \e_n(\Brpt,L_2) \leq c_{\varepsilon}n^{-r}(\log n)^{\varepsilon}\,.
\end{equation}
The influence of the small mixed smoothness $r$ seems to be dominated by the influence of the fine parameter $\theta$. 
The correct order is not known, see also Open problem 6.4 below. 

Note, that if $r>1/\theta-1/2$ the following stronger lower bound
$$
\epsilon_n(\bB^r_{\infty,\theta},L_1)  \gtrsim n^{-r}(\log
n)^{(d-1)(r+1/2-1/\theta)}
$$
follows directly from the proof of Theorem \ref{Hlo} in \cite{TE2} and the
inequality: for $f\in \Tr(\Delta Q_l)$
$$
\|f\|_{\bB^r_{p,\theta}} \lesssim l^{(d-1)/\theta}\|f\|_{\bH^r_p}.
$$
In the case $p\ge2$ and $\theta\ge 2$ the corresponding upper bound for the  $ \e_n(\Brpt,L_q)$ can be derived from
Theorem \ref{thm:ent_pq} with the help of inequalities
 $$
 \|\sum_{|\bs|_1=l}\delta_\bs(f)\|_p \lesssim \left(\sum_{|\bs|_1=l}\|\delta_\bs(f)\|_p^2\right)^{1/2} \lesssim
\left(\sum_{|\bs|_1=l}\|\delta_\bs(f)\|_p^\theta\right)^{1/\theta}l^{(d-1)(1/2-1/\theta)}.
 $$

Theorem \ref{thm:ent_Brpt} has
been proved in \cite{Vyb06} in a more general and non-periodic situation. The
author used discretization techniques (wavelet isomorphisms), entropy results
for finite dimensional spaces, and complex interpolation to obtain the above
result. 

If either $r>1/p$ and $\theta \geq p$ or $r>(1/p-1/q)_+$ and $\theta \geq
\min\{q,2\}$ the result in Theorem \ref{thm:ent_Brpt} has been shown by Dinh
D\~ung \cite{Di01}. One should keep in mind that under these parameter
assumptions it is possible to obtain that $r+1/2-1/\theta<0$. Of course, the
exponent of the logarithm can not be negative. Inspecting the proof of
\cite[Thm.\ 8]{Di01} it turns out that Dinh D\~ung proved a slightly weaker
upper bound in the situation $r>(1/p-1/q)_+$ and $\theta\geq \min\{q,2\}$,
namely
$$
    \e_n(\Brpt,L_q) \lesssim n^{-r}(\log n)^{(d-1)(r+1/\min\{q,2\}-1/\theta)}\,.
$$
 Very recently, in the paper \cite{VT156} the approximation classes $\bW^{a,b}_p$ have been studied from the point of view of the entropy numbers. 
We do not give a detailed definition of these classes here, we only note that these classes are defined in a way similar to the classes 
$\bW^{a,b}_A$ (see Lemma \ref{L2.1} below) with $\|\cdot\|_A$ replaced by $\|\cdot\|_p$. The paper \cite{VT156} develops a new method of proving the upper bounds for the entropy numbers. This method, which is based on the general Theorem 6.9, allows us to prove all known upper bounds for classes $\bW^r_p$.  

\subsection{Entropy numbers and the Small Ball Problem}\index{Entropy number}\index{Small Ball Problem}
\label{sect:ESBP}

We already pointed out that the case $r=1$, $p=2$, $q=\infty$ is equivalent to the Small Ball Problem from
probability theory. 

We discuss related results in detail. Consider the centered
Gaussian process ${\bf B}_d := (B_\bx)_{\bx\in[0,1]^d}$ with covariance
$$
 {\mathbb{E}}(B_\bx B_\by) =\prod_{i=1}^d \min\{x_i,y_i\},\quad \bx=(x_1,\dots,x_d),\quad
\by=(y_1,\dots,y_d).
$$
This process is called Brownian sheet. It is known that the sample paths of
${\bf B}_d$ are almost surely continuous. We consider them as random elements of
the space $C([0,1]^d)$. The small ball problem is the problem of the asymptotic
behavior of the small ball probabilities
$$
 {\mathbb{P}}(\sup_{\bx\in[0,1]^d}|B_\bx| \le \e)
$$
as $\e$ tends to zero. We introduce notation
$$
\phi(\e) := -\ln  {\mathbb{P}}(\sup_{\bx\in[0,1]^d}|B_\bx| \le \e).
$$
The following relation is a fundamental result of probability theory: for $d=2$
and $\e<1/2$
\begin{equation}\label{36.5}
\phi(\e) \asymp \e^{-2}(\ln (1/\e))^3.
\end{equation}
The upper bound in (\ref{36.5}) was obtained by  Lifshits and Tsirelson
\cite{LiTs} and by Bass \cite{Bass}. The lower bound in (\ref{36.5}) was
obtained by Talagrand \cite{TaE}. 

Kuelbs and Li \cite{KL} discovered the fact that there is a tight relationship
between $\phi(\e)$ and the entropy $H_\e(\bW^1_{2},L_\infty)$. We note that they
considered the general setting of a Gaussian measure on a Banach space. We only
formulate a particular result of our interest in terms convenient for us.  Roughly speaking (i.e. under the additional 
condition $\phi(\epsilon) \sim \phi(\epsilon/2)$), they
proved the equivalence relations: for any $d$
$$
\phi(\e)\lesssim \e^{-2}(\ln(1/\e))^\beta \quad \iff \quad
\e_n(\bW^1_{2},L_\infty)\lesssim n^{-1}(\ln n)^{\beta/2};
$$
$$
\phi(\e)\gtrsim \e^{-2}(\ln(1/\e))^\beta \quad \iff \quad
\e_n(\bW^1_{2},L_\infty)\gtrsim n^{-1}(\ln n)^{\beta/2}.
$$
These relations and (\ref{36.5}) imply for $d=2$
\begin{equation}\label{36.6}
\e_n(\bW^1_{2} {(\T^2)},L_\infty {(\T^2)}) \asymp n^{-1}(\ln n)^{3/2}.
\end{equation}

The proof of the most difficult part of (\ref{36.5}) -- the lower bound -- is based
on a special inequality for the Haar polynomials proved by Talagrand \cite{TaE}
(see \cite{TE5} for a simple proof). 
The reader can find this inequality in Subsection 2.6 of Section \ref{trigpol}.
We note that the lower bound in (\ref{36.6}) can be deduced directly from the
corresponding Small Ball Inequality (\ref{2.6.2}). However, this way does not
work for deducing the lower bound in (\ref{36.3}) for general $r$. This
difficulty was overcome in \cite{TE3} by proving the Small Ball Inequality for
the trigonometric system (see Subsection 2.6 of Section \ref{trigpol},
inequality (\ref{2.6.3})).  

 The other way around, results on entropy numbers can be used to estimate the small ball probabilities. From the relation in \eqref{36.4'} 
one obtains the bounds 
\begin{equation}\label{sbp}
     \epsilon^{-2}(\ln(1/\epsilon))^{2d-2} \lesssim \phi(\epsilon) \lesssim \epsilon^{-2}(\ln(1/\epsilon))^{2d-1}\,,
\end{equation}
see also \cite{LiLin99}. Note, that there is so far no ``pure'' probabilistic proof for the upper bound. By the recent improvements on the 
Small Ball Inequality (see the remark after \eqref{2.6.5} in Subsection \ref{subsect:SBI} above) one can slightly improve the lower bound in \eqref{sbp}
to get 
\begin{equation}\nonumber
     \epsilon^{-2}(\ln(1/\epsilon))^{2d-2+\delta} \lesssim \phi(\epsilon) \lesssim \epsilon^{-2}(\ln(1/\epsilon))^{2d-1}\,,
\end{equation}
for some $0<\delta = \delta(d) < 1$\,.

\subsection{Open problems}\index{Open problems!Entropy numbers}

We presented historical comments and a discussion in the above text of Section 6. We formulate here the most important
open problems on the entropy numbers of classes of functions with mixed smoothness. 
\newline
{\bf Open problem 6.1.} Find the order of
$\epsilon_n(\bW^r_{\infty},L_\infty)$ in the case $d=2$ and $0<r\le 1/2$.
\newline
{\bf Open problem 6.2.} For $d>2$ find the order of
$\epsilon_n(\bW^r_{1},L_q)$, $1\le q\le \infty$.
\newline
{\bf Open problem 6.3.} For $d>2$ find the order of
$\epsilon_n(\bW^r_{p},L_\infty)$  and $\epsilon_n(\bH^r_{p},L_\infty)$, $1\le p\le \infty$, $r>1/p$.
\newline
{\bf Open problem 6.4.} Find the correct order of the entropy numbers $\e_n(\bB^r_{p,\theta},L_2)$ in case of small smoothness $0<\theta < p\leq 2$ and 
$1/p-1/2<r\leq 1/\theta - 1/2$\,. \index{Small smoothness}

%% file: nonlinear_approx.tex
\section{Best $m$-term approximation}\index{Best $m$-term approximation}\index{Nonlinear approximation}
\label{Sect:bestmterm}

\subsection{Introduction}

The last two decades have seen great successes in studying nonlinear  
approximation
which was motivated by numerous applications.  
Nonlinear
approximation is important in applications because of its concise
representations  and increased computational efficiency.
Two types of  nonlinear approximation are frequently employed in  
applications.
Adaptive methods are used in PDE solvers, while $m$-term approximation,
considered here, is used in image/signal/data processing, as well as in  
the design
of neural networks. Another name for $m$-term approximation is {\it sparse
approximation}.

The fundamental question of nonlinear approximation  
is how to devise good constructive methods (algorithms) of nonlinear  
approximation. This problem has two levels of nonlinearity. The first level of  
nonlinearity is $m$-term approximation with regard to bases. In this problem one
can use the unique function expansion with regard to a given basis to build an
approximant. Nonlinearity enters by looking for $m$-term approximants with
terms  (i.e. basis elements in approximant) allowed to depend on a given 
function.  Since the elements of the basis used in the $m$-term approximation
are allowed to depend on the function being approximated, this type of
approximation is very efficient.  On the second level of  nonlinearity, we
replace a basis by a more general system which is not necessarily minimal (for
example, redundant system, dictionary). This setting is much  more complicated
than the first one (bases case), however, there is a  solid justification of
importance of redundant systems in both theoretical questions and in practical  
applications. We only give here a brief introduction to this important area of
research
and refer the reader to the book \cite{Tbook} for further results. 

One of the major questions in approximation (theoretical and numerical) is:
what is an optimal method? We discuss here this question in a theoretical
setting with the only criterion of quality of approximating method its accuracy.
One more important point in the setting of optimization problem is to specify
a set of methods over which we are going to optimize. Most of the problems
which approximation theory deals with are of this nature. Let us give some
examples from classical approximation theory. These examples will help us to
motivate the question we are studying in this section.

{\bf Example 1.} When we are searching for $n$-th best trigonometric
approximation of a given function we are optimizing in the sense of accuracy
over the subspace of trigonometric polynomials of degree $n$.
 
 {\bf Example 2.} When we are solving the problem on Kolmogorov's $n$-width for
a given function class we are optimizing in the sense of accuracy for a given
class over all subspaces of dimension $n$.
 
 {\bf Example 3.} When we are finding best $m$-term approximation of a given
function with regard to a given system of functions (dictionary) we are
optimizing over all $m$-dimensional subspaces spanned by elements from a given
dictionary.
 
Example 2 is a development of Example 1 in the sense that in Example 2 we are
looking for an optimal $n$-dimensional subspace instead of being confined to a
given one (trigonometric polynomials of degree $n$). Example 3 is a nonlinear
analog of Example 1, where instead of trigonometric system we take a dictionary
$\Di$ and allow approximating elements from $\Di$ to depend on a function. In
\cite{T69} we made some steps in a direction of developing Example 3 to a
setting which  is a nonlinear analog of Example 2. In other words, we want to
optimize over some sets of dictionaries. We discuss   two classical structural
properties of dictionaries:

1. Orthogonality;

2. Tensor product structure (multivariate case).

Denote by $\Di$ a dictionary in a Banach space $X$ and by
$$
\sigma_m(f,\Di)_X := \inf\limits_{\substack{g_i\in \Di,c_i\\i=1,\dots,m}} \|f -\sum_{i=1}^m
c_ig_i\|_X
$$
best $m$-term approximation of $f$ with regard to $\Di$. For a function class
$F\subset X$ and a collection $\di$ of dictionaries we consider
$$
\sigma_m(F,\Di)_X := \sup_{f\in F}\sigma_m(f,\Di)_X \quad ,
$$
$$
\sigma_m(F,\di)_X := \inf_{\Di \in \di}\sigma_m(F,\Di)_X \quad.
$$
Thus the quantity $\sigma_m(F,\di)_X$ gives the sharp lower bound for best
$m$-term approximation of a given function class $F$ with regard to any
dictionary  $\Di \in \di$.

Denote by $\og$ the set of all orthonormal dictionaries defined on a given
domain.  Kashin \cite{K} proved that for the class $H^{r,\alpha}$, $r=0,1,\dots,
\quad \alpha \in [0,1]$, of univariate functions such that
$$
\|f\|_\infty +\|f^{(r)}\|_\infty \le 1 \quad \text{and}\quad 
|f^{(r)}(x)-f^{(r)}(y)| \le |x-y|^\alpha,\quad x,y \in [0,1]
$$
we have
\be\label{tag1.1}
\sigma_m(H^{r,\alpha},\og)_{L_2} \ge C(r,\alpha)m^{-r-\alpha} .
\ee
It is well known that in the case $\alpha \in (0,1)$ the class $H^{r,\alpha}$ is equivalent to the class
$H^{r+\alpha}_\infty$. In the case $\alpha =1$ the class $H^{r,1}$ is close to the class $W^{r+1}_\infty$. 
It is interesting to remark that we cannot prove anything like (\ref{tag1.1})
with $L_2$
replaced by $L_p$, $p<2$. We proved (see \cite{KTE1}) that there exists $\Phi
\in \og$
such that for any $f\in L_1(0,1)$ we have $\sigma_1(f,\Phi)_{L_1}=0$.
It is pointed out in \cite{T69} that the proof from \cite{KTE1} also works for
$L_p$, $p<2$, instead of $L_1$: 
\begin{rem}\label{R1.1} For any $1\le p <2$ there exists a complete in
$L_2(0,1)$ orthonormal system $\Phi$ such that for each $f \in L_p(0,1)$ we have
$\sigma_1(f,\Phi)_{L_p}=0$.
\end{rem}
This remark means that to obtain nontrivial lower bounds for
$\sigma_m(f,\Phi)_{L_p}$ , $p<2$,  we need to impose additional restrictions on
$\Phi \in \og$. Some ways of imposing restrictions were discussed in \cite{KTE1}
and \cite{T69}.

\subsection{Orthogonal bases}\index{Best $m$-term approximation!Orthogonal bases}

We discuss  approximation of multivariate functions. It is convenient for us to
present results in the periodic case. We consider classes of functions with
bounded mixed derivative $\bW^r_p$   and classes with restriction of Lipschitz
type on mixed difference  $\bH^r_p$ and $\bB^r_{p,\theta}$. These classes are
well known (see for instance \cite{TBook}) for their importance in numerical
integration, in finding universal methods for approximation of functions of
several variables, in the average case setting of approximation problems for the
spaces equipped with the Wiener sheet measure (see \cite{W}) and in other
problems. It is proved in \cite{T69} that
\be\label{tag1.2}
\sigma_m(\bH^r_p,\og)_2 \gtrsim m^{-r}(\log m)^{(d-1)(r+1/2)}, \quad 1\le p
< \infty, \ee
\be\label{tag1.3}
\sigma_m(\bW^r_{p},\og)_2 \gtrsim m^{-r}(\log m)^{(d-1)r}, \quad 1\le p <
\infty. 
\ee

It is also proved in \cite{T69} that the orthogonal basis $\cU^d$ which we
construct below provides optimal upper estimates (like (\ref{tag1.2}) and
(\ref{tag1.3})) in best $m$-term approximation of the classes $\bH^r_p$ and
$\bW^r_{p}$ in the $L_q$-norm, $2\le q < \infty$. Moreover, we proved there that
for all $1<p,q<\infty$ the order of best $m$-term approximation
$\sigma_m(\bH^r_p,\cU^d)_{L_q}$ and $\sigma_m(\bW^r_{p},\mathcal U^d)_{L_q}$
can be achieved by a greedy type algorithm $G^q(\cdot,\cU^d)$.
Assume a given system $\Psi$ of functions $\psi_I$ indexed by dyadic intervals
can be enumerated in such a way that $\{\psi_{I^j}\}_{j=1}^\infty$ is a basis
for $L_q$. Then we define the Thresholding Greedy Algorithm (TGA)
 $G^q(\cdot,\Psi)$ as follows. Let
$$
f = \sum_{j=1}^\infty c_{I^j}(f,\Psi)\psi_{I^j} 
$$
and
$$
c_I(f,q,\Psi) := \|c_I(f,\Psi)\psi_I\|_q .
$$
Then $c_I(f,q,\Psi) \to 0$ as $|I| \to 0$. Denote $\L_m$ a set of $m$ dyadic
intervals
 $I$ such that
\be\nonumber
\min_{I\in \L_m} c_I(f,q,\Psi) \ge \max_{J \notin \L_m}c_J(f,q,\Psi) . 
\ee
We define $G^q(\cdot,\Psi)$ by formula
$$
G^q_m(f,\Psi) := \sum_{I\in \L_m} c_I(f,\Psi) \psi_I .
$$
 
The question of constructing a procedure (theoretical algorithm) which realizes
(in the sense of order) the best possible accuracy is a very important one and
we discuss it in detail in this section. Let $A_m(\cdot,\Di)$ be a mapping which
maps each $f\in X$ to a linear combination of $m$ elements from a given
dictionary $\Di$. Then the best we can hope for with this mapping is to have for
each $f \in X$
\be\label{tag1.5}
\|f -A_m(f,\Di)\|_X = \sigma_m(f,\Di)_X 
\ee
or a little weaker
\be\label{tag1.6}
\|f -A_m(f,\Di)\|_X \le C(\Di,X) \sigma_m(f,\Di)_X .  
\ee
There are some known trivial and nontrivial examples when (\ref{tag1.5}) holds
in a Hilbert space $X$. We do not touch this kind of relations here. Concerning
(\ref{tag1.6}) it is proved in \cite{T3} that for any basis $\Psi$ which is
$L_p$-equivalent to the univariate Haar basis we have
\be\label{tag1.7}
\|f -G^q_m(f,\Psi)\|_q\le C(p)\sigma_m(f,\Psi)_q, \quad 1<q<\infty. 
\ee
However, as it is shown in \cite{T4} and \cite{T69}, the inequality
(\ref{tag1.7}) does not hold for particular dictionaries with tensor product
structure.

We define the system $\mathcal U:=\{U_I\}$ in the univariate case. Denote
\begin{equation}\nonumber
\begin{split}
U^+_n(x) &:= \sum_{k=0}^{2^n-1}e^{ikx} = \frac{e^{i2^nx}-1}{e^{ix}-1},\quad
n=0,1,2,\dots;\\
U^+_{n,k}(x) &:= e^{i2^nx}U^+_n(x-2\pi k2^{-n}),\quad k=0,1,\dots ,2^n-1;\\
 U^-_{n,k}(x) &:= e^{-i2^nx}U^+_n(-x+2\pi k2^{-n}),\quad k=0,1,\dots ,2^n-1.
\end{split}
\end{equation}
It will be more convenient for us to normalize in $L_2$ the system of functions
$\{U^+_{m,k},U^-_{n,k}\}$ and  enumerate it by dyadic intervals. We write
\begin{equation}\nonumber
\begin{split}
 U_I(x) &:= 2^{-n/2}U^+_{n,k}(x)\quad \text{with}\quad
 I=[ (k+1/2)2^{-n}, (k+1)2^{-n});\\
U_I(x) &:= 2^{-n/2}U^-_{n,k}(x)\quad \text{with}\quad
 I=[ k2^{-n},(k+1/2)2^{-n});
\end{split}
\end{equation}
and
$$
U_{[0,1)}(x) :=1.
$$
Denote
\begin{equation}\nonumber
\begin{split}
  D^+_n &:= \{I: I=[(k+1/2)2^{-n},(k+1)2^{-n}),\quad k=0,1,\dots,2^n-1\};\\
   D^-_n&:= \{ I:I=[k2^{-n},(k+1/2)2^{-n}),\quad k=0,1,\dots,2^n-1\};\\
\end{split}
\end{equation}
and 
$$
D:= \cup_{n\ge 0}(D^+_n\cup D^-_n)\cup D_0 
$$
with $D_0 := [0,1)$. In the multivariate case of $\bx=(x_1,\dots,x_d)$ we define the system $\mathcal
U^d$
as the tensor product of the univariate systems $\mathcal U$. Let
$I=I_1\times\dots\times I_d$, $I_j \in D$, $j=1,\dots,d$, then 
$$
U_I(\bx) := \prod_{j=1}^d U_{I_j}(x_j) .
$$
We have for instance (see \cite{T69})
\be\label{tag1.8}
\sup_{f\in L_q}\|f-G^q_m(f,\cU^d)\|_{L_q}/\sigma_m(f,\cU^d)_{L_q} \gtrsim (\log
m)^{(d-1)|1/2-1/q|} . 
\ee
The inequality (\ref{tag1.8}) shows that using the algorithm $G^q(\cdot,\cU^d)$
we lose for sure for some functions $f\in L_q$, $q\neq 2$. In light of
(\ref{tag1.8}) the results of \cite{T69} look encouraging for using
$G^q(\cdot,\cU^d)$.
\begin{thm}\label{uw} Define
$$
    r(W,p,q)= \left\{\begin{array}{rcl}
                          \max\{1/p,1/2\}-1/q&:&q\geq 2\,,\\
                          (\max\{2/q,2/p\}-1)/q&:&p<2\,\,.	
                      \end{array}\right..
$$
Then for $1<p,q <\infty$ and $r>r(W,p,q)$ we have
\be\label{tag1.10}
\sup_{f\in \bW^r_{p}}\|f-G_m^q(f,\cU^d)\|_q\asymp \sigma_m(\bW^r_{p},\cU^d)_q
\asymp m^{-r}(\log m)^{(d-1)r} . 
\ee
\end{thm}

\begin{thm}\label{uh} Define
$$
    r(H,p,q) = \left\{\begin{array}{rcl}
                          (1/p-1/q)_+&:&q\geq 2\,,\\
                          (\max\{2/q,2/p\}-1)/q&:&q<2\,\,.	
                      \end{array}\right..
$$
Then for $1<p,q <\infty$ and $r>r(H,p,q)$ we have
\be\label{tag1.9}
\sup_{f\in \bH^r_p}\|f-G_m^q(f,\cU^d)\|_q\asymp \sigma_m(\bH^r_p,\cU^d)_q \asymp
m^{-r}(\log m)^{(d-1)(r+1/2)} . 
\ee
\end{thm}

Comparing (\ref{tag1.9}) with (\ref{tag1.2}) and (\ref{tag1.10}) with
(\ref{tag1.3}), we conclude that the dictionary $\cU^d$ is the best (in the
sense of order) among all orthogonal dictionaries for $m$-term approximation of
the classes $\bH^r_p$ and $\bW^r_{p}$ in $L_q$ where $1<p<\infty$ and $2\le q
<\infty$. The dictionary $\cU^d$ has one more important feature. The near best
$m$-term approximation of functions from $\bH^r_p$ and $\bW^r_{p}$ in the
$L_q$-norm can be realized by the simple greedy type algorithm
$G^q(\cdot,\cU^d)$ for all $1<p,q<\infty$.  For further results in this direction for the 
system $\mathcal{U}^d$ we refer to \cite{Bal15}, \cite{BaSm15}.

It is known that the system $\cU^d$ and its analog built on the base of the de
la Valle{\'e} Poussin kernels instead of the Dirichlet kernels play important
role in the bilinear approximation \cite{Tem15} (for bilinear approximation see
subsection 7.4 below). The de la Valle{\'e} Poussin kernels are especially
important when we deal with either $L_1$ or $L_\infty$ spaces. 
This setting has been considered by Dinh D\~ung in \cite{Di00,
Di01_2}. Instead of the system $\mathcal{U}^d$ he considered translates of the
de la Valle\'e Poussin kernels $\mathcal{V}^d$. In contrast to $\mathcal{U}^d$
this one is linearly dependent (redundant). However, it admits similar
discretization
techniques as orthonormal wavelet bases. In addition, the spaces studied
slightly differ from the ones considered by Temlyakov.  
\begin{thm}\label{thm:mterm_Dung} Let $1<p,q<\infty$, $0<\theta\leq \infty$ and
$r>0$. \\
{\em (i)} It holds
$$
      \sigma_m(\bB^r_{p,\theta},\mathcal{V}^d)_q\gtrsim m^{-r}(\log
m)^{(d-1)(r+1/2-1/\theta)}\,.
$$
(ii) If $r>(1/p-1/q)_+$ and $\theta \geq \min\{q,2\}$. Then 
$$
    \sigma_m(\bB^r_{p,\theta},\mathcal{V}^d)_q \lesssim m^{-r}(\log
m)^{(d-1)(r+1/\min\{q,2\}-1/\theta)}\,.
$$
(iii) Let $r>\max\{0,1/p-1/q,1/p-1/2\}$ and $2\leq \theta \leq \infty$. Then it
holds
$$
    \sigma_m(\bB^r_{p,\theta},\mathcal{V}^d)_q \asymp m^{-r}(\log
m)^{(d-1)(r+1/2-1/\theta)}\,.
$$
(iv) Let $r>1/p$. Then 
$$
    \sigma_m(\bB^r_{p,\theta},\mathcal{V}^d)_q \lesssim m^{-r}(\log
m)^{(d-1)(r+1/2-1/\max\{p,\theta\})}\,.
$$

\end{thm}

Dinh D\~ung essentially studied embeddings between Besov spaces and $L_q$. As a
consequence of well-known embeddings he obtained also the following result for
the $\bW$ spaces which has to be compared with Theorem \ref{uw}
above.

\begin{thm}\label{thm:mterm_Dung2} Let $1<p,q<\infty$ and
$r>\max\{0,1/p-1/q,1/p-1/2,1/2-1/q\}$. Then 
$$
    \sigma_m(\Wrp,\mathcal{V}^d)_q \asymp m^{-r}(\log m)^{(d-1)r}\,.
$$
\end{thm}

Let now $\Phi = \{\psi_I\}_I$ be a tensorized orthonormal wavelet basis (indexed
by a dyadic parallelepiped $I$) with sufficient smoothness, decay
and vanishing moments.
In the non-periodic setting the estimates for $\sigma_m$ in Theorems \ref{uw}, \ref{uh},
\ref{thm:mterm_Dung} can be also observed in a non-periodic setting, see
\cite{T69,HaSi10,HaSi11, HaSi12} using the wavelet basis $\Phi$ as a dictionary.
Here
we get the sharp bounds.

\begin{thm}\label{T7.6} Let $1<p,q<\infty$ and $0<\theta \leq \infty$ \\
{\em (i)} Let $r>\max\{1/p-1/q, 1/\min\{p,\theta\}-1/\max\{q,2\},0\}$ then 
$$
   \sigma_m(\Brpt,\Phi)_q \asymp m^{-r}(\log m)^{(d-1)(r+1/2-1/\theta)_+}\,.
$$
{\em (ii)} Let $r>(1/p-1/q)_+$ then
$$
    \sigma_m(\Wrp,\Phi)_q \asymp m^{-r}(\log m)^{(d-1)r}\,.
$$
\end{thm}

Note, that contrary to Theorems \ref{uw} and \ref{uh} Theorem \ref{T7.6} does not address the greedy approximation. As a result,
the assumptions in  Theorem \ref{T7.6} are (so far) the weakest compared to
Theorems \ref{uw}, \ref{thm:mterm_Dung}, \ref{thm:mterm_Dung2}. 

Let us also address a small smoothness effect in this framework 
(observed by Hansen, Sickel \cite[Cor.\ 5.11]{HaSi12}) according to the parameter situation 
already mentioned after Theorem \ref{thm:ent_Brpt} in the Entropy section (see also Open problem 6.4).
In fact, in case $0<\theta<p\leq 2$ and $1/p-1/2 < r \leq 1/\theta-1/2$ we have 
\begin{equation}\label{nolog}
    \sigma_m(\bB^r_{p,\theta}, \Phi)_2 \asymp m^{-r}\,.
\end{equation}
The bound is a simple consequence of recent results \cite{HaSi10,HaSi12} on best $m$-term approximation for non-compact embeddings $\bW^r_p \hookrightarrow L_q$ if $1<p<q<\infty$ and $r=1/p-1/q$, and
$\bB^r_{p,p} \hookrightarrow L_q$ if $0<p\leq \max\{p,1\}<q<\infty$ and $r=1/p-1/q$, see Theorems \ref{noncomp1},
\ref{noncomp2} below. Clearly, in the above parameter situation we may use the trivial embedding $\bB^r_{p,\theta} \hookrightarrow \bB^r_{p^{\ast},p^{\ast}}$, where $p^{\ast}$ is chosen such that 
$r = 1/p^{\ast}-1/2$ and hence we may use the non-compact embedding result \eqref{nonc} below. Note, that in \eqref{nolog} the typical logarithmic factor with the $d$ power is not present. 
From the viewpoint of high-dimensional approximation this is an important phenomenon. Let us also refer to the discussions in Section \ref{Sect:highdim} and Subsection \ref{sparsetrig} below (after Theorem \ref{T2.8I}). 

In the recent paper \cite{ByUl16} best $m$-term approximation with respect to the tensor product Faber-Schauder system $\mathcal{F}^d$, 
see Subsection \ref{timelimsamprep}, has been addressed. Via a greedy type algorithm one can prove for $1<p<q \leq \infty$, $\max\{1/p,1/2\}<r<2$ that
$$
    \sigma_m(\bW^r_p,\mathcal{F}^d)_q \lesssim m^{-r} (\log m)^{(d-1)(r+1/2)}\,.
$$
Note, that every coefficient in the Faber-Schauder expansion is given by a linear combination of (discrete) function values \eqref{f_100}. Hence, the best $m$-term approximant is build from 
$c(d)m$ properly chosen function values and corresponding hat functions.

\subsection{Some related problems}

 Approximations with respect to the basis $\cU^d$ and with respect to the
trigonometric system
are useful in the following two fundamental problems. 

{\bf Bilinear approximation.}\index{Best $m$-term approximation!Bilinear approximation} For a function $f(\bx,\by)\in L_{p_1,p_2}$,
$\bx=(x_1,\dots,x_a)$, $\by=(y_1,\dots,y_a)$ we define the best bilinear
approximation as follows
$$
\tau_m(f)_{p_1,p_2}
:=\inf_{\substack{u_i(\bx),v_i(\by)\\i=1,\dots,m}}\Big\|f(\bx,\by)-\sum_{i=1}^m
u_i(\bx)v_i(\by)\Big\|_{p_1,p_2}.
$$
where $\|\cdot\|_{p_1,p_2}$ denotes the mixed norm: $L_{p_1}$ in $\bx$ on
$\T^a$ and in $L_{p_2}$ in $\by$ on $\T^a$. 

The classical result, obtained by E. Schmidt \cite{S}, gives the following
theorem.
\begin{thm}\label{Schmidt} Let $K(\bx,\by) \in L_{2,2}$. Then
$$
\tau_m(K)_{2,2} = \Big(\sum_{j=m+1}^\infty \lambda_j(K)\Big)^{1/2},
$$
where $\{\lambda_j(K)\}_{j=1}^\infty$ is a non-increasing sequence of
eigenvalues of the operator $K^*K$ with $K$ being an integral operator with the
kernel $K(\bx,\by)$. 
\end{thm}

It is clear that for a periodic function $f\in L_q(\T^a)$,  one has
$$
 \tau_m(f(\bx-\by))_{q,\infty} \le \sigma_m(f,\Tr^a)_q.
$$
This observation and some known results on the $m$-term approximations  
allowed us to prove sharp upper bounds for the best bilinear approximation (see
\cite{T29} and \cite{Tmon}). We note here that it easily follows from the
definition of classes $\bW^r_p$ that 
$$
\lambda_m(\bW^r_1,L_q) \le \tau_m(F_r(\bx-\by))_{q,\infty}\le
\sigma_m(F_r,\Tr)_q.
$$
In the case $d=2$ ($a=1$) it is the classical problem of bilinear approximation.
It turned out that the best $m$-term approximation with respect to the
trigonometric system 
gives the best bilinear approximation in the sense of order for some classes of
functions. In the case of approximation in the
$L_2$-space the bilinear approximation problem is closely related to the problem
of singular value decomposition (also called Schmidt expansion) of the
corresponding integral operator with the kernel $f(x_1,x_2)$ (see Theorem
\ref{Schmidt} above). There are known results on the rate of decay of errors of
best bilinear approximation in $L_{p_1,p_2}$ under different smoothness
assumptions on $f$. We only mention some known results for classes of functions
which are studied in this paper. The problem of estimating $\tau_m(f)_2$ in case
$d=2$ (best $m$-term bilinear approximation $\tau_m(f)$ in $L_2$) is a classical
one and was considered for the first time by E. Schmidt \cite{S} in 1907. For
many function classes $\bF$ an asymptotic behavior of
$\tau_m(\bF)_p:=\sup_{f\in \bF}\tau_m(f)_{p,p}$ is known. For instance, the
relation
\begin{equation}\nonumber
\tau_m(\bW^r_p)_q \asymp m^{-2r + (1/p-\max\{1/2,1/q\})_+}
\end{equation}
for $r>1$ and $1\le p\le q \le \infty$ follows from more general results in
\cite{T32}.

{\bf Tensor product approximation.}\index{Best $m$-term approximation!Tensor product approximation}\index{Tensor product}
For a function $f(x_1,\dots,x_d)$ denote
$$
\Th_m(f)_X:=\inf\limits_{\substack{\{u^i_j\}\\  j=1,\dots,m\\ i=1,\dots,d}}\Big\|f(x_1,\dots,x_d) -
\sum_{j=1}^m\prod_{i=1}^d u^i_j(x_i)\Big\|_X.
$$

In the case $d>2$ almost nothing is known for the tensor product approximation.
There is (see \cite{T35}) an upper estimate in the case $q=p=2$
\begin{equation}\nonumber
\Th_m(\bW^r_2)_2 := \sup\limits_{f\in \bW^r_2} \Th_m(f)_{L_2} \lesssim
m^{-rd/(d-1)}\,. 
\end{equation}
For recent results in this direction see \cite{BT} and \cite{ScUs14}.

\subsection{A comment on Stechkin's lemma}\index{Stechkin's lemma}

The following simple lemma (see, for instance, \cite[p. 92]{Tmon}) turns out to be very useful in nonlinear approximation. Note, that here the case $p<1$ is particularly 
important and leads to better convergence rates. 
\begin{lem}\label{Lqp} Let $a_1\ge a_2\ge \cdots \ge a_M\ge 0$ and $0 < p\le q\le \infty$.
Then for all $m<M$ one has
$$
\left(\sum_{k=m}^M a_k^q\right)^{1/q} \le m^{-\beta} \left(\sum_{k=1}^M a_k^p\right)^{1/p},\quad \beta:=1/p-1/q.
$$
\end{lem}

\begin{proof} Denote
$$
A:=\left(\sum_{k=1}^M a_k^p\right)^{1/p}.
$$
Monotonicity of $\{a_k\}$ implies
$$
ma_m^p \le A^p\quad \text{and}\quad a_m \le Am^{-1/p}.
$$
Then
$$
\sum_{k=m}^M a_k^q \le a_m^{q-p}\sum_{k=m}^M a_k^p \le a_m^{q-p}A^p.
$$
The above two inequalities imply Lemma \ref{Lqp}.
\end{proof}

For the first time a version of Lemma \ref{Lqp} for $1 \leq p \leq 2, q = 2$, was proved and used in
nonlinear approximation in \cite{T29}. In the case $1 \leq p \leq q \leq \infty$ Lemma \ref{Lqp} is proved in \cite{Tmon}. The
same proof gives Lemma \ref{Lqp} (see above) for $0 < p \leq q \leq \infty$.

It has been observed recently that the constant in Lemma \ref{Lqp} is actually better than stated ($c_{p,q}=1$ in Lemma \ref{Lqp}). 
Using tools from convex optimization one can prove the following stronger version of Lemma \ref{Lqp}, see \cite{Fo12}, \cite[Thm.\ 2.5]{FoRa13}.

\begin{lem} \label{Lqp2} Let $a_1\ge a_2\ge \cdots \ge a_M\ge 0$ and $0 < p < q\le \infty$.
Then for all $m<M$ one has
$$
\left(\sum_{k=m+1}^M a_k^q\right)^{1/q} \le c_{p,q}m^{-\beta} \left(\sum_{k=1}^M a_k^p\right)^{1/p},\quad \beta:=1/p-1/q\,,
$$
where 
$$
    c_{p,q}:=\Big[\Big(\frac{p}{q}\Big)^{p/q}\Big(1-\frac{p}{q}\Big)^{1-p/q}\Big]^{1/p} \leq 1. 
$$
\end{lem}

Note, that in case $q=2$ and $p=1$ we obtain $c_{1,2}=1/2$\,. This result is useful if the $p$ is not determined in advance. In fact, in various situations it is possible 
to minimize the right-hand side in Lemma \ref{Lqp2} by choosing $p$ depending on $m$ via balancing the quantities involved in the right-hand side (including $c_{p,q}$). 

Lemma \ref{Lqp} and its versions are used in many papers dealing with nonlinear approximation and, more recently,
the approximation of high-dimensional parametric elliptic partial differential equations as in \eqref{spde} below, 
see e.g.\ \cite[Lem.\ 3.6]{CD15}, \cite[(3.13)]{CoDeSch11}, \cite[(21)]{BaCoDeMi16}, \cite{TrWeZh15} and Subsection \ref{approximation in infinite dimensions}. 
This is related to the novel field of {\em uncertainty quantification}. In some of these papers it is called Stechkin's lemma. 
We make a historical remark in this regard. As far as we know S.B. Stechkin did not formulate Lemma \ref{Lqp} even in a special case. 
Stechkin \cite{St} proved the following lemma. 
\begin{lem}\label{StL} ({\bf The Stechkin lemma}) Let $a_1\ge a_2\ge \dots$ be a sequence of
nonnegative numbers. Then the following inequalities hold
$$
\frac{1}{2} \sum_{n=1}^{\infty}n^{-1/2}\left (\sum_{k=n}^{\infty} a_k^2\right)^{1/2} \le \sum_{k=1}^{\infty}a_k \le \frac{2}{\sqrt{3}} \sum_{n=1}^{\infty}n^{-1/2}\left (\sum_{k=n}^{\infty} a_k^2\right)^{1/2}.
$$
\end{lem}

This two-sided estimate gives a criterion of absolute convergence of (multivariate) Fourier series in terms of the 
  trigonometric best $m$-term approximation in $L_2(\T^d)$. This in turn gives a characterization of the Wiener algebra \index{Wiener algebra} $\mathcal{A}(\T^d)$ 
in terms of approximation spaces $\mathbf{A}^{1/2}_1(L_2(\T^d))$, where
$$
\| \, f \, \|_{\mathbf{A}^s_p (L_2(\T^d))} := \| \, f \, \|_2+ \Big(
\sum_{m=1}^\infty \frac 1 m \, [m^s\, \sigma_m (f,\mathcal{T}^d)_2]^p\Big)^{1/p}
<\infty\,,
$$
(compare with Theorem \ref{noncomp2} below). Classically 
(see \cite{St}, \cite[Ex.\ 1 in 3.2]{Pi}, \cite{DT2}, \cite[Thm.\ 4]{D}) 
authors considered the following generalization of Lemma \ref{StL}, which is stronger in a certain sense than Lemma \ref{Lqp}. 
These generalizations are used to characterize nonlinear 
approximation spaces, see for instance Theorem \ref{noncomp2} below. For any $x \in \ell_q$ we denote 
$$
    \sigma_m(x)_q := \inf\{ \|x-z\|_q~:~z \in \ell_q, m\text{-sparse}\}\,.
$$

\begin{lem}\label{pietsch} Let $0<p<q\leq \infty$. Then $x\in \ell_q$ belongs to $\ell_p$ if and only if 
\be\label{appr_space}
    \Big[\sum\limits_{m=1}^{\infty} \big(m^{1/p-1/q} \sigma_m(x)_q\big)^p\frac{1}{m}\Big]^{1/p} <\infty\,.
\ee
Then \eqref{appr_space} is an equivalent quasi-norm in $\ell_p$. 
\end{lem}
Lemma \ref{pietsch} means in particular that for any $x\in \ell_p$ one has
\be\label{appr_space2}
    \Big[\sum\limits_{m=1}^{\infty} [m^{1/p-1/q} \sigma_m(x)_q]^p\frac{1}{m}\Big]^{1/p} \leq C(p,q)\|x\|_p\,.
\ee
One can derive Lemma \ref{Lqp} (with the constant $1$ replaced by $C(p,q)$) from (\ref{appr_space2}) by an argument, which is very similar to the direct proof of Lemma \ref{Lqp}.

\subsection{Sparse trigonometric approximation}\index{Best $m$-term approximation!Sparse trigonometric approximation}
\label{sparsetrig}

Sparse trigonometric approximation of periodic functions began by the paper of
S.B. Stechkin \cite{St}, who used it in the criterion for absolute convergence
of trigonometric series. R.S. Ismagilov \cite{I} found nontrivial estimates for
$m$-term approximation of functions with singularities of the type $|x|$ and
gave interesting and important applications to the widths of Sobolev classes. He
used a deterministic method based on number theoretical constructions. 
His method was developed by V.E. Maiorov \cite{M}, who used a method based on
Gaussian sums.  Further strong results were obtained in \cite{DT1} with the help
of a non-constructive result from finite dimensional Banach spaces due to E.D.
Gluskin \cite{G}. Other powerful non-constructive method, which is based on a
probabilistic argument, was used by Y. Makovoz \cite{Mk} and by E.S. Belinskii
\cite{Be1}  (see the book \cite{TrBe04} for a detailed description of the technique). 
Different methods were created in \cite{T29}, \cite{KTE1},
\cite{T1}, \cite{Tappr} for proving lower bounds for function classes.
It was discovered in \cite{DKTe} and \cite{T12} that greedy algorithms\index{Greedy algorithm} can be
used for constructive $m$-term approximation with respect to the trigonometric
system. We demonstrated in \cite{VT152} how greedy algorithms can be used to
prove optimal or best known upper bounds for $m$-term approximation of classes
of functions with mixed smoothness. It is a simple and powerful method of
proving upper bounds.   The reader can find a detailed study of $m$-term
approximation of classes of functions with mixed smoothness, including small
smoothness, in 
the paper \cite{Rom1} by A.S. Romanyuk and in recent papers \cite{VT152},
\cite{T152sm}. We note that in the case $2<q<\infty$ the upper bounds in
\cite{Rom1} are not constructive. 

The following two theorems are from \cite{VT152}. We use the notation 
$\beta:=\beta(p,q):= 1/p-1/q$ and $\eta:=\eta(p):= 1/p-1/2$. In the case of
trigonometric system $\Tr^d$ we drop it from the notation:
$$
\sigma_m(\bF)_q := \sigma_m(\bF,\Tr^d)_q.
$$
\begin{thm}\label{T2.8I} We have
$$
 \sigma_m(\bW^r_p)_{q}
  \asymp  \left\{\begin{array}{ll} m^{-r+\beta}(\log m)^{(d-1)(r-2\beta)}, &
1<p\le q\le 2,\quad r>2\beta,\\
 m^{-r+\eta}(\log m)^{(d-1)(r-2\eta)}, & 1<p\le 2\le q<\infty,\quad r>1/p,\\ 
 m^{-r}(\log m)^{r(d-1)}, & 2\le p\le q<\infty, \quad r>1/2.\end{array} \right.
$$
\end{thm}
The third line of Theorem \ref{T2.8I} combined with (\ref{tag1.3}) shows that
the trigonometric system is an optimal in the sense of order orthonormal system
for the $m$-term approximation of the classes $\bW^r_p$ in $L_q$ for $2\le
p,q<\infty$. The case $1<p\le q\le 2$ in Theorem \ref{T2.8I}, which corresponds to the first
line, was proved in \cite{T29} (see also \cite{Tmon}, Ch.4). The proofs from
\cite{T29} and \cite{Tmon} are constructive. 

Comparing Theorems \ref{T2.8I} and \ref{baT2.3} (below) with Theorem \ref{TBT3.2} we conclude that in the case 
$1<q\le p<\infty$ we have
$$
\sigma_m(\bW^r_p)_q \asymp E_{Q_n}(\bW^r_p)_q,\quad m\asymp |Q_n|\asymp 2^nn^{d-1}.
$$
In the case $1<p<q<\infty$ we have 
$$
\sigma_m(\bW^r_p)_q = o\left( E_{Q_n}(\bW^r_p)_q\right),\quad m\asymp |Q_n|\asymp 2^nn^{d-1}.
$$
Note, that in case $1<p<q\leq 2$ we encounter a ``multivariate phenomenon''. In fact, the improvement 
for $\sigma_m(\bW^r_p)_q$ happens in the power of the $\log$. Hence, it is not present in the univariate setting with $d=1$. 

Comparing Theorems \ref{T2.8I} and \ref{baT2.3} (below) with Theorem \ref{uw} we see that the trigonometric system
$\Tr^d$ is as good as the wavelet type system $\mathcal U^d$ for $m$-term approximation in the range of parameters
$1<q\le p<\infty$ and $2\le p\le q<\infty$ and for other parameters from $1<p,q<\infty$ the wavelet type system
$\mathcal U^d$ provides better $m$-term approximation than the trigonometric system $\Tr^d$. Probably, this phenomenon is 
related to the fact that the trigonometric system is an
orthonormal uniformly bounded system. Interestingly, for the $\bH$ classes (see below) we observe that wavelet type systems provide a better 
$m$-term approximation also in case $1<p=q<2$\,.

For $1<p\le q\le 2$ the case of small smoothness $\beta <r\le 2\beta$ is settled. In this case 
we have the constructive bound (see \cite{T152sm} and \cite{Be88})
\begin{equation}\label{nolog1}
\sigma_m(\bW^r_p)_q \asymp  m^{-(r-\beta)}.
\end{equation}
Moreover, if $1<p \leq 2 < q<\infty$ and $\beta=1/p-1/q<r\leq 1/p-q'/(2q)$ Theorem \ref{T2.12sm} below implies 
\begin{equation}\label{nolog2}
  \sigma_m(\bW^r_p)_q \asymp m^{-(r-\beta)q/2}\,.
\end{equation}
This (non-constructive) result actually goes back to Belinskii \cite{Be88} for the larger range $\beta < r \leq q'\beta$. However, the paper \cite{Be88} is hard to access. 
In the case of small smoothness the relations \eqref{nolog1} and \eqref{nolog2} show an interesting and important phenomenon. 
The asymptotic rate of decay of $\sigma_m(\bW^r_p)_q$ does not depend on dimension $d$. 
The dependence on $d$ is hidden in the constants. See also \eqref{nolog} above for a similar phenomenon in the $\bB$-classes. 

We note that in the case $q>2$
Theorem \ref{T2.8I} is proved in \cite{T1}. However, the proof there is not
constructive -- it uses a non-constructive result from \cite{DT1}. In
\cite{VT152} we provided a constructive proof, which is based on greedy
algorithms.  Also, this proof works under weaker conditions on $r$: $r>1/p$
instead of $r>1/p+\eta$ for $1<p\le2$. In \cite{VT152} we concentrated on the
case $q\ge 2$. 

Let us continue with sparse trigonometric approximation in case $q=\infty$. 
In this case the results are not as complete as those for $1<q<\infty$. We give here one result from \cite{VT152} (the reference also contains a 
historical discussion).

\begin{thm}\label{T2.9I} We have
$$
 \sigma_m(\bW^r_p)_{\infty}
  \lesssim  \left\{\begin{array}{ll}   m^{-r+\eta}(\log m)^{(d-1)(r-2\eta)+1/2},
& 1<p\le 2,\quad r>1/p,\\ 
 m^{-r}(\log m)^{r(d-1)+1/2}, & 2\le p<\infty, \quad r>1/2.\end{array} \right.
$$
\end{thm}

We used recently developed techniques on greedy approximation in
Banach spaces to prove Theorems \ref{T2.8I} and \ref{T2.9I}. It is important
that greedy approximation allows us not only to prove the above theorems but
also to provide a constructive way for building the corresponding $m$-term
approximants. 
We call this algorithm {\it constructive} because it provides an explicit construction with feasible one parameter
optimization steps. We stress that in the setting of approximation in an infinite dimensional Banach space, which is
considered in our paper, the use of term {\it algorithm} requires some explanation. In this paper we discuss only
theoretical aspects of the efficiency (accuracy) of $m$-term approximation and possible ways  to  realize this
efficiency.    The {\it greedy algorithms}   give a procedure to construct an approximant, which turns out to be a good
approximant. The procedure of
constructing a greedy approximant is not a numerical algorithm ready for computational implementation. Therefore, it
would be more precise to call this procedure a {\it theoretical greedy algorithm} or {\it stepwise optimizing process}.
Keeping this remark in mind we, however, use the term {\it greedy algorithm} in this paper because it has been used in
previous papers and has become a standard name for procedures alike Weak Chebyshev Greedy Algorithm
\index{Greedy algorithm!Weak Chebychev}(see below) and for
more general procedures of this type (see for instance
\cite{D}, \cite{Tbook}). Also, the theoretical algorithms, which we use, become 
algorithms in a strict sense if instead of an infinite dimensional setting we consider a finite dimensional setting,
replacing, for instance, the $L_p$ space by its restriction on the set of trigonometric polynomials. 
We give a precise formulation from \cite{VT152}.
\begin{thm}\label{T2.8C} For $q\in(1,\infty)$ and $\mu>0$ there exist
constructive methods $A_m(f,q,\mu)$, which provide for $f\in \bW^r_p$ an
$m$-term approximation such that
\begin{equation}\nonumber
\begin{split}
&\|f-A_m(f,q,\mu)\|_q\\
&~~~\lesssim  \left\{\begin{array}{ll} m^{-r+\beta}(\log m)^{(d-1)(r-2\beta)}, &
1<p\le q\le 2,\quad r>2\beta+\mu,\\
 m^{-r+\eta}(\log m)^{(d-1)(r-2\eta)}, & 1<p\le 2\le q<\infty,\quad
r>1/p+\mu,\\ 
 m^{-r}(\log m)^{r(d-1)}, & 2\le p\le q<\infty, \quad r>1/2+\mu.\end{array}
\right.
\end{split}
\end{equation}
\end{thm}
Similar modification of Theorem \ref{T2.9I} holds for $q=\infty$.
We do not have matching lower bounds for the upper bounds in Theorem
\ref{T2.9I} in the case of approximation in the uniform norm $L_\infty$. 

We now demonstrate some important features of new techniques used in the proof of the above Theorem \ref{T2.8C}. First
of all, constructive methods $A_m(f,q,\mu)$ are built on the base of greedy-type algorithms. The reader can find a
comprehensive study of greedy algorithms in \cite{Tbook}. One example of greedy-type algorithm -- the Thresholding
Greedy Algorithm\index{Greedy algorithm!Thresholding} --  is discussed in Subsection 7.2 above. Another example -- the Weak Chebyshev Greedy Algorithm -- is
discussed at the end of this subsection. Both of these algorithms are used in building $A_m(f,q,\mu)$. Second, it is
known that the $A$-norm with respect to a given dictionary plays an important role in greedy approximation. In the case
of the trigonometric system $\Tr^d$ the $A$-norm is defined as follows
$$
\|f\|_A := \sum_\bk |\hat f(\bk)|.
$$
The following theorem and lemma from \cite{VT152} play the key role in proving
Theorem \ref{T2.8C}. Denote $\bar m := \max\{m,1\}$. 

\begin{thm}\label{T2.5C} There exist constructive greedy-type approximation
methods $G^q_m(\cdot)$, which provide $m$-term polynomials with respect to
$\Tr^d$ with the following properties: for $2\le q<\infty$
\be\label{2.14p}
\|f-G^q_m(f)\|_q \le C_1(d)(\bar m)^{-1/2}q^{1/2}\|f\|_A,\quad \|G^q_m(f)\|_A
\le C_2(d)\|f\|_A,
\ee
  and for $q=\infty$, $f\in \Tr(\bN,d)$
\be\label{2.14'}
\|f-G^\infty_m(f)\|_\infty \le C_3(d)(\bar m)^{-1/2}(\ln
\vartheta(\bN))^{1/2}\|f\|_A,\quad \|G^\infty_m(f)\|_A \le C_4(d)\|f\|_A.
\ee
\end{thm} 

Taking into account importance of the $A$-norm for greedy approximation we prove some error bounds for approximation of
classes $\bW^{a,b}_A$ (see the definition in Lemma \ref{L2.1}), which are in a style of classes $\bW^r_p$ with parameter
$a$ being similar to parameter $r$ and parameter $b$ being the one controlling the logarithmic type smoothness. 
\begin{lem}\label{L2.1}  Define for $f\in L_1$
$$
f_l:=\sum_{|\bs|_1=l}\delta_\bs(f), \quad l\in \N_0,\quad \N_0:=\N\cup \{0\}.
$$
Consider the class
$$
\bW^{a,b}_A:=\{f: \|f_l\|_A \le 2^{-al}({\bar l})^{(d-1)b}\}.
$$
Then for $2\le q\le\infty$ and $0<\mu<a$ there is a constructive method
$A_m(\cdot,q,\mu)$ based on greedy algorithms, which provides the bound for
$f\in \bW^{a,b}_A$
\begin{equation}\nonumber
\|f-A_m(f,q,\mu)\|_q \lesssim  m^{-a-1/2} (\log m)^{(d-1)(a+b)},\quad 2\le
q<\infty,      
\end{equation}
\begin{equation}\label{2.14}
\|f-A_m(f,\infty,\mu)\|_\infty \lesssim  m^{-a-1/2} (\log m)^{(d-1)(a+b)+1/2}.  
\end{equation}

\end{lem}
 \begin{proof} We prove the lemma for $m\asymp 2^nn^{d-1}$, $n\in \N$. Let $f\in
\bW^{a,b}_A$.  
We approximate $f_l$ in $L_q$. By Theorem \ref{T2.5C} we obtain for $q\in
[2,\infty)$
\begin{equation}\label{2.15}
\|f_l-G^q_{m_l}(f_l)\|_q \lesssim (\bar m_l)^{-1/2}\|f_l\|_A \lesssim (\bar
m_l)^{-1/2}2^{-al}l^{(d-1)b}.
\end{equation}
We take $\mu\in (0,a)$ and specify
$$
m_l := [2^{n-\mu (l-n)}l^{d-1}],\quad l=n,n+1,\dots.
$$
In addition we include in the approximant 
$$
S_n(f) := \sum_{|\bs|_1\le n}\delta_\bs(f).
$$
Define
$$
A_m(f,q,\mu) := S_n(f)+\sum_{l>n} G_{m_l}^q(f_l).
$$
Then, we have built an $m$-term approximant of $f$ with 
$$
m\lesssim 2^nn^{d-1}  +\sum_{l\ge n} m_l \lesssim 2^nn^{d-1}.
$$
 The error of this approximation in $L_q$ is bounded from above by
$$ 
\|f-A_m(f,q,\mu)\|_q \le \sum_{l\ge n} \|f_l-G^q_{m_l}(f_l)\|_q \lesssim
\sum_{l\ge n} (\bar m_l)^{-1/2}2^{-al}l^{(d-1)b}
$$
$$
\lesssim \sum_{l\ge n}2^{-1/2(n-\mu(l-n))}l^{-(d-1)/2}2^{-al}l^{(d-1)b}
\lesssim 2^{-n(a+1/2)}n^{(d-1)(b-1/2)}.
$$
This completes the proof of lemma in the case $ 2\le q<\infty$.

 Let us discuss the case $q=\infty$. The proof repeats the proof in the above
case $q<\infty$ with the following change. Instead of using (\ref{2.14p}) for
estimating an $m_l$-term approximation of $f_l$ in $L_q$ we use  (\ref{2.14'})
to estimate an  $m_l$-term approximation of $f_l$ in $L_\infty$. Then bound
(\ref{2.15}) is replaced by 
\begin{equation}\label{2.16}
\|f_l-G^\infty_{m_l}(f_l)\|_\infty \lesssim (\bar m_l)^{-1/2}(\ln
2^l)^{1/2}\|f_l\|_A \lesssim (\bar m_l)^{-1/2}l^{1/2}2^{-al}l^{(d-1)b}.
\end{equation}  
The extra factor $l^{1/2}$ in (\ref{2.16}) gives an extra factor $(\log
m)^{1/2}$ in (\ref{2.14}). 
\end{proof}

We now consider the case $\sigma_m(\bW^r_{1})_p$, which is not covered by 
Theorem  \ref{T2.8I}. The function $F_r(\bx)$ belongs to the closure in $L_q$
of  $\bW^r_{1}$, $r>1-1/q$, and, therefore, on the one hand
$$
\sigma_m(\bW^r_{1})_q \ge \sigma_m(F_r(\bx))_q.
$$
On the other hand, it follows from the definition of $\bW^r_{1}$ that for any
$f\in \bW^r_{1}$
$$
\sigma_m(f)_q \le \sigma_m(F_r(\bx))_q.
$$
Thus,
\be\nonumber
\sigma_m(\bW^r_{1})_q = \sigma_m(F_r(\bx))_q.
\ee
The following results on $\sigma_m(F_r(\bx))_q$ are from \cite{VT152}.

\begin{thm}\label{T2.10} We have
$$
 \sigma_m(F_r(\bx))_q
  \asymp  \left\{\begin{array}{ll}   m^{-r+1-1/q}(\log m)^{(d-1)(r-1+2/q)}, &
1<q\le 2,\quad r>1-1/q,\\ 
 m^{-r+1/2}(\log m)^{r(d-1)}, & 2\le q<\infty, \quad r>1.\end{array} \right.
$$
The upper bounds are provided by a constructive method $A_m(\cdot,q,\mu)$ based
on greedy algorithms. 
\end{thm}
\begin{thm}\label{T2.11} We have
$$
 \sigma_m(F_r(\bx))_\infty
  \lesssim    
 m^{-r+1/2}(\log m)^{r(d-1)+1/2},   \quad r>1. 
$$
The  bounds are provided by a constructive method $A_m(\cdot,\infty,\mu)$ based
on greedy algorithms. 
\end{thm}

We now proceed to classes $\bH^r_p$ and $\bB^r_{p,\theta}$.
  The following theorem was proved in
\cite{T29} (see also \cite{Tmon}, Ch.4). The proofs from \cite{T29} and
\cite{Tmon} are constructive.

\begin{thm}\label{T2.7} Let $1<p\le q\le 2$, $r>\beta$. Then
$$
\sigma_m(\bH^r_{p})_q \asymp m^{-r+\beta}(\log m)^{(d-1)(r-\beta+1/q)}.
$$
\end{thm}

The following analog of Theorem \ref{T2.8I} for classes $\bH^r_p$ was proved
in \cite{Rom1}. The proof in \cite{Rom1} in the case $q>2$ is not constructive.
\begin{thm}\label{T2.12} One has
$$
 \sigma_m(\bH^r_{p})_{q}
  \asymp  \left\{\begin{array}{ll} m^{-r+\beta}(\log m)^{(d-1)(r-\beta+1/q)}, &
1<p\le q\le 2,\quad r>\beta,\\
 m^{-r+\eta}(\log m)^{(d-1)(r-1/p+1)}, & 1<p\le 2\le q<\infty,\quad r>1/p,\\ 
 m^{-r}(\log m)^{(d-1)(r+1/2)}, & 2\le p\le q<\infty, \quad r>1/2.\end{array}
\right.
$$
\end{thm}
The third line of Theorem \ref{T2.12} combined with (\ref{tag1.2}) shows that
the trigonometric system is an optimal one in the sense of order orthonormal system
for the $m$-term approximation of the classes $\bH^r_p$ in $L_q$ for $2\le
p,q<\infty$. Interestingly, in case $1<p<q \leq 2$ it holds, in contrast to the situation for the $\bW^r_p$ classes,
$$
    \sigma_m(\bH^r_p)_q \asymp E_{Q_n}(\bH^r_p)_q\,,\quad m\asymp |Q_n|\asymp 2^nn^{d-1}\,,
$$
see Theorem \ref{TBT3.3}. Hence, 
the ``multivariate phenomenon'' mentioned above after Theorem \ref{T2.8I} is not present here. 

The following proposition is from \cite{VT152}.
\begin{prop}\label{HCp} The upper bounds in Theorem \ref{T2.12} are provided by
a constructive method $A_m(\cdot,q,\mu)$ based on greedy algorithms. 
\end{prop}
In the case $q=\infty$ we have (see \cite{VT152}).
\begin{thm}\label{T2.13} We have
$$
 \sigma_m(\bH^r_{p})_{\infty}
  \lesssim  \left\{\begin{array}{ll}   m^{-r+\eta}(\log
m)^{(d-1)(r-1/p+1)+1/2}, & 1<p\le 2,\quad r>1/p,\\ 
 m^{-r}(\log m)^{(r+1/2)(d-1)+1/2}, & 2\le p<\infty, \quad r>1/2.\end{array}
\right.
$$
The upper bounds  are provided by a constructive method $A_m(\cdot,\infty,\mu)$
based on greedy algorithms. 
\end{thm}
For a non-constructive proof of the bound of Theorem \ref{T2.13} in the case
$2\le p<\infty$ see \cite{Be3}. 

We now proceed to classes $\bB^r_{p,\theta}$.
It will be convenient for us to use the following slight modification of
classes $\bB^r_{p,\theta}$. Define
$$
\|f\|_{\bH^r_{p,\theta}}:= \sup_n
\left(\sum_{\bs:|\bs|_1=n}\left(\|\delta_\bs(f)\|_p
2^{r|\bs|_1}\right)^\theta\right)^{1/\theta}
$$
and 
$$
\bH^r_{p,\theta}:= \{f: \|f\|_{\bH^r_{p,\theta}}\le 1\}.
$$

The best $m$-term approximations of classes $\bB^r_{p,\theta}$ are studied in
detail by A.S. Romanyuk \cite{Rom1}.

There is the following extension 
of Theorem \ref{T2.12} (see \cite{Rom1} for the $\bB^r_{p,\theta}$ classes and \cite{VT152} for the $\bH^r_{p,\theta}$ classes).
\begin{thm}\label{T2.12B} One has
$$
\sigma_m(\bB^r_{p,\theta})_{q}\asymp m^{-r+\beta}(\log
m)^{(d-1)(r-\beta+1/q-1/\theta)_+}, \quad 1<p\le q\le 2,\quad r>\beta\,,
$$
and
\begin{equation*}
\begin{split}
\sigma_m(\bB^r_{p,\theta})_{q} &\asymp  \sigma_m(\bH^r_{p,\theta})_{q} \\
  &\asymp  \left\{\begin{array}{ll}  
 m^{-r+\eta}(\log m)^{(d-1)(r-1/p+1-1/\theta)}, & 1<p\le 2\le q<\infty,\quad
r>1/p,\\ 
 m^{-r}(\log m)^{(d-1)(r+1/2-1/\theta)}, & 2\le p\le q<\infty, \quad
r>1/2.\end{array} \right.
\end{split}
\end{equation*}
\end{thm}

Interestingly, if $\theta$ is getting small in the first statement in Theorem \ref{T2.12B} 
we get rid of the $\log$ in some cases. Or, in different words, 
if $0<r-\beta< 1/\theta-1/p$ is small then the $\log$-term diappears, which is a similar small smoothness 
phenomenon as in \eqref{nolog1} and \eqref{nolog2}. The dimension $d$ plays no role in the asymptotic rate of convergence, it is 
only hidden in the constants, see also Theorem \ref{baT2.5B} where a similar effect occurs for $q\leq p$.

The following proposition is from \cite{VT152}.
\begin{prop}\label{BCp} The upper bounds in Theorem \ref{T2.12B} are provided
by a constructive method based on greedy algorithms. 
\end{prop}

The following result is from \cite{Baz26a}.

\begin{thm}\label{TBaz16} For $1\le p\le q\le 2$, $1<q$, $1\le \theta\le\infty$ one has
$$
  \sigma_m(\bH^r_{p,\theta})_{q} 
  \asymp  \left\{\begin{array}{ll}  
 m^{-r+\beta}(\log m)^{(d-1)(r-\beta+1/q-1/\theta)_+}, & \beta<r\neq 1/p-2/q+1/\theta,\\ 
 m^{-r+\beta}(\log\log m)^{1/\theta)}, & \beta<r= 1/p-2/q+1/\theta.\end{array} \right.
$$
\end{thm}

\index{Small smoothness} In another small smoothness range the following result is known (see \cite[Thm.\ 2.1]{Rom1} engl. version).
\begin{thm}\label{T2.12sm} Let $1\le p\le 2<q<\infty$ and $1\le\theta\le\infty$.
Then
$$
 \sigma_m(\bB^r_{p,\theta})_{q} 
  \asymp  \left\{\begin{array}{ll} m^{-1/2}(\log m)^{d(1-1/\theta)}, &  
r=1/p,\\
 m^{-q(r-\beta)/2}(\log m)^{(d-1)(q-1)(r-1/p+q'/(q\theta'))_+},  &  \beta<
r<1/p.\end{array} \right.
$$
\end{thm}

The case $1<p\leq 2<q<\infty$, $\theta=\infty$, $\beta<r<1/p$ was considered by Temlyakov 
in \cite[Thm.\ 3.4]{T152sm}. There the upper bounds were achieved by a constructive 
greedy-type algorithm (compared to the proof of \cite[Thm.\ 2.1]{Rom1}, non-constructive upper bounds). 
Theorem \ref{T2.12sm} is also true in more general settings, see the recent paper by Stasyuk \cite{Sta16}.

\subsubsection*{The case $q\leq p$}
We formulate some known results in the case $1<q\le p\le\infty$ and describe interesting effects when 
comparing sparse trigonometric approximation and hyperbolic cross projections. So far we have seen, that 
sparse trigonometric approximation seems to show a significant improvement in comparison with hyperbolic cross 
projections only if $p<q$. However, the comment after Theorem \ref{T2.12} shows that even this is not always the case. In addtion, we will see
below that in the framework of $\bB$-classes sparse trigonometric approximation may beat hyperbolic cross projections even in case $q\leq p$. Finally, 
we will compare sparse trigonometric approximation with best $m$-term approximation using wavelet type dictionaries. 
\begin{thm}\label{baT2.3} Let $1<q\le p<\infty$, $r>0$. Then
$$
\sigma_m(\bW^r_p)_q \asymp m^{-r}(\log m)^{(d-1)r}.
$$
\end{thm}
The upper bound in Theorem \ref{baT2.3} follows from error bounds for
approximation by the hyperbolic cross polynomials (see \cite{Tmon}, Ch.2, \S2)
$$
E_{Q_n}(\bW^r_q,L_q) \lesssim 2^{-rn},\quad 1<q<\infty.
$$
The lower bound in Theorem \ref{baT2.3} was proved in \cite{KTE1}.

The following result for $\bH^r_p$ classes is known.
\begin{thm}\label{baT2.5} Let $q\le p$, $2\le p\le \infty$, $1<q<\infty$, $r>0$.
Then
$$
\sigma_m(\bH^r_{p})_q \asymp m^{-r}(\log m)^{(d-1)(r+1/2)}.
$$
\end{thm}
The lower bound for all $q>1$
$$
\sigma_m(\bH^r_\infty)_q \gtrsim m^{-r}(\log m)^{(d-1)(r+1/2)}
$$
was obtained in \cite{KTE1}. The matching upper bounds follow from
approximation by the hyperbolic cross polynomials (see \cite{Tmon}, Ch.2,
Theorem 2.2)
$$
E_{Q_n}(\bH^r_q)_q:=\sup_{f\in\bH^r_q}E_{Q_n}(f)_q \asymp
n^{(d-1)/2}2^{-rn},\quad 2\le q<\infty.
$$
The rate observed in Theorem \ref{baT2.5} does not extend to $1<q\leq p\leq 2$. As one would expect, 
we obtain that the hyperbolic cross projections are optimal also for sparse 
trigonometric approximation in case $1<p=q\leq 2$ and $r>0$, see 
Theorem \ref{TBT3.3} and the lower bounds in Theorems \ref{T2.7} and 
\ref{T2.12}. In other words, we have in this case
$$
    \sigma_m(\bH^r_p)_p \asymp E_{Q_n}(\bH^r_p)_p \asymp m^{-r}(\log m)^{(d-1)(r+1/p)}\,,\quad m\asymp |Q_n|\asymp 2^nn^{d-1}\,.
$$
Even more interesting is the following observation. Here we have a situation where $p=q$ and wavelet type dictionaries $\Phi$ yield a substantially better 
rate of convergence than the trigonometric system. Indeed, Theorem \ref{T7.6} gives 
$$
  \sigma_m(\bH^r_p,\Phi)_p \asymp m^{-r}(\log m)^{(d-1)(r+1/2)}\quad, \quad m\in \N\,,
$$
in case $1<p=q <2$ and $r>1/p-1/2$\,. Analogously to the comment after Theorem \ref{T2.8I} this 
is a multivariate phenomenon since the improvement happens in the power of the $\log$. Such a situation 
was not known before as far as we know. Although Stasyuk \cite[Rem.\ 2]{Sta11} comments on such a phenomenon, his example 
does not work. His Theorems \cite[Thm.\ 1]{Sta11}, \cite[Thm.\ 3.1]{Sta15} do not hold in case of small smoothness. Indeed, sparse trigonometric and 
sparse Haar wavelet approximation yield the same rate for his example $\bB^r_{\infty,\theta}$, $q<\infty$, $\theta<2$ 
and $0<r<1/\theta-1/2$.
 
The following result for $\bB$ classes was proved in \cite{Rom1}. For an extension to $q<p=\infty$ we refer to 
\cite{RomRom10}.
\begin{thm}\label{baT2.5B} Let $1<q\le p<\infty$, $2\le p< \infty$,
$1<q<\infty$, $r>0$. Then
$$
\sigma_m(\bB^r_{p,\theta})_q \asymp m^{-r}(\log m)^{(d-1)(r+1/2-1/\theta)_+}.
$$
\end{thm}
\index{Small smoothness}
Let us comment on two different effects here. If $\theta<2$ we observe an improvement in comparison with the respective bound in 
Theorem \ref{Theorem[HC-Approx-Brpt]} for the approximation from the hyperbolic cross. Let us emphasize that $q\leq p$ here, 
so this effect is new and has not been observed before. The improvement happens in the $\log$ factor. Thus, we again encounter a pure multivariate phenomenon here 
which is not present if $d=1$. 

Secondly, the rates of convergence in Theorems \ref{baT2.5} and \ref{baT2.5B} for $p\geq 2$ 
coincide with the rates of convergence for sparse approximation with wavelet type 
dictionaries, see Theorem \ref{T7.6}. As we have seen above for $\bH^r_p$ classes 
this is not the case if $1<q\leq p<2$. Here the 
wavelet type dictionaries show a better behavior in the $\log$.

\subsection{Different types of Greedy Algorithms}
\noindent We are interested in the following fundamental problem of sparse
approximation. 

{\bf Problem.} How to design a practical algorithm that builds sparse
approximations comparable to best $m$-term approximations? 
 
It was demonstrated in the paper \cite{T144} that the Weak Chebyshev Greedy
Algorithm (WCGA), which we define momentarily, is a solution to the above
problem for a special class of dictionaries.  
 
Let $X$ be a real Banach space with norm $\|\cdot\|:=\|\cdot\|_X$. We say that
a set of elements (functions) $\D$ from $X$ is a dictionary  if each $g\in \D$
has norm   one ($\|g\|=1$),
and the closure of $\Span \D$ is $X$.
For a nonzero element $g\in X$ we let $F_g$ denote a norming (peak) functional
for $g$:
$$
\|F_g\|_{X^*} =1,\qquad F_g(g) =\|g\|_X.
$$
The existence of such a functional is guaranteed by the Hahn-Banach theorem.

Let $t\in(0,1]$ be a given weakness parameter. Define the Weak Chebyshev Greedy
Algorithm (WCGA) (see \cite{T15}) as a generalization for Banach spaces of the
\index{Greedy algorithm!Weak orthogonal matching pursuit}
Weak Orthogonal Matching Pursuit (WOMP). In a Hilbert space the WCGA coincides
with the WOMP. The WOMP is very popular in signal processing, in particular, in
compressed sensing. In case $t=1$,   WOMP is called  Orthogonal Matching Pursuit
(OMP).

 {\bf Weak Chebyshev Greedy Algorithm (WCGA).}\index{Greedy algorithm!Weak Chebychev}
Let $f_0$ be given. Then for each $m\ge 1$ we have the following inductive
definition.

(1) $\varphi_m :=\varphi^{c,t}_m \in \D$ is any element satisfying
$$
|F_{f_{m-1}}(\varphi_m)| \ge t\sup_{g\in\D}  | F_{f_{m-1}}(g)|.
$$

(2) Define
$
\Phi_m := \Phi^t_m := \Span \{\varphi_j\}_{j=1}^m,
$
and define $G_m := G_m^{c,t}$ to be the best approximant to $f_0$ from $\Phi_m$.

(3) Let
$
f_m := f^{c,t}_m := f_0-G_m.
$

The trigonometric system is a classical system that is known to be difficult to
study. In \cite{T144} we study among other problems the problem of nonlinear
sparse approximation with respect to it. Let  ${\mathcal R}{\mathcal T}$ denote
the real trigonometric system 
$1,\sin x,\cos x, \dots$ on $[0,2\pi]$ and let ${\mathcal R}{\mathcal T}_p$ to
be its version normalized in $L_p([0,2\pi])$. Denote ${\mathcal R}{\mathcal
T}_p^d := {\mathcal R}{\mathcal T}_p\times\cdots\times {\mathcal R}{\mathcal
T}_p$ the $d$-variate trigonometric system. We need to consider the real
trigonometric system because the algorithm WCGA is well studied for the real
Banach space. In order to illustrate performance of the WCGA we   discuss in
this section  the above mentioned problem for the trigonometric system. 
   The following 
Lebesgue-type inequality for the WCGA was proved in \cite{T144}.
\begin{thm}\label{gaT1.2} Let $\D$ be the normalized in $L_p$, $2\le p<\infty$,
real $d$-variate trigonometric
system. Then    
for any $f_0\in L_p$ the WCGA with weakness parameter $t$ gives
\begin{equation}\label{I1.4}
\|f_{C(t,p,d)m\ln (m+1)}\|_p \le C\sigma_m(f_0,\D)_p .
\end{equation}
\end{thm}
The Open Problem 7.1 (p. 91) from \cite{T18} asks if (\ref{I1.4}) holds without
an extra 
$\ln (m+1)$ factor. Theorem \ref{gaT1.2} is the first result on the
Lebesgue-type inequalities for the WCGA with respect to the trigonometric
system. It provides a progress in solving the above mentioned open problem, but
the problem is still open. 

Theorem \ref{gaT1.2} shows that the WCGA is very well designed for the
trigonometric system. It was shown in \cite{T144}  that an analog of
(\ref{I1.4}) holds for 
uniformly bounded orthogonal systems.  

As a direct corollary of Theorems \ref{gaT1.2} and \ref{T2.8I} we obtain the
following result (see \cite{VT152}).
\begin{thm}\label{TCW} Let $q\in [2,\infty)$. Apply the WCGA with weakness
parameter $t\in (0,1]$ to $f\in L_q$ with respect to the real trigonometric
system ${\mathcal R}{\mathcal T}_q^d$. If $f\in \bW^r_p$, then we have
$$
 \|f_m\|_{q}
  \lesssim  \left\{\begin{array}{ll}  
 m^{-r+\eta}(\log m)^{(d-1)(r-2\eta)+r-\eta}, & 1<p\le 2,\quad r>1/p,\\ 
 m^{-r}(\log m)^{rd}, & 2\le p<\infty, \quad r>1/2.\end{array} \right.
$$
\end{thm}

\subsection{Open problems}\index{Open problems!Nonlinear approximation}

It is well known that the extreme cases, when one of the parameters $p$ or $q$
takes values $1$ or $\infty$ are difficult in the hyperbolic cross approximation
theory. Often, study of these cases requires special techniques. Many of the
problems, which involve the extreme values of parameters, are still open. Also
the case of small smoothness is still open in many settings.\index{Small smoothness}

{\bf Open problem 7.1.} Find a constructive method, which provides the order of
$\sigma_m(\bW^r_{p})_q$, $2\le p\le q <\infty$, $\beta<r\le 1/2$. 

{\bf Open problem 7.2.} Find the order of $\sigma_m(\bW^r_{p})_\infty$, $1\le
p\le \infty $, $r>1/p$. 

{\bf Open problem 7.3.} Find the order of $\sigma_m(\bW^r_{p})_1$, $1\le p\le
\infty $, $r>0$. 

{\bf Open problem 7.4.} Find the order of $\sigma_m(\bW^r_{\infty})_q$, $1\le
q\le \infty $, $r>0$. 

We formulated the above problems for the $\bW$ classes. Those problems are open
for the $\bH$ and $\bB$ classes as well.  In addition the following problem is
open for the $\bH$ and $\bB$ classes.
 
{\bf Open problem 7.5.} Find the order of $\sigma_m(\bH^r_{p})_q$ and
$\sigma_m(\bB^r_{p,\theta})_q$ for $1\le q<p \le 2 $, $r>0$.

%% file: numerical_integr.tex
\section{Numerical integration}\index{Numerical integration} 

\label{numint}

\subsection{The problem setting}
A cubature rule $\Lambda_m(f,\xi)$ approximates the integral $I(f) =
\int_{[0,1]^d} f(\bx) \, d\bx$ 
by computing a weighted sum of finitely many function values at 
$X_m = \{\bx^1,...,\bx^m\}$ , i.e., \begin{equation}\label{cub_def}
  \Lambda_m(f,X_m) := \sum\limits_{i=1}^m \lambda_i f(\bx^i),
\end{equation}
where the $d$-variate function $f$ is assumed to belong to some (quasi-)normed
function space $\mathbf{F} \subset C([0,1]^d)$. 
The optimal error with respect to the class $\mathbf{F}$ is given
by 
\begin{equation}\label{eq:minimal}
  \kappa_m(\mathbf{F}) \,:=\,
\inf\limits_{X_m=\{\bx^1,\dots,\bx^m\}\subset[0,1]^d}\,
	\inf\limits_{\lambda_1,\dots,\lambda_m\in\R}
\Lambda_m(\mathbf{F},X_m)\,,
\end{equation}
where
$$
	\Lambda_m(\mathbf{F},X_m):=\sup\limits_{f\in\mathbf{F} }
	\left|I(f)-\sum\limits_{i=1}^m \lambda_i f(\bx^i)\right|\,.
$$
It is clear that $f\in \mathbf{F}$ has to be defined at the nodes
$\{\bx^1,...,\bx^m\}$. For that reason we always assume $\mathbf{F}
\hookrightarrow C(\T^d)$. For simplicity, in contrast to the other sections, we
represent the $d$-torus here as $\T^d = [0,1]^d$. Our main interest in this
section is to present known optimal results (in the sense of order) in the
number of nodes $m$ for numerical integration of classes of functions with
bounded mixed smoothness. 

\subsection{Lower bounds}
\label{sect:lbounds_int}

The lower bounds that we want to present are valid for \emph{arbitrary} 
cubature formulas. For this we study the quantity $\kappa_m(\mathbf{F})$ 
for the spaces $\Brpt$ and $\Wrp$. We will provide two approaches for
``fooling'' the given cubature formula in order to obtain asymptotically sharp
lower bounds. 

\subsubsection*{Fooling polynomials}\index{Fooling function!Polynomial}

Let us start with the following approach which has been a big
breakthrough at that time. The theorem is from \cite{Tem22}. The idea is to
construct ``bad'' trigonometric
polynomials which fool the given cubature formula in the following way.

\begin{thm}\label{T6.2.1a} Let $1\leq p<\infty$ and $r>1/p$. Then we have 
$$
\kappa_m(\bW_{p}^r) \gtrsim m^{-r}
(\log m)^{(d-1)/2}\,.
$$
\end{thm}

The proof of Theorem \ref{T6.2.1a} is based on the existence of special
trigonometric polynomials given by Theorem \ref{T2.2.1} above. In fact, Theorem
\ref{T2.2.1} is used to prove the following assertion.

\begin{lem}\label{L6.2.2} Let the coordinates of the vector $\bs$ be natural
numbers and $|\bs|_1 = n$. Then for any $N\le 2^{n-1}$ and an
arbitrary cubature formula $\Lambda_N(\cdot,X_N)$ with $N$ nodes there is a
$t_{ \bs}\in \Tr(2^{ \bs} ,d)$ such that
$\|t_{ \bs}\|_{\infty} \le 1$ and
\be\nonumber
\hat t_{ \bs} (0) - \Lambda_N (t_{ \bs},X_N) \ge C(d) > 0.
\ee
\end{lem}
\noindent Now we choose for a given $m$ a number $n$ such that
$$
 m \le 2^{n-1}  < 2m.
$$
Consider the polynomial
$$
 t( \bx) =\sum_{| \bs|_1=n}t_{ \bs} ( \bx) ,
$$
where $t_{ \bs}$ are polynomials from Lemma \ref{L6.2.2} with $N = m$.
Then
\be\nonumber
\hat t( 0) - \Lambda_m (t,X_m)\ge C(d)n^{d-1}.
\ee
The proof of Theorem \ref{T6.2.1a} was completed by establishing that if
$2\leq p <\infty$ then
\be\nonumber
\|t\|_{\bW^r_p} \lesssim \|t\|_{\bB^r_{p,2}} \lesssim 2^{rn} n^{(d-1)/2}.
\ee
Theorem \ref{T6.2.1a} gives the same lower bound for different parameters $1\le
p<\infty$. It is clear that the bigger the $p$ the stronger the statement. 
It was pointed out in \cite{VT152} that the above example  also provides the
lower bound for the Besov classes. 
     
\begin{thm}\label{T6.2.3} Let $1\leq p,\theta \leq \infty$ and $r>1/p$. We have
the following lower bound for the Besov classes $\Brpt$
$$
\kappa_m(\Brpt) \gtrsim
m^{-r}(\log m)^{(d-1)(1-1/\theta)},\quad 1 \le p \le \infty,\quad 1\le
\theta\le\infty .
$$
\end{thm}
\noindent Indeed, it is easy to check that
\be\nonumber
  \|t\|_{\bB^r_{p,\theta}} \lesssim 2^{rn} n^{(d-1)/\theta}.
\ee
The following lower bounds for numerical integration with respect to a special
class of nodes were obtained recently in \cite{VT152}. 
Let $\bs = (s_1,\dots,s_d)$, $s_j\in \N_0$, $j=1,\dots,d$. We associate with
$\bs$ a
$W(\bs)$ as follows: denote 
$$
w(\bs,\bx) := \prod_{j=1}^d \sin (2\pi2^{s_j}x_j)
$$
and define
$$
W(\bs) := \{\bx: w(\bs,\bx)=0\}.
$$

\begin{defi}\label{D4.1} 
We say that a set of nodes $X_m:=\{\bx^i\}_{i=1}^m$ is
an $(n,\ell)$-net if $|X_m\setminus W(\bs)| \le 2^\ell$ for all $\bs$ such that
$|\bs|_1=n$.
\end{defi}

\begin{thm}\label{T4.2} For any cubature formula $\Lambda_m(\cdot,X_m)$ with
respect to a $(n,n-1)$-net $X_m$ we have for $1\le p<\infty$ that 
$$
\Lambda_m(\bW^r_p,X_m) \gtrsim 2^{-rn}n^{(d-1)/2}\quad,\quad m\in \N\,.
$$
\end{thm}

In the same way as a slight modification of the proof of Theorem \ref{T6.2.1a}
gave Theorem \ref{T6.2.3} a similar modification of the proof of Theorem
\ref{T4.2} gives the following result.

\begin{thm}\label{T4.4} Let $1\leq p,\theta \leq \infty$ and $r>1/p$. Then for
any cubature formula $\Lambda_m(\cdot,X_m)$ with respect to a $(n,n-1)$-net
$X_m$ we have that 
$$
\Lambda_m(\bB_{p,\theta}^r,X_m) \gtrsim 2^{-rn}n^{(d-1)(1-1/\theta)}\quad,\quad
n\in \N\,.
$$
\end{thm}

In case $\mathbf{W}^r_{1,0}$ with $r>1$ things can be ``improved''. 
We obtain a larger lower bound than the one in Theorem \ref{T4.2} by slightly
shrinking the class of cubature formulas via the following assumption
$$
\sum_{\mu=1}^m |\lambda_\mu| \le B.
$$
The corresponding minimal error with respect to a class
$\mathbf{F}$ is then defined by 
$$
     \kappa_m^B(\mathbf{F}) := \inf_{\sum_{\mu=1}^m |\lambda_{\mu}|\leq
B}\inf\limits_{X_m=\{\bx^1,...,\bx^m\}}  \Lambda_m(\mathbf{F},X_m).
$$
The following result is from \cite{Tem25}.
\begin{thm}\label{Thm:Wr10} Let $r>1$. Then there is a constant $C = C(r,B,d)$
such that 
$$
   \kappa^B_m(\mathbf{W}^r_{1,0}) \geq  C(r,B,d)m^{-r}(\log m)^{d-1}\,.
$$
\end{thm}

\subsubsection*{Local fooling functions}\index{Fooling function!Local}

The idea of using fooling functions to prove lower bounds for asymptotic
characteristics of functions goes back to Bakhvalov \cite{Bakh4}. The following
approach relies on atomic decompositions, a modern tool in function space
theory, to control the norms of superpositions of local bumps, see Section
\ref{Bspline}. This
powerful approach allows as well to treat the quasi-Banach situation where
$p<1$, see for instance \cite{UU14}. However, what concerns this survey
we will restrict ourselves to the Banach space case. 
With the results from Proposition \ref{atomic} we are in a position to define
test functions of type \eqref{tf} in order to prove the required lower bounds. 
By \eqref{flelambda} we are able to control the norm $\|\cdot\|_{\Brpt}$.
Following \cite[Thm.\ 4.1]{DU14} and \cite{UU14} we will use test functions of
type 
\begin{equation}\label{g_1}
g_{r,\theta} := C2^{-r\ell}\ell^{-(d-1)/\theta}
\sum\limits_{|\bj|_1=\ell+1}\sum\limits_{\bk \in K_j(X_m)} a_{\bj,\bk}\,,
\end{equation}
where $K_{\bj}(X_m) \subset \{0,...,2^{j_1-1}\times ... \times
\{0,...,2^{j_d-1}\}$ 
depends on the set of integration nodes $X_m:=\{\bx^1,...,\bx^m\}$ with $m =
2^{\ell}$\,. 

Let us now give a different proof of Theorems \ref{T6.2.1a} and \ref{T6.2.3}.
One only needs to prove Theorem \ref{T6.2.3}. In fact, Theorem
\ref{T6.2.1a} follows from Theorem \ref{T6.2.3} together with the embeddings 
$\ensuremath{{\mathbf B}_{2,2}^r} \hookrightarrow \Wrp$ if $p\leq 2$
and $\ensuremath{{\mathbf B}_{p,2}^r} \hookrightarrow \Wrp$ if $p>2$\,.
Following the arguments in
\cite[Thm.\ 4.1]{DU14}
let $m$ be given and $X_m = \{\bx^1,...,\bx^m\} \subset [0,1]^d$ be an arbitrary
set of $m = 2^{\ell}$ points. Since $Q_{\bj,\bk} \cap Q_{\bj,\bk'} = \emptyset$
for $\bk\neq \bk'$ we have for every $|\bj|_1 = \ell+1$ a set of $2^{\ell}$
cubes of the form $2^{-\bj}\bk + 2^{-\bj}Q$ which do not intersect the nodes
$X_m$. We choose the test function \eqref{g_1} where $K_{\bj}(X_m)$ is the set
of
those $\bk$ referring to those cubes. By \eqref{flelambda} the $\Brpt$ norm of
those functions is uniformly bounded (in $\ell$). The result in case $p,\theta
\ge 1$ follows from the 
observation 
$$
   \int_{[0,1]^d} g_{r,\theta}\,d\bx \asymp 2^{-\ell
r}\ell^{(d-1)(1-1/\theta)}\,.
$$
Of course, a cubature rule admitted in \eqref{eq:minimal}, which uses the points
$X_m$, produces a zero output. For parameters $\min\{p,\theta\}<1$ the fooling function
\eqref{g_1} has to be slightly modified.

\subsection{Cubature on Smolyak grids}\index{Cubature formula!Smolyak grids}
\label{cubsmol}

 Here we will discuss cubature on Smolyak grids. The Smolyak grid of
level $\ell$ is given by 

\begin{equation}\label{eq53}
SG^d(\ell)
:= \bigcup\limits_{k_1+...+k_d \leq \ell} I_{k_1} \times ... \times I_{k_d}
\end{equation}
where $I_{k}:=\{2^{-k}n:n = 0,...,2^k-1\}$. We consider the cubature formulas 
$\Lambda_m(f,SG^d(\ell))$ on Smolyak grids $SG^d(\ell)$ given by
\begin{equation}\label{cfsmolyak}
\Lambda_m(f,SG^d(\ell))
\ = \
\sum_{\bx^i \in SG^d(\ell)}  \lambda_i f(\bx^i).
\end{equation}

\begin{figure}[H]
 \begin{center}
 \includegraphics[scale = 1.3]{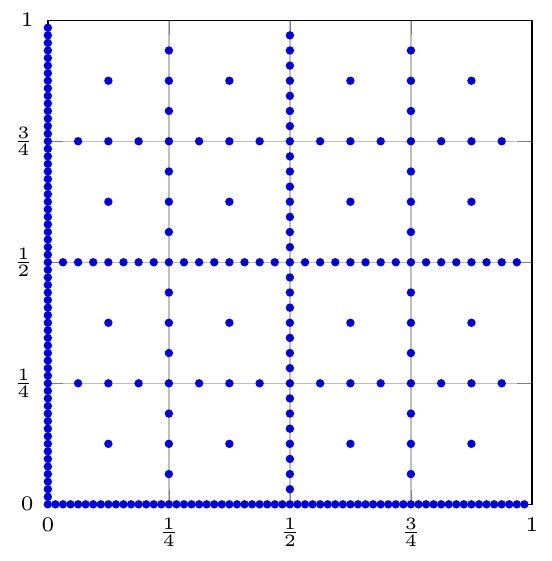}
 \caption{A sparse grid in $d=2$ with $N=256$ points}\index{Sparse grid}
 \end{center}
 \end{figure}

Here we have a degree of freedom when choosing
the weights $\{\lambda_i\}_{\bx^i \in G^d(\ell)}$. However, it turns out
that any cubature formula on Smolyak grids behaves asymptotically worse than the
optimal rules
discussed below.
The following theorem presents the approach from
\cite{DU14}. It relies on the same principle as above of using superpositions of
local bump functions to fool the algorithm. Taking into account that a sparse
grid of level $\ell$ contains $m\asymp 2^{\ell}\ell^{d-1}$ points we obtain the
following general result.

\begin{thm}\label{Thm:cubsmolyak} Let $1\leq p,\theta \leq \infty$ and
$r>1/p$.\\ 
{\em (i)} Then
\begin{equation}\label{f1}
  \begin{split}
     \Lambda_m(\Brpt,SG^d(\ell)) &\gtrsim
2^{-\ell r}\ell^{(d-1)(1-1/\theta)}
     \\&\asymp m^{-r}(\log m)^{(d-1)(r+ 1-1/\theta)}\,.
  \end{split}  
\end{equation}
{\em (ii)} If $1<p <\infty$ and $r>0$ then 
\begin{equation}\label{f2}
  \begin{split}
     \Lambda_m(\Wrp,SG^d(\ell)) &\gtrsim 2^{-\ell r}\ell^{(d-1)/2}
     \\&\asymp m^{-r}(\log m)^{(d-1)(r+1/2)}\,.
  \end{split}  
\end{equation}

\end{thm}

\begin{rem}\label{uppersmol} Those lower bounds show that numerical integration
on Smolyak grids is ``not easier'' than Smolyak sampling, see Section
\ref{Sect:Smolyak} above. However, it is also not harder. Note, that the results
in Section \ref{Sect:Smolyak} show that there is a sampling algorithm of the
form $$A_m(f,G^d(\ell))=\sum\limits_{\bx^i \in G^d(\ell)}
f(\bx^i)a_{\bx^i}(\cdot)$$
which approximates $f$ well in $L_1(\T^d)$. Taking 
$$
    \Big|\int_{\T^d}f(\bx)d\bx -  \int_{\T^d}\sum\limits_{\bx^i \in G^d(\ell)}
f(\bx^i)a_{\bx^i}(\bx) d\bx\Big| \leq \|f-A_m(f,G^d(\ell))\|_1
$$
into account we obtain with $\lambda_{\bx^i}:=\int_{\T^d}a_{\bx^i}(\bx)d\bx$ a
cubature formula of type \eqref{cfsmolyak} where the error is bounded by the
approximation error in $L_1(\T^d)$. Hence, by the results presented in Section
\ref{Sect:Smolyak} we see that the lower bounds \eqref{f1} and \eqref{f2} are
sharp in the sense, that there is a cubature formula on the Smolyak grid with
matching upper bounds. 
\end{rem}

In \cite{VT152} Temlyakov observed that the same bounds can be also obtained
from the approach based on the general Theorem \ref{T2.2.1}. Let us  briefly
discuss how Theorems \ref{T4.2} and
\ref{T4.4} imply bounds \eqref{f1} and \eqref{f2} (see \cite{VT152}). It is easy
to
check that $SG^d(n)\subset W(\bs)$ with any $\bs$ such
that $|\bs|_1=n$. Indeed, let $\xi(\bn,\bk)\in SG^d(n)$. Take any $\bs$ with
$|\bs|_1=n$. Then $|\bs|_1=|\bn|_1$ and there exists $j$ such that $s_j\ge
n_j$. For this $j$ we have
$$
\sin 2^{s_j}\xi(\bn,\bk)_j = \sin 2^{s_j}\pi k_j 2^{-n_j} =0\quad \text{and}
\quad w(\bs,\xi(\bn,\bk)=0.
$$
This means that $SG^d(n)$ is an $(n,l)$-net for any $l$.
We note that $|SG^d(n)|\asymp 2^n n^{d-1}$. It is known (see \cite{T7}) that
there
exists a cubature formula $(\Lambda,SG^d(n))$ such that
\be\label{4.6}
\Lambda(\bH^r_p,G^d(n)) \lesssim 2^{-rn}n^{d-1},\quad 1\le p\le \infty,\quad
r>1/p.
\ee
Theorem \ref{T4.4} with $\theta=\infty$ shows that the bound (\ref{4.6}) is
sharp. Moreover, 
Theorem \ref{T4.4} shows that even an addition of extra $2^{n-1}$ arbitrary 
nodes to $SG^d(n)$ will not improve the bound in (\ref{4.6}). 

Novak and Ritter \cite{NoRi96} studied Smolyak cubature based on the one-dimensional 
Clenshaw-Curtis rule with numerical experiments. Later, Gerstner and Griebel \cite{GeGr98} investigated and compared several 
variants of Smolyak cubature rules based on different one-dimensional quadrature schemes like Gauss (Patterson), trapezoidal and Clenshaw-Curtis quadrature 
\index{Cubature formula!Clenshaw-Curtis}rules. 
Their numerical experiments show that Gauss (Patterson) performs best (among the other methods mentioned) for the considered examples. 
The Clenshaw-Curtis Smolyak cubature rule has been also considered 
from the viewpoint of tractability for a class of infinite times differentiable functions. 
In Hinrichs, Novak, M. Ullrich \cite{HiNoUl14} the authors prove ``weak tractability'' in this setting based on the polynomial exactness of this cubature rule.

\subsection{The Fibonacci cubature formulas}\index{Cubature formula!Fibonacci}
\label{Sect:Fib}

The Fibonacci cubature formulas have been studied by several authors, see for
instance Bakhvalov \cite{Bakh1}, Temlyakov \cite{Tem24, VT49, Tem25},
D\~ung, Ullrich \cite{DU14} and the references in \cite[Chapt.\ IV]{TBook}.
This cubature formula behaves asymptotically optimal within the bivariate
function classes of interest here. It represents a rank-$1$-lattice rule.
Note, that for higher dimensions optimal rank-$1$-lattice rules are not known. 

\begin{thm}\label{T6.3.1} {\em (i)} Let $d=2$. For $1<p<\infty$ and
$r>\max\{1/p,1/2\}$ we
have
$$
\kappa_m(\bW^r_p(\T^2))\asymp m^{-r} (\log m)^{1/2}.
$$
{\em (ii)} For $1 \leq p,\theta \le \infty$, and $r>1/p$ we have
$$
\kappa_m(\bB^r_{p,\theta}(\T^2))\asymp m^{-r} (\log
m)^{1-1/\theta}.
$$
\end{thm}
Note that all the lower estimates are provided by the results in Subsection \ref{sect:lbounds_int} above.
The upper bounds in both above theorems are obtained by the use of the Fibonacci
cubature formulas. The Fibonacci  cubature formulas are defined as follows 
$$     
\Phi_n(f) := b_n^{-1} \sum_{\mu=1}^{b_n}  f(\mu/b_n,  \,\{\mu
b_{n-1}/b_n \}), 
$$

\begin{figure}[H]
 \begin{center}
 \includegraphics[scale = 1.3]{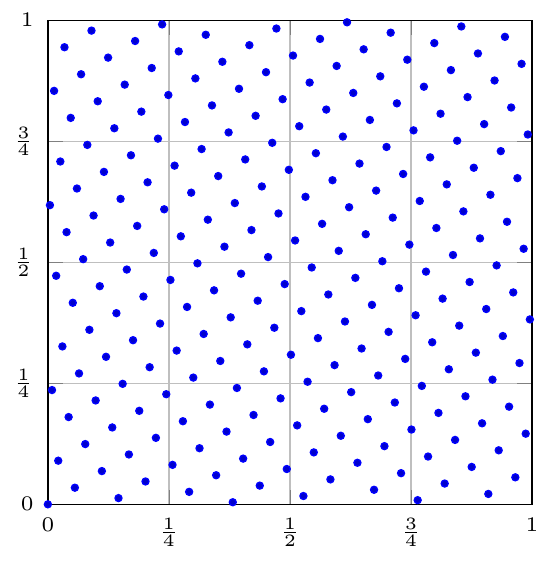}
 \caption{A Fibonacci lattice with $N=233$ points}\index{Lattice!Fibonacci}
 \end{center}
 \end{figure}

\noindent where $b_0=b_1=1$,  $\;  b_n=b_{n-1}+b_{n-2}$ are the Fibonacci
numbers and
$\{x\}$ is the fractional part of the number $x$. Those cubature formulas are
designed for periodic functions and they are exact on the trigonometric
polynomials with frequencies in the hyperbolic cross $\Gamma(\gamma b_n)$,
where $\gamma>0$ is a universal constant.  

For a function class $\bF(\T^2)$ of bivariate functions we denote
$$
     \Phi_n(\bF(\T^2)) :=    \sup_{f    \in     \bF(\T^2)}\Big|\Phi_n(f)     -  
 \int_{\T^2} \! f(\bx) \, d\bx\Big|\,,
$$
where $\T^2=[0,1]^2$ represents the $2$-torus.

\begin{prop}\label{Prop:Fib} {\em (i)} For $1<p<\infty$ and $r>\max\{1/p,1/2\}$
we
have
$$
\Phi_n(\bW^r_p(\T^2))\asymp b_n^{-r} (\log b_n)^{1/2}.
$$
{\em (ii)} For $1 \leq p,\theta \le \infty$, and $r>1/p$ we have
$$
\Phi_n(\bB^r_{p,\theta}(\T^2))\asymp b_n^{-r} (\log
b_n)^{1-1/\theta}.
$$
\end{prop}

Proposition \ref{Prop:Fib} and Theorem \ref{T6.3.1} show that the
Fibonacci cubature formulas are optimal in the sense of order in the stated situations. 

\subsubsection*{Small smoothness and limiting cases}\index{Small smoothness}

Here we are particularly interested in the limiting situations
$\mathbf{B}^{1/p}_{p,1}(\T^2)$, $\mathbf{W}^r_1(\T^2)$ and the situation of
``small smoothness'' occurring whenever we deal with $\Wrp(\T^2)$ with
$2<p<\infty$ and $1/p<r\leq 1/2$, which is not covered by Theorem \ref{T6.3.1} and Proposition \ref{Prop:Fib}.

\begin{figure}[H]
  \begin{center}
  \begin{tikzpicture}[scale=3]

\draw[->] (-0.1,0.0) -- (1.1,0.0) node[right] {$\frac{1}{p}$};
\draw[->] (0.0,-.1) -- (0.0,1.1) node[above] {$r$};
\draw (1.0,0.03) -- (1.0,-0.03) node [below] {$1$};
\draw (0.03,1.0) -- (-0.03,1.00) node [left] {$1$};
\draw (0.03,0.5) -- (-0.03,0.5) node [left] {$\frac{1}{2}$};
\draw (1.0,0) -- (1.0,1.0);
\draw (0,1) -- (1.0,1.0);
\filldraw[fill=lightgray,draw=black] (0,0.5)--(0.5,0.5)--(0,0)--(0,0.5); 
\draw (0,0.5) -- (0.5,0.5);
\draw (0,0) -- (1,1);
\draw (1/2,0.03) -- (1/2,-0.03) node [below] {$\frac{1}{2}$};
\end{tikzpicture}
\caption{The region of ``small smoothness''}
\end{center}
\end{figure}
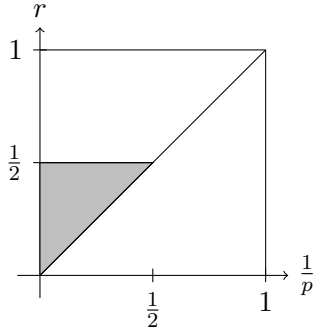

\begin{prop}\label{W1inf} {\em (i)} Let $r>1$ then it holds 
$$
    \Phi_n(\mathbf{W}^r_{1,\alpha}(\T^2)) \asymp b_n^{-r}\log b_n\,.
$$
{\em (ii)} If $r>1/2$ we have 
$$
    \Phi_n(\mathbf{W}^r_{\infty}(\T^2)) \asymp b_n^{-r}\sqrt{\log b_n}\,.
$$
{\em (iii)} If $1\leq p <\infty$ then 
$$
    \Phi_n(\mathbf{B}^{1/p}_{p,1}(\T^2) \asymp b_n^{-1/p}\,.
$$

\end{prop}

Relation (i) is proved in \cite{Tem25}. The lower bound in (i) for $\alpha = 0$ follows from Theorem \ref{Thm:Wr10}. 
The upper bound in part (iii) follows directly from (3.17) in \cite{DU14}. for the lower bound we refer to
\cite[Thm.\ 7.3]{UU14}. Note, that the embedding $\mathbf{B}^{1/p}_{p,1} \hookrightarrow C(\T^d)$ holds true, see
Lemma \ref{emb},(iii) above\,. 

\noindent Next we state the results for small smoothness. 

\begin{prop}\label{T6.3.3} Let $2<p<\infty$ and $1/p<r\leq 1/2$. Then 
$$
     \Phi_n(\bW_{p}^r) \asymp \begin{cases}
     b_n^{-r}(\log b_n)^{1-r},&1/p<r<1/2;\\
     b_n^{-r}\sqrt{(\log b_n) (\log \log b_n)},& 
 r=1/2.
\end{cases}
$$
\end{prop}
One observes a different behavior in the $\log$-power and an additional
$\log\log$ if $r=1/2$. A similar effect seems to hold in the multivariate situation
when dealing with Frolov's cubature, see Theorem \ref{cor:sobolev} below for the upper bounds. Note, that in contrast to
the result for the Fibonacci cubature rule there are so far no sharp lower bounds for the multivariate
situation. 

\subsection{The Frolov cubature formulas}\index{Cubature formula!Frolov}
\label{Sec:frolov}

The Frolov cubature formulas were introduced and studied in \cite{Fro1,Fro2}.
The reader can find a detailed discussion of this topic in \cite[Chapt.\
IV]{TBook}, \cite{Tsurv}, and \cite{MU14}. For the analysis of this method in
several Besov and Triebel-Lizorkin type spaces of mixed smoothness we refer to
Dubinin \cite{Du,Du2} and the recent papers \cite{UU14, NUU15}. The Frolov
cubature formulas are used in the proof of the upper bounds in the following
theorem. 

\begin{thm}\label{Thm:Fro1} {\em (i)} For $1<p<\infty$ and $r>\max\{1/p,1/2\}$
we
have
$$
\kappa_m(\bW^r_p)\asymp m^{-r} (\log m)^{(d-1)/2}.
$$
{\em (ii)} For $1\leq p,\theta \le \infty$, and $r>1/p$ we have
$$
\kappa_m(\bB^r_{p,\theta})\asymp m^{-r} (\log
m)^{(d-1)(1-1/\theta)}.
$$
{\em (iii)} For $1 \leq p < \infty$  we have 
$$
\kappa_m(\bB^{1/p}_{p,1})\asymp m^{-1/p}.
$$
\end{thm}

The lower bounds in (i) are provided by Theorem \ref{T6.2.1a} and in (ii) by
Theorem \ref{T6.2.3}.  The lower bound in (iii) is proven in
\cite[Thm.\ 7.3]{UU14}. 

Contrary to the case of the Fibonacci (see above) and the Korobov (see below)
cubature formulas  the Frolov cubature formulas defined below are
not designed for a direct application to periodic functions. As a result one
needs to use a two step strategy. We begin with a definition. 
 The following results, see for instance \cite[IV.4]{TBook} and \cite{KaOeUl16}, play a fundamental role in the construction of these formulas. 

\begin{thm}\label{L6.4.1} There exists a matrix $A$ such that the lattice
$L(\bm) = A\bm$, i.e., 
$$L(\bm) :=
\begin{pmatrix}
L_1(\bm)\\
\vdots\\
L_d(\bm)
\end{pmatrix}\,,$$
where $\bm$ is a {\rm(}column{\rm)}
vector with integer coordinates,
has the following properties
 
\begin{itemize}
 \item[(1)] $\qquad \left |\prod_{j=1}^d L_j(\bm)\right|\ge 1$
for all $\bm \neq \mathbf 0$;
\item[(2)] each parallelepiped $P$ with volume $|P|$
whose edges are parallel
to the coordinate axes contains no more than $|P| + 1$ lattice
points.
\end{itemize}
\end{thm}

 There is a constructive approach in choosing the lattice generating matrix $A$. In the original paper by
Frolov \cite{Fro1} a Vandermonde matrix
\begin{equation}\label{vanderm}
    A = \left(\begin{array}{cccc}
                1 & \xi_1 & \cdots & \xi_1^{d-1}\\
                1 & \xi_2 & \cdots & \xi_2^{d-1}\\
                \vdots & \vdots & \ddots & \vdots\\
                1 & \xi_{d} & \cdots & \xi_d^{d-1}
              \end{array}\right)\,
\end{equation}
has been considered, where $\xi_1,\ldots,\xi_d$ are the real roots of an irreducible polynomial over $\mathbb{Q}$,
e.g., $P_d(x):=\prod_{j=1}^d (x-2j+1)-1$\,. The general principle of this construction has been elaborated in detail by Temlyakov
in his book \cite[IV.4]{TBook} based on results on algebraic number theory, see Borevich, Shafarevich \cite{BoSh66}, Gruber, Lekkerkerker \cite{GrLekk87}, or Skriganov \cite{Sk}. 
Let us also refer to \cite{Ka16} for a detailed exposition of the construction principle based on the above mentioned references. 
The polynomial $P_d$ has a striking disadvantage from a numerical analysis point of view, namely that the real roots of the polynomials grow with $d$ and
therefore the entries in $A$ get huge due to the Vandermonde structure. In fact, sticking to the structure
\eqref{vanderm}, it seems to be a crucial task to find proper irreducible polynomials with real roots of small modulus. In
\cite[IV.4]{TBook} Temlyakov proposed the use of rescaled Chebyshev polynomials \index{Chebychev polynomials}$Q_d$ in dimensions $d=2^\ell$. To be more precise we use
for $x\in [-2,2]$
\begin{equation}\nonumber
Q_d(x) = 2 T_d(x/2)\quad\mbox{with}\quad T_d(\cdot):=\cos(d\arccos(\cdot))\,.
\end{equation}
The polynomials $Q_d$ belong to $\mathbb{\zz}[x]$ and have leading coefficient $1$. Its roots are real and given by
\begin{equation}\label{zeros}
    \xi_{k} = 2\cos\Big(\frac{\pi(2k-1)}{2d}\Big) \quad,\quad k=1,...,d\,.
\end{equation}
Let us denote the Vandermonde matrix \eqref{vanderm} with the scaled Chebyshev roots \eqref{zeros} by $T$ 
and call the corresponding lattice $\Gamma_T = T(\Z^d)$ a {\em Chebyshev lattice}.
It turned out recently, see \cite{KaOeUl16}, that a Chebychev lattice is always orthogonal. To be more precise we have the following result. 

\begin{thm}\label{cf1} The $d$-dimensional Chebyshev lattice $\Gamma_T = T(\Z^d)$ is orthogonal. In particular, there
exists a lattice representation $\tilde{T}=TS$ with $S\in \text{SL}_d(\zz)$ such that 
\begin{itemize}
 \item[(i)] $\tilde{T}_{k,\ell} \in [-2,2]$ for $k,\ell  =1,...,d$ and
 \item[(ii)] $\tilde{T}^{\ast}\tilde{T} = \mathrm{diag}(d, 2d,\ldots,2d)$.
 \end{itemize}
\end{thm}

However, Chebyshev-polynomials are not always irreducible over $\mathbb{Q}$. In fact, the polynomials $Q_d$ are irreducible if and only if $d = 2^\ell$ \cite[IV.4]{TBook}. Hence, 
a Chebyshev lattice $\Gamma_T$ is admissible in the sense of Theorem \ref{L6.4.1} if and only if $d=2^\ell$. In that case we call $\Gamma_T$ a {\em Chebyshev-Frolov lattice} and obtain the following corollary. 
\begin{cor}\label{cor_cheb} If $d=2^\ell$ for some $\ell\in \N$ the Chebyshev-Frolov lattice \index{Lattice!Chebychev-Frolov}$\Gamma_T = T(\Z^d)$ and its dual lattice are both 
orthogonal and admissible (in the sense of Theorem \ref{L6.4.1}). In particular, there is a lattice representation
for $\Gamma$ given by $\tilde{T} = Q D$ with a diagonal matrix $D = \mathrm{diag}(\sqrt{d},\sqrt{2d},...,\sqrt{2d})$  and an orthogonal matrix $Q$. For the dual
lattice $\Gamma^{\perp}$ we have the representation $\tilde{T}^{\perp} =Q D^{-1}$. 
\end{cor}

\noindent Let $a>1$ and $A$ a matrix from Lemma \ref{L6.4.1}. We consider the cubature formula
\begin{equation}\label{cub_frolov}
    \Phi(a,A)(f):=(a^d|\det A|)^{-1}\sum\limits_{\bm \in
\Z^d}f\Big(\frac{(A^{-1})^T\bm}{a}\Big)
\end{equation}
for $f$ with support in $[0,1]^d$\,. Clearly, the number $N$ of points of this
cubature formula does not exceed $C(A)a^d |\det A|$. In our framework \eqref{cub_def} the weights $\lambda_i$,
$i=1,...,m$, are all equal but do not sum up to one in general. The following figure
illustrates the construction of the Frolov points.

\begin{figure}[H]
 \begin{center}
 \includegraphics[scale = 1.3]{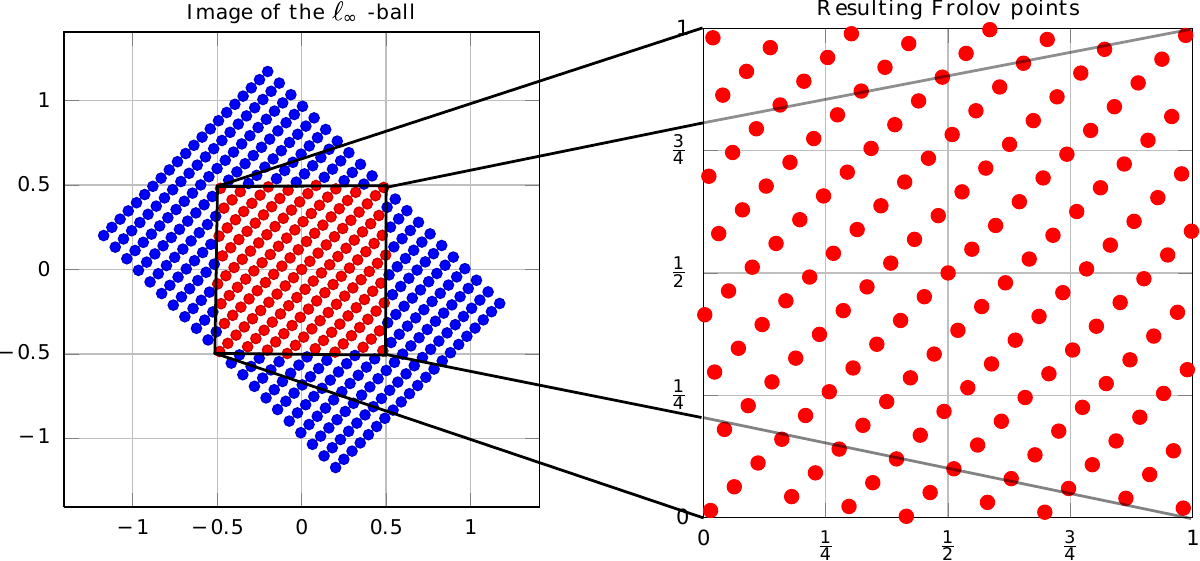}
 \caption{Generating the Frolov points}\index{Lattice!Frolov}
 \end{center}
 \end{figure}

The first step in application of the Frolov cubature formulas is to apply them
to analogs of classes $\bW^r_p$ and $\bB^r_{p,\theta}$ of respective functions
with support in $[0,1]^d$. Let us first give a precise definition of these
classes by using the approach via differences from
Lemmas \ref{diff}, \ref{fdiff}. Let $1<p<\infty$ and $m>r>0$ then
$\bW^r_p(\R^d)$ is the collection of functions $f\in L_1(\R^d)$ such that 
$$
  \|f\|^{(m)}_{\bW^r_p(\R^d)}:=\Big\|\Big(\sum\limits_{\bs\in \N_{0}^d}
2^{r|\bs|_12}\mathcal{R}^{e(\bs)}_m(f,2^{-\bs},\cdot)^2\Big)^{1/2}
\Big\|_{L_p(\R^d)}
$$
is finite. If $1\leq p,\theta \leq \infty$ and $m>r>0$
then $\bB^r_{p,\theta}(\R^d)$ is the collection of all functions $f \in
L_1(\R^d)$ such that 
$$
  \|f\|^{(m)}_{\Brpt(\R^d)} := \Big[\sum\limits_{\bs\in \N_{0}^d} 2^{r|\bs|_1
\theta}\omega_{m}^{e(\bs)}(f,2^{-\bs})^{\theta}_{L_p(\R^d)}\Big]^{1/\theta}
$$
is finite. Now we consider those functions from $\bW^r_p(\R^d)$ which are
supported in the cube $[0,1]^d$, by setting
$$
      \mathring{\bW}^r_p:=\{f\in \bW^r_p(\R^d)~:~\supp f \subset [0,1]^d\}
$$
and 
$$
      \mathring{\bB}^r_{p,\theta}:=\{f\in \Brpt(\R^d)~:~\supp f \subset
[0,1]^d\}\,.
$$
For the moment let us use the letter $A$ also for the rescaled Frolov matrix
(w.l.g. positive determinant). Then \eqref{cub_frolov} can be rewritten to (see \cite[Lem.\ IV.4.6]{TBook})
\begin{equation}\label{poi}
  \frac{1}{\det A}\sum\limits_{\bm \in \Z^d}f((A^{-1})^T\bm) = \sum\limits_{\bk\in
\Z^d} \cf f(A\bk)\,,
\end{equation}
which is a consequence of Poisson's summation formula
involving the continuous Fourier transform given by
$$
    \cf f (\by) := \int_{-\infty}^{\infty} f(\bx)e^{-2\pi i (\by,\bx)}\,d\bx\,.
$$
Before commenting on this identity let us mention that it immediately gives a
formula for the integration error since $\cf f(0) = \int_{[-1,1]^d}f(\bx)d\bx$.
In \eqref{poi} the left-hand side is a finite
sum due to the support assumption on $f$. However, the right-hand side does not
have to be unconditionally convergent. However, any convergent method of summation of the Fourier series can be used to
establish the above identity. For details we refer to \cite[Cor.\ 3.2]{UU14}. The formula \eqref{poi} is the heart of
the matter for the analysis of the
method in spaces $\mathring{\bW}^r_p$ and $\mathring{\bB}^r_{p,\theta}$. The
first step consists in proving a counterpart of Theorem
\ref{Thm:Fro1} for spaces $\mathring{\bW}^r_p$ and
$\mathring{\bB}^r_{p,\theta}$, see \cite{TBook, Tsurv, UU14}.

\begin{thm}\label{thm:besovsobolev} {\em (i)} For each $1<p<\infty$ and
$r>\max\{1/p,1/2\}$ we have 
\[
    \Phi(a,A)(\mathring{\bW}^r_p) \,\asymp\, a^{-rd} (\log
a)^{(d-1)/2}\quad,\quad a>1\,.
\]
{\em (ii)} For each $1 \leq  p,\theta\le\infty$ and $r>1/p$, we have 
\[
 \Phi(a,A)(\mathring{\bB}^r_{p,\theta}) \,\asymp\, a^{-rd} (\log
a)^{(d-1)(1-1/\theta)}\quad,\quad a>1\,.
\]
{\em (iii)} For each $1 \leq  p < \infty$  we have 
\[
   \Phi(a,A)(\mathring{\bB}^{1/p}_{p,1}) \,\asymp\, a^{-d/p}\quad,\quad
a>1\,.
\]
\end{thm}

Note, that in case $\theta = 1$ the
$\log$ term disappears and multivariate cubature shows the same asymptotical
behavior as univariate quadrature in this setting.

\subsubsection*{The case of small smoothness}\index{Small smoothness}
Here we deal with the cubature of functions from classes $\Wrp$ where
$2<p<\infty$ and $1/p<r\leq 1/2$. Using Frolov's cubature formula we can prove a
multivariate counterparts of the upper bounds in Theorem \ref{T6.3.3}. Indeed using the modified
Frolov cubature $\Phi(a,A)(f)$ we obtain 

\begin{thm}\label{cor:sobolev} Let $2<p<\infty$ and $1/p<r\leq 2$. Then the
following bounds hold true for $a>2$
$$
     \Phi(a,A)(\mathring{\bW}_{p}^r) \lesssim \begin{cases}
     a^{-rd}(\log a)^{(d-1)(1-r)},&1/p<r<1/2;\\
     a^{-rd}(\log a)^{(d-1)/2}\sqrt{\log\log a} ,& 
 r=1/2.
\end{cases}
$$
\end{thm}
\noindent In contrast to the Fibonacci cubature rules the lower bounds in this
situation are still open.

\subsection{Modifications of Frolov's method}

Frolov's method works well for functions with zero boundary condition, see Theorem \ref{thm:besovsobolev} above. 
In order to treat periodic functions from $\Brpt$ and $\Wrp$ with Frolov's
method we study the following recently developed  modification of the algorithm, 
which has been introduced in \cite{NUU15}. 
Let $\psi:\R^d
\to
[0,\infty)$
be a compactly supported function ($\supp \psi \subset \Omega$) such that 
$$
      \sum\limits_{\bk\in \Z^d} \psi(\bx+\bk) = 1\quad,\quad \bx \in \R^d\,.
$$
Then for any $1$-periodic function in each component we have 
$$
    \int_{\R^d} \psi(\bx)f(\bx)\,d\bx = \sum\limits_{\bk \in \Z^d}
\int_{[0,1]^d}\psi(\bx+\bk)f(\bx+\bk)\,d\bx
    = \int_{[0,1]^d} f(\bx)\sum\limits_{\bk \in \Z^d} \psi(\bx+\bk)\,d\bx =
\int_{[0,1]^d} f(\bx)\,d\bx\,.
$$
Let $\X_a$ be the Frolov points generated by the rescaled matrix
$\frac{1}{a}(A^{-1})^T$. The modified cubature formula is then
$$    
    \Phi(a,A,\psi)(f) := \Phi(a,A)(\psi f) = (a^d \det A)^{-1}\sum\limits_{\bx
\in \X_a \cap \Omega} \psi(\bx) f(\{\bx\})\,,
$$
where $\{\bx\}\in [0,1)^d$ denotes the fractional part of the components in $\bx$.
Proving the boundedness of the operator $M_{\psi}:\bW^r_{p} \to 
\bW^r_{p}(\R^d)$ which maps $f \mapsto \psi f$, see
\cite{NUU15}, one ends up
with
$$
   \Phi(a,A)(\mathring{\bW}^r_p) \asymp \Phi(a,A,\psi)(\bW^r_p)\quad,\quad a>1\,,
$$
as well as 
$$
  \Phi(a,A)(\mathring{\bB}^r_{p,\theta}) \asymp \Phi(a,A,\psi)(\Brpt)\quad,\quad a>1\,.
$$
which implies Theorem \ref{Thm:Fro1}. 

\subsubsection*{Change of variable}\index{Change of variable}

Let us also mention a classical modification of the algorithm which goes back to Bykovskii \cite{By}, see
also \cite{TBook, Tsurv} for details.\\\noindent Let $\ell$ be a natural number. We define the following
functions
$$
\psi_\ell (u) =
\begin{cases}
\int_0^ut^\ell(1-t)^\ell dt\biggm/\int_0^1t^\ell(1-t)^\ell dt,\qquad&
u\in[0,1],\\
0\qquad & u<0,\\
1\qquad & u>1;
\end{cases}
$$
For continuous functions of $d$ variables we define the
cubature formulas
\begin{equation}\label{modi_Frolov}
  \Phi(a,A,\ell)(f) :=
\Phi(a,A)\Big(f(\psi_{\ell}(x_1),...,\psi_{\ell}(x_d))\prod\limits_{i=1}^d
\psi'_\ell(x_i)\Big)\,.
\end{equation}
Analyzing this method boils down to mapping properties of the operator 
$$
    T_{\ell}:f \mapsto
f(\psi_{\ell}(x_1),...,\psi_{\ell}(x_d))\prod\limits_{i=1}^d
\psi'_\ell(x_i)
$$
within $\bW^r_p$ and $\bB^r_{p,\theta}$. This has been studied for spaces
$\bW^r_p$ with $r\in \N$ in \cite{TBook}, \cite{Tsurv} and
for spaces $\bB^r_{p,\theta}$, $r>1/p$, in \cite{Du2}. The full range of
relevant spaces, including fractional smoothness, is considered in the paper
\cite{NUU15}. The final result is 
$$
   \Phi(a,A)(\mathring{\bW}^r_p) \asymp \Phi(a,A,\ell)(\bW^r_p(\R^d))\quad,\quad a>1\,,
$$
as well as 
$$
  \Phi(a,A)(\mathring{\bB}^r_{p,\theta}) \asymp \Phi(a,A,\ell)(\Brpt(\R^d))\quad,\quad a>1\,,
$$
if $\ell$ is large (depending on $p$ and $r$)\,. This shows that the spaces
$\bW^r_p(\R^d), \Brpt(\R^d)$ and $\bW^r_p, \Brpt$  also show the rate of convergence given in Theorem
\ref{Thm:Fro1} with respect to $\kappa_m$.

\subsubsection*{Random Frolov}\index{Cubature formula!Random Frolov}

Further modifications of Frolov's method can be found in the recent papers Krieg, Novak 
\cite{KrNo16} and M. Ullrich \cite{MU16}. The authors define a Monte Carlo integration method $M_m^{\omega}$ with $m$ 
nodes based on the Frolov cubature formula and 
obtain the typical $\min\{1/2,1-1/p\}$ gain in the main rate of convergence. In addition, M. Ullrich \cite{MU16} 
observed the surprising phenomenon that the logarithm disappears in the rooted mean square error (variance), i.e.
$$
      \sup\limits_{f\in \bW^r_p} \big(\mathbb{E}|I(f)-M_m^{\omega}(f)|^2\big)^{1/2} \lesssim m^{-(r+1-1/p)}\,.
$$
if $1<p<\infty$ and $r\geq (1/p-1/2)_+$\,.

\subsection{Quasi-Monte Carlo cubature}\index{Quasi-Monte Carlo}
Theorems \ref{T6.2.1a}, \ref{T6.2.3} provide the lower bounds for numerical integration by cubature rules with $m$
nodes. Theorem \ref{Thm:Fro1} shows that in a large range of smoothness $r$ there exist cubature rules with $m$ nodes,
which provide the upper bounds for the minimal error of numerical integration matching the corresponding lower bounds.
This means, that in the sense of asymptotical behavior of $\kappa_m(\bF)$ the problem is solved for a large range of
function classes. In the case of functions of two variables optimal cubature rules are very simple -- the Fibonacci
cubature rules. They represent a special type of cubature rules, so-called quasi-Monte Carlo rules. In the case $d\geq
3$ the optimal (in the sense of order) cubature rules, considered in the previous Section, are constructive but not as
simple as the Fibonacci cubature formulas. In fact, for the Frolov cubature formulae
\eqref{cub_frolov} all weights $\lambda_i$, $i=1,...,m$, in \eqref{cub_def} are equal but do not sum up to one. This
means in particular that constant functions would not be integrated exactly by Frolov's method. Equal weights which
sum up to one is the main feature of quasi-Monte Carlo integration. In this Section we will discuss the problem of
finding optimal quasi-Monte Carlo rules. 

Let us start with the following result. 

\begin{thm}\label{T6.4.1} For $0<r<1$ one has
\be\nonumber
\kappa_m(\bH^r_{\infty}) \asymp m^{-r}(\log m)^{d-1}.
\ee
\end{thm}
The lower bound in the above theorem follows from Theorem \ref{T6.2.3}.
The upper bound is provided by the Korobov cubature formula\index{Cubature formula!Korobov}.  Let $m\in\N$, $\ba
= (a_1,\dots,a_d)\in\Z^d$.
We consider the cubature formulas
$$
P_m (f,\ba):= m^{-1}\sum_{\mu=1}^{m}f\left (\left \{\frac{\mu a_1}
{m}\right\},\dots,\left \{\frac{\mu a_d}{m}\right\}\right)
$$
which are called the Korobov cubature formulas. Those formulas are of a special
type. Similar to the Frolov cubature formulas all weights are equal. However,
in contrast to \eqref{cub_frolov} we have the additional feature that
the equal weights sum up to $1$. Consequently, the Korobov cubature
formulas compute exactly
integrals of constant functions. 

The above considered Fibonacci cubature formulas represent a special case of
this framework. Namely, if $d=2$, $m=b_n$, $\ba = (1,b_{n-1})$ we have
$$
P_m (f,\ba) =\Phi_n (f).
$$
Note that in the case $d>2$ the problem of finding concrete
cubature formulas of the type $P_m (f,\ba)$ as good as the Fibonacci cubature
formulas in the
case $d=2$ is unsolved. However, in the sequel we will stick to the
property that the weights sum up to $1$. In other words we are interested in
quasi-Monte Carlo cubature formulas of type 
$$
  I_m(f,X_m) := \frac{1}{m}\sum\limits_{i=1}^m f(x^i)
$$
and ask whether there is a sequence of point sets $\{X_m\}_m$ such that
$I_m(\mathbf{F},X_m)$ provides asymptotically optimal error bounds for classes
of multivariate functions $\mathbf{F}$. For more details and further
references see also the recent survey by Dick, Kuo and Sloan \cite{DiKuSl13}.

It is well-known that numerical integration of functions with mixed smoothness is closely related to the discrepancy\index{Discrepancy}
problem, see Subsection \ref{numintdisc} below. In a certain sense (duality) classical discrepancy theory corresponds to numerical integration in function
classes $\bW^r_p$ with $r=1$. We refer the reader to the survey paper \cite{Tsurv} and the book \cite{Tr10} for a
detailed discussion of the connection between numerical integration and discrepancy. In the discrepancy problem the
points $X_m$ have to be constructed such that they are distributed as evenly as possible over the unit
cube. Explicit constructions of well-distributed point sets in the unit cube
have been introduced by Sobol \cite{So67} and by Faure
\cite{Fa82}. Later Niederreiter \cite{Ni87} introduced the general concept of
$(t,n,d)$-nets\index{Quasi-Monte Carlo!$(t,n,d)$-nets}. For such point sets, it has been shown
that the star discrepancy (the $L_{\infty}$-norm of the discrepancy function), a
measure of the distribution properties of a point set, behaves well. We do not
want to go into the complicated details of the construction of such point sets.
Let us rather take a look on a $2$-dimensional example. The $2$-dimensional
van der Corput point set \cite{Co1,Co2}\index{Quasi-Monte Carlo!van der Corput points} in base $2$ is given by 
\begin{equation}\nonumber
    \mathcal{H}_n = \Big\{\Big(\frac{t_n}{2}+\frac{t_{n-1}}{2^2}+ ...
    +\frac{t_1}{2^n},
    \frac{s_1}{2}+\frac{s_2}{2^2}+...+\frac{s_n}{2^n}\Big):t_i \in
    \{0,1\}, s_i = 1-t_i, i=1,...,n  \Big\}\,.
\end{equation}
and represents a digital $(0,n,2)$-net. Here, the quality parameter is $t=0$
since every dyadic interval with volume $2^{-n}$ contains exactly $1$ point. 

\begin{figure}[H]
 \begin{center}
 \includegraphics[scale = 1.3]{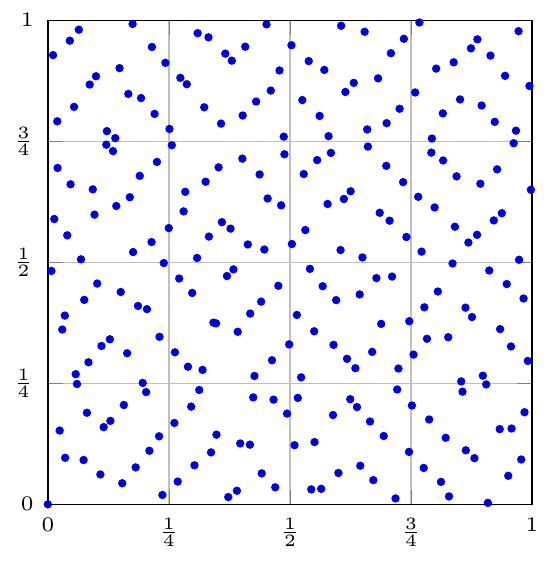}
 \caption{A digital net in $d=2$ with $N=256$ points}\index{Quasi-Monte Carlo!digital net}
 \end{center}
 \end{figure}

Based on the work of Hinrichs \cite{Hi10}, where Haar coefficients of the
discrepancy function have been computed,  the following theorem has been shown
in \cite{Ul14}
\begin{thm}\label{Ham} Let $1\leq p,\theta\leq \infty$ and $1/p<r<2$. Then
with $m=2^n$
$$
    I_m(\Brpt(\T^2), \mathcal{H}_n) \asymp m^{-r}(\log
m)^{1-1/\theta}\quad,\quad n\in \N\,.
$$
\end{thm}
The technical restriction $r<2$ comes from the hierarchical decomposition of a $2$-variate function
into the tensorized Faber-Schauder system. At least the hierarchical decomposition via the
Faber-Schauder system, see Subsection \ref{timelimsamprep} is problematic for $r>2$. 
Note, that the van der Corput quasi-Monte Carlo rule shows the same asymptotic
behavior as the Fibonacci cubature formula above in the given range for $r$,
i.e., they are asymptotically optimal. Via embeddings between $\bB$ and $\bW$
spaces, see Lemma \ref{emb}, we can
immediately deduce the corresponding asymptotically optimal result for Sobolev
spaces $\Wrp(\T^2)$ if $1<p<\infty$ and $2>r>\max\{1/p,1/2\}$\,, see
Theorem \ref{T6.3.1}(i). What concerns small smoothness\index{Small smoothness} we conjecture a counterpart of Proposition \ref{T6.3.3} 
also in this framework. 

Considering the $d$-dimensional situation, things are much more involved. Based
on the $b$-adic ($b$ large) Chen-Skriganov \cite{ChSk02} point construction
$\mathcal{CS}_n$ (which we will not discuss in detail) Markhasin \cite{Ma13}
showed the following. 

\begin{thm} Let $1\leq p,\theta \leq \infty$ and $1/p<r\leq 1$. Then with
$m=b^n$ 
$$
    I_m(\Brpt,\mathcal{CS}_n) \asymp m^{-r}(\log
    m)^{(d-1)(1-1/\theta)}\quad,\quad n\in \N\,.
$$
\end{thm}

The condition $r\leq 1$ is, of course, not satisfying. In order to prove a
$d$-dimensional counterpart of Theorem \ref{Ham} we need to use
higher order digital nets as recently introduced by Dick \cite{Di07}. Higher
order digital nets are designed to achieve higher order convergence rates if the
function possesses higher order smoothness, e.g. higher mixed partial
derivatives. In fact, classical $(t,m,d)$ nets often admit
optimal discrepancy estimates which (roughly) transfer to optimal integration
errors within the class $\bW^1_p$, see Subsection \ref{numintdisc} below. Several numerical experiments in
\cite{HiMaOeUl14} illustrate this fact. However, for higher order smoothness
corresponding statements are not known. What concerns order-$2$ digital nets
$\mathcal{DN}^2_n$ the following asymptotically optimal results have been shown
recently in \cite{HiMaOeUl14}. 

\begin{thm}\label{HiMaOeUl14} Let $1\leq p,\theta\leq \infty$ and $1/p<r<2$. Then with $m=2^n$
$$
   I_m(\Brpt,\mathcal{DN}^2_n) \asymp m^{-r}(\log
    m)^{(d-1)(1-1/\theta)}\quad,\quad n\in \N\,.
$$
\end{thm}
Compared to the Frolov cubature formulas the restriction $r<2$ is still
unsatisfactory. However, when it comes to the integration of so-called
``kink-functions'' from mathematical finance like, e.g., integrands of the
form $f(t) = \max\{0,t-1/2\}$ one observes a H\"older-Nikol'skii regularity 
$r=2$ if $p=1$. Hence the above method can take advantage of the maximal regularity of a kink.
This is what one observes in numerical experiments as well.  

Let us emphasize that the above stated results suffer from the restriction $r<2$ in case $d\geq 2$. It is not known
whether the optimal order of convergence can be achieved by a quasi-Monte Carlo rule in case of higher smoothness.
In case $d=2$ this is the case (Fibonacci). Note, that the modified Frolov method yields the optimal rate for higher
smoothness. However, this method is no QMC rule (not even the ``pure'' Frolov method). Let us mention that there is 
progress in this direction, see the recent preprints by Goda, Suzuki, Yoshiki \cite{GoSuYo15, GoSuYo16}, 
where higher order digital nets and corresponding quasi-Monte Carlo rules are used for non-periodic reproducing kernel Hilbert spaces, which extends 
Theorem \ref{HiMaOeUl14} in case $p=\theta = 2$ towards higher smoothness and non-periodic functions.\\

\subsection{Discrepancy and numerical integration}\index{Discrepancy}
\label{numintdisc}

\subsubsection*{Classical discrepancy}

In this chapter we discussed in detail the problem of numerical integration in the mixed smoothness classes. It is clear that for 
a formulation of the problem of optimal cubature formulas (optimal numerical integration rules) we need to specify a function class of 
functions which we numerically integrate. There are many different
function classes of interest. In this section we begin with a very simple class, which has a nice geometrical interpretation. The classical 
$L_\infty$ discrepancy (star-discrepancy) of a set $X_m =\{\bx^1,\dots,\bx^m\}$ is defined as follows
$$
D(X_m,L_\infty):= \sup_{\by\in [0,1]^d}\Big|\frac{1}{m}\sum_{j=1}^m \chi_{[\mathbf 0,\by]}(\bx^j) - \int_{[0,1]^d}\chi_{[\mathbf 0,\by]}(\bx)d\bx\Big|.
$$
Here, $\chi_{[\mathbf 0,\by]}(\bx):=\prod_{j=1}^d \chi_{[0,y_j]}(x_j), y_j\in [0,1], j=1,\dots,d$, where {$\chi_{[0,y]}(x)$}, { $y\in [0,1]$,} is a 
characteristic function of an interval $[0,y]$. 
Thus in this case the discrepancy problem is exactly the problem of 
numerical integration of functions from the class $\chi^d:=\{\chi_{[\mathbf 0,\by]}(\cdot)\}_{\by\in[0,1]^d}$ by cubature formulas with equal weights $1/m$. 

\begin{figure}[H]
    \begin{center}
        \begin{tikzpicture}[scale=5]
            \filldraw[fill=lightgray,draw=black] (0,0.4)--(0.65,0.4)--(0.65,0)--(0,0);
        \draw[->] (-0.1,0.0) -- (1.1,0.0);
        \draw[->] (0.0,-.1) -- (0.0,1.1);
        \draw (1.0,0.03) -- (1.0,-0.03) node [below] {$1$};
        \draw (0.03,1.0) -- (-0.03,1.00) node [left] {$1$};
        \draw (0.03,0.4) -- (-0.03,0.4) node [left] {$y_2$};
        \draw (1.0,0) -- (1.0,1.0);
        \draw (0,1) -- (1.0,1.0);

        \node at (0.1,0.1)  {$\bullet$};
        \node at (0.25,0.12)  {$\bullet$};
        \node at (0.2,0.34)  {$\bullet$};
        \node at (0.4,0.08)  {$\bullet$};
        \node at (0.37,0.28)  {$\bullet$};
        \node at (0.17,0.6)  {$\bullet$};
        \node at (0.13,0.73)  {$\bullet$};
        \node at (0.2,0.9)  {$\bullet$};
        \node at (0.33,0.65)  {$\bullet$};
        \node at (0.4,0.45)  {$\bullet$};
        \node at (0.52,0.59)  {$\bullet$};
        \node at (0.44,0.77)  {$\bullet$};
        \node at (0.58,0.2)  {$\bullet$};
        \node at (0.73,0.41)  {$\bullet$};
        \node at (0.68,0.11)  {$\bullet$};
        \node at (0.88,0.18)  {$\bullet$};
        \node at (0.84,0.83)  {$\bullet$};
        \node at (0.72,0.63)  {$\bullet$};
        \node at (0.94,0.74)  {$\bullet$};
        \node at (0.64,0.91)  {$\bullet$};
        \node at (0.68,0.45)  {$\bf y$};
        \draw (0.65,0.03) -- (0.65,-0.03) node [below] {$y_1$};
        \end{tikzpicture}
        \caption{The discrepancy function}\index{Discrepancy!Function}
    \end{center}
\end{figure}
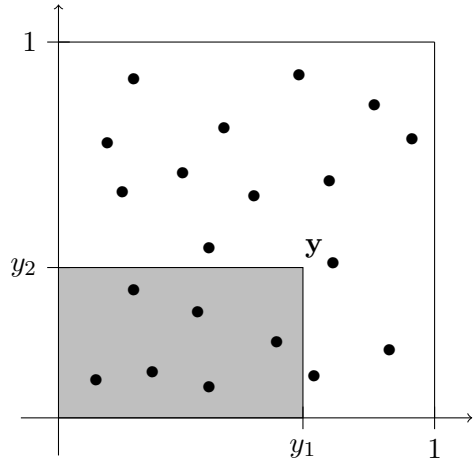

The description of the class $\chi^d$ is simple -- the functions in $\chi^d$ are labeled by $\by\in [0,1]^d$. This allows us to consider the 
$L_\infty$ norm, with respect to $\by$, of the error which represents a ``worst case''. In the case of the $L_p$ norms, $p<\infty$, of the error we may speak of 
an ``average case''. For $1\le p<\infty$ we define
$$
D(X_m,L_p):= \Big\|\frac{1}{m}\sum_{j=1}^m \chi_{[\mathbf 0,\by]}(\bx^j) - \int_{[0,1]^d}\chi_{[\mathbf 0,\by]}(\bx)d\bx\Big\|_p.
$$
It is easy to check that in this case
$$
D(X_m,L_p)= \sup_{f\in \dot{\bW}_{p'}^1 }\Big | \frac{1}{m}\sum_{j=1}^m f(\bx^j) -\int_{[0,1]^d} f(\bx) d \bx\Big|,
$$
where $ \dot{ \bW}_{p'}^1$  consists of the functions
$f(\bx)$ representable in the form
$$
f(\bx) =\int_{[0,1]^d} \chi_{[\mathbf 0,\by]}(\bx)\varphi(\by) d \by,
\qquad \|\varphi\|_{p'} \le 1.
$$
Again, in the case $1\le p<\infty$ the discrepancy problem is exactly the problem of numerical integration of functions from the class $\dot{ \bW}_{p'}^1$. Note the duality between $p$ in the discrepancy and $p'$ in the  class. Thus, the classical problem of the $L_p$ discrepancy coincides with the problem of numerical integration by the Quasi-Monte Carlo rules of 
the class $ \dot{ \bW}_{p'}^1$ of smoothness $1$. Usually, we consider the whole range of smoothness parameters $r$ in numerical integration.
Denote $\Omega_d:=[0,1]^d$. It will be convenient for us to consider the class
$\dot{ \bW}_{p}^r:= \dot{ \bW}_{p}^r (\Omega_d)$  consisting of the functions
$f(\bx)$ representable in the form
$$
f(\bx) =\int_{\Omega_d} B_r (\by, \bx)\varphi(\by) d \by,
\qquad \|\varphi\|_p \le 1,
$$
where
\begin{equation}
B_r( \by, \bx):= \prod_{j=1}^d\bigl((r-1)!\bigr)^{-1}
(y_j - x_j )_+^{r-1}\quad,\quad  \by, \bx\in\Omega_d\,.
\end{equation}
Note that in the case $r=1$ we have $B_1(\by,\bx) = \chi_{[\mathbf 0,\by)}(\bx)$. 
In connection with the definition of the class
$ \dot{ \bW}_{p}^r $ we remark here that
for the error of the cubature formula $\Lambda_m(\cdot,X_m)$ with
weights $\Lambda_m = (\lambda_1,\dots,\lambda_m)$ and nodes
$X_m = (\bx^1,\dots,\bx^m)$ the following relation holds. Let
$$
\Big |\Lambda_m(f,X_m) -\int_{\Omega_d} f( \bx) d \bx\Big| =:
R_m (\Lambda_m,X_m,f),
$$
then  
\begin{equation}\label{n1.15}
 \begin{split}
\Lambda_m\bigl( \dot{ \bW}_{p}^r,X_m\bigr)
&:= \sup_{f\in \dot{ \bW}_{p}^r}
R_m (\Lambda_m,X_m,f) \\
&=\left \|\sum_{\mu=1}^{m}\lambda_{\mu}B_r ( \by,\xi^{\mu})-
\prod_{j=1}^d (t_j^r /r!)\right\|_{p'} =:
D_r (X_m,\Lambda_m,d)_{p'}\,.
 \end{split}
\end{equation}
The quantity $D_r (X_m,\Lambda_m,d)_q$ in the case $r=1$,
$\Lambda_m = (1/m,\dots,1/m)$ is
 the classical $L_q$ discrepancy of the set of points $X_m$.  In the case 
$\Lambda_m=(1/m,\dots,1/m)$ we denote $D_r (X_m,d)_q:=D_r (X_m,(1/m,\dots,1/m),d)_q$ and call it the $r$-discrepancy (see \cite{Tem25} and \cite{Tsurv}).
Thus, the quantity $D_r (X_m,\Lambda_m,d)_{q}$ defined in (\ref{n1.15}) is a natural generalization of the concept of discrepancy 
\be\label{n1.15'}
D(X_m,L_q) = D(X_m,d)_{q} := D_1 (X_m,d)_{q}.
\ee
This generalization contains two ingredients: general weights $\Lambda_m$ instead of the special case of equal 
weights $(1/m,\dots,1/m)$ and any natural number $r$ instead of $r=1$. 
We note that in approximation theory we usually study the whole scale of smoothness classes rather than an individual smoothness class. 
The above generalization of discrepancy for arbitrary positive integer $r$ allows us to study the question: How does smoothness 
$r$ affect the rate of decay of generalized discrepancy?

The following result (see \cite{Tsurv}) connects the optimal errors of numerical integration for classes  $\dot{ \bW}_{p}^r$ 
and $\bW^r_p$. Recall the defintion of the quantities $\kappa_m$ in \eqref{eq:minimal} and \eqref{kappa2}. 
\begin{thm}\label{nT1.1} Let $1 \le p \le \infty$. Then for $r\in \N$
\be\label{n1.17}
 \kappa_m\bigl( \dot{ \bW}_p^r\bigr)\asymp
\kappa_m\bigl( \bW_p^r\bigr).
\ee
\end{thm}

\subsubsection*{General setting}  

Let us discuss two problems: (I) numerical integration for a function class defined by a kernel $K(\bx,\by)$; 
(II) discrepancy with respect to the collection of functions $\{K(\cdot,\by)\}_{\by\in[0,1]^d}$. As for Problem (I) 
let $1\le q\le \infty$. We define a set $\mathcal K_q$ of kernels possessing the following properties: 
Let $K(\bx,\by)$ be a measurable function on $\Omega^1\times\Omega^2$. We assume that for any $\bx\in\Omega^1$ we have $K(\bx,\cdot)\in L_q(\Omega^2)$; 
for any $\by\in \Omega^2$ the $K(\cdot,\by)$ is integrable over $\Omega^1$ and $\int_{\Omega^1} K(\bx,\cdot)d\bx \in L_q(\Omega^2)$. 
For a kernel $K\in \mathcal K_{p'}$ we define the class
\be\label{n2.1}
\bW^K_p :=\Big\{f:f=\int_{\Omega^2}K(\bx,\by)\varphi(\by)d\by,\quad\|\varphi\|_{L_p(\Omega^2)}\le 1\Big\}.  
\ee
Then each $f\in \bW^K_p$ is integrable on $\Omega^1$ (by Fubini's theorem) and defined at each point of $\Omega^1$. We denote for convenience
$$
J(\by):=J_K(\by):=\int_{\Omega^1}K(\bx,\by)d\bx.
$$
For a cubature formula $\Lambda_m(\cdot,X_m)$ we have
\begin{equation}\label{n2.2}
  \begin{split}
    \Lambda_m(\bW^K_p,X_m) &= \sup_{\|\varphi\|_{L_p(\Omega^2)}\le 1} \Big|\int_{\Omega^2}\bigl( J(\by)-\sum_{\mu=1}^m\lambda_\mu K(\bx^\mu,\by)\bigr)\varphi(\by)d\by\Big|\\
    &=\Big\| J(\cdot)-\sum_{\mu=1}^m\lambda_\mu K(\bx^\mu,\cdot)\Big\|_{L_{p'}(\Omega^2)}.
  \end{split}
\end{equation}
We use the a similar definition as above of the error of optimal cubature formula with $m$ nodes for a class $\bW$ 
\begin{equation}\label{kappa2}
\kappa_m(\bW):=\inf_{ \lambda_1,\dots,\lambda_m; \bx^{1},\dots,
\bx^m}\Lambda_m(\bW,X_m).
\end{equation}
Thus, by (\ref{n2.2})
\be\label{n2.2'}
\kappa_m(\bW^K_p) = \inf_{ \lambda_1,\dots,\lambda_m; \bx^{1},\dots,
\bx^m}\Big\| J(\cdot)-\sum_{\mu=1}^m\lambda_\mu K(\bx^\mu,\cdot)\Big\|_{L_{p'}(\Omega^2)}.
\ee
Note, that the error of numerical integration of those classes is closely related to the $m$-term approximation of a special 
function $J(\cdot)$ with respect to a dictionary $\{K(\bx,\cdot)\}_{\bx\in \Omega^1}$. This leads to interesting applications of the recently developed
theory of greedy approximation (cf.\ Section \ref{Sect:bestmterm} above) in numerical integration. We refer the reader to \cite[Sect.\
2]{Tsurv}.

Let us now consider Problem (II). The following definition is from \cite{VT149}.
\begin{defi}\label{mcD3.1} The $(K,q)$-discrepancy of a cubature formula $\Lambda_m$ with 
nodes $X_m = \{\bx^1,\dots,\bx^m\}$ and weights $\Lambda_m = (\lambda_1,\dots,\lambda_m)$ is defined as 
$$
 D(\Lambda_m,K,q):=\Big\|\int_{[0,1]^d} K(\bx,\by)d\bx - \sum_{\mu=1}^m \lambda_\mu K(\bx^\mu,\by)\Big\|_{L_q([0,1]^d)}.
$$
\end{defi} 
The particular case $K(\bx,\by) := \chi_{[0,\by]}(\bx):=\prod_{j=1}^d \chi_{[0,y_j]}(x_j)$
leads to the classical concept of the $L_q$-discrepancy. Therefore, we get the following ``duality'' between 
numerical integration and discrepancy (in a very general form)
$$
D(X_m,\Lambda_m,K,p') = \Lambda_m(\bW^K_p,X_m).
$$
Certainly, in the above definitions we can replace the $L_p$ space by 
more general Banach spaces $X$. Then we still have ``duality'' between discrepancy in the norm of $X$ and numerical integration of classes defined in $X'$. 

Let us define the minimal weighted $r$-discrepancy by 
$$
  D^{w}_r(m,d)_q := \inf_{X_m, \Lambda_m} D_r(X_m,\Lambda_m,d)_q\quad,\quad 1\le q\le \infty\,,
$$
whereas the minimal $r$-discrepancy is given by 
$$
  D_r(m,d)_q:=  \inf_{X_m} D_r(X_m,d)_q\quad,\quad 1\le q\le \infty\,,
$$
with the minimal (classical) discrepancy $D(m,d)_q:=D_1(m,d)_q$ as a special case. Clearly, it holds
$$
   D^{w}_r(m,d)_q \leq D_r(m,d)_q\,.
$$

\subsubsection*{Historical remarks on discrepancy}\index{Discrepancy}

To begin with let us mention that there exist several monographs and surveys on discrepancy theory, see \cite{BC, DrTi97, Ma99} to mention just a few. We would like to point out the following milestones 
regarding upper and lower bounds for the minimal discrepancy. 

{\bf Lower bounds.} Let us start with K. Roth \cite{Ro} who proved in 1954 that
\be\label{hi5.1}
D(m,d)_2 \ge C(d)m^{-1}(\log m)^{(d-1)/2}. 
\ee
In 1972 W. Schmidt \cite{Sch1} proved
\be\label{hi5.2}
D(m,2)_\infty \ge Cm^{-1}\log m . 
\ee
In 1977 W. Schmidt \cite{Sch} proved	
\be\label{hi5.3}
D(m,d)_q \ge C(d,q)m^{-1}(\log m)^{(d-1)/2},\qquad 1<q\le \infty.  
\ee
In 1981 G. Hal{\' a}sz \cite{Ha} proved
\be\label{hi5.4}
D(m,d)_1 \ge C(d)m^{-1}(\log m)^{1/2}. 
\ee
The following conjecture has been formulated in \cite{BC} as an excruciatingly difficult great open problem.
\begin{conj}\label{hiCon5.1} We have for $d\ge 3$
$$
D(m,d)_\infty \ge C(d)m^{-1}(\log m)^{d-1}. 
$$
\end{conj}
This problem is still open. Recently, in 2008, D. Bilyk and M. Lacey \cite{BL} and D. Bilyk, M. Lacey, and A. Vagharshakyan \cite{BLV} proved
\begin{equation}\label{star}
D(m,d)_\infty \ge C(d)m^{-1}(\log m)^{(d-1)/2 + \delta(d)} 
\end{equation}
with some $0<\delta(d)<1/2$, which essentially improved on the lower bound 
$$
  D(m,d)_{\infty} \ge C(d)m^{-1}(\log m)^{(d-1)/2}\Big(\frac{\log\log m}{\log \log \log m}\Big)^{\frac{1}{2d-2}}\,,
$$
see Baker \cite{Bak99}. The approach in \cite{BL, BLV} is based on an improved
version of the Small Ball Inequality\index{Inequality!Small Ball}, see Subsection \ref{subsect:SBI} above. The conjectured inequality \eqref{2.6.5} would imply 
the stronger lower bound $C(d)m^{-1}(\log m)^{d/2}$. Compared to Conjecture \ref{hiCon5.1} there would still be a large gap in the power of the 
logarithm if $d>2$. Note also the connection to metric entropy in $L_\infty$, see Conjecture \ref{small_ball_conj}, 
and the Small Ball Problem \index{Small Ball Problem}in probability theory, see Subsection \ref{sect:ESBP}.

There seems to be some further progress in connection with Conjecture \ref{hiCon5.1}. M. Levin \cite{Le13, Le14, Le15} 
recently proved that several widely used point constructions (like digital nets, Halton points, Frolov lattice) 
satisfy the lower bound proposed in Conjecture \ref{hiCon5.1}. The method of proof is deeply involved and uses 
nontrivial tools from algebraic number theory. 

The first result in estimating the weighted $r$-discrepancy was obtained in 1985 by V.A. Bykovskii \cite{By}
\be\label{hi5.5}
D_r^w(m,d)_2 \ge C(r,d)m^{-r}(\log m)^{(d-1)/2}. 
\ee
This result is a generalization of Roth's result (\ref{hi5.1}).
The generalization of Schmidt's result (\ref{hi5.3}) was obtained by V.N. Temlyakov \cite{Tem22} in 1990
\be\label{hi5.6}
D_r^w(m,d)_q \ge C(r,d,q)m^{-r}(\log m)^{(d-1)/2}, \qquad 1<q\le \infty. 
\ee
In 1994 V.N. Temlyakov \cite{Tem25} proved that for $r$ even integers we have for the minimal $r$-discrepancy  
\be\label{hi5.7}
D_r(m,d)_\infty \ge C(r,d)m^{-r}(\log m)^{d-1}. 
\ee
This result encourages us to formulate the following generalization of Conjecture \ref{hiCon5.1}.
\begin{conj}\label{hiCon5.2} For all $d,r\in \N$ we have
$$
D_r^w(m,d)_\infty \ge C(r,d)m^{-r}(\log m)^{d-1}. 
$$
\end{conj}
The above lower estimates for $D_1^w(m,d)_q$ are formally stronger than the corresponding estimates for $D(m,d)_q$ because in 
$D_1^w(m,d)_q$ we are in addition optimizing over the weights $\Lambda_m$. 

{\bf Upper bounds.} We now present the upper estimates for the discrepancy in various settings. In 1956 H. Davenport \cite{Da} proved that
$$
D(m,2)_2 \le Cm^{-1}(\log m)^{1/2}.
$$
Other proofs of this estimate were later given by I.V. Vilenkin \cite{Vil}, J.H. Halton and S.K. Zaremba \cite{HZ}, and K. Roth \cite{Ro2}. In 1979 K. Roth \cite{Ro3} proved
$$
D(m,3)_2 \le Cm^{-1}\log m
$$
and in 1980 K. Roth \cite{Ro4} and K.K. Frolov \cite{Fro3} proved
$$
D(m,d)_2 \le C(d)m^{-1}(\log m)^{(d-1)/2}.
$$
In 1980 W. Chen \cite{Che} (and later in 1994 M. Skriganov \cite{Sk}) proved
$$
D(m,d)_q \le C(d)m^{-1}(\log m)^{(d-1)/2}, \qquad 1<q<\infty.
$$
Upper bounds for the star-discrepancy have been know since 1960. J.M Hammersley \cite{Ham60} and J.H. Halton \index{Quasi-Monte Carlo!Hammersley points} 
\index{Quasi-Monte Carlo!Halton points}
\cite{Ha60} were the first who gave explicit point constructions proving that 
\begin{equation}\label{f100}
   D(m,d)_{\infty} \le C(d)m^{-1} (\log m)^{d-1}\,.
\end{equation}
As we will point out below in detail, see Subsection \ref{exconst}, a bound like \eqref{f100} is non-trivial only for $m>e^{d-1}$ (without additional knowledge on 
the involved constant). For many applications this threshold is prohibitively large. 
What concerns the ``preasymptotic range'' for $m$ let us mention a surprising (non-constructive) result by S. Heinrich, E. Novak, G. Wasilkowski, H. Wo\'zniakowski \cite{HeNoWaWo01} from 2001. 
Via probabilistic arguments based on deep results from the theory of empirical processes (Dudley 1984, Talagrand 1994) and combinatorics (Haussler 1995) the authors in \cite{HeNoWaWo01} proved the existence of a universal (but unknown)
constant $c>0$ such that for arbitrary dimension $d \in \N$ 
\begin{equation}\label{eq:He01}
    D(m,d)_{\infty} \leq c\sqrt{d/m}\quad,\quad m\in \N\,.
\end{equation}
Essentially, it is shown that $m$ points drawn uniformly at random in $[0,1]^d$ satisfy the upper bound in \eqref{eq:He01} with non-zero probability. 
A simple proof of the above result is given by Aistleitner \cite{Ai11} with constant $c=10$. Note, that the expected star-discrepancy of a random point set is of order 
$\sqrt{d/m}$ as shown by Doerr \cite{Doe14}. Hinrichs \cite{Hi04} proved a lower bound for $D(m,d)_{\infty}$ which is also polynomial in $d/m$. In fact, his proof 
shows that this lower bound is also valid for $D_1^w(m,d)_{\infty}$\,. For further comments and open problems connected with the 
star-discrepancy see Heinrich \cite{He03}.

Let us also mention the following constructive result. Greedy approximation techniques allows us to build  constructive sets $X_m$ and $X_m(p)$, 
$1\le p<\infty$, such that for $d,m\geq 2$ (see \cite{Tbook}, pp. 402--403)
\begin{equation}
 \begin{split}
    D(X_m,L_\infty) &\le c_1 d^{3/2}(\max\{\ln d,\ln m\})^{1/2}m^{-1/2}\\
    D(X_m(p),L_p)&\le c_2 p^{1/2}m^{-1/2},\quad 1\le p<\infty\,,
 \end{split}
\end{equation}
with effective absolute constants $c_1$ and $c_2$. 

We finally comment on upper bounds for the weighted $r$-discrepancy. The estimate in Theorem \ref{Thm:Fro1} above together with Theorem \ref{nT1.1} implies
$$
D_r^w(m,d)_\infty \le C(r,d)m^{-r}(\log m)^{d-1},\qquad r\ge 2.
$$
It is clear that upper estimates for $D(m,d)_q$ are stronger than the same upper estimates for $D_1^w(m,d)_q$. Let us finally emphasize 
that the Smolyak nets are very poor from the point of view of discrepancy as it was shown by 
N. Nauryzbayev and N. Temirgaliyev (see \cite{NaTe09} and \cite{NaTe12}) .

\subsection{Open problems and historical comments}\index{Open problems!Numerical integration}
\label{OP_int}
Let us begin with a list of open problems which we will partly discuss below. \\

\noindent{\bf Open problem 8.1} Find the right order of the optimal error of numerical integration
$\kappa_m(\bW^r_1)$ if $r\geq 1$ (see Conjecture \ref{C6.5.1} below). \\

\noindent{\bf Open problem 8.2} Find the right order of the optimal error of numerical integration $\kappa_m(\bW^r_p)$
in the range of small smoothness\index{Small smoothness} (see Conjecture \ref{C6.5.1b}).\\

\noindent{\bf Open problem 8.3} Find the right order of the optimal error of numerical integration
$\kappa_m(\bW^r_\infty)$ if $r>0$ (see Conjecture \ref{C6.5.2} below). \\

\noindent{\bf Open problem 8.4} Find the right order of the star-discrepancy, see \eqref{f100} and 
Conjecture \ref{hiCon5.1} above.\\

Let us give some historical comments on the subject which go back to the 1950s. 
We begin with the lower estimates for cubature formulas. The results from
Theorem \ref{T6.2.1a} have forerunners. For the class $\bW^r_2$ Theorem
\ref{T6.2.1a} was established in \cite{By} by a different method. Theorem
\ref{T6.2.1a} was proved in \cite{Tem22}. In the case $\theta=\infty$ the lower
bound in Theorem \ref{T6.2.3} has been observed by Bakhvalov \cite{Bakh4} in
1972. Theorem \ref{Thm:Wr10} was
proved in \cite{Tem25}. For recent new proofs of Theorems \ref{T6.2.1a},
\ref{T6.2.3} we refer to \cite{UU14} and \cite{DU14}.
Concerning lower bounds
there are several open problems. Let us formulate them as conjectures, see 
\cite{Tsurv}.
\begin{conj}\label{C6.5.1} For any $d\ge 2$ and any $r> 1$ we have
$$
\kappa_m(\bW^r_1) \ge C(r,d)m^{-r}(\log m)^{d-1}.
$$
\end{conj}
\begin{conj}\label{C6.5.2} For any $d\ge 2$ and any $r> 0$ we have
$$
\kappa_m(\bW^r_\infty) \ge C(r,d)m^{-r}(\log m)^{(d-1)/2}.
$$
\end{conj}
\noindent Also important is the missing lower bounds for the case of small smoothness in the
Sobolev case. There is so far only a result for the special case of the Fibonacci 
cubature formula, see Proposition \ref{T6.3.3} and \cite{Tem24} which supports the following conjecture. 

\begin{conj}\label{C6.5.1b} Let $d\geq 2$, $2<p<\infty$ and $1/p<r\leq 1/2$.
Then
we have
$$
     \kappa_m(\Wrp) \gtrsim \begin{cases}
     m^{-r}(\log m)^{(d-1)(1-r)},&1/p<r<1/2;\\
     m^{-r}(\log m)^{(d-1)/2}\sqrt{\log\log m} ,& 
 r=1/2.
\end{cases}
$$
\end{conj}
A first step could be to establish such a lower bound for the ``pure'' Frolov
method with respect to the class of Sobolev functions $\Wo$ supported in the
cube $[0,1]^d$ which is of course smaller then $\Wrp$. 
 
We turn to the upper bounds. The first result in this direction was obtained by
N.M. Korobov \cite{Korb1} in 1959. He used the cubature formulas $P_m(f,\ba)$.
His results lead to the following bound
\be\label{6.5.1}
\kappa_m(\bW^r_1)\le C(r,d)m^{-r}(\log m)^{rd}, \qquad r>1.  
\ee
In 1959 N.S. Bakhvalov \cite{Bakh1} improved (\ref{6.5.1}) to
$$
\kappa_m(\bW^r_1)\le C(r,d)m^{-r}(\log m)^{r(d-1)}, \qquad r>1. 
$$
It is worth mentioning that Korobov and Bakhvalov worked with the space $\bE^r_d$. The above mentioned result is a
consequence of the embedding $\bW^r_1 \subset \bE^r_d$, see Lemma \ref{emb:p=1} above. 

There is vast literature on
cubature formulas based on function values at the nodes of number-theoretical nets. We do
not discuss this literature in detail because these results do not provide the optimal rate of errors for numerical
integration. A typical bound differs by an extra $(\log m)^a$ factor. The reader can find many results of this type in
the books \cite{Korb2}, \cite{Kor89}, \cite{HuWa81}, \cite{TBook}. For the case of small smoothness see  \cite{VT22}. 
An interesting method of building good Korobov's cubature formulas was suggested by S. M. Voronin and N.
Temirgaliev in \cite{VoTe89}. It is based on the theory of divisors. This method was further developed in a
number of papers by N. Temirgaliev and his students \cite{TeN90}--\cite{TeN97}, \cite{ZhTeTe09}, \cite{SiTe10},
\cite{BaSiTe14}.

The first best possible upper estimate for the classes $\bW^r_p$ was obtained by
N.S. Bakhvalov \cite{Bakh2} in 1963. He proved in the case $d=2$ that
\be\label{6.5.2}
\kappa_m(\bW^r_2)\le C(r)m^{-r}(\log m)^{1/2}, \qquad r\in \N.  
\ee
N.S. Bakhvalov used the Fibonacci cubature formulas. Propositions
\ref{Prop:Fib}(i), \ref{T6.3.3} and \ref{W1inf}(i),(ii) have been proved in
\cite{Tem24} and \cite{Tem25}. Proposition \ref{Prop:Fib}(ii) in the particular
case $\theta=\infty$, $1\leq p\leq \infty$ was proved in \cite{Tem24} (see also
\cite{TBook}, Ch. 4, Theorem 2.6). The same method gives the statement of
Theorem \ref{T6.3.1}(ii) for all remaining cases including Proposition
\ref{W1inf}(iii), see \cite{DU14}. 

A. Hinrichs and J. Oettershagen \cite{HiOe14} showed that the Fibonacci points are the
globally optimal point set of size $N$ for the quasi-Monte Carlo integration in $\bW^r_2(\T^2)$  for some numbers $N
\in \N$. 

In 1976 K.K. Frolov \cite{Fro1} used the Frolov cubature formulas to extend
(\ref{6.5.2}) to the case $d>2$ :
$$
\kappa_m(\ensuremath{\mathring{\mathbf W}_2^r})\le C(r,d)m^{-r}(\log
m)^{(d-1)/2}, \qquad r\in \N.  
$$
In 1985  V.A. Bykovskii \cite{By} proved the equivalence 
\begin{equation}\label{Byk:equiv}
   \kappa_m(\ensuremath{\mathring{\mathbf W}_p^r}) \asymp
\kappa_m({\mathbf W}_p^r)
\end{equation}
 in the case $p=2$ and got the upper bound in Theorem \ref{Thm:Fro1}(i) in the
case $p=2$. Bykovskii also used the Frolov cubature formulas. 
Relation \eqref{Byk:equiv} for $1<p<\infty$ and $r\in \N$, and its extension to classes $\bH^r_p$, $1\le p\le \infty$, $r>1/p$ have been
proved in \cite{Tem22} and \cite{Tem25} (see also \cite{Tsurv}, \cite[Chapt.\
IV.4]{TBook}), and in \cite{Du2}. Note, that the matter reduces in proving the
boundedness of a certain change of variable operator between the respective
spaces. For a complete solution of this problem we refer to \cite{NUU15}. Theorem
\ref{T6.4.1} is taken from \cite{TBook}. The upper bounds in 
Theorems \ref{Thm:Fro1}(ii) and \ref{thm:besovsobolev}(ii) were 
proved by V.V. Dubinin \cite{Du, Du2} in 1992 and 1997. The upper bound in Theorem \ref{cor:sobolev} has been 
proved recently by M. Ullrich and T. Ullrich \cite{UU14}, see also \cite{NUU15}.

The Frolov cubature formulas \cite{Fro2} (see also \cite{Du, Du2, Sk} and the
recent papers \cite{UU14, NUU15}) give the following bound
\be\nonumber
\kappa_m(\bW^r_1) \le C(r,d)m^{-r}(\log m)^{d-1},\qquad r>1.  
\ee
Thus the lower estimate in Conjecture \ref{C6.5.1} is the best possible.

In 1994 M.M. Skriganov \cite{Sk} proved the following estimate
\be\nonumber
\kappa_m(\ensuremath{\mathring{\mathbf W}_p^r}) \le C(r,d,p)m^{-r}(\log
m)^{(d-1)/2},\quad 1<p\le \infty, \quad r\in \N. 
\ee
This estimate combined with \eqref{Byk:equiv} implies 
\be\nonumber
\kappa_m(\bW^r_p)  \lesssim m^{-r}(\log m)^{(d-1)/2},\quad 1<p\le \infty,
\quad r\in \N.  
\ee
Theorems
\ref{thm:besovsobolev} (for $\theta<\infty$), \ref{cor:sobolev} have been proved recently in
\cite{UU14}, see also \cite{NUU15} for the extension to periodic and non-periodic spaces on
the cube $[0,1]^d$. Together with Theorems \ref{T6.2.1a}, \ref{T6.2.3} they
imply Theorem \ref{Thm:Fro1}.

The lower bounds in Theorem \ref{Thm:cubsmolyak} have been proved in
\cite{DU14}, see also \cite{UU14}. These results can be also proven with a
different technique, see the recent paper \cite{VT152}. This technique even
shows that the lower bound will not get smaller when adding $2^{\ell-1}$
arbitrary points. 

Concerning upper bounds for Smolyak cubature, see also Remark \ref{uppersmol},
we refer to Triebel \cite{Tr10}.  The reader can find some further results on numerical integration by Smolyak type methods in papers \cite{TeKuSh09,Temir10} and \cite{BaNoRi00,NoRi96,GeGr98,NoRi99,HiNoUl14}.

%% file: related_problems.tex
\section{Related problems}

\subsection{Why classes with mixed smoothness?}\index{Bounded mixed derivative}\index{Mixed smoothness}
\label{why}

In this section we briefly discuss development of the hyperbolic cross
approximation theory with emphasis put on the development of methods and
connections to other areas of research. 
The theory of the hyperbolic cross approximation began in the papers by Babenko
\cite{Bab2} and Korobov \cite{Korb1}. Babenko studied approximation of functions
from classes $\bW^r_\infty$ in $L_\infty$ by the hyperbolic cross polynomials
and the Bernstein type inequality for the hyperbolic cross polynomials in the
$L_\infty$ norm. Korobov studied numerical integration of functions from classes
$$
\bE^r_d(C):= \{f\in L_1(\T^d): |\hat f(\bk)|\le C\prod_{j=1}^d
\max\{1,|k_j|\}^{-r}
\},
$$
see also Section \ref{sect:FS} above. One of the Korobov's motivations for studying classes $\bE^r_d(C)$ was related
to
numerical solutions of integral equations. Let $K(x,y)$ be the kernel of the
integral operator $J_K$. Then the kernel of the $(J_K)^d$ is given by
$$
K^d(x,y) = \int_{\T^{d-1}} K(x,x_1)K(x_1,x_2)\cdots K(x_{d-1},y)dx_1\cdots
dx_{d-1} .
$$
Smoothness properties of $K(x,y)$ are naturally transformed into mixed
smoothness properties of $K(x,x_1)K(x_1,x_2)\cdots K(x_{d-1},y)$. In the
simplest case of $f_j(t)$ satisfying $\|f_j'\|_\infty \le 1$, $j=1,2,\dots,d$ we
obtain $\|(f_1(x_1)\cdots f_d(x_d))^{(1,1,\dots,1)}\|_\infty \le 1$. 

It is an {\it a priori} argument about importance of classes with mixed
smoothness. There are other {\it a priori} arguments in support of importance of
classes with mixed smoothness. For instance, recent results on solutions of the
Schr{\"o}dinger equation, which we discussed in the Introduction, give such an
argument. There are also strong {\it a posteriori} arguments in favor of
thorough study of classes of mixed smoothness and the hyperbolic cross
approximation. These arguments can be formulated in the following general way.
Methods developed for the hyperbolic cross approximation are very good in
different sense. We discuss this important point in detail, beginning with
numerical integration. 

It was immediately understood that the trivial generalization of the univariate
quadrature formulas with equidistant nodes to cubature formulas with rectangular
grids does not work for classes with mixed smoothness. As a result different
fundamental methods of numerical integration were constructed: the Korobov
cubature formulas (in particular, the Fibonacci cubature formulas), the Smolyak
cubature formulas, the Frolov cubature formulas (see Section \ref{numint}).
These
nontrivial constructions are very useful in practical numerical integration,
especially, when the dimension of the model is moderate ($\le 40$). In
subsection 9.2 we discuss theoretic results on universality of these methods,
which explain such a great success of these methods in applications. 

From the general point of view the problem of numerical integration can be seen
as a problem of discretization: approximate a ``continuous object'' -- an
integral
with respect to the Lebesgue measure -- by a ``discrete object'' -- a cubature
formula. It is a fundamental problem of the discrepancy theory. It is now well
understood that the numerical integration of functions with mixed smoothness is
closely related to the discrepancy theory (see, for instance, \cite{Tsurv}). 

Other example of a fundamental problem, which falls into a category of
discretization problems is the entropy problem (see Section 6). It turns out
that the entropy problem for classes with bounded mixed derivative is equivalent
to an outstanding problem of probability theory -- the small ball problem. Both
of the mentioned above discretization problems are deep and difficult problems.
Some fundamental problems of those areas are still open. The problems that were
resolved required new interesting technique. For instance, the Korobov cubature
formulas and the Frolov cubature formulas are based on number theoretical
constructions. Study of the entropy numbers of classes of functions with mixed
smoothness uses deep results from the theory of finite dimensional Banach spaces
and geometry, including volume estimates of special convex bodies. 

One more fundamental problem of the discretization type is the sampling problem
discussed in Section 5. Alike the entropy problem and the problem of numerical
integration the sampling problem for classes with mixed smoothness required new
techniques. Study of the sampling problem is based on a combination of classical
results from harmonic analysis and new embedding type inequalities. Some
outstanding problems are still open. 

We now briefly comment on some steps in the development of the hyperbolic cross
approximation. First sharp in the sense of order results on the behavior of
asymptotic characteristics of classes with mixed smoothness were obtained in the
$L_2$ norm. The Hilbertian structure of $L_2$, in particular the Parseval
identity, was used in those results. In a step from $L_2$ to $L_p$,
$1<p<\infty$, different kind of harmonic analysis technique was used. The
classical Littlewood-Paley theorem and the Marcinkiewicz multipliers theorem are
the standard tools. Later, the embedding type inequalities between the $L_q$
norm of a function and the $L_p$ norms of its dyadic blocs were proved and
widely used. At the early stage of the hyperbolic cross approximation the
classes $\bW^r_p$ and $\bH^r_p$ were studied. Many new and interesting phenomena
(compared to the univariate approximation) were discovered. Even in the case
$1<p,q<\infty$ the study of asymptotic characteristics of classes $\bW^r_p$ and
$\bH^r_p$ in $L_q$ 
required new approaches and new techniques. 
Some of the problems are still open. A step from the case $1<p,q<\infty$ to the
case, when one (or two) of the parameters $p,q$ take extreme values $1$ or
$\infty$ turns out to be very difficult. The majority of the problems are still
open in this case. Some surprising phenomena were observed in that study. For
instance, in many cases the step from a problem for two variables to the problem
for $d\ge 3$ variables is not resolved. A number of interesting effects, which
required special approaches for being established, was discovered. We only
mention a few from a list of new methods, which were used in those discoveries:
properties of special polynomials, for example,
$$
\Big|\sum_{\bk\in\Gamma(N),\bk>0}(k_1\cdots k_d)^{-1}\sin k_1x_1 \cdots \sin
k_dx_d\Big|
\le C(d)
$$
in a combination with the Nikol'skii duality theorem; the Riesz products; the
Small Ball Inequality; the volume estimates of sets of Fourier coefficients of
bounded trigonometric polynomials. 

The sampling problem, including the numerical integration problem, turns out to
be a difficult problem for classes with mixed smoothness. Many outstanding
problems in this area are still open. In particular, the problem of numerical
integration of classes $\bW^r_1$ and $\bW^r_\infty$ is not resolved yet. 

\subsection{Universality}\index{Universality}

In this subsection we illustrate the following general observation. Methods of
approximation, which are optimal in the sense of order for the classes with
mixed smoothness, are universal for the collection of anisotropic smoothness
classes. This gives {\it a-posteriori} justification for thorough study of
classes of functions with mixed smoothness. 
The phenomenon of saturation is well known in approximation theory
\cite{DeLo93}, Ch.11. The classical example of a saturation method is the
Fej{\'e}r operator for approximation of univariate periodic functions. In the
case of the sequence of the Fej{\'e}r operators $K_n$, saturation means that the
approximation order by operators $K_n$ does not improve over the rate $1/n$ even
if we increase smoothness of functions under approximation. Methods (algorithms)
that do not have the saturation property are called unsaturated. The reader can
find a detailed discussion of unsaturated algorithms in approximation theory and
in numerical analysis in a survey paper \cite{Bab3}. We point out that the
concept of smoothness becomes more complicated in the multivariate case than it
is in the univariate case. In the multivariate case a function may have
different smoothness properties in different coordinate directions. In other
words, functions may belong to different anisotropic smoothness classes (see
H{\"o}lder-Nikol'skii 
classes below). It is known (\cite{TBook}) that approximation characteristics of
anisotropic smoothness classes depend on the average smoothness and optimal
approximation methods depend on anisotropy of classes. This motivated a study in
\cite{Te88b} of existence of an approximation method that is good for all
anisotropic smoothness classes. This is a problem of existence of a universal
method of approximation.  We note that the universality concept in learning
theory is very important and it is close to the concepts of adaptation and
distribution-free estimation in non-parametric statistics (\cite{GKKW},
\cite{BCDDT}, \cite{VT113}). 

We present in this section a discussion of known results on universal cubature
formulas. 
We define the multivariate periodic H{\"o}lder-Nikol'skii classes $\bN\bH^\br_p$
 in the following way.
The class $\bN\bH^\br_p$, $\br=(r_1,\dots,r_d)$ and $1\le p \le \infty$, 
is the set of periodic functions $f\in L_p(\T^d)$ such that for each 
$l_j = [r_j]+1$, $j=1,\dots,d$, the following relations hold
$$
\|f\|_p \le 1,\qquad \|\Delta^{l_j,j}_tf\|_p \le |t|^{r_j}, \quad j=1,\dots,d,
$$
where $\Delta^{l,j}_t$ is the $l$-th difference with step $t$ in the variable
$x_j$. In the case $d=1$, $NH^r_p$ coincides with the standard H{\"o}lder class
$H^r_p$. 

Let a vector   $ \br = (r_1 ,\dots,r_d)$, $r_j  > 0$ and a number
 $m$ be given. Denote  $g(\br):=(\sum_{j=1}^d r_j^{-1})^{-1}$.
We define numbers   $N_j:=\max\bigl([m^{\varrho_j}],1\bigr)$,
  $\varrho_j  := g( r)/r_j$,   $j  =  1,\dots,d$
and the cubature formula
$$
  q_m (f,\br) := q_{\bN} (f) ,\qquad \bN := (N_1 ,\dots,N_d) .
$$
$$
  q_{\bN}(f):=(\prod_{j=1}^d N_j)^{-1}\sum_{j_d=1}^{N_d}\dots
\sum_{j_1=1}^{N_1}f(2\pi j_1/N_1,\dots,2\pi j_d/N_d).
$$
 
It is known (\cite{Bakh1}, \cite{TBook}) that for $g(\br)>1/p$
$$
  \kappa_m(\bN\bH^\br_p) \asymp q_m (\bN\bH^\br_p,\br) \asymp m^{-g(\br)},\quad
1\le p \le \infty, 
$$
where 
$$
q_m (\bW,\br):=\sup_{f\in
\bW}\Big|q_m(f,\br)-(2\pi)^{-d}\int_{\T^d}f(\bx)d\bx\Big|.
$$

We note that the cubature formula $q_m(\cdot,\br)$ depends essentially on the
anisotropic class defined by the vector $\br$.  It is known (see \cite{VT42},
\cite{VT49}) that the Fibonacci cubature formulas (see Section \ref{numint} for
the
definition) are optimal (in the sense of order) among all cubature formulas: for
  $g(\br)>1/p$
$$
  \delta_{b_n}(\bN\bH^\br_p)\asymp \Phi_n( \bN\bH^\br_p)\asymp b_n^{-g(\br)}.
$$

Thus, the Fibonacci cubature formulas are universal for the collection 
$\{\bN\bH^\br_p:1\le p\le \infty, g(\br)>1/p\}$ in the following sense. The
$\Phi_n(\cdot)$ does not depend on the vector $\br$ and the parameter $p$  and
provides optimal (in the sense of order) error bound for each class
$\bN\bH^\br_p$ from the collection.

In the case $d>2$, just as for $d=2$, there exist universal cubature formulas
for the anisotropic H{\"o}lder-Nikol'skii classes (see  \cite{VT42},
\cite{TBook}). At the same time we emphasize that the universal cubature
formulas that we constructed for $d>2$ have an essentially different character
from the Fibonacci cubature formulas. The first cubature formulas of this type
have been constructed by K.K. Frolov, see Section \ref{numint}. Let
$Q^{\ell}_n(f)$ be the cubature formula defined in \eqref{modi_Frolov}.  The
following result has been obtained in \cite{VT42} (see also \cite{TBook}).

\begin{thm}\label{T9.1} Let $1<p\le\infty$, $g(\br)>1$. Then for
all $\br$ such that $r_j\le (\ell - 1)(1 - 1/p)-1$, $j=1,\dots,d$
the following relation holds,
$$
Q^{\ell}_n(\bN\bH_p^{\br}) \asymp n^{-g(\br)}.
$$
\end{thm}

\noindent This theorem establishes that the cubature formulas $Q^{\ell}_n$ are
universal
for the collection $$\Big\{\bN\bH^\br_p:1<p\le\infty, g(\br)>1,r_j\le (\ell -
1)(1 -
1/p)-1, j=1,\dots,d\Big\}\,.$$

We now briefly discuss universality results in the nonlinear approximation. The
following observation motivates our interest in  the universal dictionary
setting. In practice we often do not know the exact smoothness class $F$ where
our input function (signal, image) comes from. Instead, we often know that our
function comes from a class of certain structure, for instance, anisotropic
H{\"o}lder-Nikol'skii class. This is exactly the situation we are dealing with
in the following universal dictionary setting. We formulate an optimization
problem in a Banach space $X$ for a pair of a function class $F$ and a
collection $\DD$ of bases (dictionaries) $\D$. We use the following
notation from Section \ref{Sect:bestmterm}
\begin{equation}\nonumber
\begin{split}
\sigma_m(f,\D)_X &:= \inf_{\substack{g_i\in \D,c_i \\ i=1,\dots,m}}\|f-\sum_{i=1}^mc_ig_i\|_X;\\
\sigma_m(F,\D)_X &:= \sup_{f\in F} \sigma_m(f,\D)_X;\\
\sigma_m(F,\DD)_X &:= \inf_{\D\in \DD} \sigma_m(F,\D)_X.
\end{split}
\end{equation}

The universality problem is a problem of finding a method of approximation that
is good for each class from a given collection of classes. For example,  we
introduced in \cite{VT77} the following definition of universal dictionary.
\begin{def}\label{Definition 1.1} Let two collections $\F$ of function classes
and $\DD$ of dictionaries be given. We say that $\D\in \DD$ is universal for the
pair $(\F,\DD)$ if there exists a constant $C$ which may depend only on $\F$,
$\DD$, and $X$ such that for any $F\in \F$ we have
$$
\sigma_m(F,\D)_X \le C\sigma_m(F,\DD)_X .
$$
\end{def}
So, if for a collection $\F$ there exists a universal dictionary $\D_u \in \DD$,
it is an ideal situation. We can use this universal dictionary $\D_u$ in all
cases and we know that it adjusts automatically to the best smoothness class
$F\in \F$ which contains a function under approximation. Next, if a pair
$(\F,\DD)$ does not allow a universal dictionary we have a trade-off between
universality and accuracy.

 It was proved in \cite{VT77} that
\begin{equation}\label{9.2}
\sigma_m(\bN\bH^\br_p,\OO)_{L_q} \asymp m^{-g(\br)} 
\end{equation}
for
$$
 1< p< \infty,\quad 2\le q<\infty, \quad g(\br)>(1/p-1/q)_+. 
$$
It is important to remark that the basis $\cU^d$ studied in \cite{T69} (see
Section 7) realizes (\ref{9.2}) for all $\br$. Thus, the orthonormal bases
$\cU^d$ is optimal in the sense of order for $m$-term approximation of classes
of functions with mixed smoothness and also it is universal for $m$-term
approximation of classes with anisotropic smoothness.  

In this survey we discussed in detail the hyperbolic cross approximation, namely, approximation of periodic functions
by the trigonometric polynomials with frequencies in the hyperbolic crosses. There is a natural analog of the
trigonometric hyperbolic cross approximation in the wavelet approximation. In Sections 5, 7 we already discussed
wavelet type systems $\mathcal U^d$. The general construction goes along the same lines. 
We briefly explain the construction from \cite{DKT98, SiUl09}, where it is called ``Hyperbolic Wavelet Approximation''.
\index{Hyperbolic wavelet approximation}
Let $\Psi =\{\psi_I\}$ be a system of univariate functions indexed by dyadic intervals. Define $d$-variate system
$\Psi^d:=\{\psi_{\mathbf I}\}$, $\mathbf I = I_1\times\cdots\times I_d$, $\psi_{\mathbf I}(\bx) :=
\prod_{j=1}^d\psi_{I_j}(x_j)$. Then the subspace 
$$
\Psi^d(n) := \Span \{\psi_{\mathbf I}: |\mathbf I| \ge 2^{-n}\}
$$
is an analog of the $\Tr(Q_n)$. The reader can find an introduction to the hyperbolic wavelet approximation in
\cite{DKT98, SiUl09}.

\subsection{Further generalizations}

In this survey we discussed in detail some problems of linear and nonlinear approximation of 
functions with mixed smoothness. The term {\it mixed smoothness} means that in the definitions of the corresponding
classes we either use the mixed derivative (classes $\bW$) or the mixed difference (classes $\bH$ and $\bB$). Other
classical scale of classes includes 
the Sobolev classes mentioned in the Introduction and  the H{\"o}lder-Nikol'skii classes discussed above in Subsection
9.2 (for detailed study of approximation of these classes see \cite{TBook}, Chapter 2). The approximation properties and
the techniques used for studying these classes are very different. There is an interesting circle of papers by D.B.
Bazarkhanov \cite{Ba03, Baz22b, Baz22a, Ba10a, Ba10b, Ba14} where he builds a unified theory, which covers both the collection of Sobolev,
H{\"o}lder-Nikol'skii classes and classes with mixed smoothness. His approach is based on 
dividing the variable $\bx =(x_1,\dots,x_d)$ into groups $\bx^j:=(x_{k_{j-1}+1},\dots,x_{k_{j}})$, $j=1,\dots,l$,
$0=k_0<k_1<\cdots<k_l=d$ and assuming that, roughly speaking, $f(\bx)$ as a function on each $\bx^j$ belongs to, say, a
Sobolev class, and as a function on the variable $(\bx^1,\dots,\bx^l)$ has mixed smoothness. 

We mostly confined ourselves to approximation in the Banach space $L_q$, $1\le q\le \infty$, of functions from classes
defined by a restriction on the mixed derivative or mixed difference in the Banach space $L_p$, $1\le p\le \infty$. In
our discussion parameters $p$ and $q$ are scalars. In some approximation problems it is natural to consider $L_{\mathbf
p}$ spaces with vector $\mathbf p =(p_1,\dots,p_d)$ instead of a scalar $p$. For instance, considering general setting
of approximation of classes of univariate functions (see \cite{VT31})
$$
W^g_p:=\{f: f(x)=\int_\Omega g(x,y)\varphi(y) dy,\quad \|\varphi\|_{L_p(\Omega)}\le 1\}
$$
we immediately encounter the problem of approximation of the kernel $g(x,y)$ in vector norms. In particular, the
operator norm of the integral operator with the kernel $g(x,y)$ from 
$L_p(\Omega)$ to $L_\infty(\Omega)$ is equal to $\| \|g(x,\cdot)\|_{L_{p'}}\|_{L_\infty}$. 
Dinh D\~ung \cite{Di84a} -- \cite{Di87}, Galeev \cite{Ga78} -- \cite{Ga2} investigated embedding theorems, hyperbolic cross approximation and various widths for classes of functions with 
mixed smoothness defined in a vector $L_{\mathbf p}$ space in a vector $L_{\mathbf q}$ space.
G.A. Akishev \cite{Ak06} -- \cite{Ak10} conducted a detailed study of approximation of functions with 
mixed smoothness defined in a vector $L_{\mathbf p}$ space in general Lorentz-type spaces.

\subsection{Direct and inverse theorems}\index{Direct and inverse theorems}

In the univariate approximation theory classes with a given modulus of continuity (smoothness) are 
a natural generalization of classes $H^r_\infty$. These classes were studied in detail from the point of view of direct
and inverse theorems of approximation, which bear the names of 
Jackson and Bernstein. The reader can find the corresponding results in \cite{DeLo93} and \cite{TBook}. For instance it
is well known that the following conditions are equivalent
\be\nonumber
E_n(f)_p \lesssim n^{-r}
\ee
\be\nonumber
\|f\|_{H^r_p} <\infty.
\ee
Clearly, researchers tried to prove direct and inverse theorems for approximation by the hyperbolic cross approximation.
Direct theorems of that type for classes $\bW$, $\bH$ and 
$\bB$ are discussed in Section 4. There are interesting results on the Jackson type theorems for the hyperbolic cross
approximation. The first result in this direction was proved 
for approximation in the uniform metric in \cite{Te89}. Further detailed study of approximation of classes of periodic
functions of several variables with a given majorant of the mixed moduli of smoothness was conducted by 
Dinh D\~ung \cite{Di84b,Di84c,Di85,Di86,Di14}, N.N. Pustovoitov \cite{Pu13} -- \cite{Pu91} and by 
S.A. Stasyuk \cite{Sta12a} -- \cite{Sta14}. It was discovered in the early papers on the hyperbolic cross approximation
that contrary to the univariate approximation by the trigonometric polynomials we cannot characterize classes of
periodic functions of several variables with a given majorant of the mixed moduli of smoothness by the hyperbolic cross
approximation, see also \cite{ScSi04}. The problem of characterization of classes of functions with a given rate of
decay of their best approximations by the hyperbolic cross polynomials was solved in \cite{DePeTe94,Di97a,Di97b}. As a result  new concepts of a mixed modulus of smoothness were introduced in \cite{DePeTe94,Di97a}.

\subsection{Kolmogorov widths of the intersection of function classes} \index{Width!Kolmogorov}
\label{widths-diff-operator }

For $\br \in \R^d$, let the differential operator $D^{\br}$ be defined by 
$D^\br\colon f\mapsto (-i)^{|\br|_1} f^{(\br)}$, where $ f^{(\br)}$ is the fractional derivative of order $\br$ in the sense of Weil. For  a nonempty finite set $A\subset\N^d_0$ and a nonzero sequence of numbers $(c_\br)_{\br\in A}$, the polynomial $P(\bx):=  \sum_{\br\in A} c_\br \bx^\br$ induce the differential operator 
\[
P(D)=\sum_{\br \in A} c_\br D^{\br}. 
\]
Set 
\[
U^{[P]}_2:= \ \{f\in L_2: \, \|P(D)(f)\|_2 \leq 1\}.
\]

One of the most important problems in multivariate approximation is how to define smoothness function classes.
In the early paper \cite{Bab1}, Babenko suggested to define them as the functions from the finite intersection 
$\cap_{j=1}^J U^{[P_j]}$. 
There, he obtained a non-explicit upper bound of $d_n(\cap_{j=1}^J U^{[P_j]}_2,L_2)$ in terms of the eigenvalues of the
operator $\sum_{j=1}^J P_j^*P_j$. The Sobolev class of mixed smoothness $\bW^r_2$ can be considered as a particular case
of $U^{[P]}$. Tikhomirov and his school 
\cite{Di84a,Di84b,Di84c, Di85,Di86,Di00,Di01_2,Di02,Ga85,Ga88,Ga90,Ti87,Ti90} considered smoothness functions as the
functions from the intersections 
\[
\bW^A_p:= \cap_{\br \in A}\bW^\br_p, \quad  
\bH^A_p:= \cap_{\br \in A}\bH^\br_p, \quad  
\bB^A_{p,\theta}:= \cap_{\br \in A}\bB^\br_{p,\theta}
\]
for some (not necessarily finite) set $A \subset \R^d$.
They  investigated embedding theorems, hyperbolic cross approximation, various widths, entropy number and $m$-term
approximation for these classes of functions on $\T^d$ and $\R^d$. 

As an illustration, we give a typical result in \cite{Di84a} on the Kolmogorov width $d_n(\bW^A_p,L_p)$ for 
a nonempty finite set $A \subset [0,\infty)^d$ such that ${\bf 0} \in A$ and 
$\max_{\br \in A} r_i > 0$ for every $i=1,...,d$. 
Denote by $\operatorname{conv}(A)$ the convex hull of $A$. We define two quantities characterizing the smoothness of $\bW^A_p$.
\[
r(A)
:= \
\min\{t > 0: t(1,1,...,1) \in \operatorname{conv}(A)\} 
\]
and $\nu(A):= d-\mu$ where $\mu$ is the dimension of the minimal face (extremal subset) of  
$\operatorname{conv}(A)$ containing $r(1,1,...1)$. These quantities were introduced by Dinh D\~ung \cite{DD79} in a dual
form:  $1/r(A)$ is defined as the optimal value of the problem
\[
\operatorname{maximize}\, |\bx|_1, \quad 
\operatorname{subject \ to} \ \bx \in \R^d, \ (\br,\bx) \le 1, \ \forall \br \in A,
\]
and $\nu(A)-1$ as the dimension of its solutions.
In particular cases, we have $r(A)=r$, $\nu(A)=d-1$
for the class $\Wrp= \bW^A_p$ with $A = \{r(1,1,...1)\}$, and $r(A)=r_1$, $\nu(A)=\nu$,
for the class $\bW^{\br}_p= \bW^A_p$ with $A = \{\br\}$, 
$0 < r_1 = \cdots = r_\nu < r_{\nu + 1} \le r_{\nu + 2} \le \cdots \le r_d$ ($1\le \nu \le d$). These examples tell us
that  the quantities $r(A)$ and $\nu(A)$ are indeed characteristics of the smoothness of $\bW^A_p$.
\begin{thm} \label{thm[d_nW^A]}
For $1 < p < \infty$, we have
\begin{equation}\nonumber
d_n({\bf W}^A_p,L_p)
\ \asymp \
\left(\frac{\log^{\nu(A)-1}n}{n}\right)^{r(A)}\quad,\quad n\in \N\,.
\end{equation}
\end{thm}

The problem of computing asymptotic orders of 
$d_n(U^{[P]}_2,L_2)$ in the general case when $W^{[P]}_2$ is compactly 
embedded into $L_2$ has been open for a long time; see, e.g., 
\cite[Chapter~3]{TBook} for details. It has been recently solved in \cite{CoD14} for a non-degenerate
differential operator $P(D)$ (see there for a definition of non-degenerate
differential operator). Here we give a generalization of this result which can be proven in a similar way.

\begin{thm} \label{thm[d_nW^[P]]}
Let $P_j$, $j = 1,...,J$ be polynomials with the power sets $A_j$ and the coefficient sequences 
$(c_\br^j)_{\br\in A_j}$. Assume that the different operators $P_j(D)$ are non-degenerate, 
$0 \in A:= \cap_{j=1}^J A_j$ and the intersection of $A$ with the ray 
$\{\lambda e^j: \lambda > 0\}$  nonempty where $e^j$ denotes the 
$j$th standard unit vector of $\R^d$. 
Then we have
\begin{equation}\label{d_nW^[P]}
d_n(\cap_{j=1}^J U^{[P_j]}_2,L_2)
\ \asymp \
\left(\frac{\log^{\nu(A)-1}n}{n}\right)^{r(A)}\quad,\quad n\in \N.
\end{equation}
\end{thm}

Notice that in both Theorems \ref{thm[d_nW^A]} and \ref{d_nW^[P]}, the asymptotic order of the Kolmogorov width is
realized by the approximation by trigonometric polynomials with frequencies from the intersection of hyperbolic crosses
corresponding to the set $A$. For related results, surveys and bibliography on embedding theorems, hyperbolic cross
approximation, various widths,  entropy number and $m$-term approximation of classes of multivariate periodic functions
with several bounded fractional derivatives or bounded differences see    
\cite{CoD14,Di84a, Di84b, Di84c, Di85, Di86,Di87,Di00,Di01_2,Di02,DM13,DM13_C,DT79,Ga81,Ga85,Ga88,Ga90,Mag79,Mag86,Mag87,Mag88,Ti87,Ti90,Rom12}.

\subsection{Further $s$-numbers}\index{s-numbers}
\label{sect:snumbers}
It is well-known that every compact operator $T$ in a Hilbert space $H$, i.e., $T:H\to H$ can be represented through
its singular value decomposition (Schmidt expansion)
$$
    Tx = \sum\limits_{m} s_m\langle x, u_m\rangle v_m\quad,\quad x\in H\,,
$$
where $(u_m)_m$ and $(v_m)_m$ are the eigenelements of $T^*T$ and $TT^*$. Hence, we associate every compact operator a
sequence $(s_m(T))_m$ of singular numbers. 

The step in extending this concept to (quasi-)Banach spaces was done by Pietsch in 1974, see \cite[6.2.2]{Pi07} for
more historical facts. 

\begin{defi}\label{def:sfunc}
Let X,Z be quasi Banach-spaces and $Y$ be a $p$-Banach space, let $S,T\in\mathcal{L}(X,Y)$ and $R\in \mathcal{L}(Y,Z)$.
A mapping $s:T\to (s_m(T))_{m=0}^{\infty}$ with the following properties
\begin{description}
  \item[(S1)] $\|T\|_{\mathcal{L}(X,Y)}=s_0(T)\geq s_1(T)\geq\ldots\geq 0$\,,\\
  \item[(S2)] for all $m_1,m_2\in\N_0$ holds
  $$s_{m_1+m_2}(R\circ S)\leq s_{m_1}(R)s_{m_2}(S)\,,$$
  \item[(S3)] for all $m_1,m_2\in\N_0 $ holds $$s_{m_1+m_2}^p(S+T)\leq s_{m_1}^p(S)+s_{m_2}^p(T)\,,$$\\
  \item[(S4)] for all $m\in \N$ holds $s_m(\mbox{id}:\ell_2^m\to \ell_2^m)=1\,,m=1,2,...$,
  \item[(S5)] and $s_m(T)=0$ whenever $\mbox{rank } T \leq m$,
\end{description}
is called $s$-function.
\end{defi}

The definition of the widths studied in Section 4 (orthowidths, Kolmogorov widths, linear widths) can be
extended in order to approximate general linear operators $T$ instead of the identity/embedding operator. There are
some issues related with this interpretation which are discussed in \cite[6.2.6]{Pi07} and \cite[p. 30]{Pin85}.
However, for Kolmogorov and approximation numbers the above axioms are satisfied and they form an $s$-function. In
other words, they represent special sequences of $s$-numbers, $d_m(T)$ and $a_m(T)$. We use the usual notation $a_m(T)$
(instead of $\lambda_m(T)$) in order to avoid confusion with the sequence of eigenvalues. A simple consequence of (S5)
is the fact that the approximation numbers form the largest sequence of $s$-numbers. In fact, for any operator with rank
less or equal to $m$ it holds 
$$
    s_m(T) \leq s_m(S) + \|T-S\| = \|T-S\|
$$
and therefore $s_m(T) \leq a_m(T)$. As a direct implication we see that the orthowidths (discussed in Section 4) can not
be interpreted as $s$-numbers. In fact, in some situations they are asymptotically larger than approximation numbers
(linear widths), see Section 4. 

Let us emphasize that the sequence of dyadic entropy numbers $\epsilon_m(T)$ (studied in Chapter 6) does not give an
$s$-functions since it does not satisfy (S5). However, there are many more interesting examples of $s$-numbers. Let us
first discuss Gelfand numbers/widths. 

The {\em Gelfand numbers}\index{Width!Gelfand} have been introduced by Tikhomirov in 1965 (Gelfand just proposed them) 
\begin{equation}\label{gelf}
    c_m(\bF,X):=\inf\limits_{\substack{A:\bF \to \R^m\\
\text{linear}}} \sup\limits_{\substack{\|f\|_{\bF}\leq 1 \\ f\in \ker A}}
\|f\|_X\,.
\end{equation}
They can as well be defined for operators $T:\bF \to X$ such that we end up with a scale $c_m(T)$ 
of $s$-numbers (\cite[6.2.3.3]{Pi07}). They satisfy
a useful duality relation with the Kolmogorov numbers. For any compact operator $T:X\to Y$ it holds 
$$
    c_m(T^*:Y'\to X') = d_m(T:X\to Y) \quad \mbox{and} \quad d_m(T^*:Y' \to X') = c_m(T:X\to Y)\,,
$$
where $T^*$ denotes the dual operator to $T$ and $X'$, $Y'$ the dual spaces to $X$ and $Y$.
As a special case we obtain the relation
$$
    d_m(\bW^r_p,L_q) = c_m(\bW^r_{q'},L_{p'})\quad,\quad m\in\N\,.
$$
if  $1<p,q<\infty$. 
``Dualizing'' Figure \ref{dmW} gives the following order (compared to linear widths).

\begin{figure}[H]
\begin{minipage}{0.48\textwidth}
\begin{center}
\begin{tikzpicture}[scale=2.5]

\draw[->] (-0.1,0.0) -- (2.1,0.0) node[right] {$\frac{1}{p}$};
\draw[->] (0.0,-.1) -- (0.0,2.1) node[above] {$\frac{1}{q}$};

\draw (1.0,0.03) -- (1.0,-0.03) node [below] {$\frac{1}{2}$};
\draw (0.03,1) -- (-0.03,1) node [left] {$\frac{1}{2}$};

\node at (1.7,2.2) {$\lambda_m(\bW^r_p,L_q)$};

\draw (0,2) -- (2,2);
\draw (1,1) -- (2,1);

\draw (1,1) -- (2,0);
\draw (1,0) -- (1,1);
\draw (1,1) -- (2,1);
\draw (2,2) -- (2,0);

\node at (1.4,0.2) {\tiny ${\alpha = r-\frac{1}{2}+\frac{1}{q}}$};
\node at (1.6,0.8) {\tiny ${\alpha= r-\frac{1}{p}+\frac{1}{2}}$};
\node at (1,1.4) {$\alpha = r-(\frac{1}{p}-\frac{1}{q})_+$};

\draw (2,0.03) -- (2,-0.03) node [below] {$1$};
\draw (0.03,2) -- (-0.03,2) node [left] {$1$};

\end{tikzpicture}

\end{center}

\end{minipage}
\begin{minipage}{0.48\textwidth}
 \begin{center}
\begin{tikzpicture}[scale=2.5]

\draw[->] (1,0.0) -- (2.1,0.0) node[right] {$\frac{1}{p}$};
\draw (0,0.0) -- (1,0.0);
\draw[->] (0.0,-.1) -- (0.0,2.1) node[above] {$\frac{1}{q}$};

\draw (1.0,0.03) -- (1.0,-0.03) node [below] {$\frac{1}{2}$};
\draw (0.03,1) -- (-0.03,1) node [left] {$\frac{1}{2}$};

\draw (0,0) -- (1,1);
\draw (0,0) -- (2,0) -- (2,2) -- (0,2);
\draw (1,0) -- (1,1) -- (2,1);

\node at (1,1.4) {$\alpha = r$};

\node at (1.48,0.45) {\tiny $\alpha = {r-\frac{1}{2}+\frac{1}{q}}$};
\node at (0.58,0.12) {\tiny $\alpha = {r-\frac{1}{p}+\frac{1}{q}}$};
\node at (1.7,2.2) {$c_m(\bW^r_p,L_q)$};
\draw (2,0.03) -- (2,-0.03) node [below] {$1$};
\draw (0.03,2) -- (-0.03,2) node [left] {$1$};

\end{tikzpicture}

\end{center}

\end{minipage}
\caption{Comparison of $\lambda_m(\bW^r_p,L_q)$ and $c_m(\bW^r_p,L_q)$, rate $(m^{-1}(\log m)^{d-1})^{\alpha}$}
\end{figure}
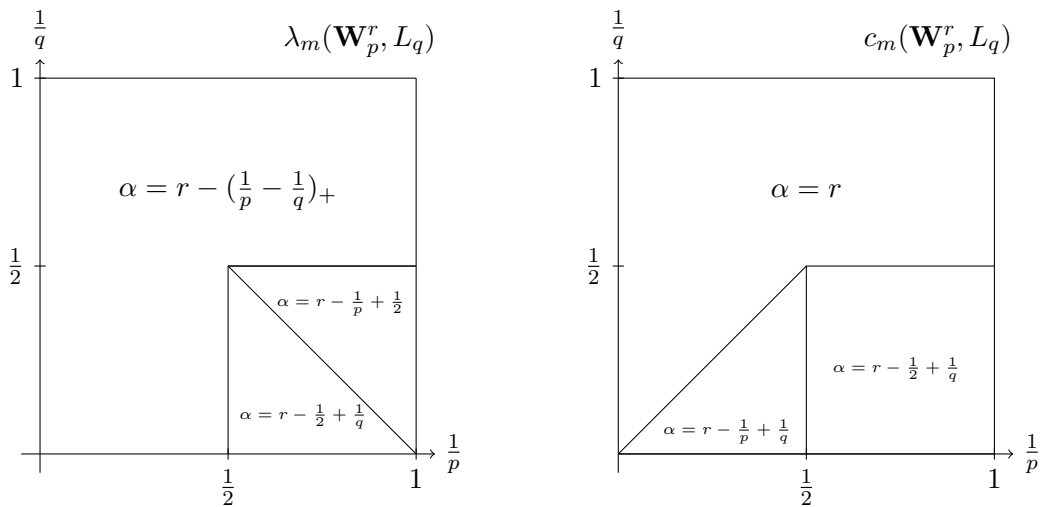

The Gelfand numbers can be interpreted from an
algorithmic point of view. Suppose, one is
interested in reconstructing an object (function) from $m$ linear samples.
The mapping $A$ serves as measurement map and is linear, whereas the reconstruction can be a nonlinear
operator using the measurement vector from $\R^m$. In that sense they directly relate to the novel field of compressed
sensing, see \cite[Chapt.\ 10]{FoRa13}, \cite[Chapt.\ 5]{Tbook}, and the recent paper \cite{FoPaRaUl10}.
An optimal measurement map minimizes \eqref{gelf} and measures the maximal ``distance'' of two instances which provide
the same output (by $A$). Hence, any reconstruction map can not distinguish between those two instances and the
reconstruction error would relate to their distance. In the picture there are parameter regions where the linear widths
are asymptotically strictly larger than the Gelfand numbers. In this region a general
(possibly nonlinear) operator/algorithm (that uses linear information) might
beat the linear algorithm which is behind the linear widths. Such an algorithm
can probably come out of an optimization like in learning theory or compressed
sensing. 

In addition, some of the $s$-numbers are useful for upper bounding the sequence of eigenvalues\index{Eigenvalue} of compact operators
(Weyl inequalities, \cite[6.4.2]{Pi07})\index{Inequality!Weyl}.  For this reason authors studied  the order of Weyl numbers
\index{s-numbers!Weyl numbers} $x_m(T)$ also
in the context of function spaces with bounded mixed difference/derivative, see Nguyen, Sickel \cite{NgSi14, Ng14,
Ng15}. In contrast to the Weyl numbers, which represent special $s$-numbers, the Bernstein number $b_m(T)$ 
\index{s-numbers!Bernstein numbers}fail to
satisfy (S2, S3). However, Bernstein numbers are important since they often serve as lower bounds for Kolmogorov and
Gelfand numbers \index{s-numbers!Gelfand numbers}, and especially, as lower bounds for the error analysis of Monte-Carlo algorithms, see the recent work
\cite{Ku15}.  

V.K. Nguyen shows in \cite[Lem.\ 3.3]{Ng15} the interesting relation
$$
      b_m(T) \leq 2\sqrt{2} \epsilon_m(T)\quad,\quad m\in \N\,.\index{Entropy number}
$$
Let us shortly write $b_m(X,Y)$ and $x_m(X,Y)$ when considering the identity mapping $b_m(I)$ and
$x_m(I)$ for $I:X\to Y$. Together with an abstract relation between Bernstein and Weyl numbers \cite{Pi08} this
leads to the sharp relation
$$
    b_m(\bW^r_p,L_q) \asymp \min\{x_m(\bW^r_p,L_q), \epsilon_m(\bW^r_p,L_q)\}\quad, \quad m\in \N\,.    
$$
The results on Bernstein and Weyl numbers can be illustrated in the following diagrams. 

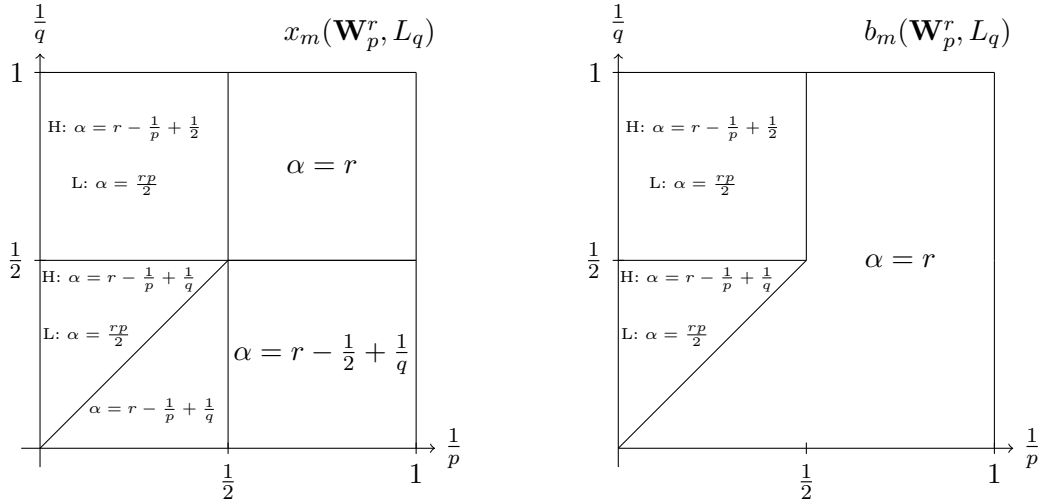
\begin{figure}[H]
\begin{minipage}{0.48\textwidth}
\begin{center}
\begin{tikzpicture}[scale=2.5]

\draw[->] (-0.1,0.0) -- (2.1,0.0) node[right] {$\frac{1}{p}$};
\draw[->] (0.0,-.1) -- (0.0,2.1) node[above] {$\frac{1}{q}$};

\draw (1.0,0.03) -- (1.0,-0.03) node [below] {$\frac{1}{2}$};
\draw (0.03,1) -- (-0.03,1) node [left] {$\frac{1}{2}$};

\node at (1.7,2.2) {$x_m(\bW^r_p,L_q)$};

\draw (0,2) -- (2,2);
\draw (1,1) -- (2,1);

\draw (1,0) -- (1,1);
\draw (1,1) -- (2,1);
\draw (2,2) -- (2,0);
\draw (1,1) -- (1,2);
\draw (0,0) -- (1,1);
\draw (0,1) -- (1,1);

\node at (1.5, 1.5) {$\alpha = r$};
\node at (1.5,0.5) {$\alpha = r-\frac{1}{2}+\frac{1}{q}$};

\node at (0.6,0.2) {\tiny $\alpha = r-\frac{1}{p}+\frac{1}{q}$};

\node at (0.42,0.9) {\tiny H: $\alpha = r-\frac{1}{p}+\frac{1}{q}$};
\node at (0.25,0.6) {\tiny L: $\alpha = \frac{rp}{2}$};

\node at (0.45,1.7) {\tiny H: $\alpha = r-\frac{1}{p}+\frac{1}{2}$};
\node at (0.4,1.4) {\tiny L: $\alpha = \frac{rp}{2}$};


\draw (2,0.03) -- (2,-0.03) node [below] {$1$};
\draw (0.03,2) -- (-0.03,2) node [left] {$1$};

\end{tikzpicture}

\end{center}

\end{minipage}
\begin{minipage}{0.48\textwidth}
 \begin{center}
\begin{tikzpicture}[scale=2.5]

\draw[->] (1,0.0) -- (2.1,0.0) node[right] {$\frac{1}{p}$};
\draw (0,0.0) -- (1,0.0);
\draw[->] (0.0,-.1) -- (0.0,2.1) node[above] {$\frac{1}{q}$};

\draw (1.0,0.03) -- (1.0,-0.03) node [below] {$\frac{1}{2}$};
\draw (0.03,1) -- (-0.03,1) node [left] {$\frac{1}{2}$};

\draw (0,2) -- (2,2);
\draw (1,1) -- (1,2);
\draw (0,0) -- (1,1);
\draw (0,1) -- (1,1);
\draw (2,2) -- (2,1);
\draw (2,1) -- (2,0);

\node at (1.5, 1.0) {$\alpha = r$};

\node at (0.42,0.9) {\tiny H: $\alpha = r-\frac{1}{p}+\frac{1}{q}$};
\node at (0.25,0.6) {\tiny L: $\alpha = \frac{rp}{2}$};

\node at (0.45,1.7) {\tiny H: $\alpha = r-\frac{1}{p}+\frac{1}{2}$};
\node at (0.4,1.4) {\tiny L: $\alpha = \frac{rp}{2}$};

\node at (1.7,2.2) {$b_m(\bW^r_p,L_q)$};
\draw (2,0.03) -- (2,-0.03) node [below] {$1$};
\draw (0.03,2) -- (-0.03,2) node [left] {$1$};

\end{tikzpicture}

\end{center}

\end{minipage}
\caption{Comparison of $x_m(\bW^r_p,L_q)$ and $b_m(\bW^r_p,L_q)$, rate $(m^{-1}(\log m)^{d-1})^\alpha$}
\end{figure}
Here, $H$ refers to the parameter domain of ``high smoothness'', $r>\frac{1/\max\{2,q\}-1/p}{p/2-1}$ and $L$
refers to ``low smoothness'', that is $r<\frac{1/\max\{2,q\}-1/p}{p/2-1}$.

\subsection{The quasi-Banach situation}\index{quasi-Banach}
\label{quasiB}

\subsubsection*{Complements to Section 3: Mixed differences}\index{Function spaces!Bounded mixed difference}

By replacing the moduli of continuity \eqref{modc}
by more regular variants like integral means \eqref{rectm} of differences we can extend the characterization in Lemma \ref{diff} to all $0<p\leq \infty$
and $r>(1/p-1)_+$.

For $\bs\in \N_0^d$ we put $2^{-\bs}:=(2^{-s_1},...,2^{-s_d})$ and $e(\bs) = \{i:s_i \neq 0\}$. We have the following characterization, see \cite{NUU15} and \cite{Ul06}.

\begin{thm}\label{thm95} Let $0<p, \theta\leq \infty$ and $s>\sigma_{p}$. Let further $m \in \N$ be a natural number with
$m>s$\,. Then 
$$
    \|f\|_{\bB^s_{p,\theta}} \asymp \Big(\sum\limits_{\bs\in \N_{0}^d} 2^{s|\bs|_1 \theta}\|
\mathcal{R}^{e(\bs)}_m(f,2^{-\bs},\cdot)\|_p^{\theta}\Big)^{1/\theta}\,.
$$
In case $\theta=\infty$ the sum above 
is replaced by the supremum over $\bs$.
\end{thm}

\subsubsection*{Complements to Section 5: Sampling representations}\index{Sampling!Representation}

An extension of Propositions \ref{samp_repr1},
\ref{samp_repr2} to all parameters
$0<p,\theta \leq \infty$ and $r>1/p$ is possible but not straight forward. 
As shown recently in \cite[Thm.\ 5.7]{ByUl15} the de la Vall\'ee Poussin sampling 
representation \eqref{lhs1} (as well as \eqref{lhs2} for mixed Triebel-Lizorkin spaces) 
works for all $1/2<p\leq 1$. In order to get a sampling
representation for arbitrary small $0<p\leq 1$ we have to use different fundamental
interpolants in \eqref{dvp}. Those have to provide a better decay than
\eqref{2.1.9} on the ``time side'' or, in different words, the Fourier
transform of the kernel has to be much smoother than just piecewise linear. 
This can be arranged, see \cite{ByUl15}, by an iterated convolution of dilates
of the characteristic function of the interval $[-1/2,1/2]$, a process related
to the construction of the so-called $up$-function, see \cite{Rv90}. However, when dealing with sampling on Smolyak grids 
it turned out, see \cite{ByUl15}, that the integrability in the target space $L_q$ determines the appropriate sampling kernel. In fact, when
approximating functions from the class $\bB^r_{p,\theta}$, with $p<1$ in $L_2$ even classical Dirichlet-Smolyak operators
will do the job, see \cite[Rem.\ 6.5, Thm.\ 5.14]{ByUl15}.\index{Kernel!Dirichlet} 

The Faber-Schauder and B-spline quasi-interpolation relations in Propositions
\ref{contvsdisc}, \ref{discvscont}, \ref{samp_repr1Q},(i)  
for the periodic space $\Brpt$ as well as its non-periodic counterpart can be extended to $0<p,\theta
\leq \infty$. In contrast to the de la Vall\'ee Poussin kernels the hat functions are perfectly localized. In other words, we have 
sufficient decay on the time side. However, the smoothness of the piecewise linear hat functions is limited. Therefore, on can not expect 
to decompose arbitrary smooth function over the Faber-Schauder basis. This explains the natural restriction $1/p<r<2$ in Proposition
\eqref{discvscont} and hence we can only go ``down'' to $p>1/2$ and only get a maximal rate of convergence (worst-case error) bounded by $2$. 
Using $B$-splines one can achieve higher order convergence by adapting the order of
the B-spline also to $p$ (and not just on $r$). Note that the restriction in Proposition \ref{samp_repr1Q} reads as 
$1/p<r<2\ell$ if $p<1$. In particular this means that we have to choose $2\ell$ larger than $1/p$, a similar
effect as described in the previous paragraph.

\subsubsection*{Complements to Section 5: Sampling widths}\index{Width!Sampling}

The method described in the proof of Theorem \ref{Theorem[T_n(B)<]}
also works for $p,\theta, q<1$ once we have the corresponding sampling
representations. For this issue we refer \cite{Di11} and \cite{ByUl15}. Here we obtain the following results.

\begin{thm} \label{Theorem[rho_n]}
Let $ 0 < p, q, \theta \le \infty$, and $r > 1/p$.
Then we have. 
\begin{itemize}
\item[{\rm (i)}] For $p \ge q$ and $\theta \le 1$,
\begin{equation*} 
\varrho_m(\bBr,L_q)
 \ \asymp \ 
 (m^{-1} \log^{d-1}m)^r\quad,\quad
\begin{cases}
 2 \le q < p < \infty, \\
1 < p = q \le \infty.
\end{cases}
\end{equation*}
\item[{\rm (ii)}] For $1 < p < q < \infty$, 
\begin{equation*} 
\varrho_m(\bBr,L_q)
 \ \asymp \ 
(m^{-1} \log^{d-1}m)^{r - 1/p + 1/q}(\log^{d-1}m)^{(1/q -
1/\theta)_+}\quad,\quad 
 \begin{cases}
2 \le p, \ 2 \le \theta \le q,  \\
q \le 2.
\end{cases}
 \end{equation*}
\end{itemize} 
\end{thm}

\subsubsection*{Complements to Section 8: Numerical integration}\index{Numerical integration}

\begin{thm}
{\em (ii)} For each $0 <  p,\theta\le\infty$ and $r>1/p$, we have 
\[
\Phi(a,A,\psi)(\Brpt) \,\asymp\, a^{-dr+d(1/p-1)_+} (\log
a)^{(d-1)(1-1/\theta)_+}\quad,\quad a>1\,.
\]
{\em (ii)} For each $0 <  p < \infty$ and $0<\theta \leq 1$, we have 
\[
\Phi(a,A,\psi)(\mathbf{B}^{1/p}_{p,\theta}) \,\asymp\,
a^{-d/\max\{p,1\}}\quad,\quad a>1\,.
\]
\end{thm}

Analogous estimates hold true for the Fibonacci-cubature formulas.

\subsubsection*{Nonlinear approximation of non-compact embeddings}
\index{Nonlinear approximation} \index{Embedding!Non-compact}

In contrast to the study of $s$-numbers it makes also sense to study non-compact
(but continuous) embeddings in connection with best $m$-term approximation. This
has been done recently in \cite{HaSi10}. Let us first clarify the precise
embedding situation. A proof may be found in \cite{Vyb06}. 

\begin{lem} Let $1<q<\infty$ and $r>0$. \\
{\em (i)} Let $1<p<\infty$. Then 
$$
    \Wrp \hookrightarrow L_q 
$$
if and only if $r\geq 1/p-1/q$.\\
{\em (ii)} Let $0<p<\infty$. Then 
$$
    \bB^r_{p,p} \hookrightarrow L_q
$$
if and only if $r\geq 1/p-1/q$\,.\\
{\em (iii)} The embeddings are compact if and only if $r>1/p-1/q$\,. The
embeddings (i) and (ii) also hold true for spaces $\bB^r_{p,p}(\R^d)$ and
$\Wrp(\R^d)$ if $p\leq q$ and 
$r\geq 1/p-1/q$ but they are never compact.
\end{lem}

The compact embeddings $r>1/p-1/q$ have been discussed above. Now we are interested in $r=1/p-1/q$. Again we use a
wavelet basis of sufficiently smooth wavelets as dictionary $\Phi$. 

\begin{thm}\label{noncomp1} {\em (i)} Let $0<p\leq \max\{p,1\}<q<\infty$ and $r = 1/p-1/q$.
Then 
$$
    \sigma_m(\bB^r_{p,p},\Phi)_q \asymp m^{-r}(\log m)^{(d-1)(r-1/p+1/2)_+}\,.
$$
{\em (ii)} Let $1<p\leq 2\leq q<\infty$ and $r=1/p-1/q>0$. Then 
$$
    \sigma_m(\Wrp,\Phi)_q \asymp m^{-r}(\log m)^{(d-1)r}\,.
$$
\end{thm}

\begin{rem} In the case $q=2$ we obtain with $r=1/p-1/2$
\begin{equation}
    \sigma_m(\bB^r_{p,p},\Phi)_2 \asymp m^{-r}\,. \label{nonc}
\end{equation}
This has been already observed in \cite{SiUl09}, see also \cite{Ni06}. The
latter reference dealt with tensor products of univariate Besov spaces
$B^r_{p,p}$ which have been recently identified as spaces with dominating mixed
smoothness also in the case $p<1$, see \cite{SiUl09}. 
\end{rem}

We can say even more in the situation $q=2$ for spaces on $\R^d$. For $s>0$ and
$p$ as above we consider the nonlinear approximation space
$\mathbf{A}^s_p (L_2(\R^d))$, see e.g. \cite{D}, as the collection of all
$f\in L_2 (\R^d)$ such that
\[
\| \, f \, \|_{\mathbf{A}^s_p (L_2(\R^d))} := \| \, f \, \|_2+ \Big(
\sum_{m=1}^\infty \frac 1 m \, [m^s\, \sigma_m (f,\Phi)_2]^p\Big)^{1/p}
<\infty\, .
\]
Mainly as a corollary of the characterization of $\mathbf{A}^{\frac 1p -
\frac 12}_p (L_2 (\R^d))$, see e.g. Pietsch \cite{Pi}, DeVore
\cite[Thm.~4]{D} or \cite[Sect.\ 1.8]{Tbook}, we obtain the following
identification.

\begin{thm}\label{noncomp2}
Let $0< p<2$. Then we have
\[
\mathbf{A}^{\frac 1p - \frac 12}_p (L_2(\R^d)) = \bB^{1/p-1/2}_{p,p}(\R^d)
\]
in the sense of equivalent quasi-norms.
\end{thm}

\subsection{Sampling along lattices}\index{Sampling!Lattices}
\label{samp_lattices}

In Section 5 (Subsection \ref{sampwidths}) we encountered several situation where Smolyak's algorithm (sparse grids) represent the optimal sampling algorithm with respect to $\varrho_m$. However, in case $q\leq p$ an optimal 
sampling algorithm is only known in case $d=2$ and $p=q=\infty$. The natural question arises whether there are other discrete point set constructions and corresponding sampling 
algorithms which might behave better than sparse grids or perform even optimal. Here one could think about the big variety of point set constructions which are used for numerical integration
 based on number theoretic constructions, see 
\cite{DiKuSl13} or Section \ref{numint} above. This question has been considered by several authors in the past. We refer to \cite{VT27}, \cite{KuSlWo06}, \cite{KuWaWo09}, \cite{KaKu12}, \cite{Ka13}, \cite{KaPoVo15}, \cite{KaPoVo15_2}. 
It turned out that the proposed methods provide certain numerical advantages since the lattice structure allows for using the classical Fast Fourier Transform for the evaluation and 
reconstruction of trigonometric polynomials. Moreover, in contrast to sparse grid discretization techniques
the use of oversampled lattice rules provides better numerical stability properties \cite{KaKu12}, \cite{KaKu11}. However, when expressing the error in terms of the used 
function values (in the sense of $\varrho_m$) it turns out, see the recent paper \cite{ByKaUlVo16}, that one only can expect half the rate of convergence of a sparse grid method, i.e.
$$
	  \varrho_m^{\text{latt}_1}(\bW^r_2,L_2) \gtrsim m^{-r/2}\,.
$$
The quantities on the left-hand side are similarly defined as $\varrho_m$ above with the difference that we take the $\inf$ over all sampling algorithms taking function values 
along rank-$1$ lattices with $m$ points. Surprisingly, such methods never perform optimal with respect to the ``information complexity''. However, as already mentioned in Subsection \ref{sampsmol}, the runtime 
of such an algorithm can be accelerated significantly by exploiting the lattice structure. 

A related framework is studied in the papers \cite{BaDaDeGr14}, \cite{NoRu16}. The authors deal with sampling recovery of the class or rank$-1$ tensors
$$
    F^r_{M,d}:=\Big\{f = \bigotimes\limits_{i=1}^d f_i~:~\|f_i\|_{\infty}\leq 1, \|f_i^{(r)}\|_{\infty} \leq M\Big\}
$$
which represents a subclass of the unit ball in the space $\bW^r_{\infty}$ if $M=1$. The authors in \cite{BaDaDeGr14} use a low discrepancy point set (Halton points, see Section 8) to propose a sampling recovery algorithm with $N$
points and $L_\infty$-error $\lesssim C(r,d)N^{-r}$ with constant $C(d,r)$ scaling like $d^{dr}$. The question after the curse of dimensionality arises. By proposing a Monte-Carlo sampling recovery algorithm 
the authors in \cite{NoRu16} prove that the curse of dimensionality is present (in the randomized setting) if and only if  $M\geq 2^rr!$. In case $M<2^rr!$ the complexity is only polynomial in the dimension. 

\subsection{Sampling recovery in energy norm}\index{Sampling!Recovery!Energy norm}
\label{Sampling recovery in energy norm}

There is a large class of solutions of the electronic Schr\"odinger equation 
in quantum chemistry, which belong to Sobolev spaces with mixed regularity and, moreover, possess  
some additional  Sobolev isotropic smoothness properties, see Yserentant's recent lecture notes 
\cite{Y2010} and the references therein. This type of regularity is precisely expressed by the spaces 
$\bW^{r,\beta}_p$ for $p=2$ introduced in \cite{GrKn09}. Here, the parameter $r$ reflects the 
smoothness in the dominating mixed sense and 
the parameter $\beta$ reflects the smoothness in the  isotropic sense. 
In \cite{BG99,BuGr04,GrKn00,Gr05,GrKn09,ScSuTo08}, the authors used Galerkin methods for the $W^1_2(\T^d)$-approximation of the solution of general elliptic variational problems. 
In particular, by use of tensor-product biorthogonal wavelet bases, 
the authors of \cite{GrKn09} constructed so-called optimized sparse grid  subspaces for 
finite element approximations of the solution having $\bW^{r,\beta}_2$-regularity, 
whereas the approximation error is measured in the energy norm of isotropic Sobolev space $W^\gamma_2$.
For non-periodic functions of mixed smoothness $r$ from $\bW^r_2$, 
linear sampling algorithms on sparse grids have been investigated by Bungartz and Griebel 
\cite{BuGr04} employing hierarchical Lagrangian polynomials multilevel basis  
and measuring the approximation error in energy $W^1_2$-norm.

The problem of sampling recovery on sparse grids in energy norm $W^\gamma_q$ of functions from  classes 
$\bW^{r,\beta}_p$ and $\bB^{r,\beta}_{p,\theta}$ has been investigated in \cite{ByDuUl14,Di14}.
The space $\bB^{r,\beta}_{p,\theta}$ is a ``hybrid'' of the space $\Brpt$ of mixed smoothness $r$ and the classical isotropic Besov space $B^\beta_{p,\theta}$ of smoothness $\beta$. 
The space $\bB^{r,\beta}_{p,\theta}$ can be seen as a Besov type generalization of the space 
$\bW^{r,\beta}_2 = \bB^{\alpha,\beta}_{2,2}$. For the definition of these spaces, see \cite{ByDuUl14,Di14}.
It turns out that for this sampling recovery, we can achieve the right order of the sampling widths.

\begin{thm}  \label{Sampling_Energy_NormW}
 Let $r, \beta, \gamma\,  \in \R$ 
 such that $\min\{r, r + \beta\} >1/2$, $\gamma \ge 0$ and $0 <  \gamma - \beta < r$.
Then we have
\[
\varrho_m(\bW^{r,\beta}_2, W^\gamma_2(\T^d)) \asymp 
m^{-(r-\gamma + \beta)} .  
\]
\end{thm}

\begin{thm}  \label{Sampling_Energy_NormB}
Let $0 < p, \theta \le \infty$, $1 < q < \infty$, $r > 0$, $\gamma \ge 0$ and  
$\beta \in \R$, $\beta \not= \gamma$. Assume that there hold the conditions $\min (r,r + \beta) > 1/p$ and
\begin{equation} \nonumber
r > 
\begin{cases}
(\gamma -  \beta)/d, \ & \beta > \gamma, \\
\gamma -  \beta, \ & \beta < \gamma.
\end{cases} 
\end{equation}
Then we have 
\begin{equation} \nonumber
\varrho_m(\bB^{r,\beta}_{p,\theta}, W^\gamma_q(\I^d)) 
\ \asymp 
\begin{cases}
m^{- r  - (\beta - \gamma)/d + (1/p - 1/q)_+},  &  \beta > \gamma, \\
m^{- r  - \beta + \gamma + (1/p - 1/q)_+},  &  \beta < \gamma.
\end{cases}
\end{equation}
\end{thm}

Theorem~\ref{Sampling_Energy_NormW} has been proven in \cite{ByDuUl14} and 
Theorem~\ref{Sampling_Energy_NormB} in \cite{Di14}. 
Special sparse grids are constructed for sampling recovery of  $\bW^{r,\beta}_p$ and $\bB^{r,\beta}_{p,\theta}$. They have much smaller number of sample points
than the corresponding standard full grids and Smolyak grids, but give the same error of the sampling recovery on the both latter ones. The construction of asymptotically optimal linear sampling algorithms 
 is essentially based on quasi-interpolation representations by Dirichlet kernel series or B-spline series of functions from  
 $\bW^{r,\beta}_p$ and $\bB^{r,\beta}_{p,\theta}$ with a discrete equivalent quasi-norm in terms of the coefficient function-valued functionals of this series. For details and more results on this topic see  
\cite{ByDuUl14,Di14}.

\subsection{Continuous algorithms in $m$-term approximation and nonlinear widths}\index{Width!Nonlinear}
\label{nonlinear widths}

Notice that if $X$ is separable Banach space and the dictionary $\Di$ is  dense in the unit ball of $X$, 
then the best $m$-term approximation of $f$ with respect to $\Di$ vanishes, i.e., $\sigma_m(f,\Di)_X = 0$ for any $f
\in X$. 
On the other hand, for almost all well-known dictionaries with good approximation properties the quantity
$\sigma_m(\bF,\Di)_X$ has reasonable lower bounds for well-known classes of functions $\bF$ having a common smoothness. 
 Therefore, the first problem which actually arises,  is to impose additional conditions on  $\Di$ and/or methods of
$m$-term approximation.
In \cite{Di98,KTE1} to obtain the right order of $\sigma_m(\bF,\Di)_X$ for certain function classes the dictionary $\Di$
is
required to satisfy some ``minimal properties''. 
 
Another approach to dealing with this problem  is to impose continuity assumptions on methods of $m$-term approximation
by the elements of $\Sigma_m(\Di)$ which is the set of all linear combinations $g$ of the form
$
g = \sum_{j=1}^m a_j g_j, \ g_j \in \Di.
$
This does not weaken the rate of the $m$-term approximation for many well-known dictionaries and function classes.
The continuity assumptions on approximation methods certainly lead to various notions of nonlinear width, in particular,
the classical  Alexandroff width and the nonlinear manifold width \cite{DHM89,Math90}. 
Thus, the classical  Alexandroff width $a_m(\bF,X)$ characterizes the best approximation of $\bF$ by continuous methods
as
mappings from $\bF$ to a complex of topological dimension $\le n$ (see, e.g., \cite{Ti90} for an exact definition). 
Let us recall one of the notions of nonlinear widths based on continuous methods of $m$-term approximation suggested in
\cite{Di00}. 

A continuous method of $m$-term approximation with regard to the dictionary $\Di$ is called a continuous mapping from
$X$ into $\Sigma_m(\Di)$. Denote by  ${\mathbb F}$ the set of all dictionaries $\Di$ in $X$ such that the intersection
of 
$\Di$ with any finite dimensional linear subspace is a finite set, and by $C(X,\Sigma_m(\Di))$ the set of all continuous
mappings $S$ from $X$ into $\Sigma_m(\Di)$. The restriction with  continuous methods of $m$-term approximation and
dictionaries $\Di \in {\mathbb F}$ leads to the notion of nonlinear width:
\[
\alpha_m(\bF,X) :=   \inf_{\Di \in {\mathbb F}} \ \inf_{S \in C(X,\Sigma_m(\Di))} \ \sup_{x \in \bF} \|f - S(f)\|_X.
\]

\begin{thm}\label{thm:tau_Brpt} Let  $1<p,q<\infty$ and $2 \le \theta \le \infty$ and $r > \max(0,1/p-1/q,1/p-1/2)$.
 Then we have
\[ 
a_m(\Wrp,L_q) 
\ \asymp \
\alpha_m(\Wrp,L_q) 
\ \asymp \ 
m^{-r}(\log m)^{(d-1)r},
\]
and
\[
a_m(\Brpt,L_q) 
\ \asymp \
\alpha_m(\Brpt,L_q) 
\ \asymp \ 
m^{-r}(\log m)^{(d-1)(r+1/2-1/\theta)}.
\]
\end{thm}

This theorem was proven in \cite{Di00}. A different concept of non-linear width based on pseudo-dimension related to
Learning Theory, but without continuity assumptions was suggested in \cite{RM98,RM99} and investigated for
function classes of mixed smoothness in  \cite{Di01,Di01_2}.  For other notions of nonlinear widths, related results,
surveys and bibliography, see \cite{Di00,Di01,Di01_2,DV96,D,DHM89,DKLT93,KTE1,Math90,RM98,RM99,St75,Ti76,Ti90}.

%% file: highdim.tex
\section{High-dimensional approximation}\index{High-dimensional approximation}
\label{Sect:highdim}
We explained in Subsections 9.1 and 9.2 that classes with mixed smoothness play
a central role among the classes of functions with finite smoothness. A typical
error bound for numerical integration, sampling, and hyperbolic cross
approximation  of, say, $\bW^r_p$ in $L_p$ is of the form $m^{-r}(\log
m)^{(d-1)\xi}$, where $\xi>0$ may depend on $r$ and $p$. 
The classical setting of the problem of the error behavior asks for dependence
on $m$, when other parameters $r$, $p$, $d$ are fixed. In this survey we mostly
discuss this setting. Then the factor $ (\log m)^{(d-1)\xi}$ plays a secondary
role compared to the factor $m^{-r}$. 
However, as we pointed out in the Introduction, the problems with really large
$d$ attract a lot of attention. In this case, for reasonably small $m$ the
factor $(\log m)^{(d-1)\xi}$ becomes  a dominating one and makes it impossible
to apply a model of classes with finite smoothness for practical applications.
It is a very important and difficult problem of contemporary numerical analysis.
One of the promising ways to resolve this problem is to use classes of functions
with special structure instead of smoothness classes. This approach is based on
the concept of {\it sparsity} (with respect to a dictionary) and is widely used
in compressed sensing and in greedy approximation. For instance, sparsity with
respect to a dictionary with tensor product  structure is a popular structural
assumption (see Section 6). 

\subsection{Anisotropic mixed smoothness} \index{Anisotropic mixed smoothness}
\label{Sect:aniso}
We point out that classes with mixed smoothness also have a potential for
applications  in high-dimensional approximation. To illustrate this point we
begin with the anisotropic version of classes $\bW$, $\bH$, $\bB$. So far we 
discussed isotropic classes of functions with mixed smoothness. Isotropic
means that all variables play the same role in the definition of our smoothness
classes. In the hyperbolic cross approximation theory anisotropic classes of
functions with mixed smoothness are of interest and importance. The framework described below 
goes back to the work of Mityagin \cite[pp.\ 397, 409]{Mit} in 1962 and Telyakovskii \cite[p.\ 438]{Tely64} in 1964
and has been later used by several authors from the former Soviet Union, see \cite[pp. 32, 36, 72]{Tmon} (English version) 
for further historical comments.

We give the corresponding definitions. Let $\br=(r_1,\dots,r_d)$ be such that
$0<r_1=r_2=\dots=r_\nu<r_{\nu+1}\le r_{\nu+2}\le\dots\le r_d$ with $1\le \nu \le
d$. 
For $\bx=(x_1,\dots,x_d)$ denote
$$
F_\br(\bx) := \prod_{j=1}^d F_{r_j}(x_j)
$$
and
$$
\bW^\br_p := \{f:f=\varphi\ast F_\br,\quad \|\varphi\|_p \le 1\}.
$$
We now proceed to classes $\bH^\br_q$ and $\bB^\br_{q,\theta}$.
Define
$$
\|f\|_{\bH^\br_q}:= \sup_\bs \|\delta_\bs(f)\|_q 2^{(\br,\bs)},
$$
and for $1\le \theta <\infty$ define
$$
\|f\|_{\bB^\br_{q,\theta}}:= \left(\sum_{\bs}\left(\|\delta_\bs(f)\|_q
2^{(\br,\bs)}\right)^\theta\right)^{1/\theta}.
$$
We will write $\bB^\br_{q,\infty}:=\bH^\br_q$. 
Denote the corresponding unit ball
$$
\bB^\br_{q,\theta}:= \{f: \|f\|_{\bB^\br_{q,\theta}}\le 1\}.
$$
It is known that in many problems of estimating asymptotic characteristics the
anisotropic classes of functions of $d$ variables with mixed smoothness behave
in the same way as isotropic classes of functions of $\nu$ variables (see, for
instance, \cite{Tmon,Di86,Telyak88}). It is clear that the above remark holds for the lower
bounds. To prove it for the upper bounds one needs to adapt the
hyperbolic cross approximation to the smoothness of the class. We pay more
samples in those directions where the smoothness is small. This results in an
anisotropic hyperbolic cross. 

What concerns numerical integration, the studied methods like the Fibonacci and
Frolov cubature formulas need not to be adapted. These methods are
able to ``detect'' the anisotropy themselves.

\begin{figure}[H]
\begin{center}

\begin{tikzpicture}[scale=0.2]
\pgfmathsetmacro{\m}{5};
\pgfmathsetmacro{\rf}{1};
\pgfmathsetmacro{\rs}{2};
\pgfmathtruncatemacro{\xl}{2^(\m / \rf)}
\pgfmathtruncatemacro{\yl}{2^(\m / \rs)}
\pgfmathtruncatemacro{\limf}{round(\m / \rf)}

\pgfmathtruncatemacro{\lima}{round(\m / \rf)}
\pgfmathtruncatemacro{\limb}{round(\m / \rs)}

\foreach \x in {1,...,\lima}
{
\pgfmathtruncatemacro{\lim}{round((\m-\rf*\x)/(\rs))}
\foreach \y in {1,...,\limb}
{
\pgfmathtruncatemacro{\xm}{\x-1}
\pgfmathtruncatemacro{\ym}{\y-1}
    \draw[dashed] (2^\xm,2^\ym)-- (2^\x,2^\ym) --(2^\x,2^\y) -- (2^\xm,2^\y) --
(2^\xm,2^\ym);
        \draw[dashed] (-2^\xm,-2^\ym)-- (-2^\x,-2^\ym) --(-2^\x,-2^\y) --
(-2^\xm,-2^\y) -- (-2^\xm,-2^\ym) ;
        \draw[dashed] (-2^\xm,2^\ym)-- (-2^\x,2^\ym) --(-2^\x,2^\y) --
(-2^\xm,2^\y) -- (-2^\xm,2^\ym) ;
        \draw[dashed] (2^\xm,-2^\ym)-- (2^\x,-2^\ym) --(2^\x,-2^\y) --
(2^\xm,-2^\y) -- (2^\xm,-2^\ym) ;
}
}

\foreach \x in {0,...,\limf}
{
\pgfmathtruncatemacro{\lim}{round((\m-\rf*\x)/(\rs))}
\foreach \y in {0,...,\lim}
{
\pgfmathtruncatemacro{\xm}{\x-1}
\pgfmathtruncatemacro{\ym}{\y-1}
\ifnum \x<1 \relax%
     \ifnum \y<1 \relax%
         draw[color=black,very thick] (-1,-1)-- (1,-1) --(1,1) -- (-1,1) --
(-1,-1);
    \else
        \draw[color=black,very thick] (-1,-2^\ym)-- (1,-2^\ym) --(1,-2^\y) --
(-1,-2^\y) -- (-1,-2^\ym) ;
        \draw[color=black,very thick] (-1,2^\ym)-- (1,2^\ym) --(1,2^\y) --
(-1,2^\y) -- (-1,2^\ym) ;
     \fi
\else%
    \ifnum \y<1 \relax%
            \ifnum \x<1 \relax%

            \else
            \draw[color=black,very thick] (2^\xm,-1)-- (2^\x,-1) --(2^\x,1) --
(2^\xm,1) -- (2^\xm,-1) ;
            \draw[color=black,very thick] (-2^\xm,-1)-- (-2^\x,-1) --(-2^\x,1)
--
(-2^\xm,1) -- (-2^\xm,-1);
            \fi
    \else
    \draw[color=black,very thick] (2^\xm,2^\ym)-- (2^\x,2^\ym) --(2^\x,2^\y) --
(2^\xm,2^\y) -- (2^\xm,2^\ym);
        \draw[color=black,very thick] (-2^\xm,-2^\ym)-- (-2^\x,-2^\ym)
--(-2^\x,-2^\y) -- (-2^\xm,-2^\y) -- (-2^\xm,-2^\ym) ;
        \draw[color=black,very thick] (-2^\xm,2^\ym)-- (-2^\x,2^\ym)
--(-2^\x,2^\y) -- (-2^\xm,2^\y) -- (-2^\xm,2^\ym) ;
        \draw[color=black,very thick] (2^\xm,-2^\ym)-- (2^\x,-2^\ym)
--(2^\x,-2^\y) -- (2^\xm,-2^\y) -- (2^\xm,-2^\ym) ;
    \fi
\fi

}}

\end{tikzpicture} 
  \caption{Anisotropic hyperbolic cross for $\br = (1,2)$ in $d=2$}\index{Hyperbolic cross!Anisotropic}
 \end{center}
\end{figure}
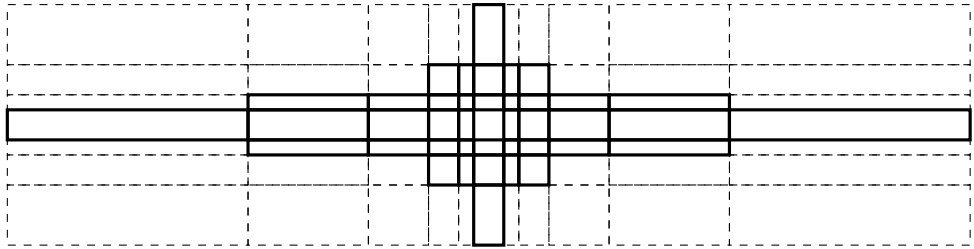

The consequence is that the logarithmic factor $(\log m)^{(d-1)\xi}$ is
replaced by the factor $(\log m)^{(\nu-1)\xi}$. Thus, if we use a model of
anisotropic class with mixed smoothness $\br$ with small $\nu$, the factor
$(\log m)^{(\nu-1)\xi}$ is not a problem anymore. In this case it is important
to study dependence on $d$ of the constants. In the recent paper \cite{DG15}, it has been shown that with a fixed $\nu$
and some moderate conditions on the anisotropic mixed smoothness $\br$ the rate of the hyperbolic cross 
approximation does not depend on the dimension $d$, when $d$ may be very large or even infinite 
(see also Subsection \ref{approximation in infinite dimensions}).

\subsection{Explicit constants and preasymptotics}\index{Preasymptotics}
\label{exconst}

As already mention there is an increasing interest in approximation problems where the
dimension $d$ is large or even huge. Since then people were not only interested in getting the right order of the respective
approximation error, also the dependence of the error bounds on the underlying dimension $d$ became crucial. 
In case of high dimensions, the traditional estimate \eqref{BabMit} becomes problematic. Let us illustrate this issue in the following
figure. Fixing $d$ and $r$ the function $f_{d}(t):= t^{-r}\,
(\log t)^{r(d-1)}$ is increasing on $[1,e^{d-1}]$ and decreasing on $[e^{d-1},\infty)$.\\

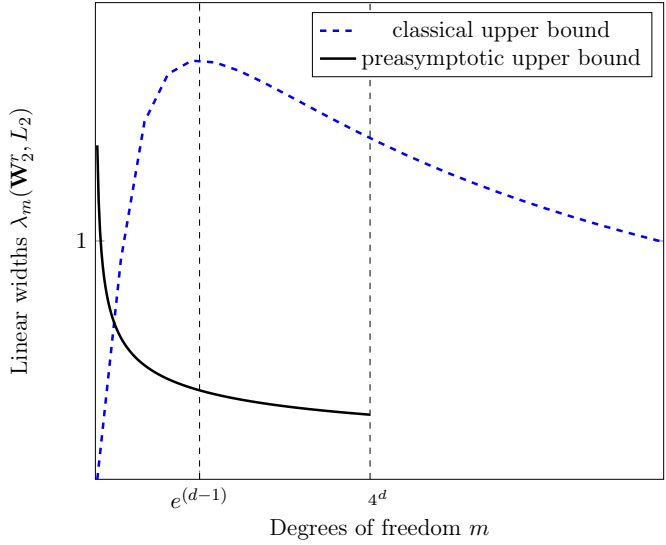
\begin{figure}[H]
\centering
\begin{tikzpicture}[trim axis left, scale = 0.8]
\begin{axis}[ xmin=0, xmax=300, domain=1:300,
    xlabel={Degrees of freedom $m$}, ylabel={Linear widths
$\lambda_m(\bW^r_2,L_2)$},
    xtick={55}, xticklabels={ $e^{(d-1)}$},
    ytick={10}, yticklabels={$1$},
    ymin=0, ymax=20,
    legend entries={classical upper bound, preasymptotic upper bound},
    width=0.7\textwidth,
]
\addplot[
	very thick,
	blue,
	no marks,
	dashed,
] {0.8*(((ln(x))^4)/x)^2};
\addplot[
	very thick,
	black,
	no marks,
	domain=1:145,
	samples=200,
] {(exp(8)/x)^0.33};
\draw[dashed] (55,0) -- (55,220);
\draw[dashed] (145,0) -- (145,220);
\end{axis}
\draw (4.7,0) node[below] {\tiny $4^d$};
\end{tikzpicture}
\caption{Asymptotics vs. preasymptotics}\index{Preasymptotic}
\label{preas}
\end{figure}

Hence, its maximum on $[1,\infty)$ is \enspace
$
\max_{t\ge 1} f_{d}(t) = f_{d}(e^{d-1}) = \big(\frac{d-1}{e}\big)^{r(d-1)}\, ,
$
which increases super-exponentially in $d$. 
That means, for large $d$ we have to wait ``exponentially long'' until the 
sequence $n^{-r}(\ln n)^{(d-1)r}$ decays, and even longer until it becomes less than one. This observation indicates the
so-called ``curse of dimensionality'', a phrase originally coined by Bellmann \cite{Be57} in 1957. For a mathematically
precise notion see \cite[p.\ 1]{NW08}. How to avoid that the logarithm increases exponentially in $d$ has been
discussed in Subsection \ref{Sect:aniso}. This could be one approach for making the
problem tractable in high dimensions. In any case it is important to control the
behavior of the constants $c(r,d)$ and $C(r,d)$ in $d$ and, in a second step, to establish {\em preasymptotic
estimates} for small $m$. 

The first result on upper bounds in this direction (to the authors'
knowledge) has been given by Wasilkowski, Wo\'zniakowski \cite[Sect.\ 4.1]{WaWo95}. There the authors studied Smolyak's
algorithm in a general framework (similar as done in Section \ref{Subs:appr_hypcross}) and obtained upper bounds for
the error based on the behavior of the participating univariate operators.

Linear and Kolmogorov widths of Sobolev
classes of a mixed smoothness in $L_2(\T^d)$-norm and so-called energy norms in the high-dimensional setting have been
first investigated by Dinh D\~ung and Ullrich in \cite{DU13} where it is stressed to treat the $d$-dependent upper
bound and lower bound of these widths together. Of course, the specific definition of the space is now of particular
relevance. Here, the class $\bW^r_2$ is defined as the set of function in $L_2(\T^d)$ 
for which the right hand side of \eqref{NeqW} is finite. 
In particular, the following almost precise upper and lower bounds explicit in the dimension $d$ have been proven in
\cite{DU13} (see also Theorem \ref{thm:lambdaH1} below). 
\begin{thm}\label{thm:DU13}
 For any $r > 0$ and $m \geq 2^d$
\begin{equation*} 
\frac{1 + \log e}{4^r(d-1)^{r(d-1)}} \biggl(\frac{m}{(\log m)^{d-1}}\biggl)^{-r}  
\le 
d_m(\bW^r_2, L_2(\T^d)) 
\le 
4^r\Big(\frac{2e}{d-1}\Big)^{r(d-1)} \biggl(\frac{m}{(\log m)^{d-1}}\biggl)^{-r}.
\end{equation*}
\end{thm}
One observes a super-exponential decay of the constants. Other results have been
obtained recently by K\"uhn, Sickel, Ullrich \cite{KSU15}. In contrast to Theorem \eqref{thm:DU13} the authors 
considered the space $\bW^r_2$ to be specifically normed as follows
\begin{equation}\label{bW_2}
    \|f\|^2_{\bW^r_2}:= \sum\limits_{\bk\in \Z^d} |\hat{f}(\bk)|^2 \prod\limits_{j=1}^d (1+|k_j|^2)^r\,.
\end{equation}
It has been shown in \cite{KSU15} that in this situation the ``asymptotic constant''\index{Asymptotic constant} behaves exactly as follows
\begin{equation}\label{asconst}
    \lim_{m \to
  \infty} \, \frac{m^r}{(\log m)^{(d-1)r}}\cdot \lambda_m(\bW^r_2,
L_2) \, =
  \Big[\frac{2^d}{(d-1)!}\Big]^r\,.
\end{equation}
This result is surprising from several points of view. First, the limit
exists, second one can compute it explicitly and third, the number on the
right-hand side decays exponentially in $d$. The next step is to ask for
estimates of type \eqref{BabMit} with precise given constants in some range
$m\geq m_0$. There are several results in this direction, see for instance
\cite[Thm.\ 3.8]{BuGr04}, \cite{DU13}, \cite[Thm.\ 5.1]{ScSuTo08}, where the super-exponential decay of the constant has
been already observed in different periodic and non-periodic settings. For a thorough discussion and comparison of the
mentioned results we refer to \cite[Section 4.5]{KSU15}. To be consistent with the setting described in \eqref{bW_2} let
us state the following result from
\cite{KSU15}. 

\begin{thm} Let $r>0$ and  $d\in \N$. Then we have\\
 $$
  \lambda_{m} (\bW^r_2,L_2) \le
\Big[\frac{(3\cdot \sqrt{2})^d}{(d-1)!}\Big]^r\frac{(\ln m)^{(d-1)r}}{m^r}\,, 
\qquad \mbox{if}\quad m \ge 27^d
 $$
and
$$
  \lambda_{m} (\bW^r_2,L_2) \ge 
  \left[
  \frac{3}{d!}\, \Big(\frac{2}{2 + \ln 12}\Big)^{d}\right]^r\, \frac{(\ln m)^{(d-1)r}}{m^r} 
  \,,
\qquad 
\mbox{if}\quad m >  (12 \, e^2)^{d} \, .
 $$
\end{thm}
\noindent Note, that the constants decay super-exponentially in $d$ as expected
from \eqref{asconst}. However, if $d$ is large we have to ``wait'' very long
until that happens. Hence, the next question is what happens in the
preasymptotical range, say for $m$ less than $2^d$. Here, we get the bound below in Theorem \ref{PA}
(see \cite[Thm.\ 4.17]{KSU15} and Figure \ref{preas} above). For similar results in more general classes as well as a 
non-periodic counterpart, see the longer arXiv version of \cite{CD13}. 

\begin{thm}\label{PA} 
Let $d\geq 2$, and $r>0$. Then for any $2\leq m \leq \frac{d}{2}4^d$ we have
the upper estimate 
$$
    \lambda_m(\bW^r_2,L_2) \leq 
\Big(\frac{e^2}{m}\Big)^{\frac{r}{4+2\log_2 d}}\,.
$$
\end{thm}
\noindent Note, that there is no hidden constant in the upper bound. This type
of error decay reflects ``quasi-polynomial'' tractability, a notion recently
introduced by Gnewuch, Wo\'zniakowski
\cite{GnWo11}. The result in Theorem \ref{PA} is based on a refined estimate for the cardinality of the hyperbolic cross 
$$
    \tilde{\Gamma}(N,d):=\Big\{k\in \Z^d~:~\prod\limits_{j=1}^d (1+|k_j|) \leq N \Big\}\quad,\quad N\in \N\,,
$$
compare with \eqref{HC} above. In \cite[Thm.\ 4.9]{KSU15} it is shown
\begin{equation}\label{cardhyp}
  |\tilde{\Gamma}(N,d)| \leq e^2 N^{2+\log_2 d}\,.
\end{equation}
In \cite{CoDeFoRa11} the authors state cardinality bounds (without proof) for slightly modified hyperbolic crosses in $d$ dimensions.

What concerns the approximation in the uniform norm $L_\infty$ in case $r>1/2$ we refer to the recent paper
Cobos, K\"uhn, Sickel \cite{CoKuSi15}. Analyzing the formula \eqref{CKS} the authors obtained the asymptotic constant 
$$
    \lim\limits_{m\to\infty} \frac{m^{r-1/2}\lambda_m(\bW^r_2,L_\infty)}{(\log m)^{(d-1)r}} =
\frac{1}{\sqrt{2r-1}}\Big[\frac{2^d}{(d-1)!}\Big]^{r}\,.
$$
Preasymptotic error bounds for isotropic periodic Sobolev spaces are given in \cite{KSU14}, \cite{KuMaUl16}.

\subsection{Approximation in the energy norm}\index{Energy norm}

Motivated by the aim to approximate the solution of a Poisson equation in the energy norm, i.e., in the norm of the
isotropic Sobolev space $H^1$, Bungartz and Griebel \cite{BG99} investigated upper estimates of the quantities
$\lambda_m(\bW^r_2, H^1)$, where $H_1$ denotes the isotropic Sobolev space in $L_2$, the so-called energy space. 
These studies have been continued in Griebel, Knapek \cite{GrKn00, GrKn09}, Bungartz, Griebel \cite{BuGr04}, Griebel
\cite{Gr05}, Schwab, S\"uli, Todor \cite{ScSuTo08}, and Dinh D\~ung, Ullrich \cite{DU13}. A first result on upper bounds in this setting is due to Griebel
$$
      \lambda_m(\bW^2_\infty([0,1]^d),H^1) \leq C_{d} m^{-1}\quad,\quad m\in \N\,.
$$
Note, that the usual $\log$-term disappears. In other words, there is no bad $d$-dependence in the rate. For
constructing an appropriate approximant we have to modify the standard hyperbolic cross according to the
energy-norm. Instead of \eqref{HC} we use projections onto the ``energy hyperbolic cross''
$$
    Q^E_n:=\bigcup\limits_{r|\bs|_1 - |\bs|_\infty\leq n} \rho(\bs)\,.
$$
Griebel \cite{Gr05} observed that in case $r=2$ the constant $C_{d}  = d^2\cdot 0.97515^d$ suffices for large $m > m_d$. The precise range for $m$ 
has not been given. The following theorem is stated in \cite[Thm.\ 4.7(ii)]{DU13}. 
\begin{thm}\label{thm:lambdaH1}
Let $1 < r \le 2$. Then one can
precisely determine a threshold $\lambda = \lambda(r)> 1$ such that for $m>\lambda^d$, the correct upper and lower
bounds explicit in the dimension $d$
\begin{equation}\label{lambdaH1}
 C_r d^{r-1}\biggl(\frac{1}{2^{1/(r-1)}-1}\biggl)^d \, m^{-(r-1)}
\leq 
\lambda_m(\tilde{\bW}^r_2,H^1(\T^d)) 
\leq C_r' d^{r-1}\biggl(\frac{1}{2^{1/(r-1)}-1}\biggl)^d \, m^{-(r-1)}
\end{equation}
hold true with some explicit positive  constants $C_r,C_r'$ depending on $r$ only. 
\end{thm}
Here $\tilde{\bW}^r_2$ is the subspace of $\bW^r_2$ containing all functions $f$ such that 
$\hat{f}(\bk) \neq 0 \implies \prod_{i=1}^d k_i \neq 0$, and $\bW^r_2$ is defined as the set of function in $L_2(\T^d)$ 
for which the right hand side of \eqref{NeqW} is finite\,. The upper bound in \eqref{lambdaH1} also holds true for every $r > 1$. Note, that the terms depending on $d$ in the both sides of \eqref{lambdaH1} are the same and decay exponentially in $d$ for $r<2$.  For $r=2$, the relations \eqref{lambdaH1} becomes\
\begin{equation}\nonumber
 C_2\, d\, m^{-1}
\leq 
\lambda_m(\tilde{\bW}^2_2,H^1(\T^d)) 
\leq C_2'\, d \, m^{-1}.
\end{equation}
For more results on this direction, see \cite{DU13}.
For a detailed comparison of the mentioned results we refer again to \cite[Sect.\ 4.5]{KSU15}.

\subsection{$\varepsilon$-dimension and approximation in infinite dimensions}
\label{approximation in infinite dimensions}

In computational mathematics, the so-called $\varepsilon$-dimension
$n_\varepsilon = n_\varepsilon(\bF,X)$ is used to quantify the computational complexity (in Information-Based Complexity
(IBC) a similar object is termed {\em information complexity} or {\em $\varepsilon$-cardinality}). 
It is defined by
\be\label{epsdim}
n_\varepsilon(\bF,X)
:= \ 
\inf \left\{m \in \N:\, \exists L_m: \, \sup_{f \in \bF} \ \inf_{g \in L_m} \|f - g\|_X \le \varepsilon \right\},
\ee 
where $L_m$  denotes a linear subspace in $X$ of dimension $\le m$. This approximation characteristic is the inverse of
$d_m(\bF,X)$.
In other words, the quantity $n_\varepsilon(\bF,X)$ is  the minimal number
$n_\varepsilon$ such that the approximation of $\bF$ by a suitably
chosen approximant $n_\varepsilon$-dimensional subspace $L$ in $X$ gives the approximation error 
$\le \varepsilon$ (see \cite{DD79,DD80}).  
The quantity $n_{\varepsilon}$ represents a modification of the information
complexity (used in IBC) which is described by the minimal number $n(\varepsilon,d)$ 
of ``linear'' information (in case of $\Lambda^{\text{all}}$) needed to solve the corresponding $d$-variate linear
approximation problem of the identity operator within accuracy 
$\varepsilon$. In contrast to \eqref{epsdim} this quantity is defined as the inverse of the well-known Gelfand width,
see \cite[4.1.4, 4.2]{NW08} and \eqref{gelf} above. For further information on this topic we refer the interested
reader to the books \cite{NW08, NoWo09, NoWo12} and the references therein. 
The following theorem on $\varepsilon$-dimensions of the Sobolev class $\bW^r_2$ in the high-dimensional setting (see
the previous subsection) has been proved in \cite{CD13}.

\begin{thm} \label{n_e[d]}
Let $r >  0$, $d \ge 2$. Then we have for every $\varepsilon \in (0,1]$, 
\begin{equation}\nonumber
n_\varepsilon(\bW^r_2,L_2(\T^d)) 
\  \leq \ 
\frac{2^d}{(d-1)!} \, \varepsilon^{-1/r}\,(\ln \varepsilon^{-1/r} + d \ln 2)^{d-1}
\end{equation}
and for every $\varepsilon \in (0, (2/3)^{rd})$
\begin{equation*} 
n_\varepsilon(\bW^r_2,L_2(\T^d)) 
\  \geq \ 
\frac{2^d}{(d-1)!} \, 
\frac{\varepsilon^{-1/r}  [\ln \varepsilon^{-1/r}  - d\ln (3/2)]^d}
{\ln \varepsilon^{-1/r}  - d\ln (3/2) + d}
\ - \ 1.
\end{equation*}
\end{thm}
\noindent 
It is worth to emphasize that in high-dimensional approximation, the form of the upper and lower bounds for  
$n_\varepsilon(\bW^r_2,L_2(\T^d))$ as in Theorem \ref{n_e[d]} is more natural and suitable than the form as used 
in the traditional form where the terms $\varepsilon^{- 1/r} |\log \varepsilon|^{(s-1)/r }$
are a priori split from constants which are actually a function of dimension parameter $d$ 
(and smoothness parameter~$r$), and therefore, any high-dimensional estimate based on them leads to a rougher bound. 
The situation is analogous for Kolmogorov $m$-widths of classes of functions having a mixed smoothness if the terms 
$m^{- r} (\log m)^{r (d-1)}$  are a priori split from constants depending on the dimension $d$. 

See also \cite{DU13} for some similar results, and \cite{CD13} for a general version of Theorem \ref{n_e[d]}
as well an extension to non-periodic functions.
From Theorem \ref{n_e[d]} we have
\begin{equation*} 
\lim\limits_{\varepsilon \to 0}\, 
\frac{n_\varepsilon(\bW^r_2,L_2(\T^d))}{\varepsilon^{-1/r} (\ln \varepsilon^{-1/r})^{d-1}} 
\  =  \ 
\frac{2^d}{(d-1)!}.
\end{equation*}

\bigskip
\noindent
The efficient approximation of a function of infinitely many variables is an important issue for a lot of problems in 
uncertainty quantification, computational finance and computational physics 
and is encountered for example in the numerical treatment of path integrals, stochastic processes, random fields and
stochastic or parametric PDEs. While the problem of quadrature of functions in weighted Hilbert spaces with infinitely
many variables has recently found a lot of interest 
in the information based complexity community, see e.g., 
\cite{BG14,CDMR09,DG14a,DG14b,G10,GMR14,HMNR,HW,KSWW, NHMR,PW,PWW,SW98,WW96}, there is much less literature on
approximation. So far, the approximation of functions in weighted Hilbert spaces with infinitely many variables has been
studied for a properly weighted $L_2$-error norm in \cite{Wa12b}. (See also \cite{Wa12a,WW11a,WW11b}, where however a
norm in a special Hilbert space was employed such that the approximation problem indeed got easier than the integration
problem.) 
It has been noticed in \cite{PW10}, one may have two options for obtaining tractability in infinite-dimensional
approximation: either by using decaying weights or by using an increasing smoothness vector $\br$, see also Section
\ref{Sect:aniso} above. 

In the recent decades, various approaches and methods have been proposed and studied for
numerical solving of the parametric and stochastic elliptic PDE 
\begin{equation}\label{spde}
-\mathrm{div}_\bx(\sigma(\bx,\by)\nabla_\bx u(\bx,\by)) = f(\bx) \quad \bx \in D \quad \by \in \Omega,
\end{equation} 
with homogeneous boundary conditions $u(\bx,\by)=0$, $\bx\in \partial D$, $\by \in \Omega$, where $D \subset \R^m$,
$\Omega \subset \R^d$ and $d$ may be very large or infinity.
See to \cite{CD15, GWZ14, SG11} for surveys and bibliography on different aspects in approximation and numerical methods
for the problem \eqref{spde}.

The recent paper  \cite{DG15} has investigated the linear hyperbolic cross approximation in infinite dimensions of
functions from spaces of mixed Sobolev type smoothness and mixed Sobolev-analytic-type smoothness in the
infinite-variate case where specific summability properties of the smoothness are fulfilled. Such function spaces appear
for example for the solution of the equation \eqref{spde}. The optimality of this approximation is studied in
terms of
the $\varepsilon$-dimension of their unit balls for which tight upper and lower bounds are given. These results then are
applied to the linear approximation to the solution of the problem \eqref{spde}.
The obtained upper and lower bounds of the approximation error as well as of the associated
$\varepsilon$-dimension are completely independent of any parametric or stochastic dimension, and of the parameters
which define the smoothness properties of the parametric or stochastic part of the solution. 
In the following, as an example let us briefly mention one of the results from  \cite{DG15} on the hyperbolic cross
approximation in infinite dimensions in the norm of the space $\mathcal{G}:=H^\beta(\T^m)\otimes L_2{(\T^\infty)}$ of
functions from the space $\mathcal{H}:= H^\alpha(\T^m)\otimes K^{\br}(\T^\infty)$, and its optimality in terms of
$\varepsilon$-dimension
$n_\varepsilon(\mathcal{U},\mathcal{G})$. The space $H^\gamma(\T^m)$, $\gamma \ge 0$, equipped with a different
equivalent norm can be identified with the isotropic Sobolev space $W^\gamma_2(\T^m)$. The space of  
anisotropic infinite-dimensional mixed smoothness space $K^{\br}(\T^\infty)$ with 
$\br = (r_1,r_2,...) \in \R^\infty$ and $0 < r_1 \le r_2, \cdots \le r_j \cdots$, is an infinite-variate generalization
of the space $K^{\br}(\T^d)$  which equipped with  a different equivalent norm can be identified with the space
$\bW^{\br}_2(\T^d)$ (see \cite{DG15} for exact definitions of these spaces).
Denote by $\mathcal{U}$   the unit ball in $\mathcal{H}$. 
\begin{thm} \label{theorem[n_e[d-inf]}
If $(\alpha - \beta)/m < r_1$ and 
$\sum_{i=1}^\infty \frac{(3/2)^{1 - mr_i/(\alpha - \beta)}}{mr_i/(\alpha - \beta) - 1} < \infty$, we have for every $\varepsilon \in (0,1]$, 
\begin{equation} \label{ineq[n_e<]-periodic}
\lfloor \varepsilon^{- 1/(\alpha - \beta)} \rfloor^m - 1
\ \le \
n_\varepsilon(\mathcal{U},\mathcal{G})
\ \le \
C \, \varepsilon^{-m/(\alpha-\beta)}, 
\end{equation}
where $C$ is a constant depending on $\alpha,\beta,m,\br$ only. 
\end{thm}

The upper bound in \eqref{ineq[n_e<]-periodic} is
realized by a linear hyperbolic cross approximation in infinite dimensions corresponding to the infinite-variate
anisotropic mixed smoothness of the spaces $\mathcal{H}$ and  $\mathcal{G}$. Depending on the regularity of the
diffusions $\sigma(\bx,\by)$ and the right hand side $f(\bx)$ in the periodic equation \eqref{spde} with $D=\T^m$ and 
$\Omega = \T^\infty$, we may assume the solution $u(\bx,\by)$
belonging to $\mathcal{H}$ for some $\alpha$ and $\br$. Then we approximate $u(\bx,\by)$ in the ``energy'' norm of
$\mathcal{G}$ for some $0 \le \beta < \alpha$. Based on  the linear hyperbolic cross approximation in infinite dimension
and \eqref{ineq[n_e<]-periodic}, we can construct a linear method of rank $\le n$ which gives the convergence rate of
approximation to the solution of  \eqref{spde} as $n^{-(\alpha - \beta)/m}$. See \cite{DG15, Di15_2} for more results, on infinite dimensional approximation and applications in  parametric and stochastic PDEs, in particular, on non-periodic and mixed versions of \eqref{spde} and of Theorem \ref{theorem[n_e[d-inf]}.

%% file: appendix.tex
\section{Appendix}

 \subsection{General notation}
 
 Let us start with introducing some notations.
For $1\le p \le \infty$ we shall denote by $p'$ the duality exponent, that
is, the number (or $\infty$) such that $1/p + 1/p' = 1$. For a vector
$\mathbf 1\le\mathbf p \le \infty$
we denote $\bp ' = (p'_1,\dots,p'_d)$ and $1/\bp = (1/p_1,\dots,1/p_d)$.

For the sake of brevity we shall write $\int f d\mu$ instead of
$(2\pi)^{-d}\int_{\T^d} f(\mathbf x) d\mathbf x$, where $\T^d = [-\pi,\pi]^d$
and $\mu$ means the normalized Lebesgue measure on $\T^d$. For functions
$f,g:\T^d\to  \C$ we define the convolution\index{Convolution}
$$
    f*g(\bx):=(2\pi)^{-d}\int_{\T^d}f(\by)g(\bx-\by)d\by\,.
$$
In the case $\mathbf p = \mathbf 1p$
we shall write
the scalar $p$ instead of the vector $\mathbf p$. 
Let further denote $L_q(\T^d)$, $0<q\leq \infty$, the space of
all measurable functions $f:\T^d\rightarrow \C$ satisfying
$$
    \|f\|_p :=
    \Big(\,\int_{\T^d}|f(\bx)|^p\,\dint \mu\Big)^{1/p} < \infty
$$
with the usual modification in case $p=\infty$. In case $1\leq p\leq \infty$
the quantity $\|\cdot\|_p$ represents a norm. In case $0<p<1$ it is a
quasi-norm. The space $C(\T^d)$ is often used as a replacement for
$L_{\infty}(\T^d)$. It denotes the collection of all continuous periodic
functions equipped with the $L_{\infty}$-topology. 

\subsection{Inequalities} \label{App_ineq}We shall mention some well-known inequalities. 

{\bf 1.1. The H\"older inequality.}\index{Inequality!H\"older} Let $1\le p\le\infty$, $f_1\in L_{p}$,
$f_2\in L_{p'}$.
Then $f_1 f_2\in L_1$ and
\be\label{7.1}
\int |f_1 f_2| d\mu\le \|f_1\|_p\|f_2\|_{p'}.
\ee
 
{\bf 1.2.} As a consequence of the relation (\ref{7.1}) we obtain the
H\"older inequality for a vector $\mathbf 1\le \mathbf p \le \infty$
$$
\int |f_1 f_2| d\mu\le \|f_1\|_{\mathbf p}\|f_2\|_{\mathbf p'}.
$$

{\bf 1.3. The H\"older inequality for several functions.} Let
$1\le p_i\le \infty$, $i = 1,\dots,m$, $1/p_1 + \dots+ 1/p_m = 1$,
$f_i\in L_{p_i}$, $i=1,\dots,m.$
Then $f_1\dots f_m\in L_1$ and
\be\label{7.3}
\int|f_1\dots f_m|d\mu\le\|f_1\|_{p_1}\dots\|f_m\|_{p_m}.
\ee

{\bf 1.4. The monotonicity of $L_p$-norms.} Let $1\le q\le p \le \infty$,
then
\be\nonumber
\|f\|_q\le\|f\|_p
\ee
and for $\mathbf 1\le \mathbf q \le \mathbf p \le \infty$,
\be\nonumber
\|f\|_{\mathbf q}\le\|f\|_{\mathbf p}.
\ee

{\bf 1.5.} Let $1\le a < p < b\le \infty$,
$\theta= (1/p - 1/b)(1/a - 1/b)^{-1}$,
then
\be\nonumber
\|f\|_p\le\|f\|_a^{\theta}|f\|_b^{1-\theta} .
\ee

{\bf 1.6.} From the inequality (\ref{7.1}) we easily obtain the H\"older
inequality for sums :
$$
\sum_{i=1}^N |a_i b_i |\le\left(\sum_{i=1}^N|a_i|^p\right)^{1/p}
\left(\sum_{i=1}^N|b_i|^{p'}\right)^{1/p'},\qquad 1 \le p \le \infty.
$$
We remark that in this inequality one can take $N = \infty$.

{\bf 1.7. The Minkowski inequality.}\index{Inequality!Minkowski} Let $1\le p\le \infty$ , $f\in L_p$,
$i = 1,\dots,m$. Then
\be\nonumber
\left\|\sum_{i=1}^mf_i\right\|_p\le\sum_{i=1}^m\|f_i\|_p.
\ee

{\bf 1.8.} It is possible to deduce the generalized Minkowski
inequality from the Minkowski inequality. Let
$\mathbf 1\le \mathbf p \le \infty$, then
\be\nonumber
\left\|\int\varphi(\cdot,\mathbf y)d\mu(\mathbf y)\right\|_{\mathbf p}
\le\int\bigl\|\varphi(\cdot,\mathbf y)
\bigr\|_{\mathbf p}d\mu(\mathbf y).
\ee

{\bf 1.9.} Let $1\le q\le p\le\infty$ , then
\be\nonumber
\left(\int\left(\int\bigl|f(\mathbf x,\mathbf y)\bigr|^qd\mu(\mathbf y)\right)
^{p/q}d\mu(\mathbf x)\right)^{1/p}\le \left(\int\left(\int\bigl|f(\mathbf
x,\mathbf y)
\bigr|^pd\mu(\mathbf x)\right)^{q/p}d\mu(\mathbf y)\right)^{1/q}.
\ee
 
{\bf 1.10. The Young inequality.}\index{Inequality!Young} Let $p$, $q$ and $a$ be real numbers
satisfying the conditions
\be\nonumber
1\le p\le q\le \infty,\qquad 1 - 1/p + 1/q = 1/a.
\ee
Let $f\in L_p$ and $K\in L_a$ be $2\pi$-periodic
functions of a single
variable. Let us consider the convolution of these functions
$$
J(x) = (2\pi)^{-1}\int_{-\pi}^{\pi} K(x - y)f(y)dy = K * f .
$$
Then
\be\nonumber
\|J\|_q\le \|K\|_a\|f\|_p .
\ee

{\bf 1.11. The Young inequality for vector $\mathbf p$, $\mathbf q$,
$\mathbf a$.}\\ Let
$\mathbf 1\le\mathbf p\le\mathbf q\le \infty$,
$\mathbf 1 - 1/\mathbf p + 1/\mathbf q = 1/\mathbf a$
$$
J(\mathbf x) =\int K(\mathbf x -\mathbf y)f(\mathbf y)d\mu(\mathbf y) = K * f .
$$
Then
\be\nonumber
\|J\|_{\mathbf q}\le\|K\|_{\mathbf a}\|f\|_{\mathbf p} .
\ee

{\bf 1.12. The Abel inequality.}\index{Inequality!Abel} For nonnegative and non-increasing
$v_1 ,\dots,v_n$ we have
$$
\left|\sum_{i=1}^{n}u_iv_i\right|\le v_1
\max_k\left|\sum_{i=1}^{k}u_i\right| .
$$

This inequality easily follows from the following formula
$$
\sum_{i=1}^{n}u_iv_i=\sum_{\nu=1}^{n-1}(v_{\nu}-v_{\nu+1})
\sum_{i=1}^{\nu} u_i+v_n\sum_{i=1}^{n}u_i,
$$
which is called the Abel transformation.

 Along with spaces $L_p$ we shall use spaces $l_p$,
$1\le p\le \infty$, of
sequences $\mathbf z = \{z_k \}_{k=1}^{\infty}$ equipped with the norm
$$
\|\mathbf z\|_p =\|\mathbf z\|_{l_p}:=\left(\sum_{k=1}^{\infty}|z_k|^p
\right)^{1/p},\quad 1 \le p < \infty ,
$$
$$
\|\mathbf z\|_{\infty}:=\|\mathbf z\|_{l_{\infty}}=\sup_k |z_k|.
$$

The spaces $l_p$ are Banach spaces.

\subsection{Duality in $L_{p}$ spaces}\index{Duality}

{\bf 2.1.} Let $f\in L_{p}$, $g\in L_{p'}$. We denote
$$
\<f,g\> := (2\pi)^{-d}\int _{\pi_d}f(\mathbf x)
\overline {g(\mathbf x)} d\mathbf x = \int f\overline g d\mu ,
$$
where $\overline z$ is the complex conjugate number to a number $z$ .

\begin{thm}\label{T7.2.1} Let $1\le p \le \infty$ and $f\in L_{p}$ then
$$
\|f\|_p = \sup_{g\in L_{p'},\|q\|_{p'}\le1} |\<f,g\>|.
$$
\end{thm}

 \begin{rem}\label{R7.1} The statement analogous to Theorem \ref{T7.2.1} is
valid for
spaces $l_p$ :
$$
\|\mathbf z\|_{l_p}=\sup_{\|\mathbf w\|_{l_{p'}}\le1}|(\mathbf z,\mathbf w)|,
\qquad 1 \le p \le \infty .
$$
\end{rem}

{\bf 2.2.} Let $F$ be a complex linear normed space and $F^{*}$ be the
conjugate (dual) space to $F$ , that is elements of $F^{*}$ are linear
functionals $\varphi$ defined on $F$ with the norm
$$
\|\varphi\| = \sup_{f\in F: \|f\|\le1}\bigl|\varphi(f)\bigr| .
$$

Let $\varPhi = \{\varphi_k \}_{k=1}^n$ be a set of functionals from
$F^{*}$. Denote
$$
F_{\varPhi}= \bigl\{ f \in F :\varphi_k (f) = 0 ,\quad k = 1,\dots,n\bigr\}.
$$

\begin{thm}\label{T7.2.2} ({\bf The Nikol'skii duality theorem.}) Let
$\varPhi = \{\varphi_k\}_{k=1}^n$ be a fixed system of functionals
from $F^{*}$. Then for any $\varphi\in F^{*}$
$$
 \inf_{\{c_k\}_{k=1}^n}
\left \|\varphi-\sum_{k=1}^{n}c_k\varphi_k\right\|=\sup_{ f\in F_{\varPhi}:
\|f\|\le1}\bigl|\varphi(f)\bigr| .
$$
\end{thm}

\begin{thm}\label{T7.2.3} Let $\varphi,\varphi_1,\dots.\varphi_n\in L_p$,
$1\le p\le \infty$ , then
$$
\inf_{c_k; k=1,\dots,n}
\left \|\varphi-\sum_{k=1}^{n}c_k\varphi_k\right\|_p=\sup_{
\|g\|_{p'}\le1:(\varphi_k,g)=0, k=1,\dots,n}
\bigl|\<\varphi,g\>\bigr| .
$$
\end{thm}

\subsection {Fourier series of functions in $L_p$}
 
{\bf 3.1.} For a function $f\in L_1(\T^d)$ we define Fourier coefficients
$$
\hat f(\mathbf k) = (2\pi)^{-d}\int_{\T^d}f(\mathbf x) e^{-i(
\bk,\bx )}
d\mathbf x= \langle f,e^{i (\mathbf k,\mathbf \cdot)}\rangle\quad,\quad \bk \in \Z^d\,.
$$
\index{Parseval's identity}There are the well-known Parseval equality: for any $f\in L_2(\T^d)$
$$
\|f\|_2 = \left (\sum_{\mathbf k}\bigl|\hat f(\mathbf k)\bigr|^2\right)^{1/2},
$$
\index{Theorem!Riesz-Fischer}and the Riesz-Fischer theorem:  if
$\sum_{\mathbf k}|c_{\mathbf k}|^2 < \infty$, then
$$
f(\mathbf x) =\sum_{\mathbf k}c_{\mathbf k} e^{i(\mathbf k,\mathbf x)}\in L_2
\qquad\text{ and }\qquad \hat f(\mathbf k) = c_{\mathbf k} .
$$

In the space $L_p$, $1 < p < \infty$, the following statement holds.

\begin{thm}\label{T7.3.1} ({\bf The Hausdorff-Young theorem.})\index{Theorem!Hausdorff-Young} Let $1 < p\le 2$,
then for any $f\in L_p$,
$$
\left (\sum_{\mathbf k}\bigl|\hat f(\mathbf k)
\bigr|^{p'}\right)^{1/p'} \le\|f\|_p .
$$
If a sequence
$\{c_{\mathbf k}\}$ is such that $\sum_{\mathbf k}|c_{\mathbf k}|^p< \infty$,
then there exists a function $f\in L_{p'}$ for which $\hat f(\mathbf k) =
c_{\mathbf k}$ and
$$
\|f\|_{p'}\le\left (\sum_{\mathbf k}\bigl|\hat f(\mathbf k)
\bigr|^p\right)^{1/p}.
$$
\end{thm}

This theorem can be derived from the following interpolation
theorem, which is a special case of the general Riesz-Thorin
theorem.

Denote the norm of an operator $T$ acting from a Banach space $E$
to a Banach space $F$ by
$$
\|T\|_{E\to F}=\sup_{\|f\|_E\le1}\|Tf\|_F .
$$

\begin{thm}\label{T7.3.2} ({\bf The Riesz-Thorin theorem.})\index{Theorem!Riesz-Thorin}
Let $E_q$ be either $L_q$ or $l_q$
and $F_p$ be either $L_p$ or $l_p$ and for $1\le q_i,p_i\le \infty$,
$$
\|T\|_{E_{q_i}\to F_{p_i}}\le M_i,\qquad i = 1,2.
$$

Then for all $0 < \theta < 1$
$$
\|T\|_{E_{q}\to F_{p}}\le M_1^{\theta}M_2^{1-\theta},
$$
where
$$
1/q = \theta/q_1 + (1-\theta)/q_2,\qquad
1/p = \theta/p_1 + (1-\theta)/p_2 .
$$
\end{thm}

{\bf 3.2.} Let $[y]$ be the integer part of the real number $y$, that is, the
largest integer $[y]$ such that $[y]\le y$. For a vector
$\mathbf s = (s_1,\dots,s_d)$
with nonnegative integer coordinates we define the set $\rho(\mathbf s)$
of vectors $\mathbf k$ with integer coordinates:
$$
\rho(\mathbf s) =\bigl\{ \mathbf k = (k_1,\dots,k_d):[2^{s_j-1}]\le
|k_j| < 2^{s_j},\qquad j = 1,\dots,d \bigr\} .
$$

For $f\in L_1$ we denote
$$
\delta_{\mathbf s}(f,\mathbf x) :=\sum_{\mathbf k\in\rho(\mathbf s)}
\hat f(\mathbf k)e^{i(\mathbf k,\mathbf x)} .
$$

\begin{thm}\label{T7.3.3} ({\bf The Littlewood-Paley theorem.})\index{Theorem!Littlewood-Paley}
Let $1 < p < \infty$. There exist
positive numbers $C_1 (d,p)$ and $C_2 (d,p)$, which depend on
$d$ and $p$, such that for each function $f\in L_p$,
$$
C_1(d,p)\|f\|_p\le\left \|\left(\sum_{\mathbf s}\bigl|\delta_{\mathbf s}
(f,\mathbf x)\bigr|^2\right)^{1/2}\right\|_p\le C_2 (d,p)\|f\|_p.
$$
\end{thm}

\begin{cor}\label{C7.1} Let $G$ be a finite set of vectors $\mathbf s$
and
let the operator $S_G$ map a function $f\in L_p$, $p > 1$, to a function
$$
S_G (f) =\sum_{\mathbf s\in G}\delta_{\mathbf s} (f) .
$$
Then
$$
\|S_G\|_{L_p\to L_p}\le C(d,p) ,\qquad 1 < p < \infty .
$$
\end{cor}
For the sake of brevity we shall write $\|T\|_{L_q\to L_p}=\|T\|_{q\to p}$.

\begin{cor}\label{C7.2} Let $p^{*} = \min \{p,2\}$; then for $f\in L_p$
we have
$$
\|f\|_p\le C(d,p)\left(\sum_{\mathbf s}\bigl\|\delta_{\mathbf s}(f,\mathbf x)
\bigr\|_p^{p^{*}}\right)^{1/p^{*}},\qquad 1 < p <\infty .
$$
\end{cor}

\begin{thm}\label{T7.3.4} ({\bf The Marcinkiewicz multiplier \index{Theorem!Marcinkiewicz multiplier}
theorem.}) Suppose that 
$\lambda_0,\lambda_1,\dots$ are Marcinkiewicz
multipliers, that is, they satisfy the conditions
$$
|\lambda_n|\le M ,\qquad n = 0,\pm 1,\dots;\qquad
\sum_{l=\pm2^{\nu}}^{\pm(2^{\nu+1}-1)}|\lambda_l-\lambda_{l+1}|\le M,
\qquad \nu = 0,1,\dots,
$$
where $M$ is a number.

Then the operator $\Lambda$ which maps a function $f$ to a function
$$
\sum_k\lambda_k\hat f(k)e^{ikx},
$$
is bounded as operator from $L_p$ to $L_p$ for $1 < p < \infty$.
\end{thm}

\begin{thm}\label{T7.3.5} ({\bf The Hardy-Littlewood-Sobolev inequality.})\index{Inequality!Hardy-Littlewood-Sobolev} Let
$1 < q <
p < \infty$,
$$
\mu = 1 - 1/q + 1/p ,\qquad \|f\|_{L_q(\R)}= \left (\int_{-\infty}^{\infty}
|f(x)|^q\,dx \right)^{1/q}<\infty,\quad 
Jf(x) :=\int_{-\infty}^{\infty}f(y)|x - y|^{-\mu} dy .
$$
Then the inequality
$$
\|Jf\|_{L_p(\R)}\le C(q,p)\|f\|_{L_q(\R)}
$$
holds.
\end{thm}

\begin{cor}\label{C7.3} Let $1 < q < p <\infty$, $\beta = 1/q - 1/p$.
Then the
operator $A_{\beta}$ which maps a function $f\in L_q$ to a function
$$
\sum_{\mathbf k}\hat f(\mathbf k)\left (\prod_{j=1}^d
\max\bigl\{1,|k_j |\bigr\}\right)^{-\beta} e^{i(\mathbf k,\mathbf x)}
$$
is bounded as operator from $L_q$ to $L_p$.
\end{cor}

%% file: bibliography.tex
\pagebreak